\documentclass[12pt, a4paper]{amsart}

\calclayout

\usepackage[english]{babel}
\usepackage[utf8]{inputenc}
\usepackage{amsmath}
\usepackage{graphicx}
\usepackage[colorinlistoftodos]{todonotes}
\usepackage{amssymb,amscd,latexsym,epsfig,xy}
\usepackage[mathscr]{euscript}
\usepackage{amsthm}
\usepackage{tikz}
\usepackage{MnSymbol}
\usetikzlibrary{matrix,arrows,decorations.pathmorphing}
\usepackage{hyperref}
\usepackage{enumerate}
\usepackage{float}
\usepackage{tikz-cd}
\usepackage{relsize}
\usepackage{mathtools}
\usepackage{dsfont}
\usepackage{cleveref}
\usepackage{cite}

\crefformat{section}{\S#2#1#3} 
\crefformat{subsection}{\S#2#1#3}
\crefformat{subsubsection}{\S#2#1#3}

\newtheorem{defn}{Definition}[section]
\newtheorem{thm}[defn]{Theorem}
\newtheorem{lem}[defn]{Lemma}
\newtheorem{cor}[defn]{Corollary}
\newtheorem{conj}[defn]{Conjecture}
\newtheorem{prop}[defn]{Proposition}

\newcommand\A{\mathbb A}
\newcommand\C{\mathbb C}

\newcommand\HH{\mathbb H}

\newcommand\NN{\mathbb N}

\newcommand\PP{\mathbb P}
\newcommand\R{\mathbb R}
\newcommand\Q{\mathbb Q}

\newcommand\Z{\mathbb Z}

\newcommand\cA{\mathcal A}

\newcommand\cC{\mathcal C}
\newcommand\cE{\mathcal E}
\newcommand\cF{\mathcal F}

\newcommand\cH{\mathcal H}

\newcommand\cO{\mathcal O}

\newcommand\fd{\mathfrak d}
\newcommand\fD{\mathfrak D}
\newcommand\fg{\mathfrak g}
\newcommand\fm{\mathfrak m}
\newcommand\fM{\mathfrak M}

\newcommand\fp{\mathfrak p}
\newcommand\Arg{\operatorname{Arg}}
\newcommand\Aut{\operatorname{Aut}}
\newcommand\Coh{\operatorname{Coh}}

\newcommand\Ext{\operatorname{Ext}}

\newcommand\Hom{\operatorname{Hom}}

\newcommand\Rea{\operatorname{Re}\,}

\newcommand\Ima{\operatorname{Im}\,}

\newcommand\ch{\operatorname{ch}}

\newcommand\cl{\operatorname{cl}}

\newcommand\td{\operatorname{td}}

\newcommand\Stab{\operatorname{Stab}}
\newcommand\D{\operatorname{D}}

\newcommand\Sing{\operatorname{Sing}}
\newcommand\Init{\operatorname{Init}}
\newcommand\RHom{\operatorname{RHom}}

\newcommand\st{\mathrm{st}}

\newcommand\iin{\mathrm{in}}
\newcommand\oout{\mathrm{out}}
\newcommand\Gal{\operatorname{Gal}}

\title[Scattering diagrams, stability conditions, and coherent sheaves on 
$\PP^2$]{Scattering diagrams, stability conditions, and coherent sheaves on 
$\PP^2$\footnote{Mathematics Subject Classification: 14N35}}

\author{Pierrick Bousseau}

\date{}

\begin{document}

\begin{abstract}
We show that a purely algebraic structure, a two-dimensional scattering diagram, describes a large part of the wall-crossing behavior of moduli spaces of Bridgeland semistable objects in the derived category of coherent sheaves on $\PP^2$. This gives a new algorithm computing the Hodge numbers of the intersection cohomology of the classical moduli spaces of Gieseker semistable sheaves
on $\PP^2$, or equivalently the refined Donaldson-Thomas invariants for compactly supported sheaves on local 
$\PP^2$.

As applications, we prove that the intersection cohomology of moduli spaces of Gieseker semistable sheaves on $\PP^2$ is Hodge-Tate, and we give the first non-trivial numerical checks of the general $\chi$-independence conjecture for refined Donaldson-Thomas invariants of one-dimensional sheaves on local $\PP^2$.
\end{abstract}

\maketitle

\setcounter{tocdepth}{1}
\setcounter{section}{-1}

\tableofcontents

\thispagestyle{empty}

\section{Introduction}

\subsection{Overview}

The concept of scattering diagram comes from the work of 
Kontsevich-Soibelman \cite{MR2181810}
and Gross-Siebert \cite{MR2846484} in mirror symmetry. In this context, a scattering diagram is an algebraic structure which is supposed to encode the behavior of holomorphic discs with boundary
on torus fibers of the Strominger-Yau-Zaslow fibration \cite{MR1429831}. Essentially the same algebraic structure appears in an a priori completely different setting: the wall-crossing behavior of Donaldson-Thomas counts of semistable objects in a Calabi-Yau triangulated category of dimension $3$, upon variation of the stability condition \cite{kontsevich2008stability, MR2951762, MR3330788}.
A precise connection between scattering diagrams and spaces of stability conditions
for quivers with potential has been established by Bridgeland \cite{MR3710055}, recently followed by Cheung-Mandel \cite{cheung2019donaldson}, and has been used by the author and Arg\"uz \cite{AB21} to prove new results on the structure of quiver Donaldson--Thomas invariants such as the flow tree formula conjectured by Alexandrov-Pioline \cite{AP19} .

The aim of the present paper is to explore further this connection between stability conditions and scattering diagrams in a specific geometric example.
We consider the complex projective plane 
$\PP^2$ and the space of Bridgeland stability conditions
$\Stab(\PP^2)$ \cite{MR2373143} on the bounded derived category $\D^b(\PP^2)$ of coherent sheaves on $\PP^2$.
The space $\Stab (\PP^2)$ is a complex manifold of complex dimension $3$. We focus on a 
particular subset $U$ of $\Stab (\PP^2)$ of complex dimension $1$. Using as input the intersection cohomology of the moduli spaces of semistable objects in $\D^b(\PP^2)$ at various points of $U$, we 
construct a scattering diagram $\fD^{\PP^2}_{u,v}$ on $U$. On the other hand, we give a purely algorithmic definition of another scattering diagram 
$S(\fD^{\iin}_{u,v})$ on $U$.

Our main result is that these two scattering diagrams coincide.

\begin{thm}[=Theorem \ref{thm_main_precise}] \label{main_thm}
We have 
\[ \fD^{\PP^2}_{u,v} = S(\fD^{\iin}_{u,v}) \,.\]
\end{thm}

We stress that the left-hand side $\fD_{u,v}^{\PP^2}$ encodes some complicated geometry of the moduli spaces 
of semistable objects in $\D^b(\PP^2)$, whereas the right-hand side $S(\fD^{\iin})$ is completely algorithmic and can be easily implemented on a computer.

As the notion of Gieseker stability can be recovered as a limiting case of Bridgeland stability condition, Theorem \ref{main_thm}
gives a new algorithm to compute  intersection cohomology of the classical \cite{huybrechts2010geometry} moduli spaces of Gieseker semistable sheaves
on $\PP^2$. Using this algorithm, we prove that the intersection cohomology of the moduli spaces of Gieseker semistable sheaves on $\PP^2$ is concentrated in Hodge bidegrees $(p,p)$ (Theorem \ref{thm_hodge_intro}), and we make the first non-trivial numerical checks of the general 
$\chi$-independence conjecture for intersection cohomology of moduli spaces of Gieseker semistable one-dimensional sheaves (Conjecture \ref{conj_chi_indep}-Theorem \ref{thm_chi_indep_test}). 

We have phrased our main result in terms of the derived category of coherent sheaves on 
$\PP^2$. One motivation for this choice is that moduli spaces of Gieseker semistable sheaves on $\PP^2$ are classical objects of algebraic geometry and the study of their intersection cohomology is an interesting topic on its own. Nevertheless, an alternative 
more modern point of view would be to consider Theorem \ref{main_thm} as a statement about $K_{\PP^2}$, the non-compact Calabi-Yau 3-fold total space of the canonical line bundle of $\PP^2$, also known as local $\PP^2$. Indeed, $U$ can also be viewed as a subspace of the space of Bridgeland stability conditions of the category $\D^b_0(K_{\PP^2})$ of coherent sheaves on $K_{\PP^2}$
set-theoretically supported on the zero-section, and the intersection cohomology of moduli spaces of semistable objects in $\D^b(\PP^2)$ involved in the definition of $\fD_{u,v}^{\PP^2}$ coincides with the refined Donaldson-Thomas invariants of $K_{\PP^2}$.
In particular, the combinatorial understanding of $\fD_{u,v}^{\PP^2}$
provided by Theorem \ref{main_thm} is an expression of the Kontsevich-Soibelman  wall-crossing formula \cite{kontsevich2008stability} for Donaldson-Thomas invariants of Calabi-Yau 3-folds.

In the follow-up paper \cite{bousseau2019takahashi}, we will combine Theorem \ref{main_thm} with the main result of 
Gr\"afnitz \cite{gabele2019tropical} to prove 
N.\ Takahashi's conjecture
\cite{MR1844627, choi2018log} on genus-$0$ Gromov-Witten theory with maximal tangency of the pair $(\PP^2,E)$, where $E$ is a smooth cubic curve in $\PP^2$.
We 
will also establish in \cite{bousseau2019takahashi} a new
sheaves/Gromov--Witten correspondence, relating Betti numbers of moduli spaces of one-dimensional Gieseker semistable sheaves on $\PP^2$, or equivalently refined genus-$0$ Gopakumar--Vafa invariants of local $\PP^2$,
with higher-genus maximal contact Gromov--Witten theory of $(\PP^2,E)$.
In combination with the work \cite{bousseau2020hae} on the Gromov-Witten side,
this will give the first mathematical proof of a non-trivial example of the general structure properties (finite generation, quasimodularity, holomorphic anomaly equation)
expected from string theory for the refined topological string on Calabi-Yau 3-folds
\cite{MR3024275}.

The rest of the introduction is organized as follows. 
In \cref{section_intro_description} we give a more detailed description of the objects 
$\fD_{u,v}^{\PP^2}$ and $S(\fD^\iin_{u,v})$
involved in the statement of Theorem
\ref{main_thm}. 
In \cref{section_intro_proof} we briefly describe the technical tools used in the proof of Theorem \ref{main_thm}. 
In \cref{section_intro_gieseker} 
we state precisely our results on moduli spaces of Gieseker semistable sheaves on 
$\PP^2$. 
Finally, we discuss connections with related works in \cref{section_intro_relation}.

\subsection{Description of the scattering diagrams $S(\fD^\iin_{u,v})$ and $\fD^{\PP^2}_{u,v}$}
\label{section_intro_description}

\subsubsection{Scattering diagrams}
Both $S(\fD^\iin_{u,v})$ and 
$\fD^{\PP^2}_{u,v}$ are scattering diagrams on $U$ for the Lie algebra $\fg_{u,v}$. Here, $U$ is the open subset of $\R^2$ defined by 
\[ U \coloneq \{ (x,y) \in \R^2\,| x^2+2y>0\} \,,\]
and $\fg_{u,v}$ the $\Q(u^{\pm \frac{1}{2}}, v^{\pm \frac{1}{2}})$-Lie algebra 
\[\fg_{u,v}\coloneq \bigoplus_{m \in \Z^2} 
\Q(u^{\pm \frac{1}{2}}, v^{\pm \frac{1}{2}}) z^m\]
with Lie bracket given by  
\[[z^m, z^{m'}] \coloneq
(-1)^{\langle m,m'\rangle} ((uv)^{\frac{\langle m,m' \rangle}{2}}
-(uv)^{-\frac{\langle m,m' \rangle}{2}})
 z^{m+m'}\,,\]
where \[\langle -,- \rangle \colon 
\Z^2 \times \Z^2 \rightarrow \Z \]
\[ \langle (a,b), (a',b') \rangle \coloneq 
3 (a'b-ab') \,.\]
For the purposes of the present paper, a scattering diagram $\fD$ on $U$ for 
$\fg_{u,v}$ is a collection of rays $\fd=(|\fd|,H_\fd)$, where:
\begin{itemize}
\item[(i)] $H_\fd \in \Q(u^{\pm \frac{1}{2}},
v^{\pm \frac{1}{2}}) z^{m_\fd}$ for some 
$m_\fd \in \Z^2$ called the class of the ray 
$\fd$.
\item[(ii)] $|\fd|$ is an oriented line segment or half-line contained in $U$, of direction $-m_\fd$, and called the support of the ray 
$\fd$.
\end{itemize}

To every ray $\fd=(|\fd|,H_\fd)$ of a scattering diagram, we attach the group element
\[\Phi_{\fd} \coloneqq \exp(H_\fd) \]
in the group $G := \exp(\fg_{u,v})$ with Lie algebra $\fg_{u,v}$
defined by the Baker-Campbell-Hausdorff formula.
In fact, in order to make sense of $G$ and of the exponential map
$\exp \colon \fg_{u,v} \rightarrow G$, one needs to work with various completions of $\fg_{u,v}$. In this introduction, we ignore this issue and we refer to  \cref{section_scattering_diagram} for details.

A scattering diagram $\fD$ on $U$ for $\fg_{u,v}$ is consistent if for every $\sigma \in U$ the product of the group elements $\Phi_\fd^{\epsilon_\fd(\sigma)}$, taken over all the rays $\fd$ of $\fD$ passing through 
$\sigma$ and in the anticlockwise order around $\sigma$, is the identity. Here, we set $\epsilon_\fd(\sigma)=\pm 1$ depending if the orientation of $\fd$ points towards $\sigma$ or not.

An elementary but important fact going back to Kontsevich-Soibelman \cite{MR2181810} gives a systematic way to construct consistent scattering diagrams.
Let us start with a scattering diagram 
$\fD$. Then $\fD$ is not necessarily consistent: there can exist a point $\sigma \in U$ such that the product of group elements $\Phi_\fd^{\epsilon_\fd(\sigma)}$ around $\sigma$ is not the identity. 
Then, Kontsevich-Soibelman \cite{MR2181810} proved that there is a unique way to add rays starting at $\sigma$ in order to form a new scattering diagram such that the composition of the group elements $\Phi_\fd^{\epsilon_\fd(\sigma)}$ around $\sigma$ is now equal to the identity. The rays that need to be added are completely determined by the consistency condition and so by the Lie bracket of the Lie algebra $\fg_{u,v}$.
The new scattering diagram is now consistent around $\sigma$, but due to the newly added rays, there are now maybe new points $\sigma' \in U$ where consistency fails, and the construction needs to be iterated by successive additions of new rays. Dealing with the convergence issues of this potentially infinite process (as done carefully in \cref{section_scattering_diagram}), we end up with a consistent scattering diagram. 
In other words, starting with any scattering diagram $\fD$, there is a canonical way to produce a consistent one, that we denote by $S(\fD)$. Each time rays of 
$\fD$ intersect, new rays are added to guarantee the consistency and the process is iterated. This explains the scattering terminology: when rays meet, they `scatter' and produce new rays in a completely algorithmic way.

\subsubsection{The scattering diagram $S(\fD^\iin_{u,v})$}
The scattering diagram $S(\fD^\iin_{u,v})$ is obtained by the consistent completion
$\fD \mapsto S(\fD)$ described above for a specific choice of initial scattering diagram $\fD^\iin_{u,v}$ which can be explicitly described. We first remark that the boundary of $U$ is the parabola in 
$\R^2$ of equation $x^2+2y=0$. The support of the rays of $\fD^{\iin}_{u,v}$ are the tangent lines to this parabola at the points
$s_n \coloneq (n,-\frac{n^2}{2})$. More precisely, we denote by $|\fd_n^+|$ the oriented tangent half-line of direction $-m_n^+$, where $m_n^+=(-1,n)$, 
and by $|\fd_n^-|$ the oriented tangent half-line of direction
$-m_n^-$, where $m_n^-=(1,-n)$.
Then the rays of $\fD^\iin_{u,v}$ are given by $\fd_{n,\ell}^+=(|\fd_n^+|, H_{n,\ell}^+)$
and  $\fd_{n,\ell}^-=(|\fd_n^-|, H_{n,\ell}^-)$, $n\in \Z$,
$\ell \geqslant 1$, where 
\[H_{n,\ell}^+ \coloneq - \frac{1}{\ell} \frac{1}{(uv)^{\frac{\ell}{2}}
-(uv)^{-\frac{\ell}{2}}} z^{\ell m_n^+} \in \Q(u^{\pm\frac{1}{2}}, v^{\pm\frac{1}{2}})z^{\ell m_n^+}\,,\]
and 
\[H_{n,\ell}^- \coloneq - \frac{1}{\ell} \frac{1}{(uv)^{\frac{\ell}{2}}
-(uv)^{-\frac{\ell}{2}}} z^{\ell m_n^-}\in \Q(u^{\pm\frac{1}{2}}, v^{\pm\frac{1}{2}})z^{\ell m_n^-}\,.\]
We refer to Figure \ref{fig:initial_scattering} for a pictorial representation of the support of $\fD^\iin_{u,v}$, and to Figure 
\ref{fig:scattering} for a pictorial representation of the support of the first steps of $S(\fD^\iin_{u,v})$.

\begin{figure}[h!]
\centering
\includegraphics[scale=0.8]{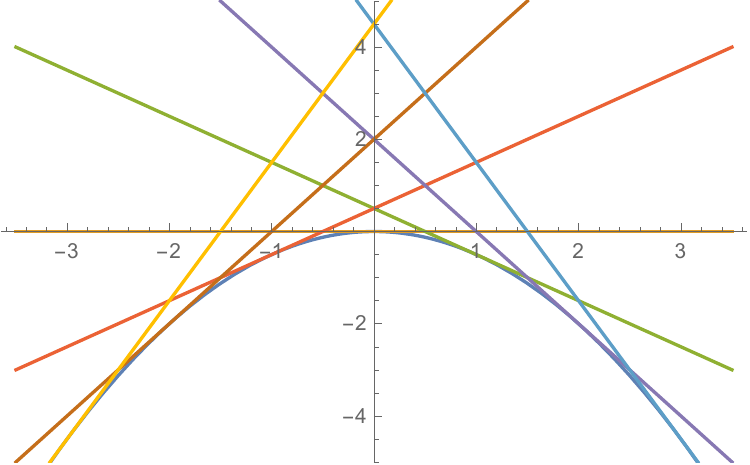}
\caption{The scattering diagram $\fD^\iin_{u,v}$}
\label{fig:initial_scattering}
\end{figure}

\subsubsection{The scattering diagram $\fD^{\PP^2}_{u,v}$}

The scattering diagram $\fD^{\PP^2}_{u,v}$ is constructed in terms of the moduli spaces of Bridgeland semistable objects in the bounded derived category $\D^b(\PP^2)$ of coherent sheaves on $\PP^2$.
We denote $\Gamma \coloneq \Z^3=K_0(\PP^2)$ and $\gamma =(r,d,\chi) \in \Gamma$, where $r$ stands for the rank, $d$ for the degree, and $\chi$ for the Euler characteristic of an object in $\D^b(\PP^2)$.

Recall from \cite{MR2373143} that every
Bridgeland stability condition $\sigma$ on
$\D^b(\PP^2)$ comes with the data of an additive map 
\[Z^\sigma \colon \Gamma \rightarrow \C\,,\]
\[ \gamma \mapsto Z_\gamma^\sigma \,,\]
called the central charge. The notion of stability specified by $\sigma$ is determined by the relative phases of the central charges $Z_\gamma^\sigma$.

Bridgeland
stability conditions for polarized surfaces, such as $\PP^2$, have been well studied \cite{MR2376815, MR3010070, MR2852118}. 
In particular, it is known how to construct 
an explicit half-space 
$\mathbb{H}\coloneq \{ (s,t) \in \R^2 \,| 
t>0\}$ of stability condition on $\D^b(\PP^2)$. The central charge for the stability $(s,t) \in \HH$ is given by 
\[  Z_\gamma^{(s,t)}= -\frac{1}{2}(s^2-t^2)r+ds+r+\frac{3}{2}d-\chi + i(d-sr)t\,.\]

The main idea to construct scattering diagrams from stability conditions is to consider loci of stability conditions $\sigma$ indexed by $\gamma \in \Gamma$ and defined by the condition that the central charge $Z_\gamma^\sigma$ remains of constant phase. In terms of the $(s,t)$ coordinates on $\mathbb{H}$, these loci are parabola in 
$\mathbb{H}$ and not straight lines.
Our main remark is that the map
\[ \mathbb{H} \rightarrow U\]
\[(s,t) \mapsto  (x,y)=(s,-\frac{1}{2}(s^2-t^2))\]
is a bijection, such that, for every
$\gamma =(r,d,\chi) \in \Gamma$, the locus 
$\Rea Z_\gamma^\sigma=0$
has equation 
\[ ry+dx+r+\frac{3}{2}d -\chi=0\,,\]
in terms of $(x,y) \in U$, and so is a straight line in $U$.
From now on, we use this identification 
$\mathbb{H} \simeq U$ to view $U$ as a space of stability conditions on $\D^b(\PP^2)$. 
In other words, we obtained $U$ from 
$\mathbb{H}$ by defining on $\mathbb{H}$ an integral affine structure such that the functions $\sigma \mapsto \Rea Z_\gamma^\sigma$ become integral affine coordinates.
The elementary change of variables
$(s,t) \mapsto (x,y)$
gives a new perspective on the standard upper half-plane $\mathbb{H}$ of stability conditions and is the key observation which makes the appearance of a scattering diagram possible.

For every $\gamma \in \Gamma$ and $\sigma \in U$, we have a projective variety $M_\gamma^\sigma$ parametrizing S-equivalence classes of $\sigma$-semistable objects of class $\gamma$ in $\D^b(\PP^2)$. We refer to \cref{section_moduli_invariants}-\cref{section_intersection_invariants} for details. The projective varieties 
$M_\gamma^\sigma$ are in general singular, due to the existence of strictly semistable objects. Nevertheless, the intersection cohomology groups $IH^{k}(M_\gamma^\sigma, \Q)$ behave as well as the cohomology groups of a smooth projective variety \cite{MR572580,MR696691,MR751966}, and in particular  
carry naturally a pure Hodge structure of weight $k$ \cite{MR1047415}. 
We denote by $Ih^{p,q}(M_\gamma^\sigma)$ the corresponding Hodge numbers, which can be organized into signed symmetrized intersection  
Hodge polynomial 
\[Ih_\gamma^\sigma(u^{\frac{1}{2}},v^{\frac{1}{2}}) \coloneq 
(-(uv)^{\frac{1}{2}})^{- \dim M_\gamma^\sigma}
\sum_{p,q=0}^{\dim M_\gamma^\sigma}
(-1)^{p+q}
Ih^{p,q}(M_\gamma^\sigma) u^p v^q \in \Z[u^{\pm
\frac{1}{2}},v^{\pm \frac{1}{2}}] \,.\]
To describe the scattering diagram $\fD^{\PP^2}_{u,v}$, it is actually more convenient to reorganize these Hodge polynomials and to define 
\[ \overline{\Omega}_\gamma^\sigma(u^{\frac{1}{2}}, v^{\frac{1}{2}}) :=-\sum_{\substack{\gamma' \in \Gamma\\ \gamma=\ell \gamma'}}
\frac{1}{\ell}
\frac{Ih_{\gamma'}^\sigma(u^{\frac{\ell}{2}}, v^{\frac{\ell}{2}})}{
(uv)^{\frac{\ell}{2}}-(uv)^{-\frac{\ell}{2}}}\,. \]
This formula might seem a bit unmotivated from the point of view of the classical geometry of moduli spaces of sheaves on $\PP^2$. However, it is the familiar way to package Donaldson-Thomas invariants in the presence of strictly semistable objects \cite{kontsevich2008stability, MR2951762}. Intuitively, the denominator in the formula comes from the fact that every stable object has a group of automorphisms equal to $\C^{*}$.

The scattering diagram $\fD^{\PP^2}_{u,v}$
is then defined as follows.
For every $\gamma \in \Gamma$, we consider the set $R_\gamma$ of points $\sigma \in U$
such that $\Rea Z_\gamma^\sigma=0$ and such that 
$\overline{\Omega}_\gamma^\sigma(u^{\frac{1}{2}}, v^{\frac{1}{2}})
\neq 0$. It happens that $R_\gamma$ is a half-line, contained in the straight line of equation $\Rea Z_\gamma^\sigma =0$. Denoting 
$m_\gamma \coloneq (r,-d) \in \Z^2$,
we attach to every point $\sigma$ of $R_\gamma$ the generating series
\[H_{\gamma}^\sigma \coloneq
 \overline{\Omega}_\gamma^\sigma(u^{\frac{1}{2}}, v^{\frac{1}{2}})
z^{m_\gamma} \,\in \Q(u^{\pm\frac{1}{2}},
v^{\pm\frac{1}{2}}) z^{m_\gamma}\,.\] 
As a function of $\sigma \in R_\gamma$, the moduli spaces $M_\gamma^\sigma$, and so the Hodge numbers 
$Ih^{p,q}(M_\gamma^\sigma)$ and the generating series $H_\gamma^\sigma$ are locally constant away from points where they jump in a discontinuous way. At such points, where $R_\gamma$ crosses walls in the space of stability conditions, the notion of semistability for objects of class $\gamma$ changes. These points of wall-crossing give a subdivision of $R_\gamma$ into line segments $R_{\gamma,j}$. We attach
to each line segment $R_{\gamma,j}$ the corresponding generating series $H_{\gamma,j} \coloneq H_\gamma^\sigma$ which is now independent of $\sigma
\in R_{\gamma,j}$ by construction.
By definition, $\fD^{\PP^2}_{u,v}$ is the scattering diagram on $U$ for 
$\fg_{u,v}$ whose rays are the  
$(R_{\gamma,j}, H_{\gamma,j})$ for all $\gamma \in \Gamma$ and for all $j$.

\subsubsection{Terminology}
It is worth pointing that we use the terminology `scattering diagram' in a slightly extended sense: $\fD^{\PP^2}_{u,v}$
is really a structure in the sense of 
\cite{MR2846484} or a wall-crossing structure in the sense of \cite{MR3330788}. In other words, it is an infinite collection of local scattering diagrams in the sense of
\cite{MR2846484}. Our terminology choices come from trying to avoid the potential confusion between the usages of the word `wall' in the mirror symmetry context and in the stability conditions context. 

One should also remark that $\fD^{\PP^2}_{u,v}$ is really a quantum scattering diagram, in the sense that the generating series $H_\gamma^\sigma$ are elements of the quantum torus Lie algebra. In the classical limit, generating series $H_\gamma^\sigma$ reduce to generating series of Euler characteristics of the intersection cohomology of moduli spaces of semistable objects.

\begin{figure}[h!]
\centering
\resizebox{0.90\textwidth}{0.90\textheight}{
\rotatebox{90}{
\begin{tikzpicture}[xscale=0.6,yscale=2.7,font=\fontsize{6}{6},define rgb/.code={\definecolor{mycolor}{RGB}{#1}}, rgb color/.style={define rgb={#1},mycolor}]
\draw[->,rgb color={255,132,0}] (-30.0,-4.00) -- (-51.0,-6.00);
\draw[->,rgb color={255,132,0}] (-30.0,-4.00) -- (-48.0,-6.00);
\draw[->,rgb color={255,132,0}] (-30.0,-4.00) -- (-48.0,-6.00);
\draw[->,rgb color={255,132,0}] (-30.0,-4.00) -- (-48.0,-6.00);
\draw[->,rgb color={255,132,0}] (-30.0,-4.00) -- (-42.0,-6.00);
\draw[->,rgb color={255,132,0}] (-30.0,-4.00) -- (-36.0,-6.00);
\draw[->,rgb color={255,132,0}] (-30.0,-4.00) -- (7.00,-4.00) node[right]{};
\draw[->,rgb color={255,132,0}] (-30.0,-4.00) -- (7.00,-4.00) node[right]{};
\draw[->,rgb color={255,132,0}] (-30.0,-4.00) -- (7.00,-4.00) node[right]{};
\draw[->,rgb color={255,132,0}] (-30.0,-4.00) -- (7.00,-4.00) node[right]{};
\draw[->,rgb color={255,132,0}] (-30.0,-4.00) -- (7.00,-4.00) node[right]{};
\draw[->,rgb color={255,132,0}] (-30.0,-4.00) -- (7.00,-4.00) node[right]{};
\draw[->,rgb color={255,132,0}] (-27.0,-4.00) -- (7.00,-4.00) node[right]{};
\draw[->,rgb color={255,132,0}] (-27.0,-4.00) -- (7.00,-4.00) node[right]{};
\draw[->,rgb color={255,132,0}] (-30.0,-4.00) -- (7.00,-4.00) node[right]{};
\draw[->,rgb color={255,132,0}] (-30.0,-4.00) -- (7.00,-1.94);
\draw[->,rgb color={255,132,0}] (-30.0,-4.00) -- (7.00,-1.94);
\draw[->,rgb color={255,132,0}] (-30.0,-4.00) -- (7.00,-1.94);
\draw[->,rgb color={255,132,0}] (-30.0,-4.00) -- (7.00,-2.24);
\draw[->,rgb color={255,132,0}] (-30.0,-4.00) -- (7.00,-2.46);
\draw[->,rgb color={255,132,0}] (-30.0,-4.00) -- (7.00,-1.76);
\draw[->,rgb color={255,132,0}] (-30.0,5.00) -- (-51.0,7.00);
\draw[->,rgb color={255,132,0}] (-30.0,5.00) -- (-48.0,7.00);
\draw[->,rgb color={255,132,0}] (-30.0,5.00) -- (-48.0,7.00);
\draw[->,rgb color={255,132,0}] (-30.0,5.00) -- (-48.0,7.00);
\draw[->,rgb color={255,132,0}] (-30.0,5.00) -- (-42.0,7.00);
\draw[->,rgb color={255,132,0}] (-30.0,5.00) -- (-36.0,7.00);
\draw[->,rgb color={255,132,0}] (-30.0,5.00) -- (7.00,5.00) node[right]{};
\draw[->,rgb color={255,132,0}] (-30.0,5.00) -- (7.00,5.00) node[right]{};
\draw[->,rgb color={255,132,0}] (-30.0,5.00) -- (7.00,5.00) node[right]{};
\draw[->,rgb color={255,132,0}] (-30.0,5.00) -- (7.00,5.00) node[right]{};
\draw[->,rgb color={255,132,0}] (-30.0,5.00) -- (7.00,5.00) node[right]{};
\draw[->,rgb color={255,132,0}] (-30.0,5.00) -- (7.00,5.00) node[right]{};
\draw[->,rgb color={255,132,0}] (-27.0,5.00) -- (7.00,5.00) node[right]{};
\draw[->,rgb color={255,132,0}] (-27.0,5.00) -- (7.00,5.00) node[right]{};
\draw[->,rgb color={255,132,0}] (-30.0,5.00) -- (7.00,5.00) node[right]{};
\draw[->,rgb color={255,132,0}] (-30.0,5.00) -- (7.00,2.94);
\draw[->,rgb color={255,132,0}] (-30.0,5.00) -- (7.00,2.94);
\draw[->,rgb color={255,132,0}] (-30.0,5.00) -- (7.00,2.94);
\draw[->,rgb color={255,132,0}] (-30.0,5.00) -- (7.00,3.24);
\draw[->,rgb color={255,132,0}] (-30.0,5.00) -- (7.00,3.46);
\draw[->,rgb color={255,132,0}] (-30.0,5.00) -- (7.00,2.76);
\draw[->,rgb color={255,132,0}] (-27.0,-4.00) -- (-39.0,-6.00);
\draw[->,rgb color={255,132,0}] (-27.0,-4.00) -- (7.00,-2.38);
\draw[->,rgb color={255,132,0}] (-27.0,5.00) -- (-39.0,7.00);
\draw[->,rgb color={255,132,0}] (-27.0,5.00) -- (7.00,3.38);
\draw[->,rgb color={255,132,0}] (-22.5,-3.75) -- (7.00,-3.75) node[right]{};
\draw[->,rgb color={255,132,0}] (-24.8,-3.75) -- (7.00,-3.75) node[right]{};
\draw[->,rgb color={255,107,0}] (-22.5,4.75) -- (7.00,4.75) node[right]{};
\draw[->,rgb color={255,107,0}] (-24.8,4.75) -- (7.00,4.75) node[right]{};
\draw[->,rgb color={255,132,0}] (-24.0,-3.67) -- (7.00,-3.67) node[right]{};
\draw[->,rgb color={255,132,0}] (-24.0,-3.67) -- (7.00,-3.67) node[right]{};
\draw[->,rgb color={255,0,0}] (-24.0,-3.50) -- (7.00,-0.917);
\draw[->,rgb color={255,0,0}] (-24.0,4.50) -- (7.00,1.92);
\draw[->,rgb color={255,107,0}] (-24.0,4.67) -- (7.00,4.67) node[right]{};
\draw[->,rgb color={255,132,0}] (-22.5,-3.50) -- (-30.0,-6.00);
\draw[->,rgb color={255,132,0}] (-22.5,-3.50) -- (7.00,-3.50) node[right]{};
\draw[->,rgb color={255,132,0}] (-22.5,-3.50) -- (7.00,-3.50) node[right]{};
\draw[->,rgb color={255,132,0}] (-21.0,-3.50) -- (7.00,-3.50) node[right]{};
\draw[->,rgb color={255,132,0}] (-18.0,-3.50) -- (7.00,-3.50) node[right]{};
\draw[->,rgb color={255,132,0}] (-22.5,-3.50) -- (7.00,-3.50) node[right]{};
\draw[->,rgb color={255,132,0}] (-22.5,-3.50) -- (7.00,-2.10);
\draw[->,rgb color={255,107,0}] (-22.5,4.50) -- (-30.0,7.00);
\draw[->,rgb color={255,107,0}] (-22.5,4.50) -- (7.00,4.50) node[right]{};
\draw[->,rgb color={255,107,0}] (-22.5,4.50) -- (7.00,4.50) node[right]{};
\draw[->,rgb color={255,107,0}] (-21.0,4.50) -- (7.00,4.50) node[right]{};
\draw[->,rgb color={255,107,0}] (-18.0,4.50) -- (7.00,4.50) node[right]{};
\draw[->,rgb color={255,107,0}] (-22.5,4.50) -- (7.00,4.50) node[right]{};
\draw[->,rgb color={255,107,0}] (-22.5,4.50) -- (7.00,3.10);
\draw[->,rgb color={255,132,0}] (-20.0,-3.33) -- (7.00,-3.33) node[right]{};
\draw[->,rgb color={255,107,0}] (-20.0,4.33) -- (7.00,4.33) node[right]{};
\draw[->,rgb color={255,107,0}] (-20.0,4.33) -- (7.00,4.33) node[right]{};
\draw[->,rgb color={255,132,0}] (-16.5,-3.25) -- (7.00,-3.25) node[right]{};
\draw[->,rgb color={255,132,0}] (-18.8,-3.25) -- (7.00,-3.25) node[right]{};
\draw[->,rgb color={255,107,0}] (-16.5,4.25) -- (7.00,4.25) node[right]{};
\draw[->,rgb color={255,107,0}] (-18.8,4.25) -- (7.00,4.25) node[right]{};
\draw[->,rgb color={255,107,0}] (-18.0,-3.00) -- (-40.5,-6.00);
\draw[->,rgb color={255,107,0}] (-18.0,-3.00) -- (-36.0,-6.00);
\draw[->,rgb color={255,107,0}] (-18.0,-3.00) -- (-36.0,-6.00);
\draw[->,rgb color={255,107,0}] (-18.0,-3.00) -- (-36.0,-6.00);
\draw[->,rgb color={255,107,0}] (-18.0,-3.00) -- (-27.0,-6.00);
\draw[->,rgb color={255,107,0}] (-18.0,-3.00) -- (-18.0,-6.00);
\draw[->,rgb color={255,107,0}] (-18.0,-3.00) -- (7.00,-3.00) node[right]{};
\draw[->,rgb color={255,107,0}] (-18.0,-3.00) -- (7.00,-3.00) node[right]{};
\draw[->,rgb color={255,107,0}] (-18.0,-3.00) -- (7.00,-3.00) node[right]{};
\draw[->,rgb color={255,107,0}] (-18.0,-3.00) -- (7.00,-3.00) node[right]{};
\draw[->,rgb color={255,107,0}] (-18.0,-3.00) -- (7.00,-3.00) node[right]{};
\draw[->,rgb color={255,107,0}] (-18.0,-3.00) -- (7.00,-3.00) node[right]{};
\draw[->,rgb color={255,107,0}] (-15.0,-3.00) -- (7.00,-3.00) node[right]{};
\draw[->,rgb color={255,107,0}] (-15.0,-3.00) -- (7.00,-3.00) node[right]{};
\draw[->,rgb color={255,107,0}] (-18.0,-3.00) -- (7.00,-3.00) node[right]{};
\draw[->,rgb color={255,107,0}] (-18.0,-3.00) -- (7.00,-1.33);
\draw[->,rgb color={255,107,0}] (-18.0,-3.00) -- (7.00,-1.33);
\draw[->,rgb color={255,107,0}] (-18.0,-3.00) -- (7.00,-1.33);
\draw[->,rgb color={255,107,0}] (-18.0,-3.00) -- (7.00,-1.61);
\draw[->,rgb color={255,107,0}] (-18.0,-3.00) -- (7.00,-1.81);
\draw[->,rgb color={255,107,0}] (-18.0,-3.00) -- (7.00,-1.15);
\draw[->,rgb color={255,107,0}] (-18.0,4.00) -- (-40.5,7.00);
\draw[->,rgb color={255,107,0}] (-18.0,4.00) -- (-36.0,7.00);
\draw[->,rgb color={255,107,0}] (-18.0,4.00) -- (-36.0,7.00);
\draw[->,rgb color={255,107,0}] (-18.0,4.00) -- (-36.0,7.00);
\draw[->,rgb color={255,107,0}] (-18.0,4.00) -- (-27.0,7.00);
\draw[->,rgb color={255,107,0}] (-18.0,4.00) -- (-18.0,7.00);
\draw[->,rgb color={255,107,0}] (-18.0,4.00) -- (7.00,4.00) node[right]{};
\draw[->,rgb color={255,107,0}] (-18.0,4.00) -- (7.00,4.00) node[right]{};
\draw[->,rgb color={255,107,0}] (-18.0,4.00) -- (7.00,4.00) node[right]{};
\draw[->,rgb color={255,107,0}] (-18.0,4.00) -- (7.00,4.00) node[right]{};
\draw[->,rgb color={255,107,0}] (-18.0,4.00) -- (7.00,4.00) node[right]{};
\draw[->,rgb color={255,107,0}] (-18.0,4.00) -- (7.00,4.00) node[right]{};
\draw[->,rgb color={255,107,0}] (-15.0,4.00) -- (7.00,4.00) node[right]{};
\draw[->,rgb color={255,107,0}] (-15.0,4.00) -- (7.00,4.00) node[right]{};
\draw[->,rgb color={255,107,0}] (-18.0,4.00) -- (7.00,4.00) node[right]{};
\draw[->,rgb color={255,107,0}] (-18.0,4.00) -- (7.00,2.33);
\draw[->,rgb color={255,107,0}] (-18.0,4.00) -- (7.00,2.33);
\draw[->,rgb color={255,107,0}] (-18.0,4.00) -- (7.00,2.33);
\draw[->,rgb color={255,107,0}] (-18.0,4.00) -- (7.00,2.61);
\draw[->,rgb color={255,107,0}] (-18.0,4.00) -- (7.00,2.81);
\draw[->,rgb color={255,107,0}] (-18.0,4.00) -- (7.00,2.15);
\draw[->,rgb color={255,107,0}] (-15.0,-3.00) -- (-24.0,-6.00);
\draw[->,rgb color={255,107,0}] (-15.0,-3.00) -- (7.00,-1.78);
\draw[->,rgb color={255,107,0}] (-15.0,4.00) -- (-24.0,7.00);
\draw[->,rgb color={255,107,0}] (-15.0,4.00) -- (7.00,2.78);
\draw[->,rgb color={255,107,0}] (-11.2,-2.75) -- (7.00,-2.75) node[right]{};
\draw[->,rgb color={255,107,0}] (-13.5,-2.75) -- (7.00,-2.75) node[right]{};
\draw[->,rgb color={255,0,0}] (-13.5,-2.50) -- (-45.0,-6.00);
\draw[->,rgb color={255,0,0}] (-13.5,-2.50) -- (7.00,-0.222);
\draw[->,rgb color={255,0,0}] (-13.5,3.50) -- (-45.0,7.00);
\draw[->,rgb color={255,0,0}] (-13.5,3.50) -- (7.00,1.22);
\draw[->,rgb color={255,77,0}] (-11.2,3.75) -- (7.00,3.75) node[right]{};
\draw[->,rgb color={255,77,0}] (-13.5,3.75) -- (7.00,3.75) node[right]{};
\draw[->,rgb color={255,107,0}] (-13.0,-2.67) -- (7.00,-2.67) node[right]{};
\draw[->,rgb color={255,107,0}] (-13.0,-2.67) -- (7.00,-2.67) node[right]{};
\draw[->,rgb color={255,77,0}] (-13.0,3.67) -- (7.00,3.67) node[right]{};
\draw[->,rgb color={255,107,0}] (-12.0,-2.50) -- (-12.0,-6.00);
\draw[->,rgb color={255,107,0}] (-12.0,-2.50) -- (7.00,-2.50) node[right]{};
\draw[->,rgb color={255,107,0}] (-12.0,-2.50) -- (7.00,-2.50) node[right]{};
\draw[->,rgb color={255,107,0}] (-10.5,-2.50) -- (7.00,-2.50) node[right]{};
\draw[->,rgb color={255,107,0}] (-7.50,-2.50) -- (7.00,-2.50) node[right]{};
\draw[->,rgb color={255,107,0}] (-12.0,-2.50) -- (7.00,-2.50) node[right]{};
\draw[->,rgb color={255,107,0}] (-12.0,-2.50) -- (7.00,-1.44);
\draw[->,rgb color={255,77,0}] (-12.0,3.50) -- (-12.0,7.00);
\draw[->,rgb color={255,77,0}] (-12.0,3.50) -- (7.00,3.50) node[right]{};
\draw[->,rgb color={255,77,0}] (-12.0,3.50) -- (7.00,3.50) node[right]{};
\draw[->,rgb color={255,77,0}] (-10.5,3.50) -- (7.00,3.50) node[right]{};
\draw[->,rgb color={255,77,0}] (-7.50,3.50) -- (7.00,3.50) node[right]{};
\draw[->,rgb color={255,77,0}] (-12.0,3.50) -- (7.00,3.50) node[right]{};
\draw[->,rgb color={255,77,0}] (-12.0,3.50) -- (7.00,2.44);
\draw[->,rgb color={255,107,0}] (-10.0,-2.33) -- (7.00,-2.33) node[right]{};
\draw[->,rgb color={255,77,0}] (-10.0,3.33) -- (7.00,3.33) node[right]{};
\draw[->,rgb color={255,77,0}] (-10.0,3.33) -- (7.00,3.33) node[right]{};
\draw[->,rgb color={255,107,0}] (-6.75,-2.25) -- (7.00,-2.25) node[right]{};
\draw[->,rgb color={255,107,0}] (-9.00,-2.25) -- (7.00,-2.25) node[right]{};
\draw[->,rgb color={255,77,0}] (-9.00,-2.00) -- (-27.0,-6.00);
\draw[->,rgb color={255,77,0}] (-9.00,-2.00) -- (-21.0,-6.00);
\draw[->,rgb color={255,77,0}] (-9.00,-2.00) -- (-21.0,-6.00);
\draw[->,rgb color={255,77,0}] (-9.00,-2.00) -- (-21.0,-6.00);
\draw[->,rgb color={255,77,0}] (-9.00,-2.00) -- (-9.00,-6.00);
\draw[->,rgb color={255,77,0}] (-9.00,-2.00) -- (7.00,-2.00) node[right]{};
\draw[->,rgb color={255,77,0}] (-9.00,-2.00) -- (7.00,-2.00) node[right]{};
\draw[->,rgb color={255,77,0}] (-9.00,-2.00) -- (7.00,-2.00) node[right]{};
\draw[->,rgb color={255,77,0}] (-9.00,-2.00) -- (7.00,-2.00) node[right]{};
\draw[->,rgb color={255,77,0}] (-9.00,-2.00) -- (7.00,-2.00) node[right]{};
\draw[->,rgb color={255,77,0}] (-9.00,-2.00) -- (7.00,-2.00) node[right]{};
\draw[->,rgb color={255,77,0}] (-6.00,-2.00) -- (7.00,-2.00) node[right]{};
\draw[->,rgb color={255,77,0}] (-6.00,-2.00) -- (7.00,-2.00) node[right]{};
\draw[->,rgb color={255,77,0}] (-9.00,-2.00) -- (7.00,-2.00) node[right]{};
\draw[->,rgb color={255,77,0}] (-9.00,-2.00) -- (3.00,-6.00);
\draw[->,rgb color={255,77,0}] (-9.00,-2.00) -- (7.00,-0.667);
\draw[->,rgb color={255,77,0}] (-9.00,-2.00) -- (7.00,-0.667);
\draw[->,rgb color={255,77,0}] (-9.00,-2.00) -- (7.00,-0.667);
\draw[->,rgb color={255,77,0}] (-9.00,-2.00) -- (7.00,-0.933);
\draw[->,rgb color={255,77,0}] (-9.00,-2.00) -- (7.00,-1.11);
\draw[->,rgb color={255,77,0}] (-9.00,-2.00) -- (7.00,-0.476);
\draw[->,rgb color={255,77,0}] (-9.00,3.00) -- (-27.0,7.00);
\draw[->,rgb color={255,77,0}] (-9.00,3.00) -- (-21.0,7.00);
\draw[->,rgb color={255,77,0}] (-9.00,3.00) -- (-21.0,7.00);
\draw[->,rgb color={255,77,0}] (-9.00,3.00) -- (-21.0,7.00);
\draw[->,rgb color={255,77,0}] (-9.00,3.00) -- (-9.00,7.00);
\draw[->,rgb color={255,77,0}] (-9.00,3.00) -- (7.00,3.00) node[right]{};
\draw[->,rgb color={255,77,0}] (-9.00,3.00) -- (7.00,3.00) node[right]{};
\draw[->,rgb color={255,77,0}] (-9.00,3.00) -- (7.00,3.00) node[right]{};
\draw[->,rgb color={255,77,0}] (-9.00,3.00) -- (7.00,3.00) node[right]{};
\draw[->,rgb color={255,77,0}] (-9.00,3.00) -- (7.00,3.00) node[right]{};
\draw[->,rgb color={255,77,0}] (-9.00,3.00) -- (7.00,3.00) node[right]{};
\draw[->,rgb color={255,77,0}] (-6.00,3.00) -- (7.00,3.00) node[right]{};
\draw[->,rgb color={255,77,0}] (-6.00,3.00) -- (7.00,3.00) node[right]{};
\draw[->,rgb color={255,77,0}] (-9.00,3.00) -- (7.00,3.00) node[right]{};
\draw[->,rgb color={255,77,0}] (-9.00,3.00) -- (3.00,7.00);
\draw[->,rgb color={255,77,0}] (-9.00,3.00) -- (7.00,1.67);
\draw[->,rgb color={255,77,0}] (-9.00,3.00) -- (7.00,1.67);
\draw[->,rgb color={255,77,0}] (-9.00,3.00) -- (7.00,1.67);
\draw[->,rgb color={255,77,0}] (-9.00,3.00) -- (7.00,1.93);
\draw[->,rgb color={255,77,0}] (-9.00,3.00) -- (7.00,2.11);
\draw[->,rgb color={255,77,0}] (-9.00,3.00) -- (7.00,1.48);
\draw[->,rgb color={255,77,0}] (-6.75,3.25) -- (7.00,3.25) node[right]{};
\draw[->,rgb color={255,77,0}] (-9.00,3.25) -- (7.00,3.25) node[right]{};
\draw[->,rgb color={255,77,0}] (-6.00,-2.00) -- (-6.00,-6.00);
\draw[->,rgb color={255,77,0}] (-6.00,-2.00) -- (7.00,-1.13);
\draw[->,rgb color={255,0,0}] (-6.00,-1.50) -- (-33.0,-6.00);
\draw[->,rgb color={255,0,0}] (-6.00,-1.50) -- (7.00,0.667);
\draw[->,rgb color={255,0,0}] (-6.00,2.50) -- (-33.0,7.00);
\draw[->,rgb color={255,0,0}] (-6.00,2.50) -- (7.00,0.333);
\draw[->,rgb color={255,77,0}] (-6.00,3.00) -- (-6.00,7.00);
\draw[->,rgb color={255,77,0}] (-6.00,3.00) -- (7.00,2.13);
\draw[->,rgb color={255,77,0}] (-3.00,-1.75) -- (7.00,-1.75) node[right]{};
\draw[->,rgb color={255,77,0}] (-5.25,-1.75) -- (7.00,-1.75) node[right]{};
\draw[->,rgb color={255,42,0}] (-3.00,2.75) -- (7.00,2.75) node[right]{};
\draw[->,rgb color={255,42,0}] (-5.25,2.75) -- (7.00,2.75) node[right]{};
\draw[->,rgb color={255,77,0}] (-5.00,-1.67) -- (7.00,-1.67) node[right]{};
\draw[->,rgb color={255,42,0}] (-5.00,2.67) -- (7.00,2.67) node[right]{};
\draw[->,rgb color={255,77,0}] (-4.50,-1.50) -- (7.00,-1.50) node[right]{};
\draw[->,rgb color={255,77,0}] (-4.50,-1.50) -- (7.00,-1.50) node[right]{};
\draw[->,rgb color={255,77,0}] (-3.00,-1.50) -- (7.00,-1.50) node[right]{};
\draw[->,rgb color={255,77,0}] (0.000,-1.50) -- (7.00,-1.50) node[right]{};
\draw[->,rgb color={255,77,0}] (-4.50,-1.50) -- (7.00,-1.50) node[right]{};
\draw[->,rgb color={255,77,0}] (-4.50,-1.50) -- (7.00,-5.33);
\draw[->,rgb color={255,77,0}] (-4.50,-1.50) -- (7.00,-0.733);
\draw[->,rgb color={255,42,0}] (-4.50,2.50) -- (7.00,2.50) node[right]{};
\draw[->,rgb color={255,42,0}] (-4.50,2.50) -- (7.00,2.50) node[right]{};
\draw[->,rgb color={255,42,0}] (-3.00,2.50) -- (7.00,2.50) node[right]{};
\draw[->,rgb color={255,42,0}] (0.000,2.50) -- (7.00,2.50) node[right]{};
\draw[->,rgb color={255,42,0}] (-4.50,2.50) -- (7.00,2.50) node[right]{};
\draw[->,rgb color={255,42,0}] (-4.50,2.50) -- (7.00,6.33);
\draw[->,rgb color={255,42,0}] (-4.50,2.50) -- (7.00,1.73);
\draw[->,rgb color={255,77,0}] (-3.00,-1.33) -- (7.00,-1.33) node[right]{};
\draw[->,rgb color={255,42,0}] (-3.00,-1.00) -- (-10.5,-6.00);
\draw[->,rgb color={255,42,0}] (-3.00,-1.00) -- (-3.00,-6.00);
\draw[->,rgb color={255,42,0}] (-3.00,-1.00) -- (-3.00,-6.00);
\draw[->,rgb color={255,42,0}] (-3.00,-1.00) -- (-3.00,-6.00);
\draw[->,rgb color={255,42,0}] (-3.00,-1.00) -- (7.00,-1.00) node[right]{};
\draw[->,rgb color={255,42,0}] (-3.00,-1.00) -- (7.00,-1.00) node[right]{};
\draw[->,rgb color={255,42,0}] (-3.00,-1.00) -- (7.00,-1.00) node[right]{};
\draw[->,rgb color={255,42,0}] (-3.00,-1.00) -- (7.00,-1.00) node[right]{};
\draw[->,rgb color={255,42,0}] (-3.00,-1.00) -- (7.00,-1.00) node[right]{};
\draw[->,rgb color={255,42,0}] (-3.00,-1.00) -- (7.00,-1.00) node[right]{};
\draw[->,rgb color={255,42,0}] (0.000,-1.00) -- (7.00,-1.00) node[right]{};
\draw[->,rgb color={255,42,0}] (0.000,-1.00) -- (7.00,-1.00) node[right]{};
\draw[->,rgb color={255,42,0}] (-3.00,-1.00) -- (7.00,-1.00) node[right]{};
\draw[->,rgb color={255,42,0}] (-3.00,-1.00) -- (7.00,-4.33);
\draw[->,rgb color={255,42,0}] (-3.00,-1.00) -- (7.00,-2.67);
\draw[->,rgb color={255,42,0}] (-3.00,-1.00) -- (7.00,0.111);
\draw[->,rgb color={255,42,0}] (-3.00,-1.00) -- (7.00,0.111);
\draw[->,rgb color={255,42,0}] (-3.00,-1.00) -- (7.00,0.111);
\draw[->,rgb color={255,42,0}] (-3.00,-1.00) -- (7.00,-0.167);
\draw[->,rgb color={255,42,0}] (-3.00,-1.00) -- (7.00,-0.333);
\draw[->,rgb color={255,42,0}] (-3.00,-1.00) -- (7.00,0.333);
\draw[->,rgb color={255,42,0}] (-3.00,2.00) -- (-10.5,7.00);
\draw[->,rgb color={255,42,0}] (-3.00,2.00) -- (-3.00,7.00);
\draw[->,rgb color={255,42,0}] (-3.00,2.00) -- (-3.00,7.00);
\draw[->,rgb color={255,42,0}] (-3.00,2.00) -- (-3.00,7.00);
\draw[->,rgb color={255,42,0}] (-3.00,2.00) -- (7.00,2.00) node[right]{};
\draw[->,rgb color={255,42,0}] (-3.00,2.00) -- (7.00,2.00) node[right]{};
\draw[->,rgb color={255,42,0}] (-3.00,2.00) -- (7.00,2.00) node[right]{};
\draw[->,rgb color={255,42,0}] (-3.00,2.00) -- (7.00,2.00) node[right]{};
\draw[->,rgb color={255,42,0}] (-3.00,2.00) -- (7.00,2.00) node[right]{};
\draw[->,rgb color={255,42,0}] (-3.00,2.00) -- (7.00,2.00) node[right]{};
\draw[->,rgb color={255,42,0}] (0.000,2.00) -- (7.00,2.00) node[right]{};
\draw[->,rgb color={255,42,0}] (0.000,2.00) -- (7.00,2.00) node[right]{};
\draw[->,rgb color={255,42,0}] (-3.00,2.00) -- (7.00,2.00) node[right]{};
\draw[->,rgb color={255,42,0}] (-3.00,2.00) -- (7.00,5.33);
\draw[->,rgb color={255,42,0}] (-3.00,2.00) -- (7.00,3.67);
\draw[->,rgb color={255,42,0}] (-3.00,2.00) -- (7.00,0.889);
\draw[->,rgb color={255,42,0}] (-3.00,2.00) -- (7.00,0.889);
\draw[->,rgb color={255,42,0}] (-3.00,2.00) -- (7.00,0.889);
\draw[->,rgb color={255,42,0}] (-3.00,2.00) -- (7.00,1.17);
\draw[->,rgb color={255,42,0}] (-3.00,2.00) -- (7.00,0.667);
\draw[->,rgb color={255,42,0}] (-3.00,2.00) -- (7.00,1.33);
\draw[->,rgb color={255,42,0}] (-3.00,2.33) -- (7.00,2.33) node[right]{};
\draw[->,rgb color={255,42,0}] (-3.00,2.33) -- (7.00,2.33) node[right]{};
\draw[->,rgb color={255,77,0}] (0.000,-1.25) -- (7.00,-1.25) node[right]{};
\draw[->,rgb color={255,77,0}] (-2.25,-1.25) -- (7.00,-1.25) node[right]{};
\draw[->,rgb color={255,42,0}] (0.000,2.25) -- (7.00,2.25) node[right]{};
\draw[->,rgb color={255,42,0}] (-2.25,2.25) -- (7.00,2.25) node[right]{};
\draw[->,rgb color={255,0,0}] (-1.50,-0.500) -- (-18.0,-6.00);
\draw[->,rgb color={255,0,0}] (-1.50,-0.500) -- (7.00,2.33);
\draw[->,rgb color={255,0,0}] (-1.50,1.50) -- (-18.0,7.00);
\draw[->,rgb color={255,0,0}] (-1.50,1.50) -- (7.00,-1.33);
\draw[->,rgb color={255,42,0}] (0.000,-1.00) -- (7.00,-3.33);
\draw[->,rgb color={255,42,0}] (0.000,-1.00) -- (7.00,-0.417);
\draw[->,rgb color={255,42,0}] (2.25,-0.750) -- (7.00,-0.750) node[right]{};
\draw[->,rgb color={255,42,0}] (0.000,-0.750) -- (7.00,-0.750) node[right]{};
\draw[->,rgb color={255,42,0}] (0.000,-0.667) -- (7.00,-0.667) node[right]{};
\draw[->,rgb color={255,42,0}] (0.000,-0.500) -- (7.00,-0.500) node[right]{};
\draw[->,rgb color={255,42,0}] (0.000,-0.500) -- (7.00,-0.500) node[right]{};
\draw[->,rgb color={255,42,0}] (1.50,-0.500) -- (7.00,-0.500) node[right]{};
\draw[->,rgb color={255,42,0}] (4.50,-0.500) -- (7.00,-0.500) node[right]{};
\draw[->,rgb color={255,42,0}] (0.000,-0.500) -- (7.00,-0.500) node[right]{};
\draw[->,rgb color={255,42,0}] (0.000,-0.500) -- (7.00,-1.67);
\draw[->,rgb color={255,42,0}] (0.000,-0.500) -- (7.00,0.0833);
\draw[->,rgb color={255,0,0}] (0.000,0.000) -- (7.00,0.000) node[right]{};
\draw[->,rgb color={255,0,0}] (0.000,0.000) -- (7.00,0.000) node[right]{};
\draw[->,rgb color={255,0,0}] (0.000,0.000) -- (7.00,0.000) node[right]{};
\draw[->,rgb color={255,0,0}] (0.000,0.000) -- (7.00,0.000) node[right]{};
\draw[->,rgb color={255,0,0}] (0.000,0.000) -- (7.00,0.000) node[right]{};
\draw[->,rgb color={255,0,0}] (0.000,0.000) -- (7.00,0.000) node[right]{};
\draw[->,rgb color={255,0,0}] (3.00,0.000) -- (7.00,0.000) node[right]{};
\draw[->,rgb color={255,0,0}] (3.00,0.000) -- (7.00,0.000) node[right]{};
\draw[->,rgb color={255,0,0}] (0.000,0.000) -- (7.00,0.000) node[right]{};
\draw[->,rgb color={255,0,0}] (0.000,0.000) -- (7.00,-4.67);
\draw[->,rgb color={255,0,0}] (0.000,0.000) -- (7.00,-2.33);
\draw[->,rgb color={255,0,0}] (0.000,0.000) -- (7.00,-2.33);
\draw[->,rgb color={255,0,0}] (0.000,0.000) -- (7.00,-2.33);
\draw[->,rgb color={255,0,0}] (0.000,0.000) -- (7.00,-1.17);
\draw[->,rgb color={255,0,0}] (0.000,0.000) -- (7.00,1.17);
\draw[->,rgb color={255,0,0}] (0.000,0.000) -- (7.00,1.17);
\draw[->,rgb color={255,0,0}] (0.000,0.000) -- (7.00,1.17);
\draw[->,rgb color={255,0,0}] (0.000,0.000) -- (7.00,-0.778);
\draw[->,rgb color={255,0,0}] (0.000,0.000) -- (7.00,0.778);
\draw[->,rgb color={255,0,0}] (0.000,0.000) -- (7.00,1.56);
\draw[->,rgb color={255,0,0}] (0.000,0.000) -- (7.00,0.583);
\draw[->,rgb color={255,0,0}] (0.000,0.500) -- (0.000,-6.00);
\draw[->,rgb color={255,0,0}] (0.000,0.500) -- (0.000,7.00);
\draw[->,rgb color={255,0,0}] (0.000,1.00) -- (7.00,1.00) node[right]{};
\draw[->,rgb color={255,0,0}] (0.000,1.00) -- (7.00,1.00) node[right]{};
\draw[->,rgb color={255,0,0}] (0.000,1.00) -- (7.00,1.00) node[right]{};
\draw[->,rgb color={255,0,0}] (0.000,1.00) -- (7.00,1.00) node[right]{};
\draw[->,rgb color={255,0,0}] (0.000,1.00) -- (7.00,1.00) node[right]{};
\draw[->,rgb color={255,0,0}] (0.000,1.00) -- (7.00,1.00) node[right]{};
\draw[->,rgb color={255,0,0}] (3.00,1.00) -- (7.00,1.00) node[right]{};
\draw[->,rgb color={255,0,0}] (3.00,1.00) -- (7.00,1.00) node[right]{};
\draw[->,rgb color={255,0,0}] (0.000,1.00) -- (7.00,1.00) node[right]{};
\draw[->,rgb color={255,0,0}] (0.000,1.00) -- (7.00,3.33);
\draw[->,rgb color={255,0,0}] (0.000,1.00) -- (7.00,3.33);
\draw[->,rgb color={255,0,0}] (0.000,1.00) -- (7.00,3.33);
\draw[->,rgb color={255,0,0}] (0.000,1.00) -- (7.00,5.67);
\draw[->,rgb color={255,0,0}] (0.000,1.00) -- (7.00,-0.167);
\draw[->,rgb color={255,0,0}] (0.000,1.00) -- (7.00,-0.167);
\draw[->,rgb color={255,0,0}] (0.000,1.00) -- (7.00,-0.167);
\draw[->,rgb color={255,0,0}] (0.000,1.00) -- (7.00,2.17);
\draw[->,rgb color={255,0,0}] (0.000,1.00) -- (7.00,-0.556);
\draw[->,rgb color={255,0,0}] (0.000,1.00) -- (7.00,0.222);
\draw[->,rgb color={255,0,0}] (0.000,1.00) -- (7.00,1.78);
\draw[->,rgb color={255,0,0}] (0.000,1.00) -- (7.00,0.417);
\draw[->,rgb color={255,0,0}] (0.000,1.50) -- (7.00,1.50) node[right]{};
\draw[->,rgb color={255,0,0}] (0.000,1.50) -- (7.00,1.50) node[right]{};
\draw[->,rgb color={255,0,0}] (1.50,1.50) -- (7.00,1.50) node[right]{};
\draw[->,rgb color={255,0,0}] (4.50,1.50) -- (7.00,1.50) node[right]{};
\draw[->,rgb color={255,0,0}] (0.000,1.50) -- (7.00,1.50) node[right]{};
\draw[->,rgb color={255,0,0}] (0.000,1.50) -- (7.00,2.67);
\draw[->,rgb color={255,0,0}] (0.000,1.50) -- (7.00,0.917);
\draw[->,rgb color={255,0,0}] (0.000,1.67) -- (7.00,1.67) node[right]{};
\draw[->,rgb color={255,0,0}] (2.25,1.75) -- (7.00,1.75) node[right]{};
\draw[->,rgb color={255,0,0}] (0.000,1.75) -- (7.00,1.75) node[right]{};
\draw[->,rgb color={255,42,0}] (0.000,2.00) -- (7.00,4.33);
\draw[->,rgb color={255,42,0}] (0.000,2.00) -- (7.00,1.42);
\draw[->,rgb color={255,42,0}] (1.00,-0.333) -- (7.00,-0.333) node[right]{};
\draw[->,rgb color={255,0,0}] (1.00,1.33) -- (7.00,1.33) node[right]{};
\draw[->,rgb color={255,0,0}] (1.00,1.33) -- (7.00,1.33) node[right]{};
\draw[->,rgb color={255,42,0}] (3.75,-0.250) -- (7.00,-0.250) node[right]{};
\draw[->,rgb color={255,42,0}] (1.50,-0.250) -- (7.00,-0.250) node[right]{};
\draw[->,rgb color={255,0,0}] (1.50,0.500) -- (7.00,0.500) node[right]{};
\draw[->,rgb color={255,0,0}] (1.50,0.500) -- (7.00,0.500) node[right]{};
\draw[->,rgb color={255,0,0}] (3.00,0.500) -- (7.00,0.500) node[right]{};
\draw[->,rgb color={255,0,0}] (6.00,0.500) -- (7.00,0.500) node[right]{};
\draw[->,rgb color={255,0,0}] (1.50,0.500) -- (7.00,0.500) node[right]{};
\draw[->,rgb color={255,0,0}] (1.50,0.500) -- (7.00,-0.111);
\draw[->,rgb color={255,0,0}] (1.50,0.500) -- (7.00,1.11);
\draw[->,rgb color={255,0,0}] (3.75,1.25) -- (7.00,1.25) node[right]{};
\draw[->,rgb color={255,0,0}] (1.50,1.25) -- (7.00,1.25) node[right]{};
\draw[->,rgb color={255,0,0}] (2.00,0.333) -- (7.00,0.333) node[right]{};
\draw[->,rgb color={255,0,0}] (2.00,0.333) -- (7.00,0.333) node[right]{};
\draw[->,rgb color={255,0,0}] (2.00,0.667) -- (7.00,0.667) node[right]{};
\draw[->,rgb color={255,0,0}] (4.50,0.250) -- (7.00,0.250) node[right]{};
\draw[->,rgb color={255,0,0}] (2.25,0.250) -- (7.00,0.250) node[right]{};
\draw[->,rgb color={255,0,0}] (4.50,0.750) -- (7.00,0.750) node[right]{};
\draw[->,rgb color={255,0,0}] (2.25,0.750) -- (7.00,0.750) node[right]{};
\draw[->,rgb color={255,0,0}] (3.00,0.000) -- (7.00,-0.667);
\draw[->,rgb color={255,0,0}] (3.00,0.000) -- (7.00,0.444);
\draw[->,rgb color={255,0,0}] (3.00,1.00) -- (7.00,1.67);
\draw[->,rgb color={255,0,0}] (3.00,1.00) -- (7.00,0.556);
\end{tikzpicture}
}
}
\caption{First steps of the scattering diagram $S(\fD^\iin_{u,v})$. Figure due to Tim Gr\"afnitz  \cite{gabele2019tropical}.} 
\label{fig:scattering}
\end{figure}

\subsection{Structure of the proof of the main result}
\label{section_intro_proof}

The scattering diagram $S(\fD^\iin_{u,v})$ is defined starting with a simple initial scattering diagram $\fD^\iin_{u,v}$ and then taking its consistent completion. The proof of Theorem \ref{main_thm}, that is, of the equality $\fD^{\PP^2}_{u,v}=S(\fD^\iin_{u,v})$, has correspondingly two parts.
We first show that $\fD^{\PP^2}_{u,v}$ is consistent, and then that $\fD^{\PP^2}_{u,v}$ has the same initial data as
$S(\fD^\iin_{u,v})$. Then, the result follows from the uniqueness of the consistent completion.

\subsubsection{Consistency from wall-crossing formula}
The most non-trivial property about 
the scattering diagram $\fD^{\PP^2}_{u,v}$ that we have to prove is its consistency. This is equivalent to a wall-crossing formula for the Hodge numbers of intersection cohomology of moduli spaces of semistable objects upon variation of the stability condition. 

The key point is the relation between intersection cohomology and Donaldson-Thomas invariants, which goes back to Meinhardt-Reineke \cite{meinhardt2017donaldson}.
This relation has been extended to Gieseker semistable sheaves on surfaces with negative canonical line bundles in 
\cite{MR3874687} and to some abstract 
framework for categories of homological dimension one in \cite{meinhardt2015donaldson}.

The crucial input that will enable us to use this class of techniques is a result of Li-Zhao \cite{MR3936077}:
for every $\sigma \in U$ and for every class $\gamma$, the stack of $\sigma$-semistable objects of class $\gamma$ is smooth. More precisely, 
the $\Ext^2$-group between two $\sigma$-semistable sheaves of class $\gamma$ vanishes if
$\gamma$ is not the class of a zero dimensional sheaf. This vanishing is well-known and obvious for Gieseker semistable sheaves, but is not obvious at all if $\sigma$ is a general stability condition in $U$, so we are really using the non-trivial content of \cite{MR3936077}. Once we know this vanishing, we can apply the machinery described in \cite{meinhardt2015donaldson}
to get that indeed Hodge numbers of intersection cohomology of moduli spaces of semistable objects satisfy the wall-crossing formula of 
Kontsevich-Soibelman
\cite{kontsevich2008stability} 
and Joyce-Song \cite{MR2951762}.

Alternatively, we could consider the derived category $\D^b_0(K_{\PP^2})$ of coherent sheaves on the total space of the canonical line bundle $K_{\PP^2}=\cO_{\PP^2}(-3)$ of $\PP^2$ which are set-theoretically supported on the zero-section $\PP^2$. As $\D^b_0(K_{\PP^2})$ is a Calabi-Yau triangulated category of dimension $3$, it is a natural place for Donaldson-Thomas theory and the wall-crossing formula. 

In fact, $U$
can also be viewed  naturally in the space $\Stab(K_{\PP^2})$ of Bridgeland stability conditions on $\D^b_0(K_{\PP^2})$. We should note that the spaces $\Stab(\PP^2)$ and 
$\Stab(K_{\PP^2})$ behave globally in very different ways, and that it is only 
in restriction to $U$ that we can identify them.
Adapting the proof of 
\cite[Proposition 3.1]{MR3861701} given for Gieseker semistable sheaves, it
follows from the vanishing result of Li-Zhao \cite{MR3936077} that for every $\sigma \in U$,
$\sigma$-semistable objects in $\D^b_0(K_{\PP^2})$ have cohomology sheaves scheme-theoretically supported on $\PP^2$, and so coincide with 
$\sigma$-semistable objects in 
$\D^b(\PP^2)$. 
Therefore, the Hodge numbers of
intersection cohomology of moduli spaces of 
$\sigma$-semistable objects in 
$\D^b(\PP^2)$ are really 
(refined) Donaldson-Thomas invariants of 
$K_{\PP^2}$.

The conclusion is that we think conceptually about $K_{\PP^2}$ but we work technically on $\PP^2$. 
Refined Donaldson-Thomas theory in the context of general 
Calabi-Yau 3-folds requires a discussion of orientation data, and we refer to \cite{shi2018orientation} for
a discussion precisely in the case of $K_{\PP^2}$, 
but thanks to the smoothness of the moduli stack of semistable objects, 
we don't have to go into these technical aspects of the general story.

\subsubsection{Initial data}
The combinatorial definition of the scattering diagram $S(\fD^{\iin}_{u,v})$
involves very simple initial data, and then a 
completely algorithmic completion which guarantees its consistency. 
Once we know that the scattering diagram 
$\fD^{\PP^2}_{u,v}$ is consistent, it is enough to show that it has the same initial data as $S(\fD^{\iin}_{u,v})$ in order to conclude the equality
$\fD^{\PP^2}_{u,v}=S(\fD^{\iin}_{u,v})$ which is the statement of Theorem \ref{main_thm}.

We will show that the initial rays
$\fd_{n,\ell}^-$ and $\fd_{n,\ell}^+$ of $S(\fD^{\iin}_{u,v})$ correspond respectively from the point of view of 
$\fD^{\PP^2}_{u,v}$ to the line bundles $\cO(n)$ and their shift
$\cO(n)[1]$, for $n \in \Z$. They come out from 
the points $s_n$ in the boundary parabola of $U$ where the central charge of $\cO(n)$ vanishes.
In order to prove that 
$\fD^{\PP^2}_{u,v}$ has the same initial data as
$S(\fD^{\iin}_{u,v})$, it is enough to show that, in the region in $U$ near the boundary parabola, 
$\fD^{\PP^2}_{u,v}$ consists only of the rays associated to $\cO(n)$ and 
$\cO(n)[1]$. 

This comes from the fact that this region can be decomposed into triangles
in correspondence with exceptional collections of objects in $\D^b(\PP^2)$.  A stability condition 
in the interior of such a triangle is equivalent to a
stability condition with quiver heart.
Using the fact that the class of an object in a quiver heart 
is a linear combination with nonnegative coefficients of the classes of the three simple quiver representations, 
we can show that no object has $\Rea Z=0$ inside each triangle, 
and 
so that the scattering diagram $\fD^{\PP^2}_{u,v}$,
defined in terms of the condition 
$\Rea Z=0$, is trivial in the interior of all the triangles.

In order to have such a simple picture, 
with initial data in correspondence with line bundles and determining everything 
else by wall-crossing, 
it is essential to restrict our attention to semistable objects with a fixed phase of the central charge, such as the semistable objects with $\Rea Z_\gamma^\sigma=0$ entering in the definition of $\fD^{\PP^2}_{u,v}$. 
Indeed, there is no point in $U$ where the set of all semistable objects is really simple, 
and so the naive idea to find a point in $U$ where the set of all semistable objects is simple, 
and then to move to a more complicated point by wall-crossing, does not work. By considering the scattering diagram 
$\fD^{\PP^2}_{u,v}$, we are only looking at semistable objects with $\Rea Z_\gamma^\sigma=0$, and according to the previous paragraph, there is a region in $U$ where the set of such semistable objects is simple.

\subsection{Applications to moduli spaces of Gieseker semistable sheaves}
\label{section_intro_gieseker}

For every $\gamma=(r,d,\chi)$, let 
$M_\gamma$ be the moduli space of S-equivalence classes of Gieseker semistable sheaves of class $\gamma$ on $\PP^2$. A reference for the classical topic of moduli spaces of Gieseker semistable sheaves is the book \cite{huybrechts2010geometry}. For every 
$\gamma \in \Gamma$, $M_\gamma$ is a projective variety, singular in general.
Nevertheless, the intersection cohomology groups $IH^{k}(M_\gamma, \Q)$
behave as well as the cohomology groups of a smooth projective variety, and in particular carry naturally a pure Hodge structure of weight $k$. We denote by $Ih^{p,q}(M_\gamma)$
the corresponding Hodge numbers and by $Ib_j(M_\gamma)$ the corresponding Betti numbers.

For every $\gamma \in \Gamma$, we have 
$M_\gamma^\sigma = M_\gamma$ for $\sigma=(x,y) \in U$ with $y$ large enough. 
Theorem \ref{main_thm} gives an algorithmic way to compute the intersection Hodge numbers $Ih^{p,q}(M_\gamma^\sigma)$, and so the intersection Hodge numbers $Ih^{p,q}(M_\gamma)$. Using this algorithm, we prove the following result.

\begin{thm}[=Theorem \ref{thm_hodge}]\label{thm_hodge_intro}
For every $\gamma \in \Gamma$, we have 
$Ih^{p,q}(M_\gamma)=0$ if $p \neq q$.
\end{thm}

If $\gamma$ is primitive
then semistability coincides with stability,
$M_\gamma$ is smooth, intersection cohomology coincides with ordinary cohomology, and in this case Theorem \ref{thm_hodge_intro} is classical \cite{MR1228610, MR1351502, MR2304330}. In general, under the extra assumption $r>0$, a different proof of Theorem \ref{thm_hodge_intro} could be extracted from \cite{MR3874687}.

The fact that the lines bundles $\cO(n)$ for 
$n \in \Z$ generate $\D^b(\PP^2)$ goes back to Beilinson \cite{MR509388}. The scattering diagram 
$\fD^{\PP^2}_{u,v}$ gives an explicit way, at the level of intersection cohomology of
moduli spaces, to reconstruct Gieseker semistable sheaves from line bundles 
by successive exact triangles in the derived category. 
In particular, we obtain a decomposition indexed by trees of 
the intersection cohomology of moduli spaces of Gieseker semistable sheaves, 
which seems to be new. We refer to
\cref{section_tree_decomposition} for details.

We also prove an analogue of Theorem \ref{thm_hodge_intro} involving the real algebraic geometry of the moduli spaces 
$M_\gamma$. We refer to \cref{section_real_moduli} for details.

\begin{thm} \label{thm_real_moduli_intro}
For every $\gamma \in \Gamma$, and for every $0 \leqslant p \leqslant \dim M_\gamma$, the complex conjugation in
$\Gal(\C/\R)$ acts as $(-1)^p$
on $IH^{2p}(M_\gamma,\Q)$.
\end{thm}

Our final results on moduli spaces of Gieseker semistable sheaves concern sheaves supported on curves, that is with $\gamma=(0,d,\chi)$ and $d \geqslant1$. The intersection Betti numbers $Ib_j(M_{(0,d,\chi)})$ are refined Donaldson-Thomas invariants for one-dimensional sheaves on $K_{\PP^2}$, and so can be viewed as refined genus-$0$
Gopakumar-Vafa invariants of $K_{\PP^2}$ (see \cite{bousseau2019takahashi} for details). 
Tensor product with $\cO(1)$ and Serre duality induce isomorphisms 
$M_{(0,d,\chi)} \simeq M_{(0,d,\chi+d)}$ and $M_{(0,d,\chi)} \simeq M_{(0,d,-\chi)}$.
On the other hand, if $d \geqslant3$ and $\chi' \neq \pm \chi \mod d$, then the algebraic varieties $M_{(0,d,\chi)}$ and $M_{(0,d,\chi')}$ are not isomorphic by \cite[Theorem 8.1]{woolf2013nef}. This makes the following conjecture particularly non-trivial.

\begin{conj}\footnote{Note added in the final version (08/2021): this conjecture has now been proved by Maulik and Shen \cite{MS20}.} \label{conj_chi_indep}
For every fixed $d \geqslant1$, the intersection Betti numbers $Ib_j(M_{d,\chi})$ do not depend on $\chi$.
\end{conj}

It is a general conjecture due to Joyce-Song \cite[Conjecture 6.20]{MR2951762}
(see also \cite[Conjecture 6.3]{MR2892766}) that (unrefined) Donaldson-Thomas invariants for one-dimensional sheaves on Calabi-Yau 3-folds should be $\chi$-independent.
This conjecture is known for $K_{\PP^2}$ (see \cite[Appendix A]{MR3861701}
and \cite{bousseau2019takahashi}). Conjecture \ref{conj_chi_indep} is the refined version of the Joyce-Song conjecture. 

Using Theorem \ref{main_thm} as an essential ingredient, we will prove in \cite{bousseau2019takahashi} a special but already quite non-trivial case of 
Conjecture \ref{conj_chi_indep}.

\begin{thm}(\cite[Theorem 0.5.2]{bousseau2019takahashi}) \label{thm_chi_indep_gcd}
For every fixed $d \geqslant1$, we have $Ib_j(M_{(0,d,\chi)})=Ib_j(M_{(0,d,\chi')})$
as long as $\gcd(d,\chi)=\gcd(d,\chi')$.
\end{thm}

Using the algorithm given by Theorem \ref{main_thm} to compute the intersection Betti numbers $Ib_j(M_{(0,d,\chi)})$, we can practically test Conjecture \ref{conj_chi_indep}
in low degrees.

\begin{thm} \label{thm_chi_indep_test}
Conjecture \ref{conj_chi_indep} holds for $d \leq 4$.
\end{thm}

As the intersection Betti numbers are computed by the scattering diagram 
$S(\fD_{u,v}^{\iin})$, one might hope for a combinatorial proof of Conjecture 
\ref{conj_chi_indep}. This seems quite difficult. In particular, we will see explicitly in the proof of Theorem \ref{thm_chi_indep_test}
in \S \ref{section_chi_indep}
that the
tree decompositions of the Betti numbers are different for different values of $\chi$, and only the total sum over the trees is $\chi$-independent in a seemingly miraculous way.

\subsection{Relations with previous and future works}
\label{section_intro_relation}
In this section we briefly mention some related works.

\subsubsection{Bridgeland}
The idea to consider objects with fixed argument of the central charge in order to make a connection between scattering diagrams
and stability conditions, in the context of quivers with potentials, is already contained in 
\cite{MR3710055}. Given $(Q,W)$ a quiver with potential, let $Q_0$ be its set of vertices and $\Gamma=\Z^{Q_0}$ the lattice of dimensions for representations of $Q$. We denote $\Gamma^\vee \coloneq \Hom(\Gamma,\Z)$ and $\Gamma^\vee_{\R}
\coloneq \Gamma^\vee \otimes_{\Z} \R$.
Let $\cC$ be the triangulated category of dg-modules with finite-dimensional cohomology
over the Ginzburg dg-algebra of $(Q,W)$ \cite{G06, KY11}.
It is a triangulated category, Calabi-Yau of dimension $3$, with a natural bounded $t$-structure of heart the abelian category 
$\cA$ of finite-dimensional representations of the Jacobi algebra of $(Q,W)$. 
We have a natural embedding 
\[ \Gamma_\R^\vee \rightarrow \Stab(\cC)\,,\]
\[ \theta \mapsto (Z_\theta, \cA) \,,\]
with the central charge $Z^\theta \colon \Gamma \rightarrow \C$ given by 
\[Z^\theta=-\theta+i \delta\,,\]
where 
\[\delta \colon \Gamma \rightarrow \Z\]
\[(\gamma_j)_{j\in Q_0} \mapsto \sum_{j \in Q_0} \gamma_j\,,\] 
is the 
total dimension for representations of $Q$.
The main result of \cite{MR3710055}
is the construction, in terms of Donaldson-Thomas invariants of 
$\cC$, of a scattering diagram $\fD^{(Q,W)}$ on 
$\Gamma_{\R}^\vee$, supported on the loci of points $\theta \in \Gamma_\R^\vee$ such that there exists
$\gamma \in \Gamma$ such that $\Rea Z^\theta_\gamma=-\theta(\gamma)=0$ and such that there exists a $\theta$-semistable object of class $\gamma$.

Our construction of the scattering diagram 
$\fD^{\PP^2}_{u,v}$ on $U$ is similar, but there is one main difference. 
In \cite{MR3710055}, the stability space $\Gamma_{\R}^\vee$ has a natural structure of vector space
and 
the scattering diagram $\fD^{(Q,W)}$ is local, that is, everything non-trivial happens near $0 \in \Gamma_\R^\vee$, corresponding to the fact that the abelian heart $\cA$ is fixed.
By contrast, in our setting, the stability space $U$ is only a piece of an integral affine manifold, and the scattering diagram
$\fD^{\PP^2}_{u,v}$ involves infinitely many such local scatterings, corresponding to the fact that the abelian heart is moving in the triangulated category 
$\D^b(\PP^2)$.

Another result of \cite{MR3710055} is a correspondence between the chamber structure of 
$\fD^{(Q,W)}$ and the mutations of 
$(Q,W)$. We will establish in a follow-up paper a similar 
correspondence between the chamber structure of 
$\fD^{\PP^2}_{u,v}$ and the exceptional collections of  
$\D^b(\PP^2)$ related by mutations, see \S
\ref{section_exceptional} for more details.

For every choice of exceptional collection $\cE$ on 
$\PP^2$, there is a quiver with potential
$(Q_\cE,W_\cE)$ describing $\D^b_0(K_{\PP^2})$, and so we can apply \cite{MR3710055} to obtain a scattering diagram $\fD^{(Q_\cE,W_\cE)}$ in $\R^3$. It would be interesting to obtain a precise comparison between $\fD^{(Q_\cE,W_\cE)}$ and $\fD^{\PP^2}_{u,v}$.

\subsubsection{Kontsevich-Soibelman} \label{subsubKS}
In
\cite{MR3330788}, Kontsevich-Soibelman introduce a general notion of wall-crossing 
structure and give several, mainly conjectural, constructions.
Under some assumptions, they construct a wall-crossing structure on the base 
$B$ of an holomorphic integrable system $\pi \colon M \rightarrow B$. 
When $\pi \colon M \rightarrow B$ is of Hitchin type, there is an associated 
noncompact Calabi-Yau 3-fold $Y$ and a conjectural embedding of (the universal cover of)
the complement $B_0$ of the discriminant locus into the space of Bridgeland stability conditions of the Fukaya category
$F(Y)$ of $Y$. General Donaldson-Thomas theory, conjecturally applied to $F(Y)$,
naturally defines a wall-crossing structure on $B$. Kontsevich-Soibelman conjecture 
that the wall-crossing structures on $B$
coming from the holomorphic integrable system $\pi \colon M \rightarrow B$ and
from Donaldson-Thomas theory of $F(Y)$ coincide. 

In \cite[\S 6]{bousseau2019takahashi}, we explain how our main result fits into this general framework, by taking for $M$
the mirror of $(\PP^2,E)$ and for $Y$ the mirror of $K_{\PP^2}$.
In particular, combining hyperkähler rotation and mirror symmetry, we will obtain a conceptual heuristic explanation for why the main connection established in \cite{bousseau2019takahashi} between sheaf counting on 
$K_{\PP^2}$ and relative Gromov-Witten theory of
$(\PP^2,E)$ should be true: 
under hyperkähler rotation, holomorphic curves in $\PP^2-E$ turn into open special Lagrangians in $M$, which after suspension turn into closed special Lagrangians in $Y$, which under mirror symmetry turn into objects of $\D^b(K_{\PP^2})$. 

Remark that $M$
is not of Hitchin type and so we are slightly outside of the strict framework of 
\cite{MR3330788} but we are in a natural extension of it. The main reason why we are able to extract non-conjectural statements from this story is that, in this specific example, $Y$ has a reasonable mirror $K_{\PP^2}$, and so we can work with the algebro-geometric $D^b(K_{\PP^2})$ instead of the a priori complicated looking symplectic Fukaya category $F(Y)$.

\subsubsection{Physics}
\label{section_physics}
Hodge numbers of intersection cohomology of 
moduli spaces of Bridgeland semistable objects in $\D^b_0(K_{\PP^2})$ are BPS indices of the $\mathcal{N}=2$ supersymmetric 4-dimensional theory obtained by compactifying 
type IIA string theory on the local Calabi-Yau 3-fold $K_{\PP^2}$. 

Even if our slice $U$ in the space of stability conditions is defined by a polynomial
central charge, we claim that the resulting scattering diagram would stay the same for the slice given by the vector multiplet moduli of the $\mathcal{N}=2$ theory, that is, by the stringy K\"ahler moduli space of $K_{\PP^2}$,
where the central charge is given by the periods solutions of the mirror Picard-Fuchs equation
(see e.g. \cite[\S 9]{MR2852118}). More details on this claim are given in 
\cite[\S 3.4]{bousseau2019takahashi}. In fact, the main constructions of the present paper were first tested using the physical central charge given by the periods. 
We choose to write the present paper using polynomial
central charge for convenience: contact with the existing mathematical litterature is easier and change of variables given by algebraic functions are easier to study than those given by transcendental functions. Nevertheless, in order to generalize the present paper to more general geometries, polynomial central charges might not be good enough in general, and the use of the honest stringy Kähler moduli space might be ultimately necessary.

The $\mathcal{N}=2$ theory obtained by compactifying type IIA string theory on 
$K_{\PP^2}$ has been studied early on \cite{MR2174156} (and in fact, as part of a series of physics works which motivated the definition of Bridgeland stability conditions). At the time of \cite{MR2174156}, the wall-crossing formula of Kontsevich-Soibelman and Joyce-Song was not known and only some qualitative aspects of the question of relating the BPS spectrum near the orbifold point to the BPS spectrum near the large volume point were studied. 
Using the wall-crossing formula, the scattering diagram $\fD^{\PP^2}_{u,v}$
gives a complete answer to this physics question, in a form which does not seem to have appeared before in the physics literature. 
At the level of terminology, the scattering diagram $\fD^{\PP^2}_{u,v}$ is made of $K$-walls in the sense of \cite{MR3115984}.

The study of the BPS spectrum of $K_{\PP^2}$ has been revisited 
more recently using spectral networks
\cite{MR3697419} and partial results have been obtained.
Our approach is orthogonal: whereas the spectral networks live on appropriate projections of the mirror curves to $K_{\PP^2}$, our scattering diagram naturally lives on the base of the family parametrizing these mirror curves (see \S \ref{subsubKS}).
In physics language, the scattering diagram $\fD_{u,v}^{\PP^2}$ is a mathematical version of the string junction description of the BPS spectrum in the $D3$-brane probe realization of the $\mathcal{N}=2$ theory (see \cite{MR1657861} and follow-ups).

\subsubsection{Manschot}
An alternative algorithm to compute Hodge numbers of intersection cohomology 
of moduli spaces of Gieseker semistable sheaves
on $\PP^2$, at least for sheaves of positive rank, has been developed by Manschot
\cite{MR2836429, MR3063951, MR3695800, MR3874687}. The idea, going back to Yoshioka 
\cite{MR1285785, MR1339925} in rank $2$, is to blow-up one point on $\PP^2$, reduce the problem to the blown-up 
surface by a blow-up formula, and then solve the problem on the blown-up surface by wall-crossing in the space of polarizations of the blown-up surface.

This algorithm and the algorithm given by  our Theorem 
\ref{main_thm} are similar in the sense that they both use wall-crossing in a space of stability conditions. But there is an important difference: the space of stability conditions in Manschot's algorithm is a space of polarizations, at the cost of working with an auxiliary geometry and using a blow-up formula, whereas our algorithm only uses the geometry intrinsic to $\PP^2$, at the cost of going deep in the interior of the space of Bridgeland stability conditions.
Another difference is that Manschot's algorithm works rank by rank, and seems limited to sheaves of positive rank, whereas our algorithm considers all ranks at once and is able to treat sheaves supported in dimension $1$.

With his approach, Manschot is able to get explicit formulas for generating series of invariants at fixed rank and degree, at least
in low ranks. It is not clear how to do the same with our algorithm and it is a possibly interesting question to explore.

\subsubsection{Exceptional vector bundles and mutations}
\label{section_exceptional}
As we will review in \cite[\S 3]{bousseau2019takahashi}, a specialization $S(\fD^\iin_{\cl^+})$
of $S(\fD^\iin_{u,v})$ naturally appears in the Gross-Siebert description of mirror symmetry for the pair 
$(\PP^2,E)$, where $E$ is a smooth cubic curve in $\PP^2$ \cite{cps}. 
The scattering diagram
$S(\fD^\iin_{\cl^+})$ has been further studied following this point of view in 
\cite{MR4038403}. In particular, the main result of \cite{MR4038403}
is the existence in the complement of the support of $S(\fD^\iin_{\cl^+})$ of a decomposition in triangle chambers naturally 
indexed and related by combinatorial mutations \cite{MR3007265}. 

On the other hand, exactly the same combinatorics of mutations is well-known to describe exceptional collections in 
$\D^b(\PP^2)$  \cite{MR816365, MR849052, 
MR916199, MR885779, MR936525}. 
As we define the scattering diagram
$\fD^{\PP^2}_{u,v}$ in terms of $\D^b(\PP^2)$, and our main result (Theorem \ref{main_thm}) is the equality 
$S(\fD^\iin_{u,v})=\fD^{\PP^2}_{u,v}$, it is natural to ask if the decomposition in triangles of \cite{MR4038403} has a
natural interpretation in terms of exceptional collections from the point of view of $\fD^{\PP^2}_{u,v}$. 

One can show that it is indeed the case, the sides of these triangles are of the form 
$\Rea Z_{\gamma(E_1)}=0$, $\Rea Z_{\gamma(E_2)}=0$, $\Rea Z_{\gamma(E_3)}=0$ for $(E_1,E_2,E_3)$ an exceptional collection in $\D^b(\PP^2)$, and the combinatorial mutations of triangles match the mutations of exceptional collections.
In particular, the equality $S(\fD^\iin_{u,v})=\fD^{\PP^2}_{u,v}$ gives a new and intrinsic to $\PP^2$ explanation of the common appearance of the combinatorics of mutations in the mirror construction for $(\PP^2,E)$ and in the structure of $\D^b(\PP^2)$. Here intrinsic means without using $\mathbb{Q}$-Gorenstein degenerations of $\PP^2$, which are naturally related to both combinatorial mutations and to exceptional collections \cite{MR3053568}. To find such explanation was in fact one of our original motivations for Theorem \ref{main_thm}.

In order to keep the length of the present paper finite, the claims of the above paragraph will be discussed in a follow-up paper. Finally,
 our proof of Theorem \ref{main_thm} 
is logically independent of the classical results of Dr\'ezet and Le Potier \cite{MR816365} on the classification of exceptional vector bundles and on the criterion for nonemptiness of the moduli spaces of Gieseker semistable sheaves. Therefore, we might expect Theorem \ref{main_thm} to give a new approach to these results, and we also refer to some future work. 

\subsection{Plan of the paper}
In 
\cref{section_scattering_diagram} we
introduce general notions about scattering diagrams and we define in a purely algorithmic way a scattering diagram $S(\fD^{\iin}_{u,v})$. 
In \cref{section_scattering_p2}
we introduce specific coordinates on a slice of the space of Bridgeland stability conditions on $\D^b(\PP^2)$ and we use them to define a scattering diagram 
$\fD^{\PP^2}_{u,v}$ in terms of intersection cohomology of Bridgeland semistable objects. 
In \cref{section_consistency_proof} 
we prove that the scattering diagram 
$\fD^{\PP^2}_{u,v}$ is consistent.
In \cref{section_initial_data_proof}
we show that the scattering diagrams 
$\fD^{\PP^2}_{u,v}$ and 
$S(\fD^\iin_{u,v})$ have the same initial data. 
In \cref{section_main_result} 
we prove
Theorem \ref{main_thm},
that is, the equality $\fD^{\PP^2}_{u,v}
=S(\fD^\iin_{u,v})$. In 
\cref{section_application_gieseker} 
we discuss various applications to moduli spaces of Gieseker semistable sheaves.

\subsection{Acknowledgments}
I thank Jinwon Choi, Michel van Garrel, Rahul Pandharipande and Junliang Shen for conversations stimulating my interest in N.\ Takahashi's conjecture, which was one of the main motivation for the present work and will be discussed in the companion paper 
\cite{bousseau2019takahashi}. I thank Richard Thomas and Tom Bridgeland for discussions at a stage where some of the ideas ultimately incorporated in this paper were still in preliminary form. I thank Yu-Shen Lin for an early discussion on the scattering diagram of \cite{cps}. I thank Tim Gr\"afnitz  for discussions on his work \cite{gabele2019tropical}. I thank Michel van Garrel, Penka Georgieva, Rahul Pandharipande, Vivek Shende and Bernd Siebert for invitations to conferences and seminars where this work has been presented and for related discussions.
I thank Yuan Yao for questions on a first version of the paper. Finally, I thank the anonymous referee for a very careful reading and many suggestions of improvement of the exposition.

I acknowledge the support of Dr.\ Max R\"ossler, the Walter Haefner Foundation and the ETH Z\"urich
Foundation.

\section{The scattering diagram $S(\fD^{\iin}_{u,v})$}\label{section_scattering_diagram}
In \cref{section_local_scattering} we review the notion of local scattering diagram.
In \cref{section_scattering_definitions} we give a definition of scattering diagram specifically designed for our future needs. A slightly unusual
feature is the consideration of real-valued and not just integral-valued
functions $\varphi$ as `regulator' of the scattering computations.
In \cref{section_initial} we give a general construction of initial scattering diagrams 
$\fD^{\iin}$ and of their consistent completions $S(\fD^{\iin})$. We specialize 
this construction in  \cref{section_scattering_final} in order to define scattering diagrams $S(\fD^\iin_{u,v})$, $S(\fD^\iin_{q^\pm})$, and $S(\fD^\iin_{\cl^\pm})$.
In \cref{subsection_psi(1)} we establish a symmetry property of the scattering diagram $S(\fD^\iin_{u,v})$.

 We follow to various degrees the notation and conventions of the references 
\cite{MR2846484, MR2667135, MR2722115, bousseau2018quantum_tropical}, except that what we call local scattering diagrams are the scattering diagrams of \cite{MR2846484}, and what we call scattering diagrams are the wall structures
of \cite{MR2846484}. This slight change in terminology is designed in order to avoid possible confusions with the terminology usually used in the context of Bridgeland stability conditions, which will enter the story in \cref{section_scattering_p2}.

\subsection{Local scattering diagrams}
\label{section_local_scattering}
In this section, we fix
$M \simeq \Z^2$ a two-dimensional lattice,
\begin{align*} \varphi \colon M &\longrightarrow \R \\
 m &\longmapsto \varphi(m)\,,\end{align*}
an additive real-valued function on $M$,
and 
\[\fg=\bigoplus_{m \in M} \fg_m \,,\] 
a $M$-graded Lie algebra over 
$\Q$ (that is, with $[\fg_{m},\fg_{m'}]\subset 
\fg_{m+m'}$) such that 
$[\fg_m,\fg_{m'}]=0$ if 
$m$ and $m'$ are collinear. 

As $\varphi$ is additive and $\fg$ is $M$-graded, the subspace $\fg_\varphi \subset \fg$ defined by 
\[ \fg_\varphi := \bigoplus_{\substack{m \in M\\ \varphi(m)>0}} \fg_m\]
is a Lie subalgebra of $\fg$.
For every $k \in \R_{>0}$, we denote by
$\fm_k$ the Lie ideal 
\[ \bigoplus_{\substack{m \in M\\\varphi(m)>k}} \fg_m\]
of
$\fg_{\varphi}$ and by $\fg_{\varphi}^k := \fg_{\varphi}/\fm_k$ the quotient Lie algebra. We say that an element $x \in \fg_\varphi$ is of \emph{order}$>k$
if $x \in \fm_k$.
We denote by 
\[ G_\varphi^k := \exp(\fg_{\varphi}^k)\] the group whose elements are formal exponentials $\exp(g)$ of elements of $\fg_{\varphi}^k$, and whose product is given by the Baker-Campbell-Hausdorff formula
\[ \exp(g_1) \exp(g_2)=\exp(g_1+g_2+\frac{1}{2}[g_1,g_2]+\dots)\,. \]
Because we work modulo $\fm_k$, the Baker-Campbell-Hausdorff contains only finitely many nonzero terms and so this definition of $G_{\varphi}^k$ makes sense.

\begin{defn} \label{def_local_naked_ray}
For every nonzero $m \in M$, a \emph{local naked ray of class $m$} is a subset of $M_\R=M \otimes \R$ of the form either 
$\R_{\geqslant 0}m$ or $-\R_{\geqslant 0}m$.
\end{defn}

\begin{defn} \label{def_local_ray}
For every nonzero $m \in M$, a \emph{local ray $\fd$ of class $m$} 
for
$(M, \varphi,
\fg)$ is a pair 
$(|\fd| , H_\fd)$, where: 
\begin{itemize}
    \item[(i)] $|\fd|$ is a local naked ray of class $m$.
    \item[(ii)] $H_\fd$ is a nonzero element of $\fg_m$.
\end{itemize}
The local ray $\fd=(|\fd|, H_\fd)$ of class $m$ is \emph{outgoing} if 
$|\fd| = -\R_{\geqslant 0}m$, and \emph{ingoing} if 
$|\fd|=\R_{\geqslant 0} m$.
The local naked ray $|\fd|$ is called the support of the local ray $\fd$.
\end{defn}

It follows from Definition 
\ref{def_local_ray} that the class of a local ray is uniquely determined by the local ray. 

\begin{defn}
We denote by $m_\fd \in M$ the class of a local ray $\fd$.
\end{defn}

\begin{defn} \label{def_local_scattering}
A \emph{local scattering diagram} for 
$(M, \varphi,\fg)$ is a collection $\fD$ of local rays $\fd=(|\fd|, H_\fd)$ for
$(M,\varphi,\fg)$, such that:
\begin{enumerate}
	\item[(i)] For every nonzero $m \in M$, there is at most one ingoing local ray of class $m$ in $\fD$, and at most one outgoing ray of class $m$ in $\fD$.
    \item[(ii)] For every ray $\fd=(|\fd|, H_\fd)$ in $\fD$, we have
    $\varphi(m_\fd) > 0$. In particular, $H_\fd \in \fg_{\varphi}$.
    \item[(iii)] For every $k \in \R_{\geqslant 0}$, there are only finitely many rays 
$\fd=(|\fd|, H_\fd)$ in $\fD$ with 
\[\varphi(m_\fd) \leqslant k\,,\] 
that is, with
\[H_\fd \neq 0 \mod \fm_k\,.\]
\end{enumerate}
\end{defn}

\begin{defn} \label{def_autom}
Let $\fD$ be a local scattering diagram
for $(M, \varphi,\fg)$. For every ray
$\fd=(|\fd|,H_\fd)$ of $\fD$ and for every positive integer $k$, we denote by $\Phi_{\fd,k}$ the group element
\[\Phi_{\fd,k} \coloneqq \exp(H_\fd) \in G_{\varphi}^k   \,,\]
where we view $H_\fd$ as an element of $\fg_{\varphi}^k$.
\end{defn}

\begin{defn}
Let $\fD$ be a local scattering diagram
for $(M, \varphi,\fg)$. We fix a positive integer $k$, and a smooth path 
\[\fp \colon [0,1] \longrightarrow M_\R - \{0\}\,\]
\[\tau \longmapsto \fp(\tau)\,,\]
with transverse intersection with respect to all the rays $\fd=(|\fd|,H_\fd)\in \fD$
with $H_{\fd} \neq 0 \mod \fm_k$. Let $\fd_1, \dots, \fd_N$ be the successive rays $\fd$ of $\fD$ with $H_\fd \neq 0 \mod \fm_k$ intersected by the path 
$\fp$ at times $\tau_1 \leqslant \dots \leqslant \tau_N$.
Then, 
the \emph{order $k$ path-orderd product along $\fp$}, denoted by $\Phi_{\fp,k}^{\fD}$, is the following product in the group $G_\varphi^k$:
\[\Phi_{\fp,k}^{\fD} \coloneqq \Phi_{\fd_N,k}^{\epsilon_N} \dots \Phi_{\fd_1,k}^{\epsilon_1}\,,\]
where, for every $j=1, \dots, N$,
\[ \epsilon_j \coloneq \mathrm{sign} (\det (\fp'(\tau_j),m_{\fd_j}))\in \{\pm 1\} \,.\]
\end{defn}

\textbf{Remarks:} 
\begin{itemize}
\item[(i)] The ordering of the rays 
$\fd_1, \dots, \fd_N$ by times of intersection with the path is slightly ambiguous as different rays can have the 
same support. By the assumption $[\fg_m,\fg_{m'}]=0$ if 
$m$ and $m'$ are collinear, the group elements associated to rays with the same support commute. It follows that 
$\Phi_{\fp,k}^{\fD}$ is well-defined despite this ambiguity.
\item[(ii)]  The definition of $\Phi_{\fp,k}^{\fD}$ makes sense as there are only finitely many rays $\fd$ in $\fD$
with $H_\fd \neq 0 \mod \fm_k$ by condition 
(2) of Definition \ref{def_local_scattering}.
\end{itemize}

\begin{defn} \label{def_local_consistency_order_k}
Let $k$ be a nonnegative integer.
A local scattering diagram $\fD$
for $(M, \varphi,\fg)$ is \emph{consistent at order $k$}
if for every smooth loop, that is,  smooth path
\[\fp \colon [0,1] \longrightarrow M_\R-\{0\}\,,\] 
\[\tau \longmapsto \fp(\tau)\,,\]
with $\fp(0)=\fp(1)$, with transverse intersection with respect to all the rays $\fd=(|\fd|,H_\fd)$ in $\fD$
with $H_\fd \neq 0 \mod \fm_k$, the order $k$ path-ordered product along $\fp$ is the identity:
\[\Phi_{\fp,k}^{\fD}=1 \in G_\varphi^k\,.\]
\end{defn}

In order to check order 
$k$ consistency of a local scattering diagram $\fD$, it is enough to show that 
$\Phi_{\fp,k}^{\fD}=1$ for $\fp$ a simple smooth loop in $M_\R-\{0\}$ encircling $0$ and with transverse intersection with respect to all rays of $\fD$.

\begin{defn} \label{def_local_consistency}
A local scattering diagram $\fD$ for $(M, \varphi,\fg)$ is \emph{consistent} if it is consistent at order $k$
for every nonnegative integer $k$.
\end{defn}

The following Proposition \ref{prop_local_consistent_completion}
states that, under some condition on the order of the initial rays, any local scattering diagram can be completed in a consistent one.
The shape of this result goes back to Kontsevich-Soibelman \cite[Theorem 6]{MR2181810}. Many variants exist in the literature, see for example 
\cite[Theorem 1.4]{MR2667135}. 
We write down below the 
`usual' proof in order to check that it goes through our slightly unusual conventions, such that the fact that $\varphi$ is $\R$-valued (and not $\Z$-valued).
 
\begin{prop} \label{prop_local_consistent_completion}
We fix $c$ a real number such that $c \geqslant 1$.
Let $\fD$ be a local scattering diagram for 
$(M,\varphi,\fg)$ such that, for every ray $\fd$ in $\fD$, we have $\varphi(m_\fd) \geqslant c$. Then, there exists a unique sequence 
$(S_k(\fD))_{k \in \NN}$ of local scattering diagram
for $(M,\varphi,\fg)$,
such that:
\begin{itemize}
\item[(i)] $S_0(\fD)$ is the set of ingoing rays of $\fD$.
\item[(ii)] For every $k \in \NN$, $S_{k+1}(\fD)$ is obtained by adding to $S_k(\fD)$ rays 
$\fd$ such that $k<\varphi(m_\fd) \leqslant k+1$ and $\varphi(m_\fd) \geqslant c$.
\item[(iii)] For every $k \in \NN$, 
$S_k(\fD)$ is consistent at order $k$.
\end{itemize}
For every $k \in \NN$, we call 
$S_k(\fD)$ the order $k$ 
\emph{consistent completion} of $\fD$. We denote by $S(\fD)$ the limiting consistent local scattering diagram for
$(M,\varphi,\fg)$
obtained for $k \rightarrow +\infty$, and we call it the \emph{consistent completion} of $\fD$.
\end{prop}

\begin{proof}
We prove the existence and uniqueness of 
$S_k(\fD)$ by induction on $k \in \NN$.

For $k=0$, we take for $S_0(\fD)$ the union of ingoing rays of $\fD$. By condition (ii) of Definition \ref{def_local_scattering}, for every ray $\fd=(|\fd|,H_\fd)$ of $\fD$, we have $H_\fd=0 \mod \fm_0$, so $\Phi_{\fd,0}=1$. In particular, $S_0(\fD)$ is consistent at order zero. The uniqueness of $S_0(\fD)$ is clear.

Let $k \in \NN$ be such that $S_k(\fD)$ has been constructed and proved to be unique. We wish to construct and prove the uniqueness of $S_{k+1}(\fD)$. Let 
\[\fp \colon [0,1] \rightarrow M_\R -\{0\}\,,\]
be a simple anticlockwise loop around $0 \in M_\R$. By the induction hypothesis, 
$S_k(\fD)$ is consistent at order $k$ and so $\Phi_{\fp,k}^{S_k(\fD)}=1$
in $G_\varphi^k$. 
It follows that the element $\Phi_{\fp,k+1}^{S_k(\fD)}$
of $G_\varphi^{k+1}$ is of the form
\[ \Phi_{\fp,k+1}^{S_k(\fD)}
=\exp(H^{(k+1)}) \,\]
where $H^{(k+1)} \in \fg_{\varphi}^{k+1}$ satisfies 
$H^{(k+1)}=0 \mod \fm_k$, and so is of the form
\[ H^{(k+1)}=\sum_{\substack{m \in M\\ 
k<\varphi(m) \leqslant k+1
}} g_m\,,\]
where $g_m \in \fg_m$.

It follows from $\varphi(m_\fd) \geqslant
c \geqslant 1$ for every ray $\fd$ of 
$S_k(\fD)$, from condition (iii) of Definition \ref{def_local_scattering} satisfied by $S_k(\fD)$, and from the Baker-Campbell-Hausdorff formula defining the product in $G_\varphi^{k+1}$ that the set $M^{(k+1)}$ of $m \in M$ such that 
$g_m \neq 0$ is finite, and that, for every such $m \in M^{(k+1)}$, we have $\varphi(m) \geqslant c$.

We define $S_{k+1}(\fD)$ as being the set of the following rays:
\begin{itemize}
\item[(i)] The ingoing rays of $\fD$, which are also the ingoing rays of $S_k(\fD)$ by induction hypothesis.
\item[(ii)] The outgoing rays of $S_k(\fD)$.
\item[(iii)] For every $m \in M^{(k+1)}$,
 the outgoing ray 
\[(-\R_{\geqslant 0}m, -g_m)\,.\]
\end{itemize} 

By construction, $S_{k+1}(\fD)$ is consistent at order $k+1$. Indeed, by the Baker-Campbell-Hausdorff formula defining the product in 
$G_\varphi^{k+1}$, it is 
consistent up to terms whose order is at least the one of the commutators of the form  
\[[H_\fd, g_m] \in \fg_{m_\fd+m}\,,\]
where $\fd$ is either an incoming ray or an outgoing ray of $S_k(\fD)$, so with 
$\varphi(m_\fd) \geqslant c$, and where 
$m \in M^{(k+1)}$ so $\varphi(m)>k$,
or of the form  
\[
[g_m, g_{m'}] \in \fg_{m+m'}\,,\]
where $m, m' \in M^{(k+1)}$, and so with $\varphi(m+m')\geqslant 2c\geqslant 2$ and $\varphi(m+m')>2k$. 
In the first case, the order of the commutator is strictly greater than $c+k \geqslant k+1$.
In the second case, 
if $k=0$, we have $\varphi(m+m')\geqslant 2>1=k+1$, and if $k \geqslant 1$, we have $\varphi(m+m')>2k 
\geqslant k+1$, and so the order of the commutator is also strictly greater than $k+1$ in any case.

It remains to show the uniqueness of 
$S_{k+1}(\fD)$. The outgoing rays $\fd$ of 
$S_{k+1}(\fD)$ with $\varphi(m_\fd) \leqslant k$ are uniquely fixed by the uniqueness of $S_k(\fD)$ given by the induction hypothesis. The same study of 
orders of commutators as above shows that the order $k+1$ group element  associated with 
the outgoing rays $\fd$ of $S_{k+1}(\fD)$
with $\varphi(m_\fd)>k$ commute with the order $k+1$ group element associated with rays of $S_k(\fD)$, and so outgoing rays $\fd$ of $S_{k+1}(\fD)$ are uniquely determined  by the condition of consistency at order $k+1$ in terms of $g_m$, as above.  
\end{proof}

\textbf{Remarks:} 
\begin{itemize}
\item[(i)] The proof of Proposition \ref{prop_local_consistent_completion} is the 
`usual' proof of existence of a consistent completion of a scattering diagram, as in \cite[Theorem 6]{MR2181810}
or 
\cite[Theorem 1.4]{MR2667135}. 
The only possibly subtle point is that $\varphi$ is 
$\R$-valued, whereas in the `usual' situation, $\varphi$ is $\Z$-valued. One might worry that this could allow a phenomenon of accumulation of values of $\varphi$ at some finite value, preventing the induction step to go from $k$ to $k+1$. 
The above proof shows that it does not happen under the additional assumption that $\varphi(m_\fd) \geqslant 1$ for every ray 
$\fd$ in $\fD$.
\item[(ii)] It is clear from the proof of Proposition
\ref{prop_local_consistent_completion} that the construction of $S(\fD)$ from $\fD$ is completely algorithmic, involving finitely many computations for any fixed order $k$.
\end{itemize}

The following Lemma motivates the scattering terminology: if ingoing rays come in a given cone of directions, then the outgoing rays of the consistent completion are `emitted forward'.

\begin{lem} \label{lem_scattering_forward}
Let $\fD$ be a local scattering diagram for 
$(M,\varphi,\fg)$ consisting entirely of ingoing rays contained in a strictly convex cone of $M_\R$ delimited by two ingoing rays 
$\fd_1^\iin$ and $\fd_2^\iin$. Then, all the outgoing rays of $S(\fD)$ are contained in the strictly convex cone $-\R_{\geqslant 0}m_{\fd_1^\iin}-\R_{\geqslant 0}m_{\fd_2^\iin}$, delimited by two outgoing rays 
$\fd_1^\oout$ and $\fd_2^\oout$, respectively of class 
$m_{\fd_1^\iin}$ and $m_{\fd_2^\iin}$, and with $H_{\fd_1^\oout}=H_{\fd_1^\iin}$
and $H_{\fd_2^\oout}=H_{\fd_2^\iin}$.
\end{lem}

\begin{proof}
We use the perturbation trick of 
\cite[\S 1.4]{MR2667135} in order to reduce the question to the case of an `elementary' local scattering diagram, as 
\cite[Lemma 1.9]{MR2667135} 
(see
\cite[Lemma 2.9]{mandel2015scattering}
for the case of a general Lie algebra $\fg$), for which the result is clear. Indeed, a consistent elementary local scattering diagram has two ingoing rays of class $m_1$ and $m_2$, and three outgoing rays of class $m_1$, $m_2$ and $m_1+m_2$. Furthermore, the outgoing rays of class $m_1$ and $m_2$ are continuations of the two ingoing rays.
\end{proof}

\subsection{Scattering diagrams} \label{section_scattering_definitions}

In this section, we fix
$M \simeq \Z^2$ a two-dimensional lattice
and 
\[\fg=\bigoplus_{m \in M} \fg_m \,,\] 
a $M$-graded Lie algebra over 
$\Q$ (that is, with $[\fg_{m},\fg_{m'}]\subset 
\fg_{m+m'}$) such that 
$[\fg_m,\fg_{m'}]=0$ if 
$m$ and $m'$ are collinear. 

We also fix $U$ an open subset of $\R^2$, of closure $\bar{U}$ in $\R^2$. For every $\sigma \in U$, we think about $M_\R
=M \otimes \R$ as being the tangent space to $U$ at $\sigma$.

\begin{defn} \label{def_naked_ray}
For every $m \in M$, a \emph{naked ray of class} $m$ in $U$ is a subset $|\fd|$ of $\bar{U}$
of the form \[|\fd|=\Init(\fd)-\R_{\geqslant 0} m \,,\] for some 
$\Init(\fd) \in \bar{U}$, or of the form 
\[|\fd|=\Init(\fd)-[0,T_\fd]m\,,\] 
for some 
$\Init(\fd) \in \bar{U}$ and some $T_\fd \in \R_{>0}$.

In both cases, we call $\Init(\fd)$ the initial point of $|\fd|$. If 
$|\fd|$ is bounded, that is, if 
$|\fd|=\Init(\fd)-[0,T_\fd]m$, we call 
$\Init(\fd)-T_\fd m$ the endpoint of $|\fd|$.
\end{defn}

\begin{defn} \label{def_ray}
For every $m \in M$, a \emph{ray $\fd$ of class} $m$ in $U$  for 
$(M,\fg)$ is a pair
$(|\fd|, H_\fd)$, where 
\begin{itemize}
    \item[(i)] $|\fd|$ is a naked ray of class $m$ in $U$,
    \item[(ii)] $H_\fd$ is a nonzero element of 
    $\fg_m$.
\end{itemize}
\end{defn}

It follows from Definition 
\ref{def_ray} that the class of a ray is uniquely determined by the ray. 

\begin{defn}
We denote by
$m_\fd \in M$ the class of a ray $\fd$.
\end{defn}

In \cref{section_local_scattering}, the notion of local scattering diagram depends on a choice of function $\varphi \colon M \rightarrow \R$.
This function plays the role of `regulator' for local scattering diagrams: for every $k$, there are only finitely many local rays of class $m$
with $\varphi(m) \leq k$. In this section, we are in a global setting and 
the `regulator' function will depend on the point 
$\sigma \in U$: for every $\sigma \in U$, we introduce a function 
$\varphi_\sigma$ that will play the role of `regulator' for the part of the 
scattering diagram localized around $\sigma$.

\begin{defn} \label{def_varphi}
For every $\sigma=(x,y) \in \bar{U}$, we denote
\begin{align*} \varphi_\sigma \colon M_\R &\longrightarrow \R \\
(a,b) &\longmapsto 2(-ax-b) \,.\end{align*}
\end{defn}

The function that we denote by $\varphi_\sigma$ should be denoted by
$-\varphi_\sigma$ using the notation of \cite{MR2846484, MR2722115}. Given our conventions in 
Definition \ref{def_scattering}, the descriptions are equivalent. The notation of \cite{MR2846484, MR2722115} is the most natural from the point of view of the mirror construction, but, for the purposes of the present paper, we prefer to incorporate a minus sign once and for all in the definition of $\varphi_\sigma$ in order to have less signs in the later formulas. 
For every $\sigma \in \bar{U}$, the function
$\varphi_\sigma$ restricted to $M$ is an additive real-valued function on $M$.

\begin{lem}\label{lem_increasing}
Let $\fd=(|\fd|,H_\fd)$ be a ray in $U$.
Then the function given by 
\begin{align*}
\R_{\geqslant 0} &\longrightarrow \R\\
\tau &\longmapsto \varphi_{\Init(\fd)-\tau m_\fd}(m_\fd)
\end{align*}
if $|\fd|=\Init(\fd)-\R_{\geqslant 0}m_\fd$, or by 
\begin{align*}
[0,T_\fd] &\longrightarrow \R\\ 
\tau &\longmapsto \varphi_{\Init(\fd)-\tau m_\fd}(m_\fd)\end{align*}
if $|\fd|=\Init(\fd)-[0,T_\fd]m_\fd$,
is increasing. More precisely, writing
$m_\fd=(a,b)$, this function is strictly increasing if $a \neq 0$ and constant if $a=0$.

Moreover, if $|\fd|=\Init(\fd)-[0,T_\fd]m_\fd$, writing $\Init(\fd)=(x_i,y_i)$ and 
$\Init(\fd)-T_\fd m_\fd=(x_f,y_f)$, we have 
\[ \varphi_{\Init(\fd) -T_\fd m_\fd}(m_\fd)
-\varphi_{\Init(\fd)}(m_\fd) \geqslant 2|x_f-x_i|\,.\]
\end{lem}

\begin{proof}
We write $m_\fd=(a,b)$. If $a \geqslant 0$, then the $x$-coordinate $x(\tau)$ of $\Init(\fd)-\tau m_\fd$
is an decreasing function of $\tau$, and so 
$\varphi_{\Init(\fd)-\tau m_\fd}(m_\fd)=2(-ax(\tau)-b)$ is an increasing function of $\tau$, strictly increasing if $a \neq 0$ and constant if $a=0$.

Similarly, if $a<0$, then the $x$-coordinate $x(\tau)$ of $\Init(\fd)-\tau m_\fd$
is an increasing function of $\tau$, and so 
$\varphi_{\Init(\fd)-\tau m_\fd}(m_\fd)=2(-ax(\tau)-b)$ is a strictly increasing function of $\tau$.
\end{proof}

\begin{defn}\label{def_normalized}
A collection $\fD$ of rays 
$\fd=(|\fd|,H_\fd)$ in $U$ for 
$(M,\fg)$ is \emph{normalized} if:
\begin{enumerate}
\item[(i)] For every $\sigma \in U$ and for every nonzero $m \in M$, there is at most 
	one ray $\fd$ in $\fD$ of class $m$ such that 
	$\sigma$ belongs to the interior of 
	$|\fd|$.
\item[(ii)] There do no exist rays $\fd_1=(|\fd_1|,H_{\fd_1})$ and $\fd_2=(|\fd_2|,H_{\fd_2})$ in $\fD$ such that the endpoint of $|\fd_1|$ coincides with the initial point of $|\fd_2|$, and such that 
$H_{\fd_1}=H_{\fd_2}$.
\end{enumerate}
\end{defn}

Given a collection 
$\fD$ of rays in $U$ for 
$(M,\fg)$, there is a canonical way to produce a normalized collection of rays that we call the 
\emph{normalization} of $\fD$. 
The normalization of $\fD$ is obtained by repeated use of the following operations:
\begin{enumerate}
\item[(i)] If $\fd_1=(|\fd_1|,H_{\fd_1})$ and $\fd_2=(|\fd_2|,H_{\fd_2})$ are two rays of the same class $m$ with $|\fd_1|\cap |\fd_2|$ of nonempty interior, replace 
$\fd_1$ and $\fd_2$ by the rays 
$(|\fd_1|-|\fd_1 \cap \fd_2|, H_{\fd_1})$,
$(|\fd_2|-|\fd_1 \cap \fd_2|, H_{\fd_2})$,
$(|\fd_1| \cap |\fd_2|, H_{\fd_1}+H_{\fd_2})$.
\item[(ii)] If $\fd_1$ and $\fd_2$ are two rays 
of $\fD$ such that the endpoint of 
$|\fd_1|$ coincides with the initial point of $\fd_2$ and such that 
$H_{\fd_1}=H_{\fd_2}$, replace 
$\fd_1$ and $\fd_2$ by the ray 
$(|\fd_1|\cup|\fd_2|,H_{\fd_1})$.
\end{enumerate}

\begin{defn}\label{def_scattering}
A \emph{scattering diagram} on $U$
for $(M,\fg)$ is a collection $\fD$ of rays
$\fd=(|\fd|,H_\fd)$ in $U$ for $(M,\fg)$, such that:
\begin{enumerate}
	\item[(i)] The collection $\fD$ is normalized.
    \item[(ii)] If $\fd=(|\fd|, H_\fd)$ is a ray in $\fD$, then $\varphi_\sigma(m_\fd)>0$ for every $\sigma \in |\fd|\cap U$.
    \item[(iii)] For every compact set $K$ in $U$ and for every $k \in \R_{\geqslant 0}$, the set of rays
$\fd=(|\fd|, H_\fd)$ in $\fD$ such that there exists $\sigma \in |\fd| \cap K$
such that $\varphi_\sigma (m_\fd)\leqslant k$ is finite.
\end{enumerate}
\end{defn}

\begin{defn}
The \emph{singular locus} $\Sing(\fD)$ of a scattering diagram $\fD$ is the set of the 
initial points of the rays and of the non-trivial intersection points of the rays, 
that is,
\[\Sing(\fD)\coloneq \bigcup_{\fd \in \fD} \{\Init(\fd)\} \cup \bigcup_{\substack{\fd_1, \fd_2 \in \fD \\ \dim |\fd_1| \cap |\fd_2| =0}} (|\fd_1| \cap |\fd_2|) \,.\]
\end{defn}

Let $\fD$ be a scattering diagram on $U$ for $(M, \fg)$, 
and let $\sigma \in U$. We explain how to define a local scattering 
diagram $\fD_\sigma$ for $(M, \varphi_\sigma, \fg)$ in the sense of \cref{section_local_scattering}.
The local scattering diagram $\fD_\sigma$ is a local picture of $\fD$ around the point $\sigma$, $M_\R=M \otimes \R$ being identified with the tangent space to $U$ at $\sigma$.
More precisely, local rays of $\fD_\sigma$ are constructed from rays of $\fD$ as follows:
\begin{itemize}
\item[(i)]
Let $\fd=(|\fd|, H_\fd)$ be a ray of class $m_\fd$ of $\fD$ such that $\Init(\fd)=\sigma$. Then we define 
$|\fd_\sigma| \coloneq -\R_{\geqslant 0}m_\fd$ and we view $\fd_\sigma \coloneq (|\fd_\sigma|, H_\fd)$ as a (outgoing) local ray of $\fD_\sigma$.
\item[(ii)]
Let $\fd=(|\fd|,H_\fd)$ be a ray of class $m_\fd$ of $\fD$ such that $\sigma \in |\fd|$ and 
$\sigma \neq \Init(\fd)$. Then we define $|\fd_{\sigma,out}| \coloneq -\R_{\geqslant 0}m_\fd$ and we view $\fd_{\sigma,out} \coloneq (|\fd_{\sigma,out}|, H_\fd)$ as a (outgoing) local ray of $\fD_\sigma$. We also define 
$|\fd_{\sigma,in}| \coloneq \R_{\geqslant 0}m_\fd$ and we view 
$\fd_{\sigma,in} \coloneq (|\fd_{\sigma,in}|, H_\fd)$ as a (ingoing) local ray of $\fD_\sigma$.
\end{itemize}

\begin{lem} \label{lem_local_scattering}
The collection of local rays $\fD_\sigma$ is a local scattering diagram for
$(M, \varphi_\sigma,\fg)$
in the sense of Definition \ref{def_local_scattering}.
\end{lem}

\begin{proof}
Conditions (i)-(2)-(3) of Definition \ref{def_local_scattering} for $\fD_\sigma$ follow respectively from conditions (i)-(2)-(3) of Definition 
\ref{def_scattering} for $\fD$.
\end{proof}

\begin{defn} \label{def_consistency_order_k}
Let $k$ be a nonnegative integer.
A scattering diagram $\fD$ on $U$
for $(M,\fg)$
is 
\emph{consistent at order $k$} if, for every $\sigma \in U \cap \Sing(\fD)$, the local scattering 
diagram $\fD_\sigma$ for 
$(M,\varphi_\sigma,\fg)$ is consistent at order $k$ in the sense of Definition \ref{def_local_consistency_order_k}.
\end{defn}

\begin{defn} \label{def_consistency}
A scattering diagram $\fD$
on $U$
for $(M,\fg)$
is 
\emph{consistent} if it is consistent at order $k$ for every nonnegative integer $k$, or equivalently if for every $\sigma \in U \cap \Sing(\fD)$, the local scattering 
diagram $\fD_\sigma$ for 
$(M,\varphi_\sigma,\fg)$ is consistent in the sense of Definition \ref{def_local_consistency}.
\end{defn}

\subsection{The scattering diagrams $S(\fD^\iin)$}\label{section_initial}
In this section we construct specific examples of scattering diagrams in the sense of  \cref{section_scattering_definitions}.
We consider the open subset $U$ of $\R^2$ given by 
\[ U=\{(x,y) \in \R^2 \,|\, y >-\frac{x^2}{2}\}\,.\]
Its closure in $\R^2$ is 
\[ \bar{U}=\{(x,y) \in \R^2 \,|\, y \geqslant -\frac{x^2}{2}\}\,,\]
and so its boundary $\partial U \coloneq \bar{U}-U$ is the parabola of equation $y=-\frac{x^2}{2}$.

\begin{lem}\label{lem_stay_in_U}
Let $\sigma \in U$ and 
$m \in M$ such that 
$\varphi_\sigma(m)>0$. Then
\[\sigma - \R_{\geqslant 0} m \subset U \,.\]
\end{lem}

\begin{proof}
We write $\sigma=(x,y) \in \R^2$ and 
$m=(a,b) \in \Z^2$. For every $\tau \in \R_{\geqslant 0}$, we have 
$\sigma -\tau m=(x-\tau a, y-\tau b)$ and
\[(x-\tau a)^2+2(y-\tau b)=x^2+2y+2\tau(-ax-b)+\tau^2 a^2 \geqslant x^2+2y+2\tau(-ax-b)\,.\]
As $\sigma \in U$, we have $x^2+2y>0$. 
On the other hand, we have $\varphi_\sigma(m)=2(-ax-b) \geqslant 0$. It follows that $(x-\tau a)^2+2(y-\tau b)>0$, that is, 
$\sigma -tm \in U$.
\end{proof}

\begin{lem} \label{lem_x2+2y_increasing}
Let $\fd$ be a ray in $U$ such that 
$\varphi_{\sigma}(m_\fd)>0$ for every 
$\sigma \in |\fd| \cap U$. 
Then, writing 
$\Init(\fd)-\tau m_\fd=(x(\tau),y(\tau))$ for every
$\tau \in \R_{\geqslant 0}$, the function given by  
\begin{align*} \R_{\geqslant 0} &\longrightarrow \R \\
\tau &\longmapsto x(\tau)^2+2y(\tau)
\end{align*}
if $|\fd|=\Init(\fd)-\R_{\geqslant 0} m_\fd$, or by 
\begin{align*} [0,T_\fd]& \longrightarrow \R\\
 \tau &\longmapsto x(\tau)^2+2y(\tau)
\end{align*}
if $|\fd|=\Init(\fd)-[0,T_\fd]m_\fd$, is strictly increasing.
\end{lem}

\begin{proof}
Writing $m_\fd=(a,b) \in \Z^2$, we have, for every 
$\sigma=(x(\tau),y(\tau)) \in |\fd| \cap U$, 
\[ \frac{d}{d\tau}(x(\tau)^2+2y(\tau))=2(-ax(\tau)-b)
=\varphi_{\sigma}(m_\fd)>0\,.\]
As $|\fd| \cap \bar{U}$ is at most a
point (indeed, $|\fd| \subset \bar{U}$ 
and $\partial{U}$ is a parabola), and 
$\tau \mapsto x(\tau)^2+2y(\tau)$ is continuous, we conclude that 
$\tau \mapsto x(\tau)^2+2y(\tau)$ is strictly increasing.
\end{proof}

\begin{defn}
For every $n \in \Z$, we denote \[s_n \coloneq (n,-\frac{n^2}{2}) \in \bar{U}-U \subset \R^2\,\] 
\[m_n^{-} \coloneq (1,-n) \in \Z^2\,\] 
\[m_n^{+} \coloneq (-1,n) \in \Z^2\,.\]
\end{defn}

\begin{defn} \label{def_first_scattering}
For every $n \in \Z$, we define naked rays
\[|\fd_n^{+}| \coloneq s_n - [0,\frac{1}{2}]m_n^+\,,\]
and 
\[|\fd_n^{-}| \coloneq s_n - [0,\frac{1}{2}]m_n^{-}\,.\]
\end{defn}

We have $|\fd_n^+|\cap U =  s_n - (0,\frac{1}{2}]m_n^+$ and $|\fd_n^{-}|\cap U \coloneq s_n - (0,\frac{1}{2}]m_n^-$, 
that is, $|\fd_n^+|$ and $|\fd_n^-|$ are contained in $U$, except for their initial point $s_n$ which is on the boundary of $U$. In particular, 
$|\fd_n^+|$ and $|\fd_n^-|$ are naked rays in $U$ in the sense of Definition 
\ref{def_naked_ray}. This is related to the fact that, for every 
$n \in \Z$, the line $\R m_n^-=\R m_n^+$
is the tangent at the point $s_n$ to the parabola 
$y=-\frac{x^2}{2}$, that is, to the boundary of $U$.

\begin{lem} \label{lem_first_scattering}
For every $n \in \Z$, we have 
\[|\fd_n^+| \cap |\fd_{n+1}^-|=s_n-\frac{1}{2}m_n^+=s_{n+1}-\frac{1}{2}m_{n+1}^-
=(n+\frac{1}{2}, -\frac{n^2}{2}-\frac{n}{2})\,.\]
\end{lem}

\begin{proof}
Elementary computation.
\end{proof}

\begin{lem} \label{lem_varphi_first-scattering}
For every $n \in \Z$, we have 
\[\varphi_{|\fd_n^+|\cap |\fd_{n+1}^-|}(m_n^+)
= \varphi_{|\fd_n^+|\cap |\fd_{n+1}^-|}(m_{n+1}^-)=1\,.\]
\end{lem}

\begin{proof}
Using Definition \ref{def_varphi} and Lemma \ref{lem_first_scattering}, we have
\[\varphi_{|\fd_n^+|\cap |\fd_{n+1}^-|}(m_n^+)=2(n+\frac{1}{2}-n)=1\,,\]
and 
\[\varphi_{|\fd_n^+|\cap |\fd_{n+1}^-|}(m_{n+1}^-)=2(-n-\frac{1}{2}+n+1)=1\,.\]
\end{proof}

\begin{figure}[h!]
\centering
\includegraphics[scale=0.8]{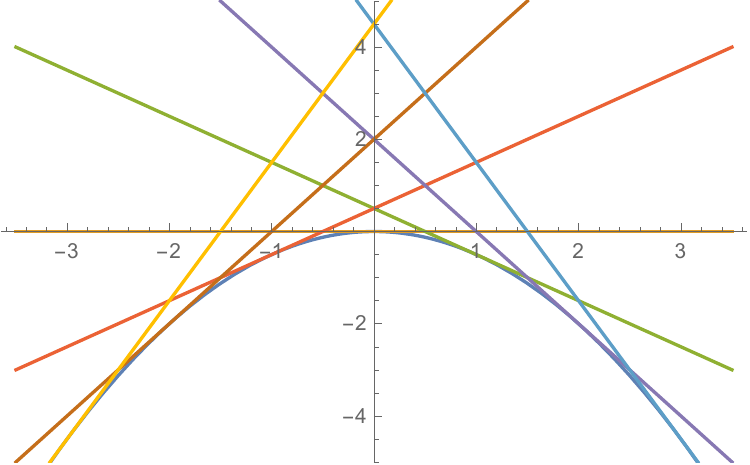}
\caption{The parabola $y=-\frac{x^2}{2}$ 
and the tangent lines
$\R m_n^-=\R m_n^+$ at the points $s_n$.}
\label{fig:universe}
\end{figure}

Let
\[\fg=\bigoplus_{m \in M} \fg_m \,,\] 
a $M$-graded Lie algebra over 
$\Q$ (that is with $[\fg_{m},\fg_{m'}]\subset 
\fg_{m+m'}$) such that 
$[\fg_m,\fg_{m'}]=0$ if 
$m$ and $m'$ are collinear. 
For every $n \in \Z$ and for every integer $\ell \geqslant 1$, 
we fix nonzero elements 
\[H_{n,\ell}^+ \in \fg_{\ell m_n^+}\]
and 
\[H_{n,\ell}^- \in \fg_{\ell m_n^-}\,.\]
For every $n \in \Z$ and for every integer 
$\ell \geqslant 1$, $(|\fd_n^+|, H_{n,\ell}^+)$
and $(|\fd_n^-|, H_{n,\ell}^-)$
are rays in $U$ in the sense of Definition 
\ref{def_ray}.

\begin{lem}
The set $\fD^{\iin}$ of rays 
$\fd_{n,\ell}^+ \coloneq (|\fd_n^+|, H_{n,\ell}^+)$, $\fd_{n,\ell}^{-} \coloneq (|\fd_n^-|, H_{n,\ell}^-)$, for all 
$n \in \Z$ and for every integer $\ell \geqslant 1$, is a scattering diagram on $U$ for $(M,\fg)$ in the sense of Definition \ref{def_scattering}.
\end{lem}

\begin{proof}
As the rays of $\fD^\iin$ have distinct classes ($\ell m_n^+=(-\ell,\ell n)$ for
$\fd_{n,\ell}^+$ and 
$\ell m_n^-=(\ell, -\ell n)$ for 
$\fd_{n,\ell}^-$),
and as there do not exist rays $\fd_1$ and $\fd_2$ of $\fD^\iin$ with the endpoint of
$\fd_1$ coinciding with the initial point of $\fd_2$, $\fD^{\iin}$ satisfies condition (i)
of Definition \ref{def_scattering}.

If $\sigma=(x,y) \in |\fd_n^+| \cap U$, then $x > n$, so, for every $\ell \geqslant 1$, $\varphi_\sigma(\ell m_n^+)=\ell (
x-n) >0$. Similarly, if $\sigma=(x,y) \in |\fd_n^-| \cap U$, then $x<n$, so, for every 
$\ell \geqslant 1$,  
$\varphi_\sigma(\ell m_n^-)=\ell(-x+n)>0$. It follows that $\fD^{\iin}$ satisfies condition (ii) of 
Definition \ref{def_scattering}.

Let $K$ be a compact in $U$ and $k \in \R_{\geqslant 0}$. As $K$ is compact in $U$, there exists $C>0$ such that $\sigma=(x,y) \in K$ implies $|x|<C$, and there exists $\epsilon >0$ such that 
$\sigma=(x,y) \in K$ implies $|x-n|>\epsilon$ for all $n \in \Z$. Therefore, if we have 
$\sigma \in K \cap |\fd_n^+|$ such that 
$\varphi_\sigma(\ell m_n^+) \leqslant k$, then 
$\ell (x-n) \leqslant k$, $|x|<C$, $|x-n|>\epsilon$, so
$|n|\leqslant k+C$ and $\ell \leqslant k/\epsilon$: in particular, there are finitely many such $n$ and $\ell$. 

Similarly, if we have 
$\sigma \in K \cap |\fd_n^-|$ such that 
$\varphi_\sigma(\ell m_n^-) \leqslant k$, then 
$\ell (n-x) \leqslant k$, $|x|<C$, $|x-n|>\epsilon$, so
$|n|\leqslant k+C$ and $\ell \leqslant k/\epsilon$. There are again only finitely many such $n$ and $\ell$. This shows that $\fD^{\iin}$ satisfies condition (iii)
of Definition \ref{def_scattering}.
\end{proof}

\begin{prop} \label{prop_consistent_completion}
There exists a unique sequence 
$(S_k(\fD^\iin))_{k \in \NN}$ of scattering
diagrams on $U$ for 
$(M,\fg)$ such that:
\begin{itemize}
\item[(i)] $S_0(\fD^\iin)=\fD^\iin$.
\item[(ii)] For every $k \in \NN$, 
$S_{k+1}(\fD^\iin)$ is obtained by adding to 
$S_k(\fD^\iin)$ rays $\fd$ such that
$|\fd|\subset U$,  $k < \varphi_{\Init(\fd)}(m_\fd) \leqslant k+1$ and $\varphi_{\Init(\fd)}(m_\fd) \geqslant 1$.
\item[(iii)] For every $k \in \NN$,
$S_k(\fD^\iin)$ is consistent at order $k$.
\end{itemize}
For every $k \in \NN$, we call $S_k(\fD^\iin)$ the order $k$ \emph{consistent completion} of $\fD^\iin$. We denote by $S(\fD^\iin)$ the limiting consistent scattering diagram on $U$ for 
$(M,\fg)$ and we call it the \emph{consistent completion} of 
$\fD^\iin$.
\end{prop}

\begin{proof}
We prove the existence and uniqueness of 
$S_k(\fD^\iin)$ by induction on 
$k \in \NN$.

For $k=0$, the scattering diagram $\fD^\iin$ is trivially consistent at order $0$, and we set $S_0(\fD^\iin) \coloneq \fD^\iin$.
According to
Lemma \ref{lem_first_scattering} and Lemma
\ref{lem_varphi_first-scattering},
the only $\sigma \in U \cap \Sing(\fD^\iin)$ are the points $|\fd_n^+| \cap |\fd_{n+1}^-|$, and at such point $\sigma$, the 
function 
$\varphi_\sigma$ evaluated on the classes of the ingoing rays of the corresponding local scattering diagram is equal to $1$.

Let $k \in \NN$ be such that 
$S_k(\fD^\iin)$ has been constructed and 
proved to be unique. We wish to construct
and prove the uniqueness of $S_{k+1}(\fD^\iin)$.

For every $\sigma \in U \cap \Sing(S_k(\fD^\iin))$, let $S_k(\fD^\iin)_\sigma$ be the corresponding local scattering diagram for
$(M,\varphi_\sigma,\fg)_\sigma$ given by Lemma \ref{lem_local_scattering}. By induction hypothesis, we know that $\varphi(m_\fd) \geqslant 1$ for every ray $\fd$ of $S_k(\fD^\iin)_\sigma$. Using Proposition \ref{prop_local_consistent_completion}, it follows that the local scattering diagram $S_k(\fD^\iin)_\sigma$
can be canonically completed in an 
order $k+1$ consistent local scattering diagram 
$S_{k+1}(S_k(\fD^\iin)_\sigma)$, by adding outgoing rays $\fd$ with $k<\varphi_\sigma(m_\fd) \leqslant k+1$ and $\varphi_\sigma(m_\fd) \geqslant 1$.

For every $\sigma \in U \cap \Sing(S_k(\fD^\iin))$ and for every 
$\fd=(-\R_{\geqslant 0}m_\fd, H_\fd)$ outgoing ray of $S_{k+1}(S_k(\fD^\iin)_\sigma)$ added to $S_k(\fD^\iin)$, we define a ray $\tilde{\fd}$ in $U$ by
\[ \tilde{\fd} \coloneq (\sigma-\R_{\geqslant 0} m_\fd, H_\fd)\,.\]
We have by construction 
$k < \varphi_{\Init(\tilde{\fd})}(m_{\tilde{\fd}}) 
\leqslant k+1$ and $\varphi_{\Init(\tilde{\fd})}(m_{\tilde{\fd}}) \geqslant 1$.
Remark that according to Lemma
\ref{lem_stay_in_U}, the condition $\varphi_\sigma(m_\fd) \geqslant 1 >0$ implies that 
$\sigma -\R_{\geqslant 0}m_\fd \subset U$, and so 
$\tilde{\fd}$ is indeed a ray in $U$.

We denote by $\tilde{S}_{k+1}(\fD^\iin)$ the union of the rays of 
$S_k(\fD^\iin)$ and of all the rays $\tilde{\fd}$ constructed above. 
We define $S_{k+1}(\fD^\iin)$ as being the normalization of 
$\tilde{S}_{k+1}(\fD^\iin)$
(see Remark following Definition 
\ref{def_normalized}), so that
$S_{k+1}(\fD^\iin)$ satisfies condition (i) of Definition \ref{def_scattering}

By construction, $S_{k+1}(\fD^\iin)$ satisfies condition (ii) of Definition 
\ref{def_scattering}. Therefore, in order to show that $S_{k+1}(\fD^\iin)$ is a scattering diagram on $U$ for $(M,\fg)$, 
it remains to check condition (iii) of Definition 
\ref{def_scattering}. It follows from the second part of Lemma 
\ref{lem_increasing} that for every compact 
subset $K$ of $U$, there exists 
a positive integer $N$ such that every ray 
$\fd =(|\fd|,H_\fd)$ in 
$S_{k+1}(\fD^\iin)$
comes by successive local scatterings from the initial rays $\fd_n^+$, $\fd_n^-$ with 
$|n| \leqslant N$. In particular, there are only finitely many such rays, and so condition (iii) of Definition \ref{def_scattering} is satisfied.

We claim that the scattering diagram $S_{k+1}(\fD^{\iin})$ is consistent at order $k+1$.
Indeed, let $\sigma \in U \cap \Sing(S_{k+1}(\fD^\iin))$ and let $S_{k+1}(\fD^{\iin})_\sigma$ be the corresponding local scattering diagram. 
If $\sigma \in U \cap 
\Sing (S_k(\fD^\iin))$, then $S_{k+1}(\fD^{\iin})_\sigma$ is obtained from $S_{k+1}(S_k(\fD^\iin)_\sigma)$ 
by adding rays $\fd$ with $\varphi_\sigma(m_\fd) >k$ and
$\varphi_\sigma(m_\fd) \geqslant 1$, coming in ingoing-outgoing pairs (local picture of a ray passing through $\sigma$: see Definition of 
$\fD_\sigma$ in \cref{section_scattering_definitions}).
By construction, $S_{k+1}(S_k(\fD^\iin)_\sigma)$
is consistent at order $k+1$. On the other hand, the group elements attached to the added rays commute with the group elements attached to the rays of $S_{k+1}(S_k(\fD^\iin)_\sigma)$
and with each other up to terms of order strictly greater than $k+1$, 
so the group elements attached to ingoing-outgoing pairs cancel each other and so 
$S_{k+1}(\fD^{\iin})_\sigma$ is consistent at order $k+1$. 

If $\sigma \notin U \cap 
\Sing (S_k(\fD^\iin))$, then all the rays 
$\fd$ of $S_{k+1}(\fD^\iin)$ are newly added rays, with $\varphi_\sigma(m_\fd) >k$
and $\varphi(m_\fd) \geqslant 1$, coming in ingoing-outgoing pairs (local picture of a ray passing through $\sigma$: see Definition of 
$\fD_\sigma$ in \cref{section_scattering_definitions}). The group elements attached to the added rays commute with each other up to terms of order strictly greater than $
\max(2k,2) \geqslant k+1$, so the group elements attached to ingoing-outgoing pairs cancel each other and so
$S_{k+1}(\fD^{\iin})_\sigma$ is consistent at order $k+1$.

It remains to prove the uniqueness of
$S_{k+1}(\fD^\iin)$.
Let $\fD$ be a scattering diagram on 
$U$ satisfying the same properties that 
$S_{k+1}(\fD^\iin)$ stated in Proposition
\ref{prop_consistent_completion}.
Let $\fd$ be a ray of $\fD$ which is not in $S_k(\fD^\iin)$. As $\varphi_{\Init(\fd)}
\leqslant k+1$, $\fd$ cannot be produced by a local scattering containing as incoming ray a ray of $\fD$ added to 
$S_k(\fD^\iin)$, and so $\Init(\fd)
\in U \cap \Sing(S_k(\fD^\iin))$.
It follows by consistency of 
$\fD$ and uniqueness of the local consistent completion given by Proposition 
\ref{prop_local_consistent_completion} that 
$\fd$ is one of the rays of 
$S_{k+1}(\fD^\iin)$.
\end{proof}

The proof of Proposition 
\ref{prop_consistent_completion}
is a variant of a rather simple case of the construction of 
wall structures in the sense of Gross-Siebert (e.g.\ see 
\cite[\S 6.3.3]{MR2722115}, or \cite{MR2846484} for a much more general context). The main difference is that we do not work with 
polyhedral decompositions and $\Z$-valued 
piecewise linear functions, but with $\R$-valued linear functions $\varphi_\sigma$.

It is clear from the proof of Proposition
\ref{prop_consistent_completion} that the construction of $S(\fD^\iin)$ from
$\fD^\iin$ is completely algorithmic:
it is an iterative application of the construction of local consistent completions,
Proposition
\ref{prop_local_consistent_completion}, which is itself completely algorithmic.

\subsection{The scattering diagrams $S(\fD^\iin_{u,v})$, $S(\fD^\iin_{q^\pm})$ and $S(\fD^\iin_{\cl^\pm})$} \label{section_scattering_final}

For every $\fg$
a $M$-graded Lie algebra over 
$\Q$ such that 
$[\fg_m,\fg_{m'}]=0$ if 
$m$ and $m'$ are collinear, 
and for every choice of nonzero elements 
$H_{n,\ell}^+ \in \fg_{\ell m_n^+}$
and $H_{n,\ell}^- \in \fg_{\ell m_n^-}$, $n\in\Z$,
$\ell \geqslant 1$, we have defined
in \cref{section_initial} a scattering diagram 
$\fD^\iin$ and constructed its consistent completion $S(\fD^\iin)$.
Below, we apply this construction in several more specific cases.

We consider the skew-symmetric non-degenerate bilinear form
\[\langle -,- \rangle \colon 
\bigwedge\nolimits^2 M \rightarrow \Z \]
given by 
\[ \langle (a,b), (a',b') \rangle = 
3 (a'b-ab') \,.\]
The choice of 
$\langle -,- \rangle$ is dictated by the skew-symmetrized Euler form of $\D^b(\PP^2)$,
see Lemma \ref{lem_skew_sym_euler_form}.

\begin{defn} \label{def_lie_algebra}
We denote by $\fg_{u,v}$ the $\Q(u^{\pm \frac{1}{2}}, v^{\pm \frac{1}{2}})$-Lie algebra 
\[\fg_{u,v}\coloneq \bigoplus_{m \in M} 
\Q(u^{\pm \frac{1}{2}}, v^{\pm \frac{1}{2}}) z^m\]
with Lie bracket given by
\[[z^m, z^{m'}] \coloneq
(-1)^{\langle m,m'\rangle} ((uv)^{\frac{\langle m,m' \rangle}{2}}
-(uv)^{-\frac{\langle m,m' \rangle}{2}})
 z^{m+m'}\,.\]
\end{defn}

\begin{defn}
We denote by $\fg_{q^-}$ the $\Q(q^{\pm \frac{1}{2}})$-Lie algebra 
\[\fg_{q^-} \coloneq \bigoplus_{m \in M} 
\Q(q^{\pm \frac{1}{2}})
z^m\]
with Lie bracket given by
\[[z^m, z^{m'}] \coloneq 
(-1)^{\langle m,m'\rangle}(q^{\frac{\langle m,m' \rangle}{2}}
-q^{-\frac{\langle m,m' \rangle}{2}}) z^{m+m'}\,.\]
\end{defn}

\begin{defn}
We denote by $\fg_{q^+}$ the $\Q(q^{\pm \frac{1}{2}})$-Lie algebra 
\[\fg_{q^+} \coloneq \bigoplus_{m \in M} 
\Q(q^{\pm \frac{1}{2}})
z^m\]
with Lie bracket given by
\[[z^m, z^{m'}] \coloneq 
(q^{\frac{\langle m,m' \rangle}{2}}
-q^{-\frac{\langle m,m' \rangle}{2}}) z^{m+m'}\,.\]
\end{defn}

\begin{defn}
We denote by $\fg_{\cl^-}$ the $\Q$-Lie algebra 
\[\fg_{\cl^-} \coloneq \bigoplus_{m \in M} \Q z^m\]
with Lie bracket given by
\[[z^m, z^{m'}] \coloneq (-1)^{\langle m,m' \rangle} 
\langle m,m' \rangle z^{m+m'}\,.\]
\end{defn}

\begin{defn}
We denote by $\fg_{\cl^+}$ the $\Q$-Lie algebra 
\[\fg_{\cl^+} \coloneq \bigoplus_{m \in M} \Q z^m\]
with Lie bracket given by
\[[z^m, z^{m'}] \coloneq 
\langle m,m' \rangle z^{m+m'}\,.\]
\end{defn}

\textbf{Remarks:} 
\begin{itemize}
\item[(i)] $\fg_{q^-}$ is the specialization $u=v=q^{\frac{1}{2}}$
of $\fg_{u,v}$.
\item[(ii)] $\fg_{\cl^-}$ is the semiclassical limit $q^{\frac{1}{2}} \rightarrow 1$ of $\fg_{q^-}$.
\item[(iii)] $\fg_{\cl^+}$ is the semiclassical
limit $q^{\frac{1}{2}} \rightarrow 1$ 
of $\fg_{q^+}$.
\item[(iv)] In the remaining part of the present paper, we will only work 
with $\fg_{u,v}$ and its specializations $\fg_{q^-}$ and $\fg_{\cl^-}$, which are relevant for describing the wall-crossing of Donaldson-Thomas invariants. However, in this section, we also introduce $\fg_{q^+}$ and $\fg_{\cl^+}$ because they appear naturally and will be used in the interpretation of the scattering diagrams in terms of log Gromov--Witten invariants presented in \cite{bousseau2019takahashi}.
\end{itemize}

\begin{defn} \label{def_scattering_D_in_u_v}
We denote by $\fD^\iin_{u,v}$ the scattering diagram $\fD^\iin$ of \cref{section_initial}
obtained for $\fg=\fg_{u,v}$, and, for every $n \in \Z$ and for every integer $\ell \geqslant 1$,
\[H_{n,\ell}^+ \coloneq - \frac{1}{\ell} \frac{1}{(uv)^{\frac{\ell}{2}}
-(uv)^{-\frac{\ell}{2}}} z^{\ell m_n^+} \in (\fg_{u,v})_{\ell m_n^+}\,,\]
and 
\[H_{n,\ell}^- \coloneq - \frac{1}{\ell} \frac{1}{(uv)^{\frac{\ell}{2}}
-(uv)^{-\frac{\ell}{2}}} z^{\ell m_n^-}\in (\fg_{u,v})_{\ell m_n^-}\,.\]
We denote by $S(\fD^\iin_{u,v})$ the consistent completion of $\fD^\iin_{u,v}$ given by 
Proposition \ref{prop_consistent_completion}.
\end{defn}

\begin{defn}\label{def_scattering_q_-}
We denote by $\fD^\iin_{q^-}$ the scattering diagram $\fD^\iin$ of  \cref{section_initial}
obtained for $\fg=\fg_{q^-}$, and, for every $n \in \Z$ and for every integer $\ell \geqslant 1$,
\[H_{n,\ell}^+ \coloneq - \frac{1}{\ell} \frac{1}{q^{\frac{\ell}{2}}
-q^{-\frac{\ell}{2}}} z^{\ell m_n^+} \in (\fg_{q^-})_{\ell m_n^+}\,,\]
and 
\[H_{n,\ell}^- \coloneq - \frac{1}{\ell} \frac{1}{q^{\frac{\ell}{2}}
-q^{-\frac{\ell}{2}}} z^{\ell m_n^-}\in (\fg_{q^-})_{\ell m_n^-}\,.\]
We denote by $S(\fD^\iin_{q^-})$ the consistent completion of $\fD^\iin_{q^-}$ given by 
Proposition \ref{prop_consistent_completion}.
\end{defn}

\begin{defn}\label{def_scattering_q_+}
We denote by $\fD^\iin_{q^+}$ the scattering diagram $\fD^\iin$ of \cref{section_initial}
obtained for $\fg=\fg_{q^+}$, and, for every $n \in \Z$ and for every integer $\ell \geqslant 1$,
\[H_{n,\ell}^+ \coloneq \frac{(-1)^{\ell-1}}{\ell} \frac{1}{q^{\frac{\ell}{2}}
-q^{-\frac{\ell}{2}}} z^{\ell m_n^+} \in (\fg_{q^+})_{\ell m_n^+}\,,\]
and 
\[H_{n,\ell}^- \coloneq \frac{(-1)^{\ell-1}}{\ell} \frac{1}{q^{\frac{\ell}{2}}
-q^{-\frac{\ell}{2}}} z^{\ell m_n^-}\in (\fg_{q^+})_{\ell m_n^-}\,.\]
We denote by $S(\fD^\iin_{q^+})$ the consistent completion of $\fD^\iin_{q^+}$ given by 
Proposition \ref{prop_consistent_completion}.
\end{defn}

\begin{defn} 
We denote by $\fD^\iin_{\cl^-}$ the scattering diagram $\fD^\iin$ of \cref{section_initial}
obtained for $\fg=\fg_{\cl^-}$, and, for every $n \in \Z$ and for every integer $\ell \geqslant 1$,
\[H_{n,\ell}^+ \coloneq - \frac{1}{\ell^2} z^{\ell m_n^+} \in (\fg_{\cl^-})_{\ell m_n^+}\,,\]
and 
\[H_{n,\ell}^- \coloneq - \frac{1}{\ell^2} z^{\ell m_n^-}\in (\fg_{\cl^-})_{\ell m_n^-}\,.\]
We denote by $S(\fD^\iin_{\cl^-})$ the consistent completion of $\fD^\iin_{\cl^-}$ given by 
Proposition \ref{prop_consistent_completion}.
\end{defn}

\begin{defn}\label{defn_scattering_cl_+}
We denote by $\fD^\iin_{\cl^+}$ the scattering diagram $\fD^\iin$ of \cref{section_initial}
obtained for $\fg=\fg_{\cl^+}$, and, for every $n \in \Z$ and for every integer $\ell \geqslant 1$,
\[H_{n,\ell}^+ \coloneq \frac{(-1)^{\ell-1}}{\ell^2} z^{\ell m_n^+} \in (\fg_{\cl^+})_{\ell m_n^+}\,,\]
and 
\[H_{n,\ell}^- \coloneq \frac{(-1)^{\ell -1}}{\ell^2} z^{\ell m_n^-}\in (\fg_{\cl^+})_{\ell m_n^-}\,.\]
We denote by $S(\fD^\iin_{\cl^+})$ the consistent completion of $\fD^\iin_{\cl^+}$ given by 
Proposition \ref{prop_consistent_completion}.
\end{defn}

\begin{lem}\label{lem_q=uv}
\begin{itemize}
\item[(i)] Under the change of variables
$q^{\frac{1}{2}}=(uv)^{\frac{1}{2}}$, we have the equality of scattering diagrams
$S(\fD^\iin_{q^-})
=S(\fD^\iin_{u,v})$.
\item[(ii)] The scattering diagram $S(\fD^\iin_{\cl^-})$ is the semiclassical limit $q^{\frac{1}{2}} \rightarrow 1$ of the scattering diagram $S(\fD^\iin_{q^-})$.
\item[(iii)] The scattering diagram $S(\fD^\iin_{\cl^+})$ is the semiclassical
limit $q^{\frac{1}{2}} \rightarrow 1$ 
of the scattering diagram $S(\fD^\iin_{q^+})$.
\end{itemize}
\end{lem}

\begin{proof}
The result is true at the level of the scattering diagrams $\fD^\iin$ from
their explicit description.
The result follows by uniqueness of the consistent completion given by 
Proposition \ref{prop_consistent_completion}.
\end{proof}

\begin{cor} \label{cor_q_rationality}
For every ray $\fd=(|\fd|,H_\fd)$ of 
$S(\fD_{u,v}^\iin)$, we have 
\[ H_\fd \in \Q(q^{\pm \frac{1}{2}}) \,,\]
where we view $\Q(q^{\pm \frac{1}{2}})$
inside 
$\Q(u^{\pm \frac{1}{2}}, v^{\pm \frac{1}{2}})$ via the change of variables 
$q^{\frac{1}{2}}
=(uv)^{\frac{1}{2}}$.
\end{cor}

\begin{proof}
It is an immediate consequence of the first point of Lemma \ref{lem_q=uv}.
\end{proof}
\subsection{Action of $\psi(1)$ on scattering diagrams}
\label{subsection_psi(1)}

In this section we establish a symmetry 
$\psi(1)$ of the scattering diagram 
$S(\fD^\iin_{u,v})$ defined in
\cref{section_scattering_final}.

\begin{defn} \label{def_psi(1)}
We denote by $\psi(1)$ the affine transformation of $\R^2$ given by
\[ \psi(1) \colon \R^2 \rightarrow \R^2\]
\[ (x,y) \mapsto \left(x+1, y-x-\frac{1}{2}\right)\,,\]
and $\mathrm{d \psi}[1]$ its linear part,
given by 
\[ \mathrm{d \psi}(1) \colon 
\R^2 \rightarrow \R^2\]
\[ (a,b) \mapsto (a,b-a) \,.\]
\end{defn}

\begin{lem} \label{lem_computation}
For every $(x,y) \in \R^2$, writing 
$\psi(1)((x,y))=(x',y')$, we have 
\[ (x')^2+2y'=x^2+2y\,.\]
\end{lem}

\begin{proof}
We have 
\[ (x')^2+2y'=(x+1)^2
+2 (y-x-\frac{1}{2})=x^2+2y\,.\]
\end{proof}

\begin{lem} \label{lem_preserv_U}
The affine transformation $\psi(1)$ of 
$\R^2$ preserves $U$.

Moreover the restriction of $\psi(1)$ to $U$ is a
bijection from $U$ to $U$.
\end{lem}

\begin{proof}
Recall that $(x,y) \in U$ if and only if 
$x^2+2y>0$. Therefore, the fact that $\psi(1)$ preserves $U$ follows directly from 
Lemma \ref{lem_computation}.

In order to show that 
$\psi(1)|_U \colon U \rightarrow U$ is a bijection, we note that 
$\psi(1) \colon \R^2 \rightarrow \R^2$ is a bijection, of inverse given by 
\[ \psi(1)^{-1} \colon 
\R^2 \rightarrow \R^2\]
\[(x,y) \mapsto (x-1, y+x-\frac{1}{2})\,.\]
The fact that $\psi(1)^{-1}$ preserves $U$
follows from Lemma
\ref{lem_computation}.
\end{proof}

\begin{defn}
Let $\fd=(|\fd|, H_\fd)$ be a ray of class 
$m_\fd$ in $U$ for
$(M,\fg_{u,v})$.
Writing $H_\fd=\overline{\Omega}_\fd z^{m_\fd} \in (\fg_{u,v})_{m_\fd}$, we define 
$\psi(1)(\fd) \coloneq
(\psi(1)(|\fd|), \psi(1)(H_\fd))$,
where 
$\psi(1)(|\fd|)$ is the image of $|\fd|$
by $\psi(1)$, and where 
$\psi(1)(H_\fd) \coloneq \overline{\Omega}_\fd z^{\mathrm{d\psi}(m_\fd)}$.
\end{defn}

The notation $\psi(1)$ will be justified in
Lemma \ref{lem_T_action}, where it will be shown that, after identification of $U$
with a space of stability conditions on 
$\D^b(\PP^2)$, the action of $\psi(1)$ on $U$ coincides with the action of the autoequivalence $- \otimes \cO(1)$
of $\D^b(\PP^2)$.

\begin{lem}
If $\fd$ is a ray of class $m_\fd$ in 
$U$ for $(M, \fg_{u,v})$, then $\psi(1)(\fd)$ is a ray of class 
$\mathrm{d\psi}(m_\fd)$ in $U$ for $(M, \fg_{u,v})$.
\end{lem}

\begin{proof}
The only not completely obvious point is that $\psi(1)(|\fd|)$ is contained in 
$\bar{U}$, but this follows from Lemma
\ref{lem_preserv_U}.
\end{proof}

\begin{defn}
Let $\fD$ be a scattering diagram on $U$ for $(M,\fg_{u,v})$. We denote by
$\psi(1)(\fD)$ the collection of rays 
$\psi(1)(\fd)$, for all $\fd$ ray of 
$\fD$.
\end{defn}

\begin{lem}
Let $\fD$ be a scattering diagram on 
$U$ for 
$(M,\fg_{u,v})$.
Then $\psi(1)(\fD)$ is a scattering diagram on $U$ for 
$(M,\fg_{u,v})$.
\end{lem}

\begin{proof}
We have to check conditions (i)-(ii)-(iii) 
of Definition \ref{def_scattering}.
Condition (i) is clear as $\psi(1)$ restricted to $U$ is a bijection from $U$ to $U$ according to Lemma \ref{lem_preserv_U}.

Let $\fd$ be a ray of $\fD$ and let 
$\sigma=(x,y) \in |\fd| \cap U$. Writing 
$m_\fd=(a,b)$, we have 
$\varphi_\sigma(m_\fd)
=2(-ax-b)>0$ by condition (ii) of Definition
\ref{def_scattering} satisfied by $\fD$.
So we have
\[ \varphi_{\psi(1)(\sigma)}
(\mathrm{d\psi}(1)(m_\fd))
=2(-a(x+1)-(b-a))=2(-ax-b) 
=\varphi_\sigma(m_\fd)>0\,.\]
It follows that 
$\psi(1)(\fD)$ satisfies condition 
(ii) of Definition \ref{def_scattering}.

Condition (iii) of Definition \ref{def_scattering} for $\psi(1)(\fD)$ follows from condition (iii) of Definition \ref{def_scattering} for $\fD$ as $\psi(1)$ is a proper map.
\end{proof}

Recall that we defined the scattering diagram $S(\fD^\iin_{u,v})$ in 
\cref{section_scattering_final}. The following result expresses the fact that $\psi(1)$ is a symmetry of 
$S(\fD^\iin_{u,v})$.

\begin{prop}\label{prop_psi(1)}
We have 
$\psi(1)(S(\fD^{\iin}_{u,v}))=S(\fD^{\iin}_{u,v})$.
\end{prop}

\begin{proof}
We first remark that we have 
$\psi(1)(\fD^\iin_{u,v})=\fD^\iin_{u,v}$.
Indeed, for every $n \in \Z$, we have 
\[ \psi(1)(s_n)=\psi(1)((n,-\frac{n^2}{2}))=(n+1,-\frac{n^2}{2}-n-\frac{1}{2})=
(n+1,-\frac{(n+1)^2}{2})= s_{n+1}\,,\] 
\[ \mathrm{d\psi}(m_n^-)=
\mathrm{d\psi}((1,-n))
=(1,-n-1)
=m_{n+1}^-\,,\]
\[ \mathrm{d \psi}(m_n^+)=
\mathrm{d\psi}((-1,n))
=(-1, n+1)
=m_{n+1}^+\,,\]
and so, for every $n \in \Z$ and for every 
integer $\ell \geqslant 1$, we have 
$\psi(1)(\fd_{n,\ell}^+)
=\fd_{n+1,\ell}^+$ and 
$\psi(1)(\fd_{n,\ell}^-)=\fd_{n+1,\ell}^-$.

The fact that $\psi(1)(S(\fD^{\iin}_{u,v}))=S(\fD^{\iin}_{u,v})$ then follows from the uniqueness of the consistent completion given by Proposition 
\ref{prop_consistent_completion}.
\end{proof}

\section{The scattering diagram $\fD^{\PP^2}_{u,v}$}
\label{section_scattering_p2}

In \cref{section_sheaves} we fix our notation for coherent sheaves on
$\PP^2$.
In \cref{section_stability_conditions} we review Bridgeland stability conditions on the derived category of coherent sheaves on $\PP^2$ and we introduce a special set of coordinates
on a particular slice of the space 
of these stability conditions. 
In \cref{section_moduli_invariants} we review 
properties of moduli spaces of Bridgeland  semistable objects.
In \cref{section_intersection_invariants}
we introduce numerical invariants of these moduli spaces, in particular Hodge numbers of the intersection cohomology.
In \cref{section_scattering_from_stability}, we
use these numerical invariants to 
construct a scattering diagram $\fD^{\PP^2}_{u,v}$. 
In \cref{subsection_psi(1)_P2} we establish a symmetry property of the scattering diagram 
$\fD_{u,v}^{\PP^2}$.

\subsection{Coherent sheaves on $\PP^2$}
\label{section_sheaves}

Let $\Coh(\PP^2)$ be the abelian category of coherent sheaves 
on $\PP^2$. Given $E$ a coherent sheaf on $\PP^2$, we denote by 
$r(E)$ its rank, $d(E)$ its degree, $\chi(E)\coloneq \dim H^0(E)-\dim H^1(E)+\dim H^2(E)$ its Euler characteristic, and 
\[\gamma(E)\coloneq (r(E),d(E),\chi(E))\in \Z^3\,.\] 
We call $\gamma(E) \in \Z^3$ the class of $E$. Remark that $E \mapsto \gamma(E)$ is additive:
for every $E$ and $F$ coherent sheaves on 
$\PP^2$, we have $\gamma(E \oplus F)=\gamma(E)+\gamma(F)$.
The data of $\gamma(E) \in \Z^3$ is equivalent to the data of the Chern character of $E$.
Indeed, by the Hirzebruch-Riemann-Roch formula, we have 
\[\ch_0(E)=r(E)\,, \ch_1(E)=d(E)\,, \ch_2(E)=\chi(E)-r(E)-\frac{3}{2}d(E)\,.\]
Let $\D^b (\PP^2)$ be the bounded derived category of 
the abelian category $\Coh(\PP^2)$. The Grothendieck group 
\[\Gamma \coloneqq K_0(\D^b(\PP^2))=K_0(\Coh(\PP^2))\]
is isomorphic to $\Z^3$ via the map
\[ \Gamma \simeq \Z^3 \,\]
\[[E] \mapsto \gamma(E)=(r(E),d(E),\chi(E))\,.\] 
For every object $E$ of 
$\D^b(\PP^2)$, we denote by $\gamma(E)$ its class in 
$\Gamma \simeq \Z^3$.

For every $n \in \Z$, we have the degree $n$ line bundle $\cO(n)$ on $\PP^2$. For every object $E$ in $\D^b(\PP^2)$, we denote 
$E(n) \coloneq E \otimes \cO(n)$.

\begin{defn} \label{def_bilinear_form}
We denote by 
\begin{align*}(-,-) \colon \Gamma \otimes \Gamma &\longrightarrow \Z \\
 \gamma \otimes \gamma' &\longmapsto (\gamma,\gamma')
 \end{align*}
the bilinear form on $\Gamma$ given by,
for $\gamma=(r,d,\chi)$ and 
$\gamma'=(r',d',\chi')$,
\[ (\gamma,\gamma') \coloneq -3dr'-rr'-dd'+r\chi'+\chi r'\,.\]
\end{defn}

\begin{lem}\label{lem_euler_form}
The bilinear form $(-,-)$ on 
$\Gamma$ coincides with the Euler form on 
$\Gamma=K_0(\D^b(\PP^2))$, that is,
for every objects $E$ and $E'$ of $\D^b(\PP^2)$, the Euler characteristic 
\[ \chi(E,E') \coloneq \dim 
\Hom(E,E')-\dim \Ext^1(E,E')+\dim \Ext^2(E,E')\,\]
is given by 
\[ \chi(E,E')=(\gamma(E),\gamma(E'))\,.\]
\end{lem}

\begin{proof}
Let $E$ and $E'$ be two objects of $\D^b(\PP^2)$. We denote $\gamma(E)=(r,d,\chi)$ and $\gamma(E')=(r',d',\chi')$.
By the Hirzebruch-Riemann-Roch formula, we have, denoting $H=c_1(\cO(1))$, 
\[\chi(E,E')=\int_{\PP^2} \ch(E^\vee)\ch(E') \td(\PP^2)
=\int_{\PP^2}(r-dH+\ch_2(E))(r'+d'H+\ch_2(E'))(1+\frac{3}{2}H+H^2)\]
\[=rr'+\frac{3}{2}(rd'-dr')+r\ch_2(E')+\ch_2(E)r'-dd'\,.\]
Using that $\ch_2(E)=\chi-r-\frac{3}{2}d$ and $\ch_2(E')=\chi'-r'-\frac{3}{2}d'$, we get the desired formula.
\end{proof}

\begin{lem} \label{lem_skew_sym_euler_form}
For every $\gamma, \gamma' \in \Gamma$, we have 
\[ (\gamma,\gamma')-(\gamma',\gamma)
=\langle m_\gamma, m_\gamma' \rangle\,.\]
\end{lem}

\begin{proof}
Let $\gamma, \gamma' \in \Gamma$. We denote 
$\gamma=(r,d,\chi)$ and $\gamma'=(r',d',\chi')$. We recall from Definition
\ref{def_m_gamma} that $m_\gamma=(r,-d)$,
$m_{\gamma'}=(r',-d')$, and from
\cref{section_scattering_final} that 
\[\langle (a,b),(a',b')\rangle=3(a'b-ab')\,.\]
Therefore, we have 
\[\langle m_\gamma,m_{\gamma'}\rangle
=3(rd'-dr')\,.\]
On the other hand, using Definition \ref{def_bilinear_form}, we have 
\[ (\gamma,\gamma')-(\gamma',\gamma)
=3(rd'-dr')\,.\]
\end{proof}

The skew-symmetric bilinear form $\gamma \otimes \gamma' \mapsto (\gamma,\gamma')-(\gamma',\gamma)$
on $\Gamma$ can be viewed as the Euler form of the Calabi-Yau-3 category $\D^b_0(K_{\PP^2})$ of coherent sheaves on $K_{\PP^2}$ set-theoretically 
supported on $\PP^2$.

\subsection{Stability conditions}
\label{section_stability_conditions}

We first recall the classical notions of 
$\mu$-stability and Gieseker stability.
We refer to \cite{huybrechts2010geometry}
for details.

\begin{defn}
Let $E$ be a coherent sheaf on $\PP^2$ with $r(E) \neq 0$. The 
\emph{slope} of $E$ is 
\[ \mu(E) \coloneq \frac{d(E)}{r(E)}
\,.\]
\end{defn}

\begin{defn}
A coherent sheaf $E$ on $\PP^2$ is \emph{$\mu$-semistable} (respectively \emph{$\mu$-stable}) if
$E$ is purely of dimension 2 (that is, the dimension of every nonzero subsheaf of $E$ is $2$), and, for every nonzero strict subsheaf $F$ of $E$, we have $\mu(F) \leqslant \mu(E)$
(respectively $\mu(F)<\mu(E)$).
\end{defn}

\begin{defn}
Let $E$ be a coherent sheaf on $\PP^2$. The reduced Hilbert polynomial is the monic polynomial
\[ p_E(n) \coloneq \frac{\chi(E(n))}{\alpha_E} \,,\]
where $\alpha_E$ is the leading coefficient of the Hilbert polynomial 
$\chi(E(n))$.
\end{defn}

\begin{defn}
A coherent sheaf $E$ on $\PP^2$ is 
\emph{Gieseker semistable} (respectively \emph{stable}) if $E$ is of pure dimension
(that is, every nonzero subsheaf of $E$ has support of dimension equal to the dimension of the support of $E$), and, for every nonzero strict subsheaf $F$
of $E$, we have $p_F(n) \leqslant p_E(n)$
(respectively $p_F(n)<p_E(n)$) for $n$ large enough.
\end{defn}

We now recall the notion stability condition
in the sense of Bridgeland \cite{MR2373143},
in the particular case of the triangulated category $\D^b (\PP^2)$.

\begin{defn}
A \emph{prestability condition} $\sigma$ on $\D^b(\PP^2)$ consists
of a pair $\sigma=(Z, \cA)$, such that:
\begin{itemize}
    \item[(i)] $\cA$ is the heart of a bounded t-structure on 
    $\D^b(\PP^2)$.
    \item[(ii)] $Z$ is a linear map 
    $Z \colon \Gamma \rightarrow \C$, called the \emph{central charge}.
    \item[(ii)] For every nonzero object $E$ of $\cA$, we have 
    $Z(E)=\rho(E) e^{i \pi \phi(E)}$ with $\rho(E) \in \R_{>0}$,
    and $0 < \phi(E) \leqslant 1$, that is $Z(E)$ is contained in the upper half-plane minus the nonnegative real axis.
    \item[(iii)] A nonzero object $F$ of $\cA$ is \emph{$\sigma$-semistable} if for every nonzero subobject $F'$ of $F$ in $\cA$, we have $\phi(F') \leqslant \phi(F)$.
    We require the Harder-Narasimhan property, that is, that every nonzero object $E$ of $\cA$ admits a finite filtration 
    \[0 \subset E_0 \subset E_1 \dots \subset E_n =E\] in $\cA$, with each factor $F_i \coloneq E_i/E_{i-1}$ $\sigma$-semistable
    and $\phi(F_1) > \phi(F_2) > \dots >\phi(F_n)$.
    \end{itemize}
\end{defn}

\begin{defn}
Given $\sigma=(Z,\cA)$ a prestability condition on $\D^b(\PP^2)$, an object 
$F$ of $\D^b(\PP^2)$ is called $\emph{$\sigma$-stable}$ if $F$
is a nonzero object of $\cA$, and, for every nonzero strict subobject $F'$ of $F$ in $\cA$, we have $\phi(F') < \phi(F)$.
\end{defn}

\begin{defn}
A stability condition $\sigma=(Z,\cA)$ on $\D^b(\PP^2)$ is a prestability condition 
satisfying the \emph{support property}, that is, such that there exists a quadratic form $Q$ on the $\R$-vector space $\Gamma \otimes \R$
    such that:
    \begin{itemize}
        \item[(i)] The kernel of $Z$ in $\Gamma \otimes \R$ is negative
        definite with respect to $Q$,
        \item[(ii)] For every $\sigma$-semistable object, we have 
        $Q(\gamma(E)) \geqslant 0$.
    \end{itemize}
\end{defn}

\textbf{Remarks:} 
\begin{itemize}
\item[(i)] If $\sigma=(Z,\cA)$ is a stability condition on $\D^b(\PP^2)$
which satisfies the support property, then the image by $Z \colon \Gamma \rightarrow \C$
of the set of $\gamma \in \Gamma$ such that there exists a $\sigma$-semistable 
object of class $\gamma$ is discrete in $\C$.
\item[(ii)] The support property is a notion due to Kontsevich-Soibelman \cite[\S 1.2]{kontsevich2008stability}. 
We refer to 
\cite[Appendix B]{MR2852118} for a comparison with earlier notions of finiteness  for stability conditions.
\end{itemize}

We denote by $\Stab(\PP^2)$ the set of stability conditions on $\D^b(\PP^2)$.
According to \cite{MR2373143}, $\Stab(\PP^2)$ has a natural structure of complex manifold of dimension $3$, such that the map
\[ \Stab(\PP^2) \rightarrow \Hom(\Gamma,\C) \simeq \C^3\]
\[ \sigma=(Z,\cA) \mapsto Z\]
is a local isomorphism of complex manifolds
(locally on $\Stab(\PP^2)$).

Following \cite{MR2376815, MR2998828, MR2852118, MR3010070},
we review a standard way to construct examples of stability conditions on 
$\D^b (\PP^2)$.

\begin{defn} \label{defn_heart}
For every $s \in \R$, we denote:
\begin{itemize}
    \item[(i)] $\Coh^{\leqslant s}(\PP^2)$ the subcategory of 
    $\Coh(\PP^2)$ generated (by extensions) by 
    $\mu$-semistable sheaves of slope $\leqslant s$.
    \item[(ii)] $\Coh^{>s}(\PP^2)$ the subcategory of $\Coh(\PP^2)$
    generated (by extensions) by $\mu$-semistable sheaves of slope
    $>s$ and torsion sheaves.
    \item[(iii)] $\Coh^{\texttt{\#} s}(\PP^2)$ the subcategory of $\D^b(\PP^2)$
    of objects $E$ such that $\cH^i(E)=0$ for $i \neq -1, 0$,
    $\cH^{-1}(E)$ is an object of $\Coh^{\leqslant s}(\PP^2)$,
    and $\cH^{0}(E)$ is an object of $\Coh^{>s}(\PP^2)$.
\end{itemize}
\end{defn}

The category $\Coh^{\texttt{\#}s}(\PP^2)$ is obtained from 
$\Coh(\PP^2)$ by tilt of the torsion pair 
\[(\Coh^{> s}(\PP^2), \Coh^{\leqslant s}(\PP^2))\,.\]
In particular, $\Coh^{\texttt{\#}s}(\PP^2)$ is an abelian category and the heart of a bounded
t-structure on $\D^b(\PP^2)$.

\begin{defn}
For every $(s,t) \in \R^2$ with $t>0$, let $\tilde{Z}^{(s,t)} \colon \Gamma \rightarrow \C$
be the linear map defined by
\[\gamma = (r,d,\chi) \mapsto \tilde{Z}_\gamma^{(s,t)} \,,\]
  \[  \tilde{Z}_\gamma^{(s,t)}\coloneq -\frac{r}{2}(s+it)^2  +d(s+it)+r+\frac{3}{2}d-\chi\,. \]
\end{defn}

We have
\[  \tilde{Z}_\gamma^{(s,t)}=r \left( -\frac{1}{2}(s^2-t^2)-ist \right)  +d (s+it)+r+\frac{3}{2}d-\chi\]
\[= -\frac{1}{2}(s^2-t^2)r+ds+r+\frac{3}{2}d-\chi + i(d-sr)t\,.\]
If $E$ is an object of $\D^b(\PP^2)$ of class $\gamma(E) \in \Gamma$, then we can write
\[\tilde{Z}_{\gamma(E)}^{(s,t)}=-\int_{\PP^2}
e^{-(s+it)H} \ch(E)\,,\]
where $H \coloneq c_1(\cO(1))$.

According to \cite{MR2376815, MR2998828, MR2852118}, for every 
$(s,t) \in \R^2$ with $t>0$, the pair 
\[ (\tilde{Z}^{(s,t)}, \Coh^{\texttt{\#} s}(\PP^2))\] is a stability condition on
$\D^b(\PP^2)$. In particular, we get
an embedding of the upper half-plane 
$\{(s,t) \in \R^2|\, t>0\}$ into 
$\Stab(\PP^2)$.

We now do something new, which is to describe the same slice inside $\Stab(\PP^2)$, but using coordinates $(x,y)$ different from the usual coordinates $(s,t)$. This will replace the upper half-plane 
$\{(s,t) \in \R^2|\,t>0\}$ by the 
`upper parabola' $U=\{(x,y) \in \R^2|\,y>-\frac{x^2}{2}\}$, which already appeared in  \cref{section_initial}.

\begin{defn} \label{def_central_charge_formula}
For every $(x,y) \in U =\{(x,y) \in \R^2| \,y>-\frac{x^2}{2}\}$,
let $Z^{(x,y)} \colon \Gamma \rightarrow \C$ be the linear map defined by 
\[\gamma = (r,d,\chi) \mapsto Z_\gamma^{(x,y)} \,,\]
\[Z_\gamma^{(x,y)} \coloneqq r y 
+dx
+r+\frac{3}{2}d-\chi
+i
(d-rx)\sqrt{x^2+2y}\,.\]
\end{defn}

\begin{prop} \label{prop_stability}
For every $\sigma=(x,y) \in U$, the pair 
$(Z^{(x,y)}, \Coh^{\texttt{\#}x}(\PP^2))$
is a stability condition on $\D^b(\PP^2)$.
\end{prop}

\begin{proof}

According to \cite{MR2376815, MR2998828, MR2852118}, for every $(s,t) \in \R^2$ with $t >0$, the pair 
\[(\tilde{Z}^{(s,t)}, \Coh^{\texttt{\#} s}(\PP^2))\]
is a stability condition on $\D^b(\PP^2)$.

The map 
\[(s,t) \mapsto  (s,-\frac{1}{2}(s^2-t^2))\]
defines a bijection between the upper half-plane 
\[\{(s,t) \in \R^2|\,t>0\}\] and 
\[U
=\{(x,y) \in \R^2|\,y>-\frac{x^2}{2}\}\,,\]
of inverse 
\[\psi \colon (x,y) \mapsto (x, \sqrt{x^2+2y})\,,\]
such that $Z^{(x,y)}=\tilde{Z}^{\psi(x,y)}$. 
\end{proof}

\textbf{Remarks:}
\begin{itemize}
\item[(i)] Proposition \ref{prop_stability} gives a map 
\[U \rightarrow \Stab(\PP^2)\,\]
\[\sigma=(x,y) \mapsto (Z^{(x,y)}, \Coh^{\texttt{\#}x}(\PP^2))\,\] which is injective.
Viewing $U$ as an open subset of $\R^2 \simeq \C$, this map is holomorphic. From now on, we use the same notation $\sigma$ for a point 
$\sigma=(x,y)$ in $U$ or for the corresponding stability condition 
$\sigma=(Z^\sigma, \cA^\sigma)
\coloneq (Z^{(x,y)}, \Coh^{\texttt{\#}x}(\PP^2))
$.
\item[(ii)] If $E$ is a skyscraper sheaf $k(p)$, $p \in \PP^2$, then 
$\gamma(E)=(0,0,1)$ and $Z_{\gamma(E)}^\sigma=-1$ for every $\sigma \in U$.
\item[(iii)] The choice of the coordinates 
$(x,y)$ instead of the more traditional
coordinates $(s,t)$ is motivated by the fact that the real part of the central charge,
\[ \Rea Z_\gamma^\sigma=ry+dx+r+\frac{3}{2}d -\chi \,, \]
is affine in $(x,y)$ (but quadratic in $(s,t)$). This will be the key point enabling us to make contact with scattering diagrams
in \cref{section_scattering_from_stability}.
\end{itemize}

\subsection{Moduli spaces and walls}

\label{section_moduli_invariants}

For every $\sigma \in U$ and $\gamma \in \Gamma$, we denote by
$\fM_\gamma^\sigma$ the moduli stack of $\sigma$-semistable 
objects in $\D^b(\PP^2)$ of class $\gamma$.
According to \cite{MR3010070}, the Artin stack $\fM_\gamma^\sigma$ admits a good
moduli space $M_\gamma^\sigma$ which is in fact a projective scheme whose
closed points are in one-to-one correspondence with 
$S$-equivalence classes of $\sigma$-semistable objects of class 
$\gamma$. This follows from a description of $M_\gamma^\sigma$
as moduli space of quiver representations,
see  
\cite[Proposition 7.5]{MR3010070}.
The moduli space $M_\gamma^{\sigma-\st}$ of $\sigma$-stable objects of class 
$\gamma$ is a quasiprojective scheme, open in $M_\gamma^\sigma$.

We introduce the rank $1$ sublattice 
\[\Gamma^0 \coloneq \{(0,0,\chi)|\, \chi \in \Z\} \subset \Gamma \,. \] 
If $E$ is a coherent sheaf on $\PP^2$, then 
$\gamma(E) \in \Gamma^0$ if and only if $E$ is supported in dimension $0$.
According to \cite{MR3936077}, we have 
$\Ext^2(E,E)=0$ for every $\sigma$-semistable object $E$ with $\gamma(E) \notin 
\Gamma^0$, and so the stacks
$\fM_\gamma^\sigma$ of $\sigma$-semistable objects and the moduli spaces 
$M_\gamma^{\sigma-\st}$ of $\sigma$-stable objects are smooth if 
$\gamma \notin \Gamma^0$.

If $E$ is a $\sigma$-semistable object of class $\gamma \notin \Gamma^0$, then,
using that $\Ext^2(E,E)=0$ and Lemma \ref{lem_euler_form},
we find that the dimension of 
$\fM_\gamma^{\sigma}$ at $E$ is
$\dim \Ext^1(E,E)-\dim \Ext^0(E,E)
=-(\gamma,\gamma)$. In short, if
$\gamma \notin \Gamma^0$, the moduli stack
$\fM_\gamma^{\sigma}$ is smooth and equidimensional of dimension
\[ \dim \fM_\gamma^\sigma=-(\gamma,\gamma)\,.\]

If $E$ is a $\sigma$-stable object of class $\gamma \notin \Gamma^0$, then,
using that $\Ext^2(E,E)=0$, that $\dim \Ext^0(E,E)=1$
as a stable object is simple, and Lemma \ref{lem_euler_form},
we find that the dimension of $M_\gamma^{\sigma-\st}$ at $E$ is $\dim \Ext^1(E,E)=1-\chi(E,E)=1-(\gamma,\gamma)$.
In short, if $\gamma \notin \Gamma^0$, the moduli space $M_\gamma^{\sigma-\st}$ is smooth and equidimensional of dimension
\[ \dim M_\gamma^{\sigma-\st}
=1-(\gamma,\gamma) \,.\]
In fact, $M_\gamma^{\sigma-\st}$ is also smooth and equidimensional if $\gamma \in \Gamma^0$, as it is 
$\PP^2$ if $\gamma=(0,0,1)$, and the empty set else (this follows from Lemma 6.3 and Lemma 10.1 of \cite{MR2376815}).
The moduli spaces
$M_\gamma^\sigma$ of $\sigma$-semistable objects are singular in general.

Our main interest is in the study of $M_\gamma^\sigma$ as a function of 
$\sigma$.
Such study, mainly from the birational point of view, has been done in 
\cite{MR3010070, MR3275289, MR3455422, MR3271294, MR3381441, MR3635357, MR3419956, MR3800356, MR3921322, MR3615584, MR2681544, woolf2013nef}.
The result is that, for a fixed $\gamma$, there are finitely many 
curves in $U$, called \emph{actual walls} for $\gamma$, such that if $\sigma$
moves in a given connected component of the complement of the walls in $U$,
then $M_\gamma^{\sigma}$ does not change. In other words, 
$M_\gamma^\sigma$, as a function of $\sigma$, only changes when 
$\sigma$ crosses an actual wall. 

Every $\gamma' \in \Gamma$ not collinear with $\gamma$ defines a \emph{potential wall} $W_{\gamma,\gamma'}$ for $\gamma$, defined
as the set of $\sigma \in U$ such that $Z_{\gamma}^{\sigma}$ and $Z_{\gamma'}^{\sigma}$
are positively collinear. It follows from the explicit formula giving 
$Z^\sigma$, Definition \ref{def_central_charge_formula}, that each potential wall is either a parabola or a vertical 
line in $U$. Every actual wall for $\gamma$ is contained in a potential wall for $\gamma$.
Not every potential wall for $\gamma$ is an actual wall for $\gamma$: there are only 
finitely many actual  walls for $\gamma$ but in general infinitely many potential walls for $\gamma$.
An algorithm to determine the actual walls for $\gamma$ is presented in \cite{MR3921322}.

We will not be interested in the study of $M_\gamma^\sigma$ as a function of $\sigma$ from the birational point of
view, but from the point of view of numerical cohomological invariants:
Euler characteristics, Betti numbers, Hodge numbers. We will consider these
invariants for intersection cohomology rather than for singular cohomology:
as the moduli spaces $M_\gamma^\sigma$ are in general singular, we will obtain invariants with better properties this way.

\subsection{Intersection cohomology invariants}
\label{section_intersection_invariants}

We refer to \cite{MR751966, MR2525735, MR1047415}, for general notions on intersection cohomology, perverse sheaves and mixed Hodge modules, and to \cite{meinhardt2017donaldson, meinhardt2015donaldson}
for applications in the context of Donaldson-Thomas theory, which will be ultimately relevant for us.

We recall that given $X$ a projective variety, $Z$ an equidimensional subvariety of
$X$, and $L$ a variation of pure Hodge structures on an open subset $Z^\circ$
of the smooth part of $Z$, then there is a canonical pure Hodge module 
$IC_Z(L)$ on $X$ such that $IC_Z(L)|_{Z^\circ}= L$.

For every $\sigma \in U$ and $\gamma \in \Gamma$, we apply this construction to 
$X=M_\gamma^{\sigma}$, $Z=\overline{M_\gamma^{\sigma-\st}}$ the closure
of $M_\gamma^{\sigma-\st}$ in $M_\gamma^{\sigma}$, $Z^\circ = M_\gamma^{\sigma-\st}$, and 
$L=\Q$ viewed as a trivial variation of pure Hodge structures of type $(0,0)$.
Remark that 
$M_\gamma^{\sigma-\st}$ is equidimensional by \cref{section_moduli_invariants}.
We obtain a mixed Hodge module $IC_{\overline{M_\gamma^{\sigma-\st}}}(\Q)$
on $M_\gamma^\sigma$.
By Saito theory of mixed Hodge modules \cite{MR1047415}, for every 
$k \in \Z$, the cohomology group
\[IH^{k}(M_\gamma^\sigma, \Q) \coloneq H^k(M_\gamma^\sigma, IC_{\overline{M_\gamma^{\sigma-\st}}}(\Q))\] 
is naturally a pure Hodge structure of weight $k$,
and so has Hodge numbers, that we denote $Ih^{p,q}(M_\gamma^\sigma)$.
We organize these intersection Hodge numbers into a signed symmetrized intersection  
Hodge polynomial 
\[Ih_\gamma^\sigma(u^{\frac{1}{2}},v^{\frac{1}{2}}) \coloneq 
(-(uv)^{\frac{1}{2}})^{- \dim M_\gamma^\sigma}
\sum_{p,q=0}^{\dim M_\gamma^\sigma}
(-1)^{p+q}
Ih^{p,q}(M_\gamma^\sigma) u^p v^q \in \Z[u^{\pm
\frac{1}{2}},v^{\pm \frac{1}{2}}] \,.\]
We organize the intersection Betti numbers
\[Ib_j(M_\gamma^\sigma) \coloneq \dim H^j(M_\gamma^\sigma, IC_{\overline{M_\gamma^{\sigma-\st}}}(\Q)) = \sum_{p+q=j} Ih^{p,q}(M_\gamma^\sigma)\,,\]
into a signed symmetrized intersection 
Poincar\'e polynomial 
\[ Ib_\gamma^\sigma(q^{\frac{1}{2}}) \coloneq 
(-q^{\frac{1}{2}})^{-\dim M_\gamma^\sigma}
\sum_{j=0}^{2\dim M_\gamma^\sigma}
(-1)^j
Ib_j(M_\gamma^\sigma) q^{\frac{j}{2}} \in \Z[q^{\pm \frac{1}{2}}]\,.\]
Finally, we can consider the intersection Euler characteristic
\[Ie_\gamma^{+,\sigma} \coloneq 
\sum_{j=0}^{2\dim M_\gamma^\sigma}
(-1)^j Ib_j(M_\gamma^\sigma) \in \Z\,,\]
and its signed version
\[Ie_\gamma^{-,\sigma} \coloneq 
(-1)^{\dim M_\gamma^\sigma}
\sum_{j=0}^{2\dim M_\gamma^\sigma}
(-1)^j Ib_j(M_\gamma^\sigma) \in \Z\,.\]
We have the obvious specialization relations:
\[Ib_\gamma^\sigma(q^{\frac{1}{2}}) = Ih_\gamma^\sigma(u^{\frac{1}{2}}
=q^{\frac{1}{2}}, v^{\frac{1}{2}}=q^{\frac{1}{2}})\,,\]
\[Ie_\gamma^{-,\sigma}=Ib_\gamma^\sigma(q^{\frac{1}{2}}=1)\,.\]

If $M_\gamma^{\sigma-\st}=M_\gamma^\sigma$, which happens for 
example if $\gamma$ is a primitive element of the lattice 
$\Gamma$, then $M_\gamma^\sigma$ is a smooth projective variety and the intersection Hodge numbers, Betti numbers, Euler characteristic of 
$M_\gamma^\sigma$ coincide with the usual Hodge numbers, Betti numbers, 
Euler characteristic of $M_\gamma^\sigma$.

\subsection{Scattering diagrams from stability conditions}.
\label{section_scattering_from_stability}

In this section we use the intersection Hodge polynomials 
$Ih_\gamma^\sigma(u^{\frac{1}{2}}, v^{\frac{1}{2}})$,
the intersection Poincar\'e polynomials $Ib_\gamma^\sigma(q^{\frac{1}{2}})$,
and the intersection Euler characteristics 
$Ie_\gamma^\sigma$, defined in  \cref{section_intersection_invariants},
to construct scattering diagrams on $U$, in the sense of 
Definition \ref{def_scattering}, $\fD^{\PP^2}_{u,v}$,
$\fD^{\PP^2}_{q^\pm}$ and $\fD^{\PP^2}_{\cl^\pm}$, respectively for
$(M,\fg_{u,v})$, $(M,\fg_{q^\pm})$, 
and
$(M, \fg_{\cl^{\pm}})$. We recall that the Lie algebras
$\fg_{u,v}$, $\fg_{q^\pm}$ and $\fg_{\cl^\pm}$ have been defined in
\cref{section_scattering_final}.

The scattering diagram  $\fD^{\PP^2}_{u,v}$ will be constructed in terms of particular combinations 
$\overline{\Omega}_\gamma^\sigma(u^{\frac{1}{2}}, v^{\frac{1}{2}})$ of the intersection Hodge polynomials $Ih_\gamma^\sigma(u^{\frac{1}{2}}, v^{\frac{1}{2}})$, defined as follows.

\begin{defn}
For every $\gamma \in \Gamma$,
we denote
\[ \Gamma_\gamma \coloneq \{\gamma' \in \Gamma\,|
\gamma=\ell \gamma' \text{for some }
\ell \in \Z_{\geqslant 1}\}\,,\]
that is, $\Gamma_\gamma$ is the finite set of elements in the lattice $\Gamma$
dividing $\gamma$.
\end{defn}

\begin{defn} \label{def_omega}
For every $\gamma \in \Gamma$ and $\sigma \in U$, we define
\[ \overline{\Omega}_\gamma^\sigma(u^{\frac{1}{2}}, v^{\frac{1}{2}}) :=-\sum_{\substack{\gamma' \in \Gamma_\gamma\\ \gamma=\ell \gamma'}}
\frac{1}{\ell}
\frac{Ih_{\gamma'}^\sigma(u^{\frac{\ell}{2}}, v^{\frac{\ell}{2}})}{
(uv)^{\frac{\ell}{2}}-(uv)^{-\frac{\ell}{2}}}\,. \]
\end{defn}

The formula defining $\overline{\Omega}_\gamma^\sigma(u^{\frac{1}{2}}, v^{\frac{1}{2}})$ might seem a bit unmotivated from the point of view of the classical geometry of moduli spaces of sheaves on $\PP^2$. However, it is the familiar way to package Donaldson-Thomas invariants in the presence of strictly semistable objects \cite{kontsevich2008stability, MR2951762}. Intuitively, the denominator in the formula comes from the fact that every stable object has a group of automorphisms equal to $\C^{*}$.

By the Möbius inversion formula, we have 
\[Ih_{\gamma}^\sigma(u^{\frac{1}{2}}, v^{\frac{1}{2}})
:=-\sum_{\substack{\gamma' \in \Gamma_\gamma\\ \gamma=\ell \gamma'}}
\frac{\mu(\ell)}{\ell}
\frac{\overline{\Omega}_{\gamma'}^\sigma(u^{\frac{\ell}{2}}, v^{\frac{\ell}{2}})}{
(uv)^{\frac{\ell}{2}}-(uv)^{-\frac{\ell}{2}}}\,,\]
where $\mu$ is the Möbius function. In particular, the knowledge of the invariants
$(Ih_{\gamma}^\sigma(u^{\frac{1}{2}}, v^{\frac{1}{2}}))_{\gamma'\in \Gamma_\gamma}$
is equivalent to the knowledge of the invariants 
$(\overline{\Omega}_\gamma^\sigma(u^{\frac{1}{2}}, v^{\frac{1}{2}}))_{\gamma'\in \Gamma_\gamma}$.

For every $\gamma =(r,d,\chi) \in \Gamma$, we define
\[L_\gamma \coloneq \{ \sigma \in U| \, Z_{\gamma}^{\sigma} \in i \R_{>0} \}\,.\]
Using the explicit formula for $Z_\gamma^\sigma$ given in Definition 
\ref{def_central_charge_formula}, we find that
\[L_\gamma
=\{(x,y) \in U |\, ry+dx+r+\frac{3}{2}d-\chi=0 \,, -rx+d>0 \}\,.\]
Remark that $L_\gamma = L_{\gamma'}$ if $\gamma$ and $\gamma'$
are positively collinear in $\Gamma$.

More explicitly, we have
the following cases according to the sign of $r$ and $d$:
\begin{itemize}
    \item[(i)] If $r > 0$, then $L_\gamma$ is the intersection of the line of 
equation 
\[ry+dx+r+\frac{3}{2}d-\chi=0\] with $U$ and with the right half-plane $x < \frac{d}{r}$.
    \item[(ii)] If $r<0$, then $L_\gamma$ is the intersection of the line of equation 
\[ry+dx+r+\frac{3}{2}d-\chi=0\] 
with $U$ and with the left half-plane 
$x> \frac{d}{r}$.
   \item[(iii)] If $r=0$ and $d>0$, then $L_\gamma$ is the intersection of the vertical line of equation
\[x=\frac{3}{2}-\frac{\chi}{d}\]
with $U$.
   \item[(iv)] If $r=0$ and $d \leqslant 0$, then $L_\gamma$ is empty.
\end{itemize}

For every $\gamma \in \Gamma$, we denote by $\bar{L}_\gamma$ the closure of $L_\gamma$ in $\bar{U}$.

\begin{defn}
For every $\gamma \in \Gamma$, we define 
\[R_\gamma \coloneqq \{ \sigma \in U \,| \, Z_\gamma^{\sigma} \in i \R_{>0} \,,
\overline{\Omega}_\gamma^\sigma(u^{\frac{1}{2}},v^{\frac{1}{2}}) \neq 0 \}\,. \]
\end{defn}

By definition, for every $\gamma \in \Gamma$, $R_\gamma$ is a subset of $L_\gamma$.
We denote by $\bar{R}_\gamma$ the closure of $R_\gamma$ in $\bar{U}$: it is a subset of 
$\bar{L}_\gamma$.

The moduli spaces $M_{\gamma'}^\sigma$, and so the intersection Hodge 
polynomials $Ih_{\gamma'}^\sigma(u^{\frac{1}{2}}, v^{\frac{1}{2}})$, are constant 
as a function of $\sigma$ for $\sigma$ lying in a given connected component of the complement in $U$ of the finitely many actual
walls for the classes $\gamma' \in \Gamma_\gamma$. 
It follows that, for every $\gamma \in \Gamma$, we have a natural decomposition
\[R_\gamma = \bigcup_{j \in J_\gamma} R_{\gamma,j}\,,\]
where:
\begin{itemize}
\item[(i)] $J_\gamma$ is the finite set indexing connected 
components of $R_\gamma$ minus the intersection points with the finitely many actual walls for classes $\gamma'\in \Gamma_\gamma$ where $\overline{\Omega}_\gamma^\sigma(u^{\frac{1}{2}},v^{\frac{1}{2}})$ jumps. 
\item[(ii)] $R_{\gamma,j}$ is the closure in $R_\gamma$ of the corresponding connected
component indexed by $j \in J_\gamma$.  We denote by $\bar{R}_{\gamma,j}$ the closure of $R_{\gamma,j}$ in 
$\bar{U}$.
\end{itemize}
Each $\bar{R}_{\gamma,j}$ is either a bounded line segment contained in $\bar{L}_\gamma$
or a half-line contained in $\bar{L}_\gamma$.
For every $j \in J_\gamma$, we denote by 
$\overline{\Omega}_{\gamma,j}(u^{\frac{1}{2}},v^{\frac{1}{2}})$ the common value of the invariants 
$\overline{\Omega}_\gamma^\sigma(u^{\frac{1}{2}},v^{\frac{1}{2}})$
for $\sigma$ in the interior of $R_{\gamma,j}$.


\begin{defn} \label{def_m_gamma}
For every $\gamma=(r,d,\chi) \in \Gamma$, we denote
$m_\gamma \coloneq  
(r,-d) \in M =\Z^2$.
\end{defn} 

Remark that, for every $\gamma \in \Gamma$, 
the integral vector $m_\gamma$ is a direction for $L_\gamma$.

If $\bar{R}_{\gamma,j}$ is a bounded line segment, then there
exists a unique way to write 
\[\bar{R}_{\gamma,j}=\Init(\fd_{\gamma,j}) - [0,T_{\fd_{\gamma,j}}] m_\gamma \,,\]
with
$\Init(\fd_{\gamma,j}) \in \bar{U}$, and
$T_{\fd_{\gamma,j}}$ a positive real number.
If $\bar{R}_{\gamma,j}$ is a half-line, then there exists a unique way to write
\[\bar{R}_{\gamma,j}=\Init(\fd_{\gamma,j}) - \R_{\geqslant 0} m_\gamma \,,\]
with  $\Init(\fd_{\gamma,j}) \in \bar{U}$.
It follows that for every $\gamma \in \Gamma$ and $j \in J_\gamma$, we can naturally view 
$\bar{R}_{\gamma,j}$ as a naked ray of class $m_\gamma$ in $U$ in the sense of Definition 
\ref{def_naked_ray}. We denote by $|\fd_{\gamma,j}|$ this naked ray of class $m_\gamma$.

\begin{defn} \label{def_ray_stability_scattering}
For every $\gamma \in \Gamma$
and $j \in J_\gamma$, we denote by $\fd_{\gamma,j}$ the ray of class 
$m_\gamma$ in $U$ for 
$(M, \fg_{u,v})$, in the sense of Definition \ref{def_ray}, given by 
$(|\fd_{\gamma,j}|, H_{\fd_{\gamma,j}})$,
where 
\[H_{\fd_{\gamma,j}}\coloneq
\overline{\Omega}_{\gamma,j}(u^{\frac{1}{2}},v^{\frac{1}{2}})
z^{m_\gamma} \,\in (\fg_{u,v})_{m_\gamma}\,.\]
\end{defn}

\begin{defn}\label{defn_scattering_stability}
We denote by $\fD^{\PP^2}_{u,v}$ the collection of rays 
$\fd_{\gamma,j}$, where $\gamma \in \Gamma$, 
$j \in J_\gamma$.
\end{defn}

\begin{prop}
The collection of rays $\fD^{\PP^2}_{u,v}$ is a scattering diagram on 
$U$ for $(M,\fg_{u,v})$ in the sense of Definition \ref{def_scattering}.
\end{prop}

\begin{proof}
Let $\sigma=(x,y) \in U$ and let $m \in M$.
Let $\fd$ be a ray of 
$\fD_{u,v}^{\PP^2}$ of class $m$ and such that $\sigma$ belongs to the interior of 
$|\fd|$. We have $\fd=\fd_{\gamma,j}$ for some $\gamma \in \Gamma$ and $j \in J_\gamma$. The class $\gamma=(r,d,\chi)$ is uniquely determined by the condition 
$m_\gamma=(r,-d)=m$ (using Definition \ref{def_m_gamma}) and the condition 
$\sigma \in L_\gamma$, that is, 
$\chi=ry+dx+r+\frac{3}{2}d$. The element $j \in J_\gamma$ is uniquely determined by the fact that 
$\sigma$ belongs to the interior of $|\fd_{\gamma,j}|$ and that the interiors of $|\fd_{\gamma,j_1}|$ and 
$|\fd_{\gamma,j_2}|$ are distinct for 
$j_1,j_2 \in J_\gamma$, $j_1 \neq j_2$.
It follows that $\fD^{\PP^2}_{u,v}$
satisfies the first condition of Definition 
\ref{def_normalized}.
As
$J_\gamma$ is defined in terms of the jumps of 
the invariants $\overline{\Omega}_{\gamma,j}(u^{\frac{1}{2}},
v^{\frac{1}{2}})$, $\fD^{\PP^2}_{u,v}$
also satisfies the second condition of 
Definition \ref{def_normalized}. In conclusion, 
$\fD^{\PP^2}_{u,v}$ is normalized, that is,
$\fD^{\PP^2}_{u,v}$ satisfies condition (i) of Definition \ref{def_scattering}.

If $\sigma=(x,y) \in |\fd_{\gamma,j}| \cap U$, then, by definition of 
$L_\gamma$, we have 
$Z_\gamma^{\sigma} \in i \R_{>0}$, so, using Definitions \ref{def_varphi} and 
\ref{def_central_charge_formula}
\[\varphi_\sigma ( m_\gamma)=2(-rx+d)=2 \frac{\Ima Z_\gamma^{\sigma}}{\sqrt{x^2+2y}}>0\,.\]
It follows that $\fD^{\PP^2}_{u,v}$ satisfies condition (ii) of Definition \ref{def_scattering}.

We fix $K$ a compact set in $U$ and $k \in \R_{\geqslant 0}$. As 
$K$ is compact in $U$, there exists $C >0$ such that $\frac{1}{2}\sqrt{x^2+2y}<C$ for every $(x,y)\in K$. In particular, if 
$\sigma=(x,y) \in K$, then, for every $\gamma \in \Gamma$, we have
$\Ima Z_{\gamma}^{\sigma}=\frac{1}{2}\sqrt{x^2+2y} \, \varphi_\sigma (m_\gamma)$
and so 
\[\Ima Z_{\gamma}^\sigma<C \varphi_\sigma (m_\gamma) \,.\]
If $\sigma=(x,y) \in K$, then, the set of $\gamma \in \Gamma$ such that $M_\gamma^{\sigma}$ is nonempty,
$\Rea Z_{\gamma}^{\sigma} \in [-1,1]$ and $\varphi_\sigma(m_\gamma) \leqslant k$ is contained in the set $F_\sigma$ of $\gamma \in \Gamma$ 
such that $M_\gamma^\sigma$ is nonempty, $\Rea Z_{\gamma}^{\sigma} \in [-1,1]$ and 
$\Ima Z_{\gamma}^\sigma \leqslant Ck$, which is finite
according to the support property satisfied by the stability condition 
$\sigma$. 

By definition of the topology on the space of stability conditions and by the support property, there exists an open subset $U_\sigma$ of $K$ containing $\sigma$ such that $F_{\sigma'}\subset F_{\sigma}$ for every $\sigma' \in U_\sigma$.
By compactness of $K$, there exists finitely many $\sigma_1, \dots,\sigma_l$ such that $K=\cup_{j=1}^l U_{\sigma_j}$. In particular, the set of $\gamma \in \Gamma$  such that there exists $\sigma \in K$ such that
$M_\gamma^\sigma$ is nonempty,
$\Rea Z_{\gamma}^\sigma=0$ and $\Ima Z_{\gamma}^\sigma \leqslant Ck$
is contained in $\cup_{j=1}^l F_{\sigma_j}$, which is finite. It follows that $\fD^{\PP^2}_{u,v}$ satisfies condition (iii) 
of Definition \ref{def_scattering}.
\end{proof}

\begin{defn} \label{def_scattering_stability_q}
We define similarly a scattering diagram $\fD_{q^-}^{\PP^2}$, resp.\
$\fD_{q^+}^{\PP^2}$,
on $U$ for $(M, \fg_{q^-})$, resp.\
$(M, \fg_{q^+})$,
by repeating the previous construction with 
\[R_\gamma \coloneqq \{ \sigma \in U\,| \, Z_\gamma^{\sigma} \in i \R_{>0} \,,
Ib_\gamma^\sigma(q^{\frac{1}{2}}) \neq 0 \}\,, \]
and 
\[H_{\fd_{\gamma,j}}
\coloneq
\left(
-\sum_{\substack{\gamma' \in \Gamma_\gamma
\\ \gamma = \ell \gamma'}}
\frac{1}{\ell}
\frac{Ib_{\gamma',j}(q^{\frac{\ell}{2}})}{
q^{\frac{\ell}{2}}-q^{-\frac{\ell}{2}}}
\right) 
z^{m_\gamma} \,\in (\fg_{q^-})_{m_\gamma}\,,\]
resp.\
\[H_{\fd_{\gamma,j}}
\coloneq
(-1)^{(\gamma,\gamma)-1}
\left(
\sum_{\substack{\gamma' \in \Gamma_\gamma
\\ \gamma = \ell \gamma'}}
\frac{1}{\ell}
\frac{Ib_{\gamma',j}(q^{\frac{\ell}{2}})}{
q^{\frac{\ell}{2}}-q^{-\frac{\ell}{2}}}
\right) 
z^{m_\gamma} \,\in (\fg_{q^+})_{m_\gamma}\,.\] 
\end{defn}

Recall that we introduced the Euler bilinear form $(-,-)\colon \Gamma \otimes 
\Gamma \rightarrow \Z$ in Definition 
\ref{def_bilinear_form}.

\begin{defn}\label{def_scattering_stability_cl}
We define similarly a scattering diagram $\fD_{\cl^-}^{\PP^2}$, 
resp.\ $\fD_{\cl^+}^{\PP^2}$,
on $U$ for $(M, \fg_{\cl^-})$, resp.\ 
$(M, \fg_{\cl^+})$,
by repeating the previous construction with 
\[R_\gamma \coloneqq \{ \sigma \in U \,| \, Z_\gamma^{\sigma} \in i \R_{>0} \,,
Ie_\gamma^\sigma \neq 0 \}\,, \]
and 
\[H_{\fd_{\gamma,j}}\coloneq
\left(
-\sum_{\substack{\gamma' \in \Gamma_\gamma \\ \gamma=\ell \gamma'}}
\frac{1}{\ell^2}
Ie_{\gamma',j}^-
\right)
z^{m_\gamma}
\,\in (\fg_{\cl^-})_{m_\gamma}\,,\]
resp.\
\[H_{\fd_{\gamma,j}}\coloneq
(-1)^{(\gamma,\gamma)-1}
\left(
\sum_{\substack{\gamma' \in \Gamma_\gamma \\ \gamma=\ell \gamma'}}
\frac{1}{\ell^2}Ie^-_{\gamma',j}
\right)
z^{m_\gamma}
\,\in (\fg_{\cl^+})_{m_\gamma}\,.\]
\end{defn}

\subsection{Action of $\psi(1)$ on $\fD_{u,v}^{\PP^2}$}
\label{subsection_psi(1)_P2}

In this section we establish a symmetry 
$\psi(1)$ of the scattering diagram 
$\fD^{\PP^2}_{u,v}$ defined in
\cref{section_scattering_from_stability}.

There is a natural action of the group of autoequivalences $\Aut (\D^b(\PP^2))$ on $\Stab (\PP^2)$ (see \cite[Lemma 8.2]{MR2373143}):
if $T \in \Aut (\D^b(\PP^2))$ and 
$\sigma=(Z,\cA) \in \Stab(\PP^2)$, the action of $T$ on $\sigma$ is defined by 
$T \cdot \sigma \coloneq (Z \circ T^{-1},
T(\cA))$.

In the remaining of this section we take $T=(-\otimes \cO(1)) \in 
\Aut (\D^b(\PP^2))$. 
Recall that we defined in
\cref{subsection_psi(1)} an action 
$\psi(1)$ on $U$.

\begin{lem} \label{lem_T_action}
The action of $T$ on $\Stab(\PP^2)$ 
preserves $U$. Moreover, the action of $T$
restricted to $U$ on $U$
coincides with the action of $\psi(1)$
on $U$.
\end{lem}
 
\begin{proof}
Let $\sigma=(x,y) \in U$. We have to compute $T \cdot \sigma$. 
$T \cdot \sigma$.
We have $T(\cA^{(x,y)})=
T( \Coh^{\texttt{\#}x}(\PP^2))
= \Coh^{\texttt{\#}x+1}(\PP^2)
=
\cA^{\psi(1)((x,y))}$. Indeed, if $E$ is a 
$\mu$-semistable sheaf of slope $x$, then 
$E(1)$ is $\mu$-semistable of slope
\[ \frac{d(E(1))}{r(E(1))}=
\frac{d(E)+r(E)}{r(E)}=\frac{d(E)}{r(E)}+1=x+1\,.\]
Recall from Definition \ref{def_central_charge_formula} that if 
$\gamma=(r,d,\chi)$, then 
\[Z_\gamma^{(x,y)} \coloneqq r y 
+dx
+r+\frac{3}{2}d-\chi
+i
(d-rx)\sqrt{x^2+2y}\,.\]
Using that
$r(E(-1))=r(E)$, $d(E(-1))=d(E)-r(E)$, $\chi(E(-1))=\chi(E)+\frac{r(E)}{2}-\left(\frac{3}{2}r(E)+d(E)\right)$,
we find 
\[ T^{-1}(\gamma)=T^{-1}
((r,d,\chi))=(r,d-r,\chi-r-d)\,.\]
Thus, we have 
\[ Z_{T^{-1}(\gamma)}^{(x,y)}
=ry+(d-r)x+r+\frac{3}{2}(d-r)
-\chi+r+d+i(d-r-rx)\sqrt{x^2+2y}\]
\[=r(y-x+\frac{1}{2})+d(x+1)
+r+\frac{3}{2}d-\chi+i(d-r(x+1))\sqrt{x^2+2y}\,,\]
and so, using Lemma \ref{lem_computation},
we have
\[ Z_{T^{-1}(\gamma)}^{(x,y)}
=Z_\gamma^{\psi(1)((x,y))}\,.\]
We conclude that $T \cdot \sigma =
\psi(1)(\sigma)$.
\end{proof}

We recall that in 
\cref{subsection_psi(1)} 
we also defined an operation 
$\psi(1)$ on scattering diagrams 
on $U$ for $(M,\fg_{u,v})$.

\begin{prop} \label{prop_psi(1)_P2}
We have 
$\psi(1)(\fD_{u,v}^{\PP^2})
=\fD_{u,v}^{\PP^2}$.
\end{prop}

\begin{proof}
The scattering diagram $\fD_{u,v}^{\PP^2}$
is defined in terms of the moduli spaces
$M_\gamma^\sigma$ of $\sigma$-semistable objects in $\D^b(\PP^2)$ for 
$\sigma \in U$. As
$T\in \Aut(\D^b(\PP^2))$, we have 
$M_\gamma^\sigma=M_{T(\gamma)}^{T \cdot \sigma}$ for every $\sigma \in U$ and for every $\gamma \in \Gamma$.
According to Lemma \ref{lem_T_action}, we have 
$T \cdot \sigma = \psi(1)(\sigma)$ for every $\sigma \in U$. It follows that
$\psi(1)(\fD_{u,v}^{\PP^2})
=\fD_{u,v}^{\PP^2}$.
\end{proof}

\section{Consistency of $\fD_{u,v}^{\PP^2}$}
\label{section_consistency_proof}

\subsection{Statement of Theorem \ref{thm_consistency}}
We introduced the scattering diagram 
$\fD_{u,v}^{\PP^2}$ in
\cref{section_scattering_from_stability}.
The main result of the present section is the consistency of $\fD_{u,v}^{\PP^2}$:

\begin{thm} \label{thm_consistency}
The scattering diagram $\fD_{u,v}^{\PP^2}$ on $U$ for $(M, \fg_{u,v})$ is consistent in the sense of Definition \ref{def_consistency}.
\end{thm}

The proof of Theorem \ref{thm_consistency} takes the remaining part of this section. 
According to Definition
\ref{def_consistency}, we have to show that for every $\sigma \in U \cap \Sing(\fD^{\PP^2}_{u,v})$, the local scattering diagram 
$(\fD_{u,v}^{\PP^2})_\sigma$ is consistent in the sense of Definition \ref{def_local_consistency}.
Using the framework of 
\cite{meinhardt2017donaldson} and
 \cite{meinhardt2015donaldson}, we will show that it is essentially a version of the wall-crossing 
formula in Donaldson-Thomas theory.

We fix $\sigma =(x_0,y_0)\in U \cap \Sing(\fD^{\PP^2})$, $k$ a nonnegative integer, and we will show that 
the local scattering diagram
$(\fD_{u,v}^{\PP^2})_\sigma$ is consistent at order $k$ in the sense of Definition 
\ref{def_local_consistency_order_k}.
By definition, $(\fD_{u,v}^{\PP^2})_\sigma$ is a local scattering diagram in $M_\R=\R^2$ identified with the tangent space to 
$U$ at $\sigma$.

We start proving general results on
$(\fD_{u,v}^{\PP^2})_\sigma$ in
\cref{section_local_wall}. 
In 
\cref{section_numerical_mixed_Hodge_theory} we review motivic numerical invariants which can be extracted from mixed Hodge theory. In \cref{section_DT_formalism} we use the Donaldson-Thomas theory framework of
\cite{meinhardt2017donaldson} and
 \cite{meinhardt2015donaldson}
to prove Proposition \ref{prop_dt_ic}.
In \cref{section_wcf} we prove Proposition 
\ref{prop_wcf}, which is a form of the wall-crossing formula in Donaldson-Thomas theory. Finally, we combine Proposition \ref{prop_dt_ic} and Proposition \ref{prop_wcf} to end the proof of Theorem 
\ref{thm_consistency} in \cref{section_proof_consistency}.
 
\subsection{Local structure near a wall}
\label{section_local_wall}
Recall from Definition \ref{def_varphi} that  $\varphi_\sigma \colon M_\R \rightarrow \R$ is given by $(a,b) \mapsto 2(-ax-b)$. 
Recall from Definition \ref{def_local_ray} that a ray $\fd$ of $(\fD_{u,v}^{\PP^2})_\sigma$ is outgoing if $|\fd|=-\R_{\geqslant 0}m_\fd$, and ingoing if $|\fd|=\R_{\geqslant 0} m_\fd$. 
Also, by Definition 
\ref{def_local_scattering}
of a local scattering diagram, we have $\varphi_\sigma(m_\fd)>0$ for every ray $\fd$ of $(\fD_{u,v}^{\PP^2})_\sigma$.
Thus, the ingoing rays of $(\fD_{u,v}^{\PP^2})_\sigma$ are contained 
in the open half-plane 
$\varphi_\sigma >0$ of $M_\R$,
whereas the outgoing rays of $(\fD_{u,v}^{\PP^2})_\sigma$
are contained in the open half-plane 
$\varphi_\sigma <0$ of $M_\R$.

We label $\fd_1^{\iin}, \dots, \fd_K^{\iin}$ the finitely many ingoing rays $\fd$ of $(\fD_{u,v}^{\PP^2})_\sigma$ with $\varphi_\sigma(m_\fd) \leqslant k$,
in such way that 
\[\langle m_{\fd_a^\iin}, m_{\fd_{a'}^\iin} \rangle \leqslant 0\] if
$a \leqslant a'$.
We label $\fd_1^{\oout},\dots,\fd_L^{\oout}$ the finitely many outgoing rays $\fd$ of $\fD_{\sigma}^{\PP^2}$
 with $\varphi_\sigma(m_\fd) \leqslant k$, in such the way that 
 \[\langle m_{\fd_b^\oout}, m_{\fd_{b'}^\oout} \rangle \geqslant 0\] if
$b \leqslant b'$.
The order $k$ consistency of $(\fD_{u,v}^{\PP^2})_\sigma$
is then to equivalent to the equality of order $k$ path-ordered products
\[\Phi_{\fd_K^{\iin}} \dots \Phi_{\fd_1^{\iin}}
= \Phi_{\fd_L^{\oout}} \dots \Phi_{\fd_1^{\oout}}\]
in the group $G_{\varphi_\sigma}^k$.

By definition, $(\fD_{u,v}^{\PP^2})_\sigma$ is a local picture of $\fD^{\PP^2}_{u,v}$
at $\sigma$. It follows that, for every 
$a=1,\dots, K$ (resp.\ $b=1,\dots, L$), there exists 
unique
$\gamma \in \Gamma$ and $j \in J_\gamma$, 
such that $\sigma \in |\fd_{\gamma,j}|$,
$m_{\fd_a^{\iin}}=m_\gamma$, 
and 
$H_{\fd_a^{\iin}}=H_{\fd_{\gamma,j}}$ 
(resp.\
$m_{\fd_b^{\oout}}=m_\gamma$, and 
$H_{\fd_b^{\oout}}=H_{\fd_{\gamma,j}}$).
We denote $\gamma_a^{\iin} \coloneq \gamma$ and 
$j_a^{\iin} \coloneq j$
(resp.\ $\gamma_b^{\oout}\coloneq \gamma$
and $j_b^{\oout} \coloneq j$).

\begin{lem} \label{lem_rank_two}
The sublattice $\Gamma^\sigma$ of $\Gamma$ generated by $\gamma_a^{\iin}$, $1 \leqslant a \leqslant K$, and $\gamma_b^{\oout}$, $1 \leqslant b \leqslant L$, is of rank $2$.
\end{lem}

\begin{proof}
As $\sigma \in U \cap \Sing (\fD^{\PP^2}_{u,v})$, the sublattice $\Gamma^\sigma$
has rank $\geqslant 2$.
As the lattice $\Gamma$ is of rank $3$, it remains to show that the sublattice 
$\Gamma^\sigma$
is not of rank 3. 
As the rays $\fd_{\gamma_a^{\iin},
j_a^{\iin}}$, $1\leqslant a \leqslant K$,
and $\fd_{\gamma_b^{\oout},
j_b^{\oout} }$, 
$1 \leqslant b \leqslant L$, meet at $\sigma$, all the central charges $Z_{\gamma_a^{\iin}}^\sigma$,
$1 \leqslant a \leqslant K$, and $Z_{\gamma_b^{\oout}}^\sigma$,
$1 \leqslant b \leqslant L$, are purely imaginary. If the sublattice $\Gamma^\sigma$ had rank 3, we would conclude 
that $Z_\gamma^\sigma$ is purely imaginary for every $\gamma \in \Gamma$. 
But it cannot be the case as $Z_{(0,0,1)}^{\sigma'}=-1$ for every $\sigma' \in U$, 
see Remarks following Proposition \ref{prop_stability}.
\end{proof}

The fact that all the rays $\fd_{\gamma_a^{\iin},
j_a^{\iin} }$, $1 \leqslant a \leqslant K$
and $\fd_{\gamma_b^{\oout},
j_b^{\oout} }$, 
$1 \leqslant b \leqslant L$, meet at $\sigma$ implies 
that all the central charges $Z_{\gamma_a^{\iin}}^\sigma$,
$1 \leqslant a \leqslant K$, and $Z_{\gamma_b^{\oout}}^\sigma$,
$1 \leqslant b \leqslant L$, are purely imaginary, with positive imaginary parts, and 
so in particular are positively collinear. 
Therefore, there exists a potential wall $W_\sigma$ passing through $\sigma$, characterized by the alignement of the central charges $Z_\gamma$
for $\gamma \in \Gamma^\sigma$. According to Lemma \ref{lem_rank_two}, 
the lattice $\Gamma^\sigma$ is of rank $2$, and so $W_\sigma$ is the only potential wall passing through $\sigma$ which is a potential wall for some $\gamma \in \Gamma^\sigma$.

By the local finiteness of actual walls
reviewed in \cref{section_moduli_invariants}, there exists $U_\sigma$ open convex neighborhood of $\sigma$ in $U$ such that every actual wall for any of
$\gamma_{\fd_a^{\iin}}$, $1 \leqslant a \leqslant K$,
$\gamma_{\fd_b^{\oout}}$, $1 \leqslant b \leqslant L$, intersecting 
$U_\sigma$, necessarily coincides with 
$W_\sigma$ in restriction to $U_\sigma$. 
Up to shrinking further $U_\sigma$, we also assume that $W_\sigma$
is the only potential wall intersecting $U_\sigma$ of the form 
$W_{\gamma_{\fd_a^{\iin}},\gamma_{\fd_{a'}^{\iin}}}$, 
$1 \leqslant a, a' \leqslant K$, or
$W_{\gamma_{\fd_b^{\oout}},\gamma_{\fd_{b'}^{\oout}}}$, 
$1 \leqslant b, b' \leqslant L$.

\begin{lem} \label{key_lem_walls}
The tangent line to the wall $W_\sigma$ at the point 
$\sigma=(x_0,y_0)$ is the line in $U$ of equation
\[(x-x_0)x_0+(y-y_0)=0\,.\]
In other words, identifying $M_\R$ with the tangent
space to $U$ at $\sigma$, it is the line of equation 
$\varphi_\sigma=0$
\end{lem}

\begin{proof}
By definition of a wall, there exists 
$\gamma_1=(r_1,d_1,\chi_1), \gamma_2=(r_2,d_2,\chi_2)
\in \Gamma$ with 
$r_2 d_1-r_1 d_2 \neq 0$ such that $W_\sigma$ is defined by the condition that $Z_{\gamma_1}$ and $Z_{\gamma_2}$ are positively collinear. It follows from the explicit formula given in Proposition \ref{def_central_charge_formula} for the central charge that $W_\sigma$ is defined by the equation 
$F(x,y)=0$, where 
\[F(x,y)=(d_1-r_1x)(r_2 y +d_2 x+r_2+\frac{3}{2}d_2-\chi_2)
-(d_2-r_2 x)(r_1 y +d_1 x+r_1+\frac{3}{2}d_1-\chi_1)\,,\]
and so the tangent line to $W_{\sigma}$ in a point of coordinates 
$(x_0,y_0)$ is given by the equation 
\[(\partial_x F)(x_0,y_0)(x-x_0)+(\partial_y F)(x_0,y_0)(y-y_0)=0\,.\] 

If we take 
$(x_0,y_0)=\sigma$, defined by the conditions 
$\Rea Z_{\gamma_1}=\Rea Z_{\gamma_2}=0$, then we find the particularly simple formulas
\[(\partial_x F)(x_0,y_0)=(r_2 d_1-r_1 d_2) x_0 \,\,,
(\partial_y F)(x_0,y_0)=r_2d_1-r_1d_2\,.\]
Using that $r_2 d_1-r_1 d_2 \neq 0$, we deduce that the tangent line to $W_\sigma$ at $\sigma$ is given by the equation $(x-x_0)x_0+(y-y_0)=0$.
\end{proof}

It follows from Lemma \ref{key_lem_walls} that all the rays 
$\fd_{\gamma^\iin_a, j^\iin_a}$, $1 \leqslant a \leqslant K$, are on one side of the wall 
$W_\sigma$ and that all the rays
$\fd_{\gamma^\oout_b, j^\oout_b}$, $1 \leqslant b \leqslant L$, are on the other side of the wall
$W_\sigma$. Let $U_\sigma^{\iin}$ be the connected component of 
$U_\sigma -(W_\sigma \cap U_\sigma)$ containing all 
the rays $\fd_{\gamma^\iin_a, j^\iin_a}$, $1 \leqslant a \leqslant K$, 
and let 
$U_\sigma^{\oout}$ be the connected component of 
$U_\sigma - (W_\sigma \cap U_\sigma)$
containing all the rays $\fd_{\gamma^\oout_b, j^\oout_b}$, $1\leqslant b \leqslant L$.

We choose a point $\sigma_{\iin} \in U_\sigma^{\iin}$ between 
$W_\sigma$ and $\fd_{\gamma_1^\iin,j_1^\iin}$, and a point $\sigma_{\oout}
\in U_\sigma^{\oout}$ between $W_\sigma$ and $\fd_{\gamma_1^\oout, j_1^\oout}$.
As there are only countably many potential walls, we can assume that 
$\sigma_\iin$ and $\sigma_\oout$ are away from every potential wall. Finally, as
according to Lemma \ref{key_lem_walls} the tangent line to $W_\sigma$ at $\sigma$ is never vertical, we can assume that 
$\sigma_\iin$ and $\sigma_\oout$ have the same $x$-coordinates, that is, that the line passing through $\sigma_\iin$ and $\sigma_\oout$ is vertical.

Figure \ref{figure_notation} gives a schematic summary of the notation introduced above.

\begin{figure}[h!]
\centering
\setlength{\unitlength}{1.2cm}
\begin{picture}(10,5)
\thicklines
\put(3.5,0){$\fd_{\gamma_1^\iin,j_1^\iin}$}
\put(5.5,0){$\fd_{\gamma_K^\iin,j_K^\iin}$}
\put(3.5,4.75){$\fd_{\gamma_1^\oout,j_1^\oout}$}
\put(5.5,4.75){$\fd_{\gamma_L^\oout,j_L^\oout}$}
\put(9.25,2.5){$\{\varphi_\sigma=0\}$}
\put(9.25,1){$W_\sigma$}
\put(3,1.5){\circle*{0.1}}
\put(3,1.2){$\sigma_\iin$}
\put(3,3.5){\circle*{0.1}}
\put(3,3.75){$\sigma_\oout$}
\put(1,0.2){$U_\sigma^\iin$}
\put(1,4){$U_\sigma^\oout$}
\put(5.15,2.25){$\sigma$}
\put(5,2.5){\line(-1, 2){1}}
\put(5,2.5){\line(1, 2){1}}
\put(5,2.5){\line(0,1){2}}
\put(5,2.5){\line(1,-2){1}}
\put(5,2.5){\line(-1,-2){1}}
\put(5,2.5){\line(1,0){4}}
\put(5,2.5){\line(-1,0){4}}
\qbezier(1,1)(5,4)(9,1)
\end{picture}
\caption{}
\label{figure_notation}
\end{figure}

\begin{lem}\label{lem_ordering} We have
$\Rea Z_{\gamma_a^\iin}^{\sigma_\iin}<0$ (resp.\ $\Rea Z_{\gamma_b^\oout}^{\sigma_\oout}<0$) for every $1 \leqslant a \leqslant K$ (resp.\ $1 \leqslant b \leqslant L$)
\end{lem}

\begin{proof}
The support of each ray $\fd_{\gamma_a^\iin, j_a^\iin}$ is contained in the line of equation 
$\Rea Z_{\gamma_a^\iin}=0$. 
Using that, if $\gamma=(r,d,\chi)$, $\Rea Z_\gamma=ry+dx+r+\frac{3}{2}d-\chi$, and 
$m_\gamma=(r,-d)$, we see that the condition 
\[\langle m_{\gamma_{\fd_a^\iin}}, m_{\gamma_{\fd_{a'}^\iin}} \rangle \leqslant 0 \,\,\,
(\text{resp}.\ \langle m_{\gamma_{\fd_b^\oout}}, m_{\gamma_{\fd_{b'}^\oout}} \rangle \geqslant 0)\,,\]
if $a \leqslant a'$ (resp.\ $b \leqslant b'$), is equivalent to the fact that the support of 
 $\fd_{\gamma_a^\iin, j_a^\iin}$
 (resp.\  $\fd_{\gamma_b^\oout, j_a^\oout}$)
 is contained in the half-plane
 $\Rea Z_{\gamma_{a'}^\iin} \leqslant 0$ 
 (resp.\ $\Rea Z_{\gamma_{b'}^\iin} \leqslant 0$)
 if 
 $a \leqslant a'$ (resp.\ $b \leqslant b'$). As $\sigma_\iin$
 (resp.\ $\sigma_\oout$)
 is between $W_\sigma$ and 
 $\fd_{\gamma_1^\iin, j_1^\iin}$
 (resp.\ $\fd_{\gamma_1^\oout, j_1^\oout}$), it follows that $\Rea Z_{\gamma_a^\in}^{\sigma_\iin}<0$ 
 (resp.\ $\Rea Z_{\gamma_b^\oout}^{\sigma_\oout}<0$) for every $1 \leqslant a \leqslant K$ (resp.\ $1 \leqslant b \leqslant L$).
\end{proof}

\subsection{Numerical invariants from mixed Hodge theory}
\label{section_numerical_mixed_Hodge_theory}
In this section we review some definitions and facts which will be useful in the proof 
Theorem \ref{thm_consistency} in \cref{section_proof_consistency}.

For every $X$ quasiprojective variety over 
$\C$, the cohomology groups with compact support
$H^j_c(X,\Q)$ come with a natural
mixed Hodge structure
\cite{MR0498551, MR0498552}.
In particular, we have an increasing weight filtration $W$ on $H^j_c(X,\Q)$ and a decreasing Hodge filtration $F$ on 
$H^j_c(X,\C)$. For every 
$p, q\in \Z$, we define virtual Hodge numbers 
\[ h^{p,q}_{\mathrm{vir}}(X)\coloneq
\sum_{j=0}^{2 \dim X}
(-1)^j \dim \mathrm{Gr}_F^p
\mathrm{Gr}^W_{p+q} H^j_c(X,\C) \in \Z
\,.\]
If $X$ is smooth and projective, the virtual Hodge numbers coincide with the usual Hodge numbers up to the sign $(-1)^{p+q}$.
We organize the virtual Hodge numbers in a virtual Hodge polynomial (also called $E$-polynomial):
\[ h(X)(u,v)\coloneq \sum_{p,q} h^{p,q}_{\mathrm{vir}}(X) 
u^p v^q \in \Z[u,v]\,.\]
Setting $u=v=q^{\frac{1}{2}}$, we get a virtual Poincar\'e polynomial:
\[ b(X)(q^{\frac{1}{2}})=\sum_j b_{j,\mathrm{vir}}(X)q^{\frac{j}{2}}\in\Z[q^{\frac{1}{2}}]\,,\]
where the virtual Betti numbers are
given by 
\[ b_{j,\mathrm{vir}}(X)=\sum_{p+q=j}
h_{\mathrm{vir}}^{p,q}(X) \in \Z\,.\]
The key property that we will use is that the virtual Hodge polynomial is motivic, in the sense that:
\begin{itemize}
\item[(i)] If $Z$ is a closed subvariety of a quasiprojective variety $X$ over $\C$, then 
\[h(X)=h(Z)+h(X-Z)\,.\]
\item[(ii)] If $X$ and $Y$ are two quasiprojective varieties over $\C$, then
$h(X \times Y)=h(X)h(Y)$.
\end{itemize}
The motivic property follows from the compatibility of the mixed Hodge structure with the excision long exact sequence 
in cohomology with compact support and with the Künneth formula. The virtual Hodge polynomial
$ X \mapsto h(X)$ is uniquely determined 
by its values on the smooth projective varieties and by the motivic property.

According to \cite[Theorem 3.10]{toen2005grothendieck}
(see also \cite{MR2354923}), the 
virtual Hodge polynomial \[X \mapsto h(X) 
\in \NN[u,v]\] with $X$ quasiprojective variety over $\C$ can be naturally extended 
in a virtual Hodge rational function 
\[ X \mapsto h(X) \in 
\Z[u,v][(uv)^{-1},\{((uv)^n-1)^{-1}\}_{n \geqslant 1}] \subset \Z(u,v)\,,\]
with $X$ Artin stack of finite type with 
affine stabilizers.
Setting $u=v=q^{\frac{1}{2}}$, we get a 
virtual Poincar\'e rational function 
\[ X \mapsto b(X) \in \Z[q^{\frac{1}{2}}]
[q^{-1},\{(q^n-1)^{-1}\}_{n\geqslant 1}]
\subset \Z(q^{\frac{1}{2}})\,.\]

We have for example, denoting $B \C^{*}$ the classifying stack of $\C^{*}$,
\[ h(B \C^{*})(u,v)=\frac{1}{uv-1}\,,
\]
and 
\[ b(B\C^{*})(q^{\frac{1}{2}})=\frac{1}{q-1}\,.\]

If $X$ is an Artin stack of finite type with affine stabilizers, writing $X$ as a
finite
disjoint union of locally closed equidimensional substacks $X_j$, we define a symmetrized version of the virtual Hodge rational function by
\[ \tilde{h}(X)(u^{\frac{1}{2}},v^{\frac{1}{2}})\coloneq \sum_j(-(uv)^{\frac{1}{2}})^{-\dim X_j} h(X_j)(u,v) \in \Z[u,v]((uv)^{\frac{1}{2}}) \subset \Z(u^{\frac{1}{2}},v^{\frac{1}{2}})\,.\]
Setting $u=v
=q^{\frac{1}{2}}$, $(uv)^{\frac{1}{2}}=q^{\frac{1}{2}}$, we get a symmetrized version 
$\tilde{b}(X)(q^{\frac{1}{2}}) \in \Z(q^{\frac{1}{2}})$ of the virtual Poincar\'e rational function.

We have for example, using that $\dim 
B\C^{*}=-1$, 
\[ \tilde{h}(B\C^{*})(u^{\frac{1}{2}},v^{\frac{1}{2}})
=-\frac{(uv)^{\frac{1}{2}}}{uv-1}
=-\frac{1}{(uv)^{\frac{1}{2}}-(uv)^{-\frac{1}{2}}}\,,\]
and 
\[ \tilde{b}(B\C^{*})(q^{\frac{1}{4}})
=-\frac{1}{q^{\frac{1}{2}}-q^{-\frac{1}{2}}}\,.\]
Beware that $X \mapsto \tilde{h}(X)$ does not satisfy the same motivic property 
than $X \mapsto h(X)$.

\subsection{Donaldson-Thomas formalism}
\label{section_DT_formalism}

In this section we prove Proposition 
\ref{prop_dt_ic}, which will be used in the the proof of 
Theorem \ref{thm_consistency}
in \cref{section_proof_consistency}. We use 
systematically the notation introduced in \cref{section_local_wall}.

Let
$\cA^{\sigma_{\iin}}$ be the abelian category heart of the stability condition $\sigma_{\iin}$, and let  $\cA^{\sigma_{\oout}}$ be the abelian category heart of the stability condition
$\sigma_{\oout}$.
As we have chosen $\sigma_\iin$ and $\sigma_\oout$ with the same $x$-coordinate, it follows from Definition \ref{defn_heart} that $\cA^{\sigma_\iin}=\cA^{\sigma_\oout}$. In what follows, we denote by $\cA$ this abelian category.

For every nonzero $\gamma \in \Gamma$ and $\sigma' \in U$, we denote
\[ \phi_\gamma^{\sigma'} \coloneq \frac{1}{\pi} \Arg Z_\gamma^{\sigma'}\,.\]
In particular, we have $\phi_\gamma^{\sigma'}=\frac{1}{2}$ if and only if $Z_\gamma^{\sigma'} \in i \R_{>0}$.

Shrinking $U_\sigma$ if necessary, it follows from the support property for stability conditions that there exists an interval $I \subset (0,1)$ containing $\frac{1}{2}$
and closed in $[0,1]$ such that:
\begin{itemize}
    \item[(i)] The set of $\gamma \in \Gamma$
such that $\varphi_{\sigma}(m_\gamma) \leqslant k$ 
and $\phi^{\sigma'}_\gamma \in I$
does not depend on $\sigma' \in U_\sigma$.
We denote by $\Lambda$ the union of this finite set with $\{0\}$
(which depends on $\sigma$ and $k$).
    \item[(ii)]  For every 
$\sigma' \in U_\sigma$, 
the set of $\gamma \in \Gamma$
such that $\varphi_{\sigma}(m_\gamma) \leqslant k$, $\phi^{\sigma'}_\gamma \in I$
and $Ih_\gamma^{\sigma'}(u^{\frac{1}{2}}, v^{\frac{1}{2}}) \neq 0$ is contained in the set of 
$\gamma_{\fd_a^{\iin}}$, $1 \leqslant a 
\leqslant K$,
and $\gamma_{\fd_b^{\oout}}$, $1 \leqslant b\leqslant L$.
\end{itemize}
In general, $I$ is a small interval around $\frac{1}{2}$.
The objects $U_\sigma$, $\Lambda$ and $I$ will play a fundamental in our proof
of the consistency of $\fD_{u,v}^{\PP^2}$ around $\sigma$, as 
$U_\sigma$ parametrizes the set of stability conditions in a neighborhood of $\sigma$, $I$ contains all the relevant phases
and $\Lambda$ parametrizes all the relevant classes locally around $\sigma$.

\begin{defn}For every $\phi \in I$, we denote 
$\cA^{\sigma_{\iin}}(\phi)$ (resp.\ $\cA^{\sigma_{\oout}}(\phi)$) the abelian subcategory of
$\cA$ whose objects are $0$ and the 
nonzero $\sigma_{\iin}$(resp.\ $\sigma_{\oout}$)-semistable objects $E$
 with 
$\phi_{\gamma(E)}^{\sigma_\iin}= \phi$ (resp.\ $\phi_{\gamma(E)}^{\sigma_\oout}= \phi$).
\end{defn}

The categories $\cA^{\sigma_{\iin}}(\phi)$ and \ $\cA^{\sigma_{\oout}}(\phi)$
are indeed abelian by \cite[Lemma 5.2]{MR2373143}.

\begin{defn}
As $\sigma_\iin$ (resp.\ 
$\sigma_\oout$) is away from every 
potential wall, for every $\phi \in I$, the lattice of 
$\gamma \in \Gamma$ such that $\phi_\gamma^{\sigma^\iin}=\phi$
(resp.\ $\phi_\gamma^{\sigma^\oout}=\phi$) is of rank $1$. In particular, there exists $\gamma^\iin_\phi \in \Gamma$
(resp.\ $\gamma^\oout_\phi \in \Gamma$) such that 
\[ \{ \gamma \in \Gamma\,|\,
\phi_\gamma^{\sigma_\iin}=\phi \,,
\fM_\gamma^{\sigma_\iin} \neq \emptyset\}
\subset \NN \gamma^\iin_\phi \,\]
(resp.\ $\{ \gamma \in \Gamma\,|\,
\phi_\gamma^{\sigma_\oout}=\phi \,,
\fM_\gamma^{\sigma_\oout} \neq \emptyset \}
\subset \NN \gamma^\oout_\phi$).
\end{defn}

Note that by varying $\phi \in I$, each $\gamma_{\fd_a^\iin}$, $1 \leq a \leq K$
(resp. $\gamma_{\fd_b^\oout}$, $1 \leq b \leq L$) will appear in some $\NN \gamma_\phi^\iin$ (resp.\ $\NN \gamma_\phi^\oout$).

We apply the formalism of 
Donaldson-Thomas theory,
such as described by Meinhardt \cite{meinhardt2015donaldson} (see in particular Example 3.34 for a discussion of surfaces)
to the abelian categories $\cA^{\sigma_{\iin}}(\phi)$
and $\cA^{\sigma_{\oout}}(\phi)$.

\begin{lem} \label{lem_technical_meinh}
For every $\phi \in I$, the abelian categories $\cA^{\sigma_{\iin}}(\phi)$ and 
$\cA^{\sigma_{\oout}}(\phi)$ satisfy the technical conditions
required to apply
\cite[Theorem 1.1]{meinhardt2015donaldson}.
\end{lem}

\begin{proof}
Technical conditions 1)-6) of \cite{meinhardt2015donaldson} are general assumptions about moduli spaces and deformation theory of objects in an abelian category. They follow for
the abelian categories
 $\cA^{\sigma_{\iin}}(\phi)$ and 
$\cA^{\sigma_{\oout}}(\phi)$
from the fact that the moduli stacks and moduli spaces 
of $\sigma'$-semistable in $\D^b(\PP^2)$ have a description as moduli stacks and moduli spaces 
of quiver representations for every $\sigma' \in U$, see \cite[Corollary 7.6]{MR3010070}.

Technical condition 7) of \cite{meinhardt2015donaldson} is the smoothness of the moduli stacks of 
objects, or equivalently the locally constant behaviour of 
$\dim \Hom (E,F)-\dim \Ext^1(E,F)$. As $Z_\gamma^{\sigma'} \in \R$ for every $\gamma \in 
\Gamma^0$
and for every $\sigma' \in U$ (see Remarks after Proposition
\ref{prop_stability}), and as $\varphi$ is neither $0$ or $1$, the class $\gamma$ of an object in 
$\cA^{\sigma_{\iin}}(\phi)$ or 
$\cA^{\sigma_{\oout}}(\phi)$ is never in $\Gamma^0$. 
Li-Zhao \cite{MR3936077} have shown that, for every 
$\sigma' \in U$, and for every $\sigma'$-semistable objects $E$ and $F$
with $\phi_{\gamma(E)}^{\sigma'}=\phi_{\gamma(F)}^{\sigma'}$ and 
$\gamma(E), \gamma(F) \notin \Gamma^0$, we have $\Ext^2(E,F)=0$.
It follows that $\cA^{\sigma_{\iin}}(\phi)$ and 
$\cA^{\sigma_{\oout}}(\phi)$ satisfy technical condition 7).

Technical condition 8) of \cite{meinhardt2015donaldson} is the symmetry of the pairing 
$\dim \Hom(E,F)-\dim \Ext^1(E,F)$. Let $E$ and $F$ be objects of $\cA^{\sigma_{\iin}}(\phi)$ (resp. 
$\cA^{\sigma_{\oout}}(\phi)$). Then we have seen in the check of technical condition 7)
that $\Ext^2(E,F)=0$. So we have 
\[\dim \Hom(E,F)-\dim \Ext^1(E,F)=\chi(E,F)\,,\] 
which by the Hirzebruch-Riemann-Roch formula is given by
\[\chi(E, F)=\int_{\PP^2} \ch(E^\vee) \ch(F) \td(\PP^2)\,.\]
As we have chosen $\sigma_\iin$ (resp.\ $\sigma_\oout$) away from every potential wall, the fact that $\phi_{\gamma(E)}^{\sigma_\iin}=\phi_{\gamma(F)}^{\sigma_\iin}=\phi$
(resp.\ $\phi_{\gamma(E)}^{\sigma_\oout}=\phi_{\gamma(F)}^{\sigma_\oout}=\phi$)
implies that $\gamma(E)$ and $\gamma(F)$, and so $\ch(E)$ and $\ch(F)$, are collinear, 
which by the above formula implies the symmetry $\chi(E,F)=\chi(F,E)$.
\end{proof}

The application of the Donaldson-Thomas formalism of \cite{meinhardt2015donaldson} to Gieseker
semistable sheaves on del Pezzo surfaces is discussed in 
\cite[Example 3.34]{meinhardt2015donaldson} and in \cite{MR3874687}. Lemma \ref{lem_technical_meinh}
will able us to apply this formalism to Bridgeland semistable objects in $\D^b(\PP^2)$.

Recall from \cref{section_moduli_invariants} that for every 
$\gamma \in \Gamma$ and 
$\sigma' \in U$, the moduli stack 
$\fM_\gamma^{\sigma'}$ of 
$\sigma'$-semistable objects of class 
$\gamma$ in $\D^b(\PP^2)$ is smooth. 
The symmetrized virtual Hodge rational functions 
$\tilde{h}(\fM_\gamma^{\sigma'})(u^{\frac{1}{2}},v^{\frac{1}{2}})$ are defined according to \cref{section_numerical_mixed_Hodge_theory}.

\begin{prop}\label{prop_dt_ic}
For every $\phi \in I$,
we have the equality
\[\sum_{n\geqslant 0}
\tilde{h}(\fM_{n\gamma^\iin_\phi}^{\sigma_{\iin}})(u^{\frac{1}{2}},v^{\frac{1}{2}})
z^{nm_{\gamma^\iin_\phi}}= \exp
\left(-\sum_{n \geqslant 1} 
\sum_{\ell \geqslant 1} \frac{1}{\ell}
\frac{Ih_{n\gamma^\iin_\phi}^{\sigma_{\iin}}(u^{\frac{\ell}{2}}, v^{\frac{\ell}{2}})}{(uv)^{\frac{\ell}{2}}-(uv)^{-\frac{\ell}{2}}} 
z^{\ell n m_{\gamma^\iin_\phi}} \right) \,,\]
of power series in $z^{m_{\gamma^\iin_\phi}}$ with coefficients in 
$\Q(u^{\frac{1}{2}}, v^{\frac{1}{2}})$,
and the equality
\[\sum_{n \geqslant 0}
\tilde{h}(\fM_{n\gamma^\oout_\phi}^{\sigma_{\oout}})(u^{\frac{1}{2}},v^{\frac{1}{2}})
z^{n m_{\gamma^\oout_\phi}} = \exp
\left(-\sum_{n \geqslant 1} \sum_{\ell \geqslant 1} 
\frac{1}{\ell}
\frac{
Ih_{n\gamma^\oout_\phi}^{\sigma_{\oout}}(u^{\frac{\ell}{2}}, v^{\frac{\ell}{2}})}{(uv)^{\frac{\ell}{2}}-(uv)^{-\frac{\ell}{2}}} 
z^{\ell n m_{\gamma^\oout_\phi}} \right) \,,\]
of power series in $z^{m_{\gamma^\oout_\phi}}$ with coefficients in 
$\Q(u^{\frac{1}{2}}, v^{\frac{1}{2}})$.
\end{prop}

\begin{proof}
We apply 
\cite[Theorem 1.1]{meinhardt2015donaldson} to 
$\cA^{\sigma_{\iin}}(\phi)$ and $\cA^{\sigma_{\oout}}(\phi)$.
The required assumptions have been checked in Lemma \ref{lem_technical_meinh}.

More precisely, the main result of \cite{meinhardt2015donaldson}
(Theorem 1.1 combined with the formula defining the Donaldson-Thomas invariants in 
Lemma 5.1) is an equality between generating series with coefficients 
in the Grothendieck group of mixed Hodge structures over $\C$ 
with addition of a square root 
$\mathbb{L}^{\frac{1}{2}}$ of the Tate motive, and inversion of 
$\mathbb{L}^{\frac{1}{2}}$, 
$(\mathbb{L}^n-1)$, $n \geqslant 1$.
We get numerical identities by application of the virtual Hodge function (extended such that $h(\mathbb{L}^{\frac{1}{2}})=-(uv)^{\frac{1}{2}}$).
\end{proof}

\subsection{Wall-crossing formula}
\label{section_wcf}
In this section, we prove Proposition 
\ref{prop_wcf}, which will be used in the the proof of 
Theorem \ref{thm_consistency}
in \cref{section_proof_consistency}.
We continue using the notation introduced in \cref{section_local_wall}-- \cref{section_DT_formalism}.

Recall from the definition of the interval $I$ 
in \cref{section_DT_formalism} that we denote by $\Lambda$ the finite subset of $\Gamma$ given by
\[ \Lambda=\{0\} \cup \{ \gamma \in \Gamma\,|\,
\varphi_\sigma(m_\gamma) \leqslant k\,,
\phi_\gamma^{\sigma_\iin} \in I\}
=\{0\} \cup \{\gamma \in \Gamma \,|\,
\varphi_\sigma(m_\gamma) \leqslant k\,,
\phi_\gamma^{\sigma_\oout} \in I\}\,.\]

\begin{defn}\label{def_algebra_A_Lambda}
We denote by $A_\Lambda$ the associative noncommutative
$\Q(u^{\frac{1}{2}},v^{\frac{1}{2}})$-algebra, given as a $\Q(u^{\frac{1}{2}},
v^{\frac{1}{2}})$-vector space by 
\[ A_\Lambda \coloneq \bigoplus_{\gamma \in \Lambda} 
\Q(u^{\frac{1}{2}},v^{\frac{1}{2}})z^{m_\gamma} \,,\]
and with the product defined by 
\[ z^{m_\gamma} \cdot z^{m_{\gamma'}}= (-1)^{\langle m_\gamma, m_{\gamma'} \rangle} (uv)^{\frac{\langle m_\gamma, m_{\gamma'}\rangle}{2}} z^{m_\gamma+m_{\gamma'}}\]
if $\gamma+\gamma' \in \Lambda$, and 
\[ z^{m_\gamma} \cdot z^{m_{\gamma'}}=0\]
if $\gamma+\gamma' \notin \Lambda$.
\end{defn}

\begin{defn}
For every $\phi \in I$, we define 
\[ \Lambda_\phi^\iin \coloneq 
\{0\} \cup \{ \gamma \in \Lambda \,|\, 
\phi_\gamma^{\sigma_\iin}=\phi\}\,\]
and 
\[ \Lambda_\phi^\oout \coloneq 
\{0\} \cup \{ \gamma \in \Lambda \,|\, 
\phi_\gamma^{\sigma_\oout}=\phi\}\,.\]
We have 
\[ \Lambda=\bigcup_{\phi \in I} \Lambda_\phi^\iin=\bigcup_{\phi \in I} \Lambda_\phi^\oout\,.\]
\end{defn}

\begin{defn}
We denote 
\[\prod_{\phi \in I}^{\rightarrow}\]
for an ordered product where the factors
with higher value of $\phi$ are on the left of those with lower value of $\phi$.
\end{defn}

\begin{prop} \label{prop_wcf}
We have the following equality in 
the $\Q(u^{\frac{1}{2}},v^{\frac{1}{2}})$-algebra $A_\Lambda$:
\[
\prod_{\phi \in I}^{\rightarrow} 
\left(
\sum_{\gamma \in \Lambda_\phi^\iin}
\tilde{h}(\fM_\gamma^{\sigma_{\iin}})(u^{\frac{1}{2}},
v^{\frac{1}{2}})
z^{m_\gamma}
\right)
= 
\prod_{\phi \in I}^{\rightarrow} 
\left(
\sum_{\gamma \in \Lambda_\phi^{\oout}}
\tilde{h}(\fM_\gamma^{\sigma_{\oout}})(u^{\frac{1}{2}},
v^{\frac{1}{2}})
z^{m_\gamma}
\right)\,.\]
\end{prop}

The proof of Proposition \ref{prop_wcf} takes the remaining part of
\cref{section_wcf}.

We follow the logic of the proof of
\cite[Proposition 6.20]{MR2357325},
in which Joyce considers Gieseker semistable sheaves on a surface $S$ with $K_S^{-1}$ nef.
We refer to \cite{MR2854172} and 
\cite{MR2303235} for details on definitions and on the use of motivic Hall algebras.

For every $\gamma \in \Gamma$, let 
$\fM_\gamma$ be the algebraic stack of objects of $\cA$ of class $\gamma$, and let 
\[ \fM \coloneq \bigcup_{\gamma \in \Gamma} \fM_\gamma \,.\]
Let $H(\fM)$ be the corresponding 
motivic Hall algebra. Elements of $H(\fM)$ are motivic stack functions on $\fM$, that is, classes
$[Z \rightarrow \fM]$ defined 
up to scissor relationships, with
$Z$ an Artin stack of finite type with affine stabilizers (see \cite{MR2854172}).  As a vector space, we have a $\Gamma$-grading
\[ H(\fM)=\bigoplus_{\gamma \in \Gamma}
H(\fM_\gamma) \,,\]
where $H(\fM_\gamma)$ is the space of motivic stack functions supported on
$\fM_\gamma$, that is of $[Z \rightarrow \fM]$ factoring through $\fM_\gamma \hookrightarrow \fM$.
The associative product $\star$ on 
$H(\fM)$ is $\Gamma$-graded, that is,
has components, for every $\gamma_1$, $\gamma_2 \in \Gamma$,
\[ \star \colon H(\fM_{\gamma_1}) \otimes H(\fM_{\gamma_2}) \rightarrow H(\fM_{\gamma_1+\gamma_2})\,,\]
determined by pullback and pushforward to and from the stack of extensions in $\cA$ of objects of class $\gamma_2$ by objects of class $\gamma_1$. 

We denote 
\[ H_\Lambda(\fM) \coloneq \bigoplus_{\gamma \in \Lambda} H(\fM_\gamma)\,\]
with product, for $\gamma_1, \gamma_2 \in \Lambda$, 
\[ \star \colon H(\fM_{\gamma_1}) \otimes H(\fM_{\gamma_2}) \rightarrow H(\fM_{\gamma_1+\gamma_2})\,,\]
restricted from $H(\fM)$ if $\gamma_1+\gamma_2 \in \Lambda$, and set to $0$ if $\gamma_1+\gamma_2 \notin \Lambda$. This defines a structure of associative algebra on
$H_\Lambda(\fM)$.

For every $\gamma \in \Lambda$, the characteristic function of the stack
$\fM_\gamma^{\sigma_\iin}$ (resp.\ 
$\fM_\gamma^{\sigma_\oout}$)
of $\sigma_\iin$-semistable 
(resp.\ $\sigma_\oout$-semistable) objects in $\cA$ of class $\gamma$ defines an element
\[\delta_\gamma^{\iin}
=[\fM_\gamma^{\sigma_\iin} \hookrightarrow \fM]\] 
(resp.\ 
$\delta_\gamma^{\oout}=[\fM_\gamma^{\sigma_\oout} \hookrightarrow \fM]$) of $H_\Lambda(\fM)$. For every $\phi \in I$, we define 
\[ \delta_\phi^{\iin} 
\coloneq \sum_{\gamma \in \Lambda_\phi^\iin} \delta_\gamma^{\iin} \in H_\Lambda(\fM) \,,\]
and 
\[ \delta_\phi^{\oout} 
\coloneq \sum_{\gamma \in \Lambda_\phi^\oout} \delta_\gamma^{\oout} \in H_\Lambda(\fM) \,.\]

From the
existence and uniqueness of the Harder-Narasimhan filtration for the stability conditions $\sigma_\iin$ and $\sigma_\oout$ of heart $\cA$, we have the identity
\[ \prod_{\phi \in I}^{\rightarrow} \delta_{\phi}^{\iin}
=\prod_{\phi \in I}^{\rightarrow} \delta_{\phi}^{\oout}\,\]
in $H_\Lambda(\fM)$.
If we denote by $\phi_1^{\iin},\dots,
\phi_N^{\iin}$ the values of $\phi \in I$ such that $\delta_\phi^{\iin} \neq 1 \in H_\Lambda(\fM)$, ordered such that 
\[ \phi_1^{\iin}>\dots >\phi_N^{\iin}\,,\]
and by $\phi_1^{\oout},\dots,
\phi_{N'}^{\oout}$ the values of 
$\phi$ such that 
$\delta_\phi^{\sigma_\oout} \neq 1 \in H_\Lambda(\fM)$, ordered such that 
\[ \phi_1^{\oout}>\dots>\phi_{N'}^{\oout}\,,\]
we can rewrite the above identity as 
\[ \delta_{\phi_1^{\iin}}^{\iin}
\star \dots \star \delta_{\phi_N^{\iin}}^{\iin}
=\delta_{\phi_1^{\oout}}^{\oout}
\star \dots \star \delta_{\phi_N^{\oout}}^{\oout}\,.\]

Given the motivic property reviewed in
\cref{section_numerical_mixed_Hodge_theory}, we can apply the virtual Hodge polynomial $h$ to an element $[Z \rightarrow \fM]$ of the motivic Hall algebra $H(\fM)$ to obtain a 
$\Z(u,v)$-valued constructible function $h([Z \rightarrow \fM])$ on $\fM$.
Recall that we introduced the bilinear Euler form $(-,-) \colon \Gamma \otimes \Gamma \rightarrow \Z$ in Definition \ref{def_bilinear_form}.
Multiplying further by $(-(uv)^{\frac{1}{2}})^{(\gamma,\gamma)}$ the restrictions to each component $\fM_\gamma$, 
we get an equality of $A_\Lambda$-valued constructible functions on 
$\fM$:
\[ \sum_{\gamma \in \Lambda}
(-(uv)^{\frac{1}{2}})^{(\gamma,\gamma)}
\,h\left( (\delta_{\phi_1^{\iin}}^{\iin}
\star \dots \star \delta_{\phi_N^{\iin}}^{\iin})|_{\fM_\gamma}\right)z^{m_\gamma}\]
\[=\sum_{\gamma \in \Lambda}
(-(uv)^{\frac{1}{2}})^{(\gamma,\gamma)}\,
h\left( (\delta_{\phi_1^{\oout}}^{\oout}
\star \dots \star \delta_{\phi_{N'}^{\oout}}^{\oout})|_{\fM_\gamma}\right)z^{m_\gamma}\,.
\]
For every $\underline{\gamma}
=(\gamma_1,\dots,\gamma_N) \in \Lambda^N$, we denote by
$\fM_{\underline{\gamma}}^{\sigma_\iin}$
the stack of objects $E$ with 
$\sigma_\iin$-Harder-Narasimhan factors
$E_1,\dots,E_N$ 
of class $\gamma_1,\dots,\gamma_N$.
Similarly, for every $\underline{\gamma}
=(\gamma_1,\dots,\gamma_{N'}) \in \Lambda^{N'}$, we denote by
$\fM_{\underline{\gamma}}^{\sigma_\oout}$
the stack of objects $E$ with 
$\sigma_\oout$-Harder-Narasimhan factors
$E_1,\dots,$ $E_{N'}$ 
of class $\gamma_1,\dots,\gamma_{N'}$.

As we have  
\[ \sum_{\gamma \in \Lambda}
(-(uv)^{\frac{1}{2}})^{(\gamma,\gamma)}
\,h\left( (\delta_{\phi_1^{\iin}}^{\iin}
\star \dots \star \delta_{\phi_N^{\iin}}^{\iin})|_{\fM_\gamma}\right)z^{m_\gamma}\]
\[=\sum_{\underline{\gamma}
\in \Lambda^N}
(-(uv)^{\frac{1}{2}})^{(\gamma_1+\dots+
\gamma_N,\gamma_1+\dots+\gamma_N)}\,h\left( (\delta_{\gamma_1}^{\iin}
\star \dots \star \delta_{\gamma_N}^{\iin})|_{\fM_{\underline{\gamma}}^{\sigma_\iin}}\right)
z^{m_{\gamma_1}+\dots +m_{\gamma_N}}\,,\]
and 
\[ \sum_{\gamma \in \Lambda}
(-(uv)^{\frac{1}{2}})^{(\gamma,\gamma)}
\,h\left( (\delta_{\phi_1^{\oout}}^{\oout}
\star \dots \star \delta_{\phi_{N'}^{\oout}}^{\oout})|_{\fM_\gamma}\right)z^{m_\gamma}\]
\[=\sum_{\underline{\gamma}
\in \Lambda^{N'}}
(-(uv)^{\frac{1}{2}})^{(\gamma_1+\dots
+\gamma_{N'},\gamma_1+\dots+\gamma_{N'})}\,h\left( (\delta_{\gamma_1}^{\oout}
\star \dots \star \delta_{\gamma_N}^{\oout})|_{\fM_{\underline{\gamma}}^{\sigma_\oout}}\right)
z^{m_{\gamma_1}+\dots +m_{\gamma_{N'}}}\,,\]
Proposition \ref{prop_wcf} follows from the following Lemma \ref{lem_wcf}.

\begin{lem}\label{lem_wcf}
For every $\underline{\gamma}
=(\gamma_1,\dots,\gamma_N)
\in \Lambda^N$, we have an equality of $A_\Lambda$-valued constructible functions on 
$\fM$:
\[ (-(uv)^{\frac{1}{2}})^{
(\gamma_1+\dots+\gamma_N,
\gamma_1+\dots+\gamma_N)}\,h (\delta_{\gamma_1}^{\iin}
\star \dots \star \delta_{\gamma_N}^{\iin})
z^{m_{\gamma_1}+\dots +m_{\gamma_N}}\]
\[=\left((-(uv)^{\frac{1}{2}})^{-\dim \fM_{\gamma_1}^{\sigma^\iin}} h(\delta_{\gamma_1}^\iin)\right)
z^{m_{\gamma_1}}
\dots\left((-(uv)^{\frac{1}{2}})^{-\dim \fM_{\gamma_N}^{\sigma_\iin}} h(\delta_{\gamma_N}^\iin)\right)
z^{m_{\gamma_N}}\,.
\]
Similarly, for every $\underline{\gamma}
=(\gamma_1,\dots,\gamma_{N'})
\in \Lambda^{N'}$, we have an equality of $A_\Lambda$-valued constructible functions on 
$\fM$:
\[ (-(uv)^{\frac{1}{2}})^{(\gamma_1+\dots+\gamma_{N'},\gamma_1+\dots+\gamma_{N'})}\,
h(\delta_{\gamma_1}^{\oout}
\star \dots \star \delta_{\gamma_{N'}}^{\oout})
z^{m_{\gamma_1}+\dots +m_{\gamma_{N'}}}\]
\[=\left((-(uv)^{\frac{1}{2}})^{-\dim \fM_{\gamma_1}^{\sigma_\oout}} h(\delta_{\gamma_1}^\oout)\right)
z^{m_{\gamma_1}}
\dots \left((-(uv)^{\frac{1}{2}})^{-\dim\fM_{\gamma_{N'}}^\oout} h(\delta_{\gamma_{N'}}^\oout)\right)
z^{m_{\gamma_{N'}}}\,.
\]
\end{lem}

\begin{proof}
We prove the formula for $\sigma_\iin$. The proof of the formula for $\sigma_\oout$ is formally identical.

According to \cite{MR3936077}, for every 
$E$ $\sigma_\iin$-semistable object of class $\gamma$, we have \[\Ext^2(E,E)=0\,,\] and so, using Lemma \ref{lem_euler_form}, we have
\[ \dim \fM_\gamma^{\sigma_\iin}
=\dim \Ext^1(E,E)-\dim \Hom(E,E)=-\chi(E,E)
=-(\gamma,\gamma)\,.\]
On the other hand, by Definition
\ref{def_algebra_A_Lambda} of the product in $A_\Lambda$, we have 
\[ z^{m_{\gamma_1}}\dots z^{m_{\gamma_N}}
=(-(uv)^{\frac{1}{2}})^{\sum_{i<j} \langle \gamma_i,\gamma_j \rangle} z^{m_{\gamma_1}
+\dots+m_{\gamma_N}}\,,\]
which can be rewritten using Lemma
\ref{lem_skew_sym_euler_form}
as 
\[ z^{m_{\gamma_1}}\dots z^{m_{\gamma_N}}
=(-(uv)^{\frac{1}{2}})^{\sum_{i<j}
(\gamma_i,\gamma_j) - \sum_{i>j}
(\gamma_i,\gamma_j)}
z^{m_{\gamma_1}
+\dots+m_{\gamma_N}}\,.\]
Thus, given the identity
\[ (\sum_i \gamma_i,\sum_j \gamma_j)=
\sum_{i,j}(\gamma_i,\gamma_j)
=\sum_i(\gamma_i,\gamma_i)
+\sum_{i<j}(\gamma_i,\gamma_j)
+\sum_{i>j}(\gamma_i,\gamma_j)\]
\[=\sum_i(\gamma_i,\gamma_i)
+\left(\sum_{i<j}(\gamma_i,\gamma_j)
-\sum_{i>j}(\gamma_i,\gamma_j)\right)
+2\sum_{i>j}(\gamma_i,\gamma_j)\,,\]
Lemma \ref{lem_wcf} follows from the following equality of 
$\Z(u,v)$-constructible functions on $\fM$:
\[ h(\delta_{\gamma_1}^{\iin}
\star \dots \star \delta_{\gamma_N}^{\iin})
=(uv)^{-\sum_{N \geqslant i>j \geqslant 1}(\gamma_i,\gamma_j)}
h(\delta_{\gamma_1}^\iin)
\dots h(\delta_{\gamma_N}^\iin)\,.
\]
We prove by induction over $n$ that, for every $1 \leqslant n \leqslant N$, we have
\[ h(\delta_{\gamma_1}^{\iin}
\star \dots \star \delta_{\gamma_n}^{\iin})
=(uv)^{-\sum_{n \geqslant i>j \geqslant 1}(\gamma_i,\gamma_j)}
h(\delta_{\gamma_1}^\iin)
\dots h(\delta_{\gamma_n}^\iin)\,.
\] 
The case $n=1$ is trivial. Let us assume that the result is known for $n-1$ and that we wish to prove it for $n$.
By existence and uniqueness of the 
$\sigma_\iin$-Harder-Narasimhan filtration, an object $E$ in the support of $(\delta_{\gamma_1}^{\iin}
\star \dots \star \delta_{\gamma_{n-1}}^\iin) \star \delta_{\gamma_n}^\iin$ can be uniquely written as 
an extension 
\[ 0 \rightarrow F \rightarrow E \rightarrow E_n \rightarrow 0\,\]
with $E_n$ $\sigma_\iin$-semistable of class $\gamma_n$, and with $Y$ of Harder-Narasimhan factors $E_1, \dots, E_{n-1}$ of class 
$\gamma_1, \dots, \gamma_{n-1}$.
It follows from the arguments used in \cite{MR3936077} that, for every 
$G$ and $G'$ $\sigma_\iin$-semistable objects with 
$\phi_{\gamma(G)}^{\sigma_\iin} 
<\phi_{\gamma(G')}^{\sigma_\iin}$, we have \[\Ext^2(G,G')=0\,.\]
In particular, we have $\Ext^2(E_n, E_j)=0$, for every $1\leqslant j \leqslant n-1$, and as $F$ is obtained by successive extensions of the $E_j$, $1 \leqslant j \leqslant n-1$, we have also $\Ext^2(E_n,F)=0$.

From the explicit description of the product in the motivic Hall algebra 
(see \cite[Proposition 6.2]{MR2854172}
or \cite[Corollary 5.15]{MR2303235}), the fiber of  $(\delta_{\gamma_1}^{\iin}
\star \dots \star \delta_{\gamma_{n-1}})\star \delta_{\gamma_n}^\iin$ at the point 
$(F,E_n)$ of $\fM_{(\gamma_1,\dots,\gamma_{n-1})}^{\sigma_\iin} \times \fM_{\gamma_n}^{\sigma_\iin}$ is given by 
\[ [\Ext^1(E_n,F)/\Hom(E_n,F)]=[\A^1]^{\dim \Ext^1(E_n,F)-\dim \Hom(E_n,F)}\,.\] 
As $\Ext^2(E_n,F)=0$, we have, using 
Lemma \ref{lem_euler_form}, 
\[ \dim \Ext^1(E_n,F)-\dim \Hom (E_n,F)
=-\chi(E_n,F)=-(\gamma_n, \gamma_1+\dots +\gamma_{n-1})\,,\]
hence, using 
$h(\A^1)=uv$, the desired relation
\[h(\delta_{\gamma_1}^{\iin}
\star \dots \star 
\delta_{\gamma_{n-1}}^\iin \star \delta_{\gamma_n}^{\iin})
=(uv)^{-(\gamma_n,\gamma_1+\dots+\gamma_{n-1})} h(\delta_{\gamma_1}^{\iin}
\star \dots \star 
\delta_{\gamma_{n-1}}^\iin)
h(\delta_{\gamma_n}^{\iin})\,.\]
\end{proof}

\subsection{End of the proof of Theorem \ref{thm_consistency}}
\label{section_proof_consistency}

\begin{lem} \label{lem_in_out}
We have the following equalities in the 
$\Q(u^{\frac{1}{2}}, v^{\frac{1}{2}})$-algebra $A_\Lambda$:
\[ \prod_{\phi \in I}^{\rightarrow}  \exp
\left(-\sum_{\gamma \in \Lambda_\phi^\iin} 
\sum_{\ell \geqslant 1} \frac{1}{\ell}
\frac{Ih_\gamma^{\sigma_{\iin}}(u^{\frac{\ell}{2}}, v^{\frac{\ell}{2}})}{(uv)^{\frac{\ell}{2}}-(uv)^{-\frac{\ell}{2}}} 
z^{\ell m_\gamma} \right)
= \prod_{1 \leqslant a \leqslant K}^{\rightarrow}
\exp \left(
H_{\fd_{\gamma_a^\iin, j_a^\iin}} \right) \,,\]
and 
\[ \prod_{\phi \in I}^{\rightarrow}  \exp
\left(- \sum_{\gamma \in \Lambda_\phi^\oout} 
\sum_{\ell \geqslant 1} \frac{1}{\ell}
\frac{Ih_\gamma^{\sigma_{\oout}}(u^{\frac{\ell}{2}}, v^{\frac{\ell}{2}})}{(uv)^{\frac{\ell}{2}}-(uv)^{-\frac{\ell}{2}}} 
z^{\ell m_\gamma} \right)
= \prod_{1 \leqslant b \leqslant L}^{\rightarrow}
\exp \left(
H_{\fd_{\gamma_b^\oout,j_b^\oout}} \right) \,.\]
\end{lem}

\begin{proof}
As $\sigma_{\iin}$ and $\fd_{\gamma_a^{\iin},j_a^\iin}$
(resp.\ $\sigma_{\oout}$ and $\fd_{\gamma_a^{\oout},j_a^\oout}$)
are both 
in $U_\sigma^{\iin}$ (resp.\ $U_\sigma^{\oout}$), and so are not separated by walls for $\gamma_a^{\iin}$ (reps.\ $\gamma_a^{\oout}$), we have $M_{\gamma_a^{\iin}}^{\sigma_{\iin}}=M_{\gamma_a^{\iin}}^{\sigma'}$
(resp.\ $M_{\gamma_a^{\oout}}^{\sigma_{\oout}}=M_{\gamma_a^{\oout}}^{\sigma'}$), for every 
$\sigma' \in U_\sigma \cap \fd_a^{\iin}$ (resp.\ $\sigma' \in U_\sigma \cap \fd_a^{\oout}$).
By the construction of $I$, every $\gamma \in \Gamma$ with $\phi_\gamma^{\sigma_\iin} \in I$
(resp.\ $\phi_\gamma^{\sigma_\oout} \in I$)
and $\varphi_\sigma(m_{\gamma}) \leqslant k$ is of the form $\gamma_a^\iin$ (resp.\ $\gamma_b^{\oout}$) for some $1 \leqslant a \leqslant K$ (resp.\ $1 \leqslant b \leqslant L$).

It remains to show that the ordering according to decreasing value of 
$\phi_\gamma^{\sigma_\iin}$ (resp.\ $\phi_\gamma^{\sigma_\oout}$)
agrees with the ordering according to decreasing value of 
$1 \leqslant a \leqslant K$ (resp.\ $1 \leqslant b \leqslant L$).

We consider a small parametrized path $\fp_\iin$ (resp.\ $\fp_\oout$)
in $U_\sigma^\iin$ (resp.\ $U_\sigma^\oout$) starting at $\sigma_\iin$
(resp.\ $\sigma_\oout$) and intersecting successively 
the rays 
$\fd_{\gamma_{a}^\iin,j_a^\iin}$ (resp.\ $\fd_{\gamma_{b}^\oout,j_b^\oout}$)
in the order of increasing $1 \leqslant a \leqslant K$ (resp.\ 
$1 \leqslant b \leqslant L$). 
According to Lemma \ref{lem_ordering}, we have
$\Rea Z_{\gamma_a^\in}^{\sigma_\iin}<0$ (resp.\ $\Rea Z_{\gamma_a^\oout}^{\sigma_\oout}<0$) for every $1 \leqslant a \leqslant K$ (resp.\ $1 \leqslant b \leqslant L$).
Recall that by definition of $U_\sigma$, there is no potential wall of the form 
$W_{\gamma_{\fd_a^{\iin}},\gamma_{\fd_{a'}^{\iin}}}$, 
$1 \leqslant a, a' \leqslant K$,
(resp.\ $W_{\gamma_{\fd_b^{\oout}},\gamma_{\fd_{b'}^{\oout}}}$, 
$1 \leqslant b, b' \leqslant L$) intersecting $U_\sigma^\iin$
(resp.\ $U_\sigma^\oout$).
It follows that the relative ordering of the phases
$\phi_{\gamma_a^\iin}^{\sigma_\iin}$ coincides with the relative 
ordering according to which $\fp_\iin$ (resp.\ $\fp_\oout$) intersects 
the rays $\fd_{\gamma_{a}^\iin,j_a^\iin}$ (resp.\ $\fd_{\gamma_{b}^\oout,j_b^\oout}$).
\end{proof}

We can now end the proof of Theorem
\ref{thm_consistency}.
According to \cref{section_local_wall}, we need to show the equality 
\[\Phi_{\fd_K^{\iin}}  \dots  \Phi_{\fd_1^{\iin}}
= \Phi_{\fd_L^{\oout}}  \dots \Phi_{\fd_1^{\oout}}\]
of order $k$ path-ordered products in the group 
$G_{\varphi_\sigma}^k$.
By Definition \ref{def_autom}, the group element $\Phi_\fd$ attached to a ray 
$\fd$ is given by 
\[\Phi_{\fd}\coloneqq \exp(H_\fd)   \,.\]
As the commutator in the associative algebra $A_\Lambda$ coincides with the Lie bracket in the Lie algebra $\fg_{u,v}$
(compare Definitions \ref{def_algebra_A_Lambda} and \ref{def_lie_algebra}), the above equality in the group $G_{\varphi_\sigma}^k$ is equivalent to the following equality in $A_\Lambda$:
\[  \prod_{1 \leqslant a \leqslant K}^{\rightarrow}
\exp \left(
H_{\fd_{\gamma_a^\iin, j_a^\iin}} \right)
= 
\prod_{1 \leqslant b \leqslant L}^{\rightarrow}
\exp \left(
H_{\fd_{\gamma_b^\oout,j_b^\oout}} \right)\,.\]
According to Lemma \ref{lem_in_out}, this is equivalent to 
\[ \prod_{\phi \in I}^{\rightarrow}  \exp
\left(-\sum_{\gamma \in \Lambda_\phi^\iin} 
\sum_{\ell \geqslant 1} \frac{1}{\ell}
\frac{Ih_\gamma^{\sigma_{\iin}}(u^{\frac{\ell}{2}}, v^{\frac{\ell}{2}})}{(uv)^{\frac{\ell}{2}}-(uv)^{-\frac{\ell}{2}}} 
z^{\ell m_\gamma} \right)
=
 \prod_{\phi \in I}^{\rightarrow}  \exp
\left(-\sum_{\gamma \in \Lambda_\phi^\oout} 
\sum_{\ell \geqslant 1} \frac{1}{\ell}
\frac{Ih_\gamma^{\sigma_{\oout}}(u^{\frac{\ell}{2}}, v^{\frac{\ell}{2}})}{(uv)^{\frac{\ell}{2}}-(uv)^{-\frac{\ell}{2}}} 
z^{\ell m_\gamma} \right)\,,\]
which by Proposition \ref{prop_dt_ic}, is equivalent to 
\[
\prod_{\phi \in I}^{\rightarrow} 
\left(
\sum_{\gamma \in \Lambda_\phi^\iin}
\tilde{h}(\fM_\gamma^{\sigma_{\iin}})(u^{\frac{1}{2}},
v^{\frac{1}{2}})
z^{m_\gamma}
\right)
= 
\prod_{\phi \in I}^{\rightarrow} 
\left(
\sum_{\gamma \in \Lambda_\phi^{\oout}}
\tilde{h}(\fM_\gamma^{\sigma_{\oout}})(u^{\frac{1}{2}},
v^{\frac{1}{2}})
z^{m_\gamma}
\right)\,.\]
But this last equality is exactly Proposition \ref{prop_wcf}, and this ends the proof of Theorem \ref{thm_consistency}.

\section{Initial data for $\fD^{\PP^2}_{u,v}$}
\label{section_initial_data_proof}

In this section, we prove that the scattering diagrams $\fD^{\PP^2}_{u,v}$ and 
$S(\fD_{u,v}^\iin)$ have the same initial data. 
More precisely, we prove Theorem \ref{thm_coincide_initial}, according to which the scattering diagrams 
$\fD^{\PP^2}_{u,v}$ and 
$S(\fD_{u,v}^\iin)$ coincide in restriction to a neighbourhood 
$\bar{U}^\iin$ of the boundary of $U$.
We define $\bar{U}^\iin$ in  
\cref{section_initial_region} and we state 
Theorem \ref{thm_coincide_initial} 
in \cref{section_statement_initial}.
We introduce a description of 
$\D^b(\PP^2)$ in terms of quiver representations in
\cref{section_quiver_description}, which is then used in \cref{section_proof_initial} to prove 
Theorem \ref{thm_coincide_initial}.

\subsection{The initial region $\bar{U}^\iin$}
\label{section_initial_region}

Recall from \cref{section_initial} that for every 
$n \in \Z$, we introduced naked rays 
$|\fd_n^+|=s_n-[0,\frac{1}{2}] m_n^+$ and 
$|\fd_n^-|=s_n-[0,\frac{1}{2}] m_n^-$,
which were then used to define the initial scattering diagrams
$\fD_{u,v}^\iin$.

\begin{lem}
For every $n \in \Z$, the line segments 
$|\fd_n^+|$ and $|\fd_n^-|$ are contained 
in the line of equation
\[ y+nx-\frac{n^2}{2}=0\,,\]
that is,
\[ \Rea Z_{\gamma(\cO(n))}^{(x,y)}=0\,.\]
\end{lem}

\begin{proof}
The fact that $|\fd_n^+|$ and $|\fd_n^-|$ are contained 
in the line of equation
\[ y+nx-\frac{n^2}{2}=0\]
is immediate given that $|\fd_n^+|=s_n-[0,\frac{1}{2}] m_n^+$ and 
$|\fd_n^-|=s_n-[0,\frac{1}{2}] m_n^-$, with
$s_n = (n,-\frac{n^2}{2})$, 
$m_n^{-} =(1,-n)$,
$m_n^{+}=(-1,n)$.

It remains to show that $y+nx-\frac{n^2}{2}=0$ is equivalent to $\Rea Z_{\gamma(\cO(n))}^{(x,y)}=0$. Recall
from Definition
\ref{def_central_charge_formula} that, writing $\gamma=(r,d,\chi)$, we have 
\[ \Rea Z_\gamma^{(x,y)}=ry+dx+r+\frac{3}{2}d-\chi\,.\]
On the other hand, we have 
\[ \gamma(\cO(n))
=(1,n,\frac{n^2}{2}+\frac{3n}{2}+1)\,.\]
 
\end{proof}

\begin{defn}\label{def_initial_region}
For every $n \in \Z$, we define 
\[ \bar{U}_n^\iin
\coloneq \{ \sigma=(x,y) \in \bar{U} |
\Rea Z_{\gamma(\cO(n-1))}^\sigma \leqslant 0 \,,
\Rea Z_{\gamma(\cO(n+1))}^\sigma \leqslant 0\}\]
\[ =\{(x,y) \in \bar{U}| y\leqslant -(n-1)x+\frac{1}{2}(n-1)^2\,, y\leqslant -(n+1)x+\frac{1}{2}(n+1)^2\}\,.\]
\end{defn}

For every $n \in \Z$, the boundary of 
$\bar{U}_n^\iin$ is the union of:
\begin{itemize}
\item[(i)] The line segment $s_{n-1}-[0,1]m_{n-1}^+$, contained in the line of equation 
\[y+(n-1)x-\frac{1}{2}(n-1)^2=0\,,\] 
that is, $\Rea Z_{\gamma(\cO(n-1))}^\sigma=0$.
\item[(ii)] The line segment 
$s_{n+1}-[0,1]m_{n+1}^-$, contained in the line of equation \[y+(n+1)x-\frac{1}{2}(n+1)^2=0\,,\]
that is $\Rea Z_{\gamma(\cO(n+1))}^\sigma=0$.
\item[(iii)] The arc of the parabola 
$y=-\frac{x^2}{2}$, boundary of 
$\bar{U}$, delimited by the points $s_{n-1}$ and $s_{n+1}$.
\end{itemize}

Remark that the line segments 
$s_{n-1}-[0,1]m_{n-1}^+$ and 
$s_{n+1}-[0,1]m_{n+1}^-$
intersect at the point
\[ (n-1,-\frac{1}{2}(n-1)^2)
-(-1,n-1)=
(n+1,-\frac{1}{2}(n+1)^2)-(1,-(n+1))
=(n, -\frac{1}{2}(n^2-1))\,,\]
which is the
intersection point of the lines $\Rea Z_{\gamma(\cO(n-1))}^\sigma=0$ and 
$\Rea Z_{\gamma(\cO(n+1))}^\sigma=0$.

The line of equation 
$y+nx-\frac{n^2}{2}=0$, that is, 
$\Rea Z_{\gamma(\cO(n))}^\sigma=0$, divides $\bar{U}_n^\iin$ into three regions.
We have 
\[ \bar{U}_n^\iin=\bar{U}_{n,T}^\iin
\cup \bar{U}_{n,L}^\iin \cup 
\bar{U}_{n,R}^\iin\,,\]
where:
\begin{itemize}
\item[(i)] $\bar{U}_{n,T}^\iin$ is the triangle delimited by the three lines 
$\Rea Z_{\gamma(\cO(n-1))}^\sigma=0$,
$\Rea Z_{\gamma(\cO(n))}^\sigma=0$ and 
$\Rea Z_{\gamma(\cO(n+1)}^\sigma=0$.
Remark that we have 
$\{ \Rea Z_{\gamma(\cO(n))}^\sigma=0\}
\cap \bar{U}_n^\iin=|\fd_n^+|\cup 
|\fd_n^-|$.
\item[(ii)] $\bar{U}_{n,L}^\iin$ is delimited 
by the lines $\Rea Z_{\gamma(\cO(n-1))}^\sigma=0$, $\Rea Z_{\gamma(\cO(n))}^\sigma=0$, and the arc of the parabola 
$y=-\frac{x^2}{2}$ delimited by the points 
$s_{n-1}$ and $s_n$. Remark that we have 
$\{ \Rea Z_{\gamma(\cO(n-1))}^\sigma 
=0\} \cap \bar{U}_{n,L}^\iin=|\fd_{n-1}^+|$ and 
$\{ \Rea Z^\sigma_{\gamma(\cO(n))}=0\} 
\cap \bar{U}^\iin_{n,L}=|\fd_n^-|$.
\item[(iii)] $\bar{U}_{n,R}^\iin$ is delimited by the lines 
$\Rea Z_{\gamma(\cO(n+1))}^\sigma=0$, 
$\Rea Z_{\gamma(\cO(n)}^\sigma=0$,
and the arc of the parabola 
$y=-\frac{x^2}{2}$ delimited by the points 
$s_n$ and $s_{n+1}$. Remark that we have 
$\{ \Rea Z_{\gamma(\cO(n+1))}^\sigma 
=0\} \cap \bar{U}_{n,R}^\iin=|\fd_{n+1}^-|$ and 
$\{ \Rea Z^\sigma_{\gamma(\cO(n))}=0\} 
\cap \bar{U}^\iin_{n,R}=|\fd_n^+|$.
\end{itemize}

Figure \ref{figure_notation_2}
gives a schematic summary of some of the notation introduced above.

\begin{figure}[h!] 
\centering
\setlength{\unitlength}{1.2cm}
\begin{picture}(10,5)
\thicklines
\put(9.2,2.5){$\Rea Z_{\gamma
(\cO(n))}^\sigma=0$}
\put(9.25,0.9){$y=-\frac{x^2}{2}$}
\put(2,1.65){\circle*{0.1}}
\put(2,1.65){\line(2,1){5}}
\put(7.2,4){$\Rea Z_{\gamma(\cO(n-1))}^\sigma=0$}
\put(2,1.65){\line(-2,-1){1}}
\put(2,1.45){$s_{n-1}$}
\put(8,1.65){\circle*{0.1}}
\put(8,1.65){\line(-2,1){5}}
\put(0.45,4){$\Rea Z_{\gamma(\cO(n+1))}^\sigma=0$}
\put(8,1.65){\line(2,-1){1}}
\put(7.6,1.45){$s_{n+1}$}
\put(5,2.5){\circle*{0.1}}
\put(5,2.25){$s_n$}
\put(5,2.5){\line(1,0){4}}
\put(5,2.5){\line(-1,0){4}}
\qbezier(1,1)(5,4)(9,1)
\put(4.7,2.7){$\bar{U}_{n,T}^\iin$}
\put(3.4,1.9){$\bar{U}_{n,L}^\iin$}
\put(6.3,1.9){$\bar{U}_{n,R}^\iin$}
\end{picture}
\caption{}
\label{figure_notation_2}
\end{figure}

\begin{defn}
We denote 
\[ \bar{U}^\iin \coloneq \bigcup_{n \in \Z}
\bar{U}_n^\iin\,.\]
\end{defn}

\begin{lem}
$\bar{U}^\iin$ is a neighbourhood in 
$\bar{U}$ of 
the boundary $\partial \bar{U}=\{(x,y)|y=-\frac{x^2}{2}\}$ of $\bar{U}$.
\end{lem}

\begin{proof}
For every $n \in \Z$, $\bar{U}_n^\iin$
is a neighbourhood in 
$\bar{U}$ of the arc of the parabola 
$y=-\frac{x^2}{2}$ defined by $n-1<x<n+1$.
The result follows from the fact that every point $(x,y)$ of the parabola 
$y=-\frac{x^2}{2}$ satisfies $n-1<x<n+1$ for some $n \in \Z$. Equivalently, one can remark that 
\[ \{(x,y)\in\bar{U}|\,y<-\frac{x^2}{2}+\frac{1}{8}\} \subset \bar{U}^\iin\,.\]
\end{proof}

\subsection{Statement of Theorem \ref{thm_coincide_initial}}
\label{section_statement_initial}

\begin{thm} \label{thm_coincide_initial}
The scattering diagrams $\fD_{u,v}^{\PP^2}$ and $S(\fD_{u,v}^\iin)$
coincide in restriction to $\bar{U}^\iin$.
\end{thm}

The proof of Theorem
\ref{thm_coincide_initial} takes the remaining part of 
\cref{section_initial_data_proof}.

Recall that we defined an action
$\psi(1)$ on $U$ in
\cref{subsection_psi(1)}.
We first remark that, for every $n \in \Z$, we have $\psi(1)(\bar{U}_n^\iin)=\bar{U}_{n+1}^\iin$. We also defined in
\cref{subsection_psi(1)} an action 
$\psi(1)$ on scattering diagrams on $U$.
According to Proposition \ref{prop_psi(1)}, 
$\psi(1)$ preserves $S(\fD_{u,v}^\iin)$
and according to Proposition
\ref{prop_psi(1)_P2}, $\psi(1)$ also 
preserves $\fD_{u,v}^{\PP^2}$.
Therefore, in order to prove Theorem
\ref{thm_coincide_initial}, it is enough to show that the scattering diagrams $\fD_{u,v}^{\PP^2}$ and $S(\fD_{u,v}^\iin)$
coincide in restriction to $\bar{U}_0^\iin$.

We describe $S(\fD_{u,v}^\iin)$ in restriction to $\bar{U}_0^\iin$.
It follows from Lemma \ref{lem_scattering_forward} that $S(\fD_{u,v}^\iin)$
restricted to $\bar{U}_0^\iin$ consists in the following rays:
\begin{itemize}
\item[(i)] For every integer $\ell \geqslant 1$, 
$\fd_{0,\ell}^+=(|\fd_0^+|, H_{0,\ell}^+)$ and 
$\fd_{0,\ell}^-=(|\fd_0^-|, H_{0,\ell}^-)$
(see \cref{section_initial}-\cref{section_scattering_final}).
\item[(ii)] For every integer 
$\ell \geqslant 1$, 
$\fd_{-1,\ell}^+=(|\fd_{-1}^+|, H_{-1,\ell}^+)$ and 
$\fd_{1,\ell}^-=(|\fd_1^-|, H_{1,\ell}^-)$
(see \cref{section_initial}-
\cref{section_scattering_final}).
\item[(iii)] For every integer $\ell \geqslant 1$, 
\[\fd_{-1,\ell}^{+,(1)} \coloneq 
\left(s_{-1}-[1/2,1]m_{-1}^+, 
H_{-1,\ell}^+ \right)\,,\] and 
\[\fd_{1,\ell}^{-,(1)} \coloneq 
\left(s_{1}-[1/2,1]m_{1}^-, 
H_{1,\ell}^-\right)\,.\]
\end{itemize}

In order to prove 
Theorem
\ref{thm_coincide_initial},
it remains to show that 
$\fD_{u,v}^{\PP^2}$ restricted to $\bar{U}_0^{\iin}$ has an identical description.

\subsection{Quiver description}
\label{section_quiver_description}

In this section we review a description of $\D^b(\PP^2)$ in terms of quiver representations, which will be used in the
proof of Theorem \ref{thm_coincide_initial}
in \cref{section_proof_initial}. 
A similar discussion can be found in 
\cite[\S 4.3]{MR2681544}. In order to compare with the notation of \cite{MR2681544}, we recall that the exterior product 
$\Omega_{\PP^2} \otimes \Omega_{\PP^2} 
\rightarrow \wedge^2 \Omega_{\PP^2}
=K_{\PP^2}=\cO(-3)$ induces an isomorphism
$T_{\PP^2} \simeq \Omega_{\PP^2}(3)$.

We denote by $T_{\PP^2}$ the tangent bundle of 
$\PP^2$ and we consider the strong exceptional collection
\cite{MR885779} 
\[ \cO(1)\,,T_{\PP^2}\,,\cO(2) \,,\]
of objects in $\D^b(\PP^2)$.
We denote 
\[ \mathbf{T} \coloneq \cO(1)\oplus T_{\PP^2}\oplus\cO(2)\,.\]
It follows from the Beilinson spectral sequence \cite{MR509388} that, introducing the algebra
\[ A_0 \coloneq \Hom(\mathbf{T},\mathbf{T})^{\mathrm{op}}\,,\]
and $\cA_0$ the abelian category of 
finitely generated left $A_0$-modules, the 
functor 
\[ \D^b(\PP^2) \rightarrow \D^b(\cA_0)\,\]
\[ E \mapsto \RHom(\mathbf{T},E) \]
is an equivalence of triangulated categories. Using this equivalence, we view 
$\cA_0$ as a subcategory of $\D^b(\PP^2)$.

As 
\[ \Hom(\cO(1),T_{\PP^2})=H^0(T_{\PP^2}(-1)) \simeq \C^3 \,,\]
\[ \Hom(T_{\PP^2},\cO(2))
=H^0(\Omega_{\PP^2}(2)) \simeq \C^3 \,,\]
\[ \Hom(\cO(1),\cO(2)) =H^0(\cO(1))
\simeq \C^3\,,\]
the algebra $A_0$ is the path-algebra of the quiver $Q_0$, consisting of three vertices $v_{-1}$, $v_0$, $v_1$, three arrows from $v_{-1}$ to $v_0$, three arrows from 
$v_0$ to $v_1$, and six linearly independent relations coming from the kernel of the composition map
\[ \Hom(\cO(1),T_{\PP^2}) \otimes \Hom(T_{\PP^2},\cO(2)) \rightarrow \Hom(\cO(1),\cO(2))\,.\]
Denoting by $\delta_0$, 
$\delta_1$, $\delta_2$ the three arrows from $v_{-1}$ to $v_0$, and 
by $\gamma_0$,
$\gamma_1$, $\gamma_2$ the three arrows from $v_0$ to $v_1$, one can write the relations as 
$\gamma_i \delta_j+\gamma_j \delta_i=0$, for every $i,j=0,1,2$, see 
\cite[\S 4.3]{MR2681544}.

\[Q_0=\]
\begin{center}
\begin{tikzpicture}[>=angle 90]
\matrix(a)[matrix of math nodes,
row sep=3em, column sep=5em,
text height=1.5ex, text depth=0.25ex]
{v_{-1}&v_0&v_1\\};
\path[->](a-1-1) edge [bend left] (a-1-2);
\path[->](a-1-1) edge  (a-1-2);
\path[->](a-1-1) edge [bend right] (a-1-2);
\path[->](a-1-2) edge [bend left]  (a-1-3);
\path[->](a-1-2) edge (a-1-3);
\path[->](a-1-2) edge [bend right] (a-1-3);
\end{tikzpicture}
\end{center}
The abelian category $\cA_0$ is the category of finite-dimensional 
representations of $Q_0$. 
In particular, the dimension of a quiver representation, that is, the triple of dimensions of the vector spaces attached to the three vertices, defines an isomorphism
\[ \dim \colon K_0(\cA_0) \simeq \Z^3\,,\]
such that $\dim([V]) \in \NN^3$ if 
$V$ is an object of $\cA_0$.

The simple objects $S_{-1}$, $S_0$, $S_1$
of $\cA_0$, one-dimensional representations of $Q_0$ supported at the vertices 
$v_{-1}$, $v_0$, $v_1$ of $Q_0$, of dimensions $(1,0,0)$, 
$(0,1,0)$, $(0,0,1)$, 
correspond respectively to the objects 
\[ \cO(-1)[2]\,,\cO[1] \,, \cO(1) \,,\]
of $\D^b(\PP^2)$. The indecomposable 
projective objects 
$P_{-1}$, $P_0$, $P_1$ of 
$\cA_0$, characterized by 
$\Ext^{\bullet}(P_j,S_k)=\delta_{jk} \C$,
correspond respectively to the objects
\[ \cO(2)\,, T_{\PP^2} \,,\cO(1) \,,\]
of $\D^b(\PP^2)$.
They correspond to quiver representations of dimension 
$(1,3,3)$, $(0,1,3)$, $(0,0,1)$ respectively.
In $\cA_0$, we have natural projective resolutions of the simple objects:
\[ 0 \rightarrow P_1 \rightarrow S_1 \rightarrow 0\,,\]
\[ 0 \rightarrow P_1^{ \oplus 3} \rightarrow P_0 \rightarrow S_0 \rightarrow 0\,,\]
\[ 0 \rightarrow P_1^{\oplus 6} \rightarrow 
P_0^{\oplus 3} \rightarrow P_{-1}
\rightarrow S_{-1} \rightarrow 0\,.\]

We now focus on the simple objects
$\cO(-1)[2]$, $\cO[1]$, $\cO(1)$,
of $\cA_0$.
We can describe $\cA_0$ as the extension-closed
subcategory of $\D^b(\PP^2)$ generated by 
$ \cO(-1)[2]\,,\cO[1] \,, \cO(1)$.
The only nonzero
$\Ext$-groups between these objects are 
\[ \Ext^1(\cO(-1)[2],\cO[1])
=H^0(\cO(1)) \simeq \C^3 \,,\]
\[ \Ext^1(\cO[1],\cO(1))=H^0(\cO(1))
\simeq \C^3\,,\]
\[ \Ext^2(\cO(-1)[2],\cO(1))
=H^0(\cO(2)) \simeq \C^6\,.\]
In particular, the ordered collection 
\[ \cO(-1)[2]\,,\cO[1] \,, \cO(1) \,,\]
is Ext-exceptional in the sense of 
\cite[Definition 3.10]{MR2335991}.
Therefore, by \cite[Lemma 3.16]{MR2335991}, 
if $\sigma=(Z,\cA)$ is a stability condition on $\D^b(\PP^2)$ such that 
$\cO(-1)[2]$, $\cO[1]$, $\cO(1)$ belong to 
$\cA$, then $\cA=\cA_0$. 

\begin{defn}
We denote by $K_3$ the quiver with two vertices $V_1$, $V_2$ and three arrows from 
$V_1$ to $V_2$.
\end{defn}

\[K_3=\]
\begin{center}
\begin{tikzpicture}[>=angle 90]
\matrix(a)[matrix of math nodes,
row sep=3em, column sep=5em,
text height=1.5ex, text depth=0.25ex]
{V_1&V_2\\};
\path[->](a-1-1) edge [bend left] (a-1-2);
\path[->](a-1-1) edge  (a-1-2);
\path[->](a-1-1) edge [bend right] (a-1-2);
\end{tikzpicture}
\end{center}

One way to obtain the quiver $K_3$ is to restrict the set of arrows of $Q_0$ to those starting and ending at $v_{-1}$ and $v_0$, or to restrict the set of arrows of $Q_0$ to those starting and ending at $v_0$ and $v_1$. In other words, $K_3$ is in two possible ways a subquiver of $Q_0$.

\begin{lem}\label{lem_quiver_kronecker}
Let $\sigma$ be a Bridgeland stability condition on the derived category of representations of $K_3$, of heart the category of representations of $K_3$. If $V$ is a $\sigma$-semistable representation of $K_3$ of dimension 
$(n_1,n_2) \in \NN^2$ with $n_1 \geqslant 1$, then we have 
\[n_2 \leqslant 3n_1\,.\]
\end{lem}

\begin{proof}
If the moduli space of $\sigma$-semistable representations of $K_3$ of dimension $(n_1,n_2)$ is nonempty, then 
(e.g.\ see \cite{MR1315461}) this moduli space has dimension 
$3n_1n_2-n_1^2-n_2^2+1$. In particular, we have $3n_1 n_2-n_1^2-n_2^2+1\geqslant 0$,
so for $n_1 \geqslant 1$, we have 
$n_2(3n_1-n_2) \geqslant n_1^2-1 \geqslant 0$, so as $n_2 \geqslant 0$, we have 
$3n_1 \geqslant n_2$.
\end{proof}

We end this section by a review of a natural operation on stability conditions that will be useful in
\cref{section_proof_initial}.
It is a particular case of a more general action of the group  $\widetilde{GL}^+_2(\R)$, the universal covering space of 
$GL^+_2(\R)$, on spaces of stability conditions, see 
\cite[Lemma 8.2]{MR2373143}.
For every $0 <\phi<1$ and $\sigma=(Z,\cA)$ a stability condition on $\D^b(\PP^2)$, we can construct a new stability condition
$\sigma[\phi]\coloneq (Z[\phi],\cA[Z,\phi])$
on $\D^b(\PP^2)$. For every $\gamma \in \Gamma$, we define
$Z[\phi]_\gamma \coloneq e^{-i\pi \phi} Z_\gamma$.
Let $\cA_\phi$
be the subcategory of 
$\cA$ generated (by extensions)
by the $\sigma$-semistable objects $E$ with  $\frac{1}{\pi}\Arg Z(E)
>\phi$, 
and let 
$\cF_\phi$ is the subcategory of $\cA$
generated (by extensions) by the 
$\sigma$-semistable objects $E$ with $\frac{1}{\pi}
\Arg Z(E) \leqslant \phi$.
Then, denoting $\cH^i_\cA$ the cohomology functors with respect to the bounded 
$t$-structure of heart $\cA$,
$\cA[Z,\phi]$ 
is the subcategory of $\D^b(\PP^2)$ of objects $E$ such that 
$\cH^i_\cA=0$ for $i \neq -1,0$, 
$\cH^{-1}_\cA(E)$ is an object of 
$\cF_\phi$, and $\cH^0_\cA(E)$ is an object of $\cA_\phi$.

For every $0<\phi<1$ and $\sigma=(Z,\cA)$
a stability condition on $\D^b(\PP^2)$, 
moduli spaces of $\sigma$-semistable (resp.\
$\sigma$-stable) objects coincide with moduli spaces of $\sigma[\phi]$-semistable
(resp.\ $\sigma[\phi]$-stable) objects.
More precisely:
\begin{itemize}
\item[(i)] $\sigma$-semistable (resp. $\sigma$-stable) objects with
$\frac{1}{\pi} \Arg=\psi > \phi$ are identified with $\sigma[\phi]$-semistable
(resp.\ $\sigma[\phi]$-stable) objects with $\frac{1}{\pi} \Arg Z[\phi] = \psi - \phi$.
\item[(ii)] $\sigma$-semistable (resp.\ 
$\sigma$-stable) objects with 
$\frac{1}{\pi}\Arg Z = \psi \leqslant \phi$ are identified with $\sigma[\phi]$-semistable  (resp.\ $\sigma[\phi]$-stable) objects with
$\frac{1}{\pi} \Arg Z[\phi] = 1 + (\psi- \phi)$ via $E \mapsto E[1]$.
\end{itemize}

\subsection{Proof of Theorem \ref{thm_coincide_initial}}
\label{section_proof_initial}

In this section we end the proof of Theorem 
\ref{thm_coincide_initial}, that is, that the scattering diagrams $\fD^{\PP^2}_{u,v}$ and 
$S(\fD_{u,v}^\iin)$ coincide in restriction to $\bar{U}^\iin$. 
We will use the notions and notation introduced in the previous \cref{section_quiver_description}.

For every $\sigma=(x,y) \in U$ such that 
$-1<x<1$, it follows from Definition
\ref{defn_heart} that $\cO(-1)[1]$ and 
$\cO(1)$ belong to the heart
$\cA^\sigma$ of the stability condition defined by $\sigma$.

As $\gamma(\cO(-1))=(1,-1,0)$, 
$\gamma(\cO)=(1,0,1)$, and 
$\gamma(\cO(1))=(1,1,3)$, we have by Definition \ref{def_central_charge_formula}, for every 
$(x,y) \in U$:
\[ Z_{\gamma(\cO(-1))}^{(x,y)}
=y-x-\frac{1}{2}-i(x+1)\sqrt{x^2+2y}\,,\]
\[Z_{\gamma(\cO)}^{(x,y)}=y-ix\sqrt{x^2+2y}\,,\]
\[Z_{\gamma(\cO(1))}^{(x,y)}
=y+x-\frac{1}{2}-i(x-1)\sqrt{x^2+2y}\,.\]

\begin{lem} \label{lem_coincide_int_T}
The scattering diagram $\fD^{\PP^2}_{u,v}$ is empty in restriction to the interior of 
$\bar{U}^\iin_{0,T}$.
In other words, the scattering diagrams 
$\fD^{\PP^2}_{u,v}$ and 
$S(\fD^\iin_{u,v})$ coincide in restriction to the interior of 
$\bar{U}^\iin_{0,T}$.
\end{lem}

\begin{proof}
Let $\sigma$ be a point in the interior of $\bar{U}^\iin_{0,T}$.
Then $\cO(-1)[1]$ and $\cO(1)$ belong to 
$\cA^\sigma$, with 
\[\Rea Z^\sigma_{\gamma(\cO(-1)[1])}=-\Rea Z^\sigma_{\gamma(\cO(-1))}>0\,,\]
and
\[\Rea Z^\sigma_{\gamma(\cO(1))}<0\,.\]
We also have 
\[ \Rea Z^\sigma_{\gamma(\cO)}>0\,,\]
and, depending if $x \geqslant 0$ or 
$x<0$, $\cO$ belongs to $\cA^\sigma$ or 
$\cO[1]$ belongs to 
$\cA^\sigma$.
In any case, the objects $\cO(-1)[2]$, $\cO[1]$, and $\cO(1)$ belong to $\cA^\sigma[Z^\sigma,1/2]$, and so 
$\cA^\sigma[Z^\sigma, 1/2]=\cA_0$.
On the other hand, we have 
\[ \Rea Z_{\gamma(\cO(-1)[2])}^\sigma 
<0\,, \Rea Z_{\gamma(\cO[1])}^\sigma<0 \,,
\Rea Z_{\gamma(\cO(1))}^\sigma <0\,.\]

As $\cA_0$ is a category of quiver representations, with simple objects 
$\cO(-1)[2]$, $\cO[1]$ and $\cO(1)$, we have that, for every $E$ nonzero object of 
$\cA^\sigma[Z^\sigma,1/2]$, the central charge
$Z^\sigma[1/2](E)$ is contained in the cone of linear combinations with 
nonnegative coefficients of 
$Z^\sigma[1/2]_{\gamma(\cO(-1))[2])}$,
$Z^\sigma[1/2]_{\gamma(\cO[1])}
$, 
$Z^\sigma[1/2]_{\gamma(\cO(1))}$.
As $Z^\sigma[1/2]_\gamma=-i Z^\sigma_\gamma$ for every $\gamma \in \Gamma$, 
we have  
\[ \Ima Z^\sigma[1/2]_{\gamma(\cO(-1))[2])}>0\,,
\Ima Z^\sigma[1/2]_{\gamma(\cO[1])}>0
\,,
\Ima Z^\sigma[1/2]_{\gamma(\cO(1))}>0\,.\]
and so
we conclude that 
\[\Ima Z^\sigma[1/2](E)>0\,.\]
Thus, if $E$ is a $\sigma$-semistable object of $\cA^\sigma$, then 
\begin{itemize}
\item[(i)] either, $\frac{1}{\pi} \Arg Z^\sigma(E)>\frac{1}{2}$, and so 
$\Rea Z^\sigma(E)<0$,
\item[(ii)] or, $\frac{1}{\pi} \Arg Z^{\sigma}(E) 
\leqslant \frac{1}{2}$, and so $E[1]$
is a $\sigma[1/2]$-semistable object 
of $\cA^\sigma[Z^\sigma, \frac{1}{2}]$, so
$\Ima Z^\sigma[1/2](E[1])>0$, and so
$\Rea Z^\sigma(E)>0$.
\end{itemize}
In any case, we have $\Rea Z^\sigma(E) \neq 0$. It follows that $\fD^{\PP^2}_{u,v}$, restricted to the interior of 
$\bar{U}_{0,T}^{\iin}$, is empty, as 
$S(\fD^\iin_{u,v})$.

Figure: an example of configuration of central charges for $\sigma=(x,y)$ in the interior of $\bar{U}^\iin_{0,T}$ with $x <0$. If $E$ is a $\sigma$-semistable object of $\cA^\sigma$, then $Z^\sigma(E)$ belongs to the dotted region.
\begin{center}
\setlength{\unitlength}{1.2cm}
\begin{picture}(10,6)
\put(5,3){\line(1,0){4}}
\put(5,3){\line(-1,0){4}}
\put(5,3){\line(0,1){3}}
\put(5,3){\line(0,-1){3}}
\put(5,3){\vector(1,1){2}}
\put(7.1,5){$Z^\sigma_{\gamma(\cO(-1)[1])}$}
\put(5,3){\vector(-1,-1){2}}
\put(2.5,0.8){$Z^\sigma_{\gamma(\cO(-1)[2])}$}
\put(5,3){\vector(-1,3){1}}
\put(3,5.5){$Z^\sigma_{\gamma(\cO(1))}$}
\put(5,3){\vector(-2,1){2}}
\put(1.8,4){$Z^\sigma_{\gamma(\cO[1])}$}
\put(5.25,3.25){\circle*{0.05}}
\put(5.5,3.25){\circle*{0.05}}
\put(5.75,3.25){\circle*{0.05}}
\put(6,3.25){\circle*{0.05}}
\put(6.25,3.25){\circle*{0.05}}
\put(6.5,3.25){\circle*{0.05}}
\put(6.75,3.25){\circle*{0.05}}
\put(7,3.25){\circle*{0.05}}
\put(7.25,3.25){\circle*{0.05}}
\put(7.5,3.25){\circle*{0.05}}
\put(7.75,3.25){\circle*{0.05}}
\put(8,3.25){\circle*{0.05}}
\put(8.25,3.25){\circle*{0.05}}
\put(8.5,3.25){\circle*{0.05}}
\put(5.5,3.5){\circle*{0.05}}
\put(5.75,3.5){\circle*{0.05}}
\put(6,3.5){\circle*{0.05}}
\put(6.25,3.5){\circle*{0.05}}
\put(6.5,3.5){\circle*{0.05}}
\put(6.75,3.5){\circle*{0.05}}
\put(7,3.5){\circle*{0.05}}
\put(7.25,3.5){\circle*{0.05}}
\put(7.5,3.5){\circle*{0.05}}
\put(7.75,3.5){\circle*{0.05}}
\put(8,3.5){\circle*{0.05}}
\put(8.25,3.5){\circle*{0.05}}
\put(8.5,3.5){\circle*{0.05}}
\put(5.75,3.75){\circle*{0.05}}
\put(6,3.75){\circle*{0.05}}
\put(6.25,3.75){\circle*{0.05}}
\put(6.5,3.75){\circle*{0.05}}
\put(6.75,3.75){\circle*{0.05}}
\put(7,3.75){\circle*{0.05}}
\put(7.25,3.75){\circle*{0.05}}
\put(7.5,3.75){\circle*{0.05}}
\put(7.75,3.75){\circle*{0.05}}
\put(8,3.75){\circle*{0.05}}
\put(8.25,3.75){\circle*{0.05}}
\put(8.5,3.75){\circle*{0.05}}
\put(6,4){\circle*{0.05}}
\put(6.25,4){\circle*{0.05}}
\put(6.5,4){\circle*{0.05}}
\put(6.75,4){\circle*{0.05}}
\put(7,4){\circle*{0.05}}
\put(7.25,4){\circle*{0.05}}
\put(7.5,4){\circle*{0.05}}
\put(7.75,4){\circle*{0.05}}
\put(8,4){\circle*{0.05}}
\put(8.25,4){\circle*{0.05}}
\put(8.5,4){\circle*{0.05}}
\put(6.25,4.25){\circle*{0.05}}
\put(6.5,4.25){\circle*{0.05}}
\put(6.75,4.25){\circle*{0.05}}
\put(7,4.25){\circle*{0.05}}
\put(7.25,4.25){\circle*{0.05}}
\put(7.5,4.25){\circle*{0.05}}
\put(7.75,4.25){\circle*{0.05}}
\put(8,4.25){\circle*{0.05}}
\put(8.25,4.25){\circle*{0.05}}
\put(8.5,4.25){\circle*{0.05}}
\put(6.5,4.5){\circle*{0.05}}
\put(6.75,4.5){\circle*{0.05}}
\put(7,4.5){\circle*{0.05}}
\put(7.25,4.5){\circle*{0.05}}
\put(7.5,4.5){\circle*{0.05}}
\put(7.75,4.5){\circle*{0.05}}
\put(8,4.5){\circle*{0.05}}
\put(8.25,4.5){\circle*{0.05}}
\put(8.5,4.5){\circle*{0.05}}
\put(1.5,3){\circle*{0.05}}
\put(1.5,3.25){\circle*{0.05}}
\put(1.5,3.5){\circle*{0.05}}
\put(1.5,3.75){\circle*{0.05}}
\put(1.5,4){\circle*{0.05}}
\put(1.5,4.25){\circle*{0.05}}
\put(1.5,4.5){\circle*{0.05}}
\put(1.5,4.75){\circle*{0.05}}
\put(1.5,5){\circle*{0.05}}
\put(1.5,5.25){\circle*{0.05}}
\put(1.5,5.5){\circle*{0.05}}
\put(1.5,5.75){\circle*{0.05}}
\put(1.5,6){\circle*{0.05}}
\put(1.75,3){\circle*{0.05}}
\put(1.75,3.25){\circle*{0.05}}
\put(1.75,3.5){\circle*{0.05}}
\put(1.75,3.75){\circle*{0.05}}
\put(1.75,4.5){\circle*{0.05}}
\put(1.75,4.75){\circle*{0.05}}
\put(1.75,5){\circle*{0.05}}
\put(1.75,5.25){\circle*{0.05}}
\put(1.75,5.5){\circle*{0.05}}
\put(1.75,5.75){\circle*{0.05}}
\put(1.75,6){\circle*{0.05}}
\put(2,3){\circle*{0.05}}
\put(2,3.25){\circle*{0.05}}
\put(2,3.5){\circle*{0.05}}
\put(2,3.75){\circle*{0.05}}
\put(2,4.5){\circle*{0.05}}
\put(2,4.75){\circle*{0.05}}
\put(2,5){\circle*{0.05}}
\put(2,5.25){\circle*{0.05}}
\put(2,5.5){\circle*{0.05}}
\put(2,5.75){\circle*{0.05}}
\put(2,6){\circle*{0.05}}
\put(2.25,3){\circle*{0.05}}
\put(2.25,3.25){\circle*{0.05}}
\put(2.25,3.5){\circle*{0.05}}
\put(2.25,3.75){\circle*{0.05}}
\put(2.25,4.5){\circle*{0.05}}
\put(2.25,4.75){\circle*{0.05}}
\put(2.25,5){\circle*{0.05}}
\put(2.25,5.25){\circle*{0.05}}
\put(2.25,5.5){\circle*{0.05}}
\put(2.25,5.75){\circle*{0.05}}
\put(2.25,6){\circle*{0.05}}
\put(2.5,3){\circle*{0.05}}
\put(2.5,3.25){\circle*{0.05}}
\put(2.5,3.5){\circle*{0.05}}
\put(2.5,3.75){\circle*{0.05}}
\put(2.5,4.5){\circle*{0.05}}
\put(2.5,4.75){\circle*{0.05}}
\put(2.5,5){\circle*{0.05}}
\put(2.5,5.25){\circle*{0.05}}
\put(2.5,5.5){\circle*{0.05}}
\put(2.5,5.75){\circle*{0.05}}
\put(2.5,6){\circle*{0.05}}
\put(2.75,3){\circle*{0.05}}
\put(2.75,3.25){\circle*{0.05}}
\put(2.75,3.5){\circle*{0.05}}
\put(2.75,3.75){\circle*{0.05}}
\put(2.75,4.5){\circle*{0.05}}
\put(2.75,4.75){\circle*{0.05}}
\put(2.75,5){\circle*{0.05}}
\put(2.75,5.25){\circle*{0.05}}
\put(2.75,5.5){\circle*{0.05}}
\put(2.75,5.75){\circle*{0.05}}
\put(2.75,6){\circle*{0.05}}
\put(3,3){\circle*{0.05}}
\put(3,3.25){\circle*{0.05}}
\put(3,3.5){\circle*{0.05}}
\put(3,3.75){\circle*{0.05}}
\put(3,4){\circle*{0.05}}
\put(3,4.25){\circle*{0.05}}
\put(3,4.5){\circle*{0.05}}
\put(3,4.75){\circle*{0.05}}
\put(3,5){\circle*{0.05}}
\put(3,5.25){\circle*{0.05}}
\put(3,6){\circle*{0.05}}
\put(3.25,3){\circle*{0.05}}
\put(3.25,3.25){\circle*{0.05}}
\put(3.25,3.5){\circle*{0.05}}
\put(3.25,3.75){\circle*{0.05}}
\put(3.25,4){\circle*{0.05}}
\put(3.25,4.25){\circle*{0.05}}
\put(3.25,4.5){\circle*{0.05}}
\put(3.25,4.75){\circle*{0.05}}
\put(3.25,5){\circle*{0.05}}
\put(3.25,5.25){\circle*{0.05}}
\put(3.25,6){\circle*{0.05}}
\put(3.5,3){\circle*{0.05}}
\put(3.5,3.25){\circle*{0.05}}
\put(3.5,3.5){\circle*{0.05}}
\put(3.5,3.75){\circle*{0.05}}
\put(3.5,4){\circle*{0.05}}
\put(3.5,4.25){\circle*{0.05}}
\put(3.5,4.5){\circle*{0.05}}
\put(3.5,4.75){\circle*{0.05}}
\put(3.5,5){\circle*{0.05}}
\put(3.5,5.25){\circle*{0.05}}
\put(3.5,6){\circle*{0.05}}
\put(3.75,3){\circle*{0.05}}
\put(3.75,3.25){\circle*{0.05}}
\put(3.75,3.5){\circle*{0.05}}
\put(3.75,3.75){\circle*{0.05}}
\put(3.75,4){\circle*{0.05}}
\put(3.75,4.25){\circle*{0.05}}
\put(3.75,4.5){\circle*{0.05}}
\put(3.75,4.75){\circle*{0.05}}
\put(3.75,5){\circle*{0.05}}
\put(3.75,5.25){\circle*{0.05}}
\put(3.75,5.75){\circle*{0.05}}
\put(3.75,6){\circle*{0.05}}
\put(4,3){\circle*{0.05}}
\put(4,3.25){\circle*{0.05}}
\put(4,3.5){\circle*{0.05}}
\put(4,3.75){\circle*{0.05}}
\put(4,4){\circle*{0.05}}
\put(4,4.25){\circle*{0.05}}
\put(4,4.5){\circle*{0.05}}
\put(4,4.75){\circle*{0.05}}
\put(4,5){\circle*{0.05}}
\put(4,5.25){\circle*{0.05}}
\put(4,6){\circle*{0.05}}
\put(4.25,3){\circle*{0.05}}
\put(4.25,3.25){\circle*{0.05}}
\put(4.25,3.5){\circle*{0.05}}
\put(4.25,3.75){\circle*{0.05}}
\put(4.25,4){\circle*{0.05}}
\put(4.25,4.25){\circle*{0.05}}
\put(4.25,4.5){\circle*{0.05}}
\put(4.25,4.75){\circle*{0.05}}
\put(4.25,5){\circle*{0.05}}
\put(4.25,5.25){\circle*{0.05}}
\put(4.5,3){\circle*{0.05}}
\put(4.5,3.25){\circle*{0.05}}
\put(4.5,3.5){\circle*{0.05}}
\put(4.5,3.75){\circle*{0.05}}
\put(4.5,4){\circle*{0.05}}
\put(4.5,4.25){\circle*{0.05}}
\put(4.5,4.5){\circle*{0.05}}
\put(4.75,3){\circle*{0.05}}
\put(4.75,3.25){\circle*{0.05}}
\put(4.75,3.5){\circle*{0.05}}
\put(4.75,3.75){\circle*{0.05}}

\end{picture}
\end{center}

\end{proof}

\begin{lem}\label{lem_coincide_bound_T}
We have $\fD^{\PP^2}_{u,v}=S(\fD_{u,v}^\iin)$
in restriction to the boundary of 
$\bar{U}^\iin_{0,T}$.
\end{lem}

\begin{proof}
The proof is similar to the proof of Lemma
\ref{lem_coincide_int_T}.
Let $\sigma$ be a point in the interior of the boundary of $\bar{U}^\iin_{0,T}$: we have $\sigma \in \{\Rea Z^\sigma_{\gamma(\cO(-1))}=0\}$, or
$\sigma \in \{ \Rea Z^\sigma_{\gamma(\cO)}
=0\}$, or 
$\sigma \in \{ \Rea Z^\sigma_{\gamma(\cO(1))}=0\}$.
We apply a limit version of the argument for
$\sigma$ in the interior of 
$\bar{U}^\iin_{0,T}$
given in the proof of Lemma 
\ref{lem_coincide_int_T}. 
There exists 
$\epsilon >0$ such that 
$\cA^\sigma[Z^\sigma,\frac{1}{2}-\epsilon]=
\cA_0$, and one shows that the only $\gamma \in \Gamma$ such that $Z_\gamma^\sigma \in i \R_{>0}$ are positive multiples of 
$\gamma(\cO(-1)[1])$, or $\gamma(\cO(1))$, 
or $\gamma(\cO)$, or $\gamma(\cO[1])$. 
In each case, the corresponding moduli space of stable objects is a moduli space 
of representations of $Q_0$ of dimension vector $(n,0,0)$, or $(0,n,0)$, or 
$(0,0,n)$, so is empty if $n>1$
(a representation of dimension $(n,0,0)$ is 
necessarily the direct sum of $n$ copies of the simple representation of dimension $(1,0,0)$, and so cannot be stable if $n>1$), and is a point if $n=1$.
So, using Definition 
\ref{def_ray_stability_scattering}
and Definition \ref{def_scattering_D_in_u_v},
the scattering diagrams $\fD^{\PP^2}_{u,v}$ and $S(\fD^\iin_{u,v})$ coincide in restriction to the boundary of 
$\bar{U}^\iin_{0,T}$.
\end{proof}

\begin{lem}\label{lem_coincide_int_R}
The scattering diagram
$\fD^{\PP^2}_{u,v}$ is empty in restriction to the interior of 
$\bar{U}^\iin_{0,R}$ and in restriction to the interior of
$\bar{U}^\iin_{0,L}$.
In other words, we have 
$\fD^{\PP^2}_{u,v}
=S(\fD^\iin_{u,v})$ in restriction to the interior of 
$\bar{U}^\iin_{0,R}$ and in restriction to the interior of 
$\bar{U}^\iin_{0,L}$. 
\end{lem}

\begin{proof}
Let $\sigma=(x,y)$ be a point in the interior of $\bar{U}^\iin_{0,R}$.
Then $\cO(-1)[1]$, $\cO(1)$, and $\cO[1]$
belong to $\cA^\sigma$, with 
\[ \Rea Z^\sigma_{\gamma(\cO(-1)[1])}>0\,,
\Rea Z^\sigma_{\gamma(\cO(1))}<0\,,
\Rea Z^\sigma_{\gamma(\cO[1])}>0\,.\]
We claim that 
$\Arg Z^{\sigma}_{\gamma(\cO[1])}>
\Arg Z^\sigma_{\gamma(\cO(-1)[1])}$.
Indeed, using $y<0$, $x>0$, 
$-y+x+\frac{1}{2}>0$,
\[ Z_{\gamma(\cO(-1))[1]}^{(x,y)}
=-y+x+\frac{1}{2}+i(x+1)\sqrt{x^2+2y}\,,\]
\[Z_{\gamma(\cO[1])}^{(x,y)}=-y+ix\sqrt{x^2+2y}\,,\]
this inequality is equivalent to 
\[ \frac{x}{-y}>\frac{x+1}{-y+x+\frac{1}{2}}\,,\]
that is,
\[ y>-x^2-\frac{x}{2}\,.\]
But as $(x,y)$ is in the interior of
$\bar{U}^\iin_{0,R}$, we have $y>-\frac{x^2}{2}$, and $x>0$, so $-\frac{x^2}{2}>
-x^2-\frac{x}{2}$.

\begin{center}
\setlength{\unitlength}{1.2cm}
\begin{picture}(10,6)
\put(5,3){\line(1,0){4}}
\put(5,3){\line(-1,0){4}}
\put(5,3){\line(0,1){3}}
\put(5,3){\line(0,-1){3}}
\put(5,3){\vector(-1,-1){2}}
\put(2.5,0.8){$Z^\sigma_{\gamma(\cO(-1)[2])}$}
\put(5,3){\vector(-1,3){1}}
\put(3,5.5){$Z^\sigma_{\gamma(\cO(1))}$}
\put(5,3){\vector(1,2){0.5}}
\put(5.05,4.25){$Z^\sigma_{\gamma(\cO[1])}$}
\put(5,3){\vector(1,1){2}}
\put(7.1,5){$Z^\sigma_{\gamma(\cO(-1)[1])}$}
\end{picture}
\end{center}

It follows from the previous inequalities that, denoting $\phi \coloneq \frac{1}{\pi}\Arg 
Z^\sigma_{\gamma(\cO(-1)[1])}$, the objects 
$\cO(-1)[2]$, $\cO[1]$, and $\cO(1)$ belong to 
$\cA^\sigma[Z^\sigma,\phi]$, and so 
$\cA^\sigma[Z^\sigma, \phi]=\cA_0$.
In terms of quiver representations, $\cO(1)$ correspond to the simple 
representation $S_1$ of $Q_0$, which is a subrepresentation of every representation of $Q_0$ of dimension $(n_1,n_2,n_3)$ with 
$n_3 \geqslant 1$. So if $V$ is a stable representation of $Q_0$ of dimension 
$(n_1,n_2,n_3)$ with $n_3 \geqslant 1$, then $\Arg Z^\sigma[\phi](V) \geqslant 
\Arg Z^\sigma[\phi](\cO(1))$.

If $V$ is a stable representation of 
$Q_0$ of dimension $(n_1,n_2,n_3)$ with 
$n_3=0$ and $n_1 \neq 0$, it follows from Lemma \ref{lem_quiver_kronecker} that $n_2 \leqslant 3n_1$.
For every nonnegative integers $n_1$, $n_2$ with $n_1 \neq 0$ and $n_2 \leqslant 3 n_1$, we have 
\[ n_1 \Rea Z^\sigma_{\gamma(\cO(-1))[2])}
+n_2 \Rea Z^\sigma_{\gamma(\cO[1])} 
\leqslant 
n_1 \Rea Z^\sigma_{\gamma(\cO(-1))[2])}
+3n_1 \Rea Z^\sigma_{\gamma(\cO[1])}\]
\[\leqslant n_1(y-x- \frac{1}{2}-3y)
=n_1(-2y-x-\frac{1}{2})\,,\]
and, for $0<x<1$ and $y>-\frac{x^2}{2}$, 
\[ -2y-x-\frac{1}{2}<x^2-x-\frac{1}{2}=(x-\frac{1}{2})^2-\frac{3}{4}<-\frac{1}{2}<0\,.\]

It follows that, for every $E$ nonzero $\sigma$-stable object of $\cA^\sigma$, we have either $\Arg Z^\sigma(E) \leqslant \Arg Z^\sigma_{\gamma(\cO(-1))[1]}$ and so
$\Rea Z^\sigma(E)>0$, or $\Rea Z^\sigma(E) <0$. In particular, 
$\Rea Z^\sigma(E) \neq 0$.

Thus, the scattering diagram $\fD_{u,v}^{\PP^2}$ restricted to the interior of 
$\bar{U}_{0,R}^\iin$ is empty, as
$S(\fD^\iin_{u,v})$.

Figure: an example of configuration of central charges for 
$\sigma=(x,y)$ in the interior of 
$\bar{U}^\iin_{0,R}$. 
If $E$ is a $\sigma$-stable object of 
$\cA^\sigma$, then 
$Z^\sigma(E)$ belongs to the dotted region (it is indeed possible to show that 
$\Arg(Z^\sigma_{\gamma(\cO(-1))[2])}+3
Z^\sigma_{\gamma(\cO[1])})>
\Arg Z^\sigma_{\gamma(\cO(1))}$).
\begin{center}
\setlength{\unitlength}{1.2cm}
\begin{picture}(10,6)
\put(5,3){\line(1,0){4}}
\put(5,3){\line(-1,0){4}}
\put(5,3){\line(0,1){3}}
\put(5,3){\line(0,-1){3}}
\put(5,3){\vector(-1,-1){2}}
\put(2.5,0.8){$Z^\sigma_{\gamma(\cO(-1)[2])}$}
\put(5,3){\vector(-1,3){1}}
\put(3,5.5){$Z^\sigma_{\gamma(\cO(1))}$}
\put(5,3){\vector(1,2){0.5}}
\put(5.05,4.25){$Z^\sigma_{\gamma(\cO[1])}$}
\put(5,3){\vector(1,1){2}}
\put(7.1,4.8){$Z^\sigma_{\gamma(\cO(-1)[1])}$}
\put(5.25,3.25){\circle*{0.05}}
\put(5.5,3.25){\circle*{0.05}}
\put(5.75,3.25){\circle*{0.05}}
\put(6,3.25){\circle*{0.05}}
\put(6.25,3.25){\circle*{0.05}}
\put(6.5,3.25){\circle*{0.05}}
\put(6.75,3.25){\circle*{0.05}}
\put(7,3.25){\circle*{0.05}}
\put(7.25,3.25){\circle*{0.05}}
\put(7.5,3.25){\circle*{0.05}}
\put(7.75,3.25){\circle*{0.05}}
\put(8,3.25){\circle*{0.05}}
\put(8.25,3.25){\circle*{0.05}}
\put(8.5,3.25){\circle*{0.05}}
\put(5.5,3.5){\circle*{0.05}}
\put(5.75,3.5){\circle*{0.05}}
\put(6,3.5){\circle*{0.05}}
\put(6.25,3.5){\circle*{0.05}}
\put(6.5,3.5){\circle*{0.05}}
\put(6.75,3.5){\circle*{0.05}}
\put(7,3.5){\circle*{0.05}}
\put(7.25,3.5){\circle*{0.05}}
\put(7.5,3.5){\circle*{0.05}}
\put(7.75,3.5){\circle*{0.05}}
\put(8,3.5){\circle*{0.05}}
\put(8.25,3.5){\circle*{0.05}}
\put(8.5,3.5){\circle*{0.05}}
\put(5.75,3.75){\circle*{0.05}}
\put(6,3.75){\circle*{0.05}}
\put(6.25,3.75){\circle*{0.05}}
\put(6.5,3.75){\circle*{0.05}}
\put(6.75,3.75){\circle*{0.05}}
\put(7,3.75){\circle*{0.05}}
\put(7.25,3.75){\circle*{0.05}}
\put(7.5,3.75){\circle*{0.05}}
\put(7.75,3.75){\circle*{0.05}}
\put(8,3.75){\circle*{0.05}}
\put(8.25,3.75){\circle*{0.05}}
\put(8.5,3.75){\circle*{0.05}}
\put(6,4){\circle*{0.05}}
\put(6.25,4){\circle*{0.05}}
\put(6.5,4){\circle*{0.05}}
\put(6.75,4){\circle*{0.05}}
\put(7,4){\circle*{0.05}}
\put(7.25,4){\circle*{0.05}}
\put(7.5,4){\circle*{0.05}}
\put(7.75,4){\circle*{0.05}}
\put(8,4){\circle*{0.05}}
\put(8.25,4){\circle*{0.05}}
\put(8.5,4){\circle*{0.05}}
\put(6.25,4.25){\circle*{0.05}}
\put(6.5,4.25){\circle*{0.05}}
\put(6.75,4.25){\circle*{0.05}}
\put(7,4.25){\circle*{0.05}}
\put(7.25,4.25){\circle*{0.05}}
\put(7.5,4.25){\circle*{0.05}}
\put(7.75,4.25){\circle*{0.05}}
\put(8,4.25){\circle*{0.05}}
\put(8.25,4.25){\circle*{0.05}}
\put(8.5,4.25){\circle*{0.05}}
\put(6.5,4.5){\circle*{0.05}}
\put(6.75,4.5){\circle*{0.05}}
\put(7,4.5){\circle*{0.05}}
\put(7.25,4.5){\circle*{0.05}}
\put(7.5,4.5){\circle*{0.05}}
\put(7.75,4.5){\circle*{0.05}}
\put(8,4.5){\circle*{0.05}}
\put(8.25,4.5){\circle*{0.05}}
\put(8.5,4.5){\circle*{0.05}}
\put(1.5,3){\circle*{0.05}}
\put(1.5,3.25){\circle*{0.05}}
\put(1.5,3.5){\circle*{0.05}}
\put(1.5,3.75){\circle*{0.05}}
\put(1.5,4){\circle*{0.05}}
\put(1.5,4.25){\circle*{0.05}}
\put(1.5,4.5){\circle*{0.05}}
\put(1.5,4.75){\circle*{0.05}}
\put(1.5,5){\circle*{0.05}}
\put(1.5,5.25){\circle*{0.05}}
\put(1.5,5.5){\circle*{0.05}}
\put(1.5,5.75){\circle*{0.05}}
\put(1.5,6){\circle*{0.05}}
\put(1.75,3){\circle*{0.05}}
\put(1.75,3.25){\circle*{0.05}}
\put(1.75,3.5){\circle*{0.05}}
\put(1.75,3.75){\circle*{0.05}}
\put(1.75,4){\circle*{0.05}}
\put(1.75,4.25){\circle*{0.05}}
\put(1.75,4.5){\circle*{0.05}}
\put(1.75,4.75){\circle*{0.05}}
\put(1.75,5){\circle*{0.05}}
\put(1.75,5.25){\circle*{0.05}}
\put(1.75,5.5){\circle*{0.05}}
\put(1.75,5.75){\circle*{0.05}}
\put(1.75,6){\circle*{0.05}}
\put(2,3){\circle*{0.05}}
\put(2,3.25){\circle*{0.05}}
\put(2,3.5){\circle*{0.05}}
\put(2,3.75){\circle*{0.05}}
\put(2,4){\circle*{0.05}}
\put(2,4.25){\circle*{0.05}}
\put(2,4.5){\circle*{0.05}}
\put(2,4.75){\circle*{0.05}}
\put(2,5){\circle*{0.05}}
\put(2,5.25){\circle*{0.05}}
\put(2,5.5){\circle*{0.05}}
\put(2,5.75){\circle*{0.05}}
\put(2,6){\circle*{0.05}}
\put(2.25,3){\circle*{0.05}}
\put(2.25,3.25){\circle*{0.05}}
\put(2.25,3.5){\circle*{0.05}}
\put(2.25,3.75){\circle*{0.05}}
\put(2.25,4){\circle*{0.05}}
\put(2.25,4.25){\circle*{0.05}}
\put(2.25,4.5){\circle*{0.05}}
\put(2.25,4.75){\circle*{0.05}}
\put(2.25,5){\circle*{0.05}}
\put(2.25,5.25){\circle*{0.05}}
\put(2.25,5.5){\circle*{0.05}}
\put(2.25,5.75){\circle*{0.05}}
\put(2.25,6){\circle*{0.05}}
\put(2.5,3){\circle*{0.05}}
\put(2.5,3.25){\circle*{0.05}}
\put(2.5,3.5){\circle*{0.05}}
\put(2.5,3.75){\circle*{0.05}}
\put(2.5,4){\circle*{0.05}}
\put(2.5,4.25){\circle*{0.05}}
\put(2.5,4.5){\circle*{0.05}}
\put(2.5,4.75){\circle*{0.05}}
\put(2.5,5){\circle*{0.05}}
\put(2.5,5.25){\circle*{0.05}}
\put(2.5,5.5){\circle*{0.05}}
\put(2.5,5.75){\circle*{0.05}}
\put(2.5,6){\circle*{0.05}}
\put(2.75,3){\circle*{0.05}}
\put(2.75,3.25){\circle*{0.05}}
\put(2.75,3.5){\circle*{0.05}}
\put(2.75,3.75){\circle*{0.05}}
\put(2.75,4){\circle*{0.05}}
\put(2.75,4.25){\circle*{0.05}}
\put(2.75,4.5){\circle*{0.05}}
\put(2.75,4.75){\circle*{0.05}}
\put(2.75,5){\circle*{0.05}}
\put(2.75,5.25){\circle*{0.05}}
\put(2.75,5.5){\circle*{0.05}}
\put(2.75,5.75){\circle*{0.05}}
\put(2.75,6){\circle*{0.05}}
\put(3,3){\circle*{0.05}}
\put(3,3.25){\circle*{0.05}}
\put(3,3.5){\circle*{0.05}}
\put(3,3.75){\circle*{0.05}}
\put(3,4){\circle*{0.05}}
\put(3,4.25){\circle*{0.05}}
\put(3,4.5){\circle*{0.05}}
\put(3,4.75){\circle*{0.05}}
\put(3,5){\circle*{0.05}}
\put(3,5.25){\circle*{0.05}}
\put(3,6){\circle*{0.05}}
\put(3.25,3){\circle*{0.05}}
\put(3.25,3.25){\circle*{0.05}}
\put(3.25,3.5){\circle*{0.05}}
\put(3.25,3.75){\circle*{0.05}}
\put(3.25,4){\circle*{0.05}}
\put(3.25,4.25){\circle*{0.05}}
\put(3.25,4.5){\circle*{0.05}}
\put(3.25,4.75){\circle*{0.05}}
\put(3.25,5){\circle*{0.05}}
\put(3.25,5.25){\circle*{0.05}}
\put(3.25,6){\circle*{0.05}}
\put(3.5,3){\circle*{0.05}}
\put(3.5,3.25){\circle*{0.05}}
\put(3.5,3.5){\circle*{0.05}}
\put(3.5,3.75){\circle*{0.05}}
\put(3.5,4){\circle*{0.05}}
\put(3.5,4.25){\circle*{0.05}}
\put(3.5,4.5){\circle*{0.05}}
\put(3.5,4.75){\circle*{0.05}}
\put(3.5,5){\circle*{0.05}}
\put(3.5,5.25){\circle*{0.05}}
\put(3.5,6){\circle*{0.05}}
\put(3.75,3){\circle*{0.05}}
\put(3.75,3.25){\circle*{0.05}}
\put(3.75,3.5){\circle*{0.05}}
\put(3.75,3.75){\circle*{0.05}}
\put(3.75,4){\circle*{0.05}}
\put(3.75,4.25){\circle*{0.05}}
\put(3.75,4.5){\circle*{0.05}}
\put(3.75,4.75){\circle*{0.05}}
\put(3.75,5){\circle*{0.05}}
\put(3.75,5.25){\circle*{0.05}}
\put(3.75,5.75){\circle*{0.05}}
\put(3.75,6){\circle*{0.05}}
\put(4,3){\circle*{0.05}}
\put(4,3.25){\circle*{0.05}}
\put(4,3.5){\circle*{0.05}}
\put(4,3.75){\circle*{0.05}}
\put(4,4){\circle*{0.05}}
\put(4,4.25){\circle*{0.05}}
\put(4,4.5){\circle*{0.05}}
\put(4,4.75){\circle*{0.05}}
\put(4,5){\circle*{0.05}}
\put(4,5.25){\circle*{0.05}}
\put(4,6){\circle*{0.05}}
\put(4.25,3){\circle*{0.05}}
\put(4.25,3.25){\circle*{0.05}}
\put(4.25,3.5){\circle*{0.05}}
\put(4.25,3.75){\circle*{0.05}}
\put(4.25,4){\circle*{0.05}}
\put(4.25,4.25){\circle*{0.05}}
\put(4.25,4.5){\circle*{0.05}}
\put(4.25,4.75){\circle*{0.05}}
\put(4.25,5){\circle*{0.05}}
\put(4.25,5.25){\circle*{0.05}}
\put(4.5,3){\circle*{0.05}}
\put(4.5,3.25){\circle*{0.05}}
\put(4.5,3.5){\circle*{0.05}}
\put(4.5,3.75){\circle*{0.05}}
\put(4.5,4){\circle*{0.05}}
\put(4.5,4.25){\circle*{0.05}}
\put(4.5,4.5){\circle*{0.05}}
\put(4.75,3){\circle*{0.05}}
\put(4.75,3.25){\circle*{0.05}}
\put(4.75,3.5){\circle*{0.05}}
\put(4.75,3.75){\circle*{0.05}}
\end{picture}
\end{center}

The proof for $\bar{U}^\iin_{0,L}$ instead of $\bar{U}^\iin_{0,R}$ is completely analogous.
Let $\sigma=(x,y)$ be a point in the interior of $\bar{U}^\iin_{0,L}$.
Then $\cO(-1)[1]$, $\cO(1)$, and $\cO$
belongs to $\cA^\sigma$, with 
\[ \Rea Z^\sigma_{\gamma(\cO(-1)[1])}>0\,,
\Rea Z^\sigma_{\gamma(\cO(1))}<0\,,
\Rea Z^\sigma_{\gamma(\cO)}<0\,.\]
We claim that $\Arg Z^\sigma_{\gamma(\cO(1))} >\Arg Z^\sigma_{\gamma(\cO)}$.
Indeed, using $y<0$, $x<0$, $-y-x+\frac{1}{2}>0$,
\[Z_{\gamma(\cO)}^{(x,y)}=y-ix\sqrt{x^2+2y}\,,\]
\[Z_{\gamma(\cO(1))}^{(x,y)}
=y+x-\frac{1}{2}-i(x-1)\sqrt{x^2+2y}\,,\]
this inequality is equivalent to 
\[ \frac{-x}{-y}>\frac{1-x}{-y-x+\frac{1}{2}}\,,\]
that is, 
\[ y>-x^2+\frac{x}{2}\,.\]
But as $(x,y)$ is in the interior of 
$\bar{U}^\iin_{0,L}$, we have 
$y>-\frac{x^2}{2}$, and $x<0$, so
$-\frac{x^2}{2}>-x^2+\frac{x}{2}$.

\begin{center}
\setlength{\unitlength}{1.2cm}
\begin{picture}(10,6)
\put(5,3){\line(1,0){4}}
\put(5,3){\line(-1,0){4}}
\put(5,3){\line(0,1){3}}
\put(5,3){\line(0,-1){3}}
\put(5,3){\vector(-1,-1){2}}
\put(2.5,0.8){$Z^\sigma_{\gamma(\cO(-1)[2])}$}
\put(5,3){\vector(-3,4){2}}
\put(2,5.5){$Z^\sigma_{\gamma(\cO(1))}$}
\put(5,3){\vector(-1,3){0.3}}
\put(4.1,4.2){$Z^\sigma_{\gamma(\cO)}$}
\put(5,3){\vector(1,-3){0.3}}
\put(5.5,2){$Z^\sigma_{\gamma(\cO[1])}$}
\put(5,3){\vector(1,1){2}}
\put(7.1,4.8){$Z^\sigma_{\gamma(\cO(-1)[1])}$}
\end{picture}
\end{center}

It follows from the previous inequalities
that, denoting $\phi \coloneq \frac{1}{\pi} \Arg Z^\sigma_{\gamma(\cO)}$, the objects $\cO(-1)[2]$,
$\cO[1]$, and $\cO(1)$ belong to 
$\cA^\sigma[Z^\sigma,\phi]$, and so 
$\cA^\sigma[Z^\sigma,\phi]=\cA_0$.
In terms of quiver representations, 
$\cO(-1)[2]$ correspond to the simple representation $S_{-1}$ of 
$Q_0$, which is a representation quotient of every representation of $Q_0$ of dimension $(n_1,n_2,n_3)$ with $n_1 \geqslant 1$. 
So, if $V$ is a stable representation of $Q_0$ of dimension $(n_1,n_2,n_3)$ with $n_1 \geqslant 1$, then 
$\Arg Z^\sigma[\phi](\cO(-1)[2]) \geqslant 
\Arg Z^\sigma[\phi](V)$.

If $V$ is a stable representation of $Q_0$ of dimension 
$(n_1,n_2,n_3)$ with $n_1=0$ and 
$n_3 \neq 0$, it follows from 
Lemma \ref{lem_quiver_kronecker} that 
$n_2 \leqslant 3n_3$.
For every nonnegative integers $n_2$, $n_3$
with $n_3 \neq 0$ and $n_2 \leqslant 3n_3$, we have 
\[ n_2 \Rea Z^\sigma_{\gamma(\cO[1])}
+n_3 \Rea Z^\sigma_{\gamma(\cO(1))}
\leqslant 3n_3 \Rea Z^\sigma_{\gamma(\cO[1])}
+n_3 \Rea Z^\sigma_{\gamma(\cO(1))} \]
\[ \leqslant n_3(-3y+y+x-\frac{1}{2})
=n_3(-2y+x-\frac{1}{2})\,,\]
and, for $-1<x<0$ and $y >-\frac{x^2}{2}$, 
\[ -2y+x-\frac{1}{2}<x^2+x-\frac{1}{2}
=(x+\frac{1}{2})^2-\frac{3}{4}<-\frac{1}{2}<0\,.\]

It follows that, for every $E$ nonzero $\sigma$-stable object of $\cA^\sigma$, we have either $\Arg Z^\sigma(E) \leqslant 
\Arg Z^{\sigma}_{\gamma(\cO(-1))[1]}$ and so $\Rea Z^\sigma(E)>0$, or $\Rea Z^\sigma(E)<0$. In particular, 
$\Rea Z^\sigma(E) \neq 0$.

Thus, the scattering diagram $\fD_{u,v}^{\PP^2}$ restricted to the interior of 
$\bar{U}^\iin_{0,L}$ is empty, as 
$S(\fD^\iin_{u,v})$.

Figure: an example of configuration of central charges for 
$\sigma=(x,y)$ in the interior of 
$\bar{U}^\iin_{0,L}$. 
If $E$ is a $\sigma$-stable object of 
$\cA^\sigma$, then 
$Z^\sigma(E)$ belongs to the dotted region (it is indeed possible to show that 
$
\Arg Z^\sigma_{\gamma(\cO(-1)[2])}
>\Arg(Z^\sigma_{\gamma(\cO(1))}+3
Z^\sigma_{\gamma(\cO[1])})$).
\begin{center}
\setlength{\unitlength}{1.2cm}
\begin{picture}(10,6)
\put(5,3){\line(1,0){4}}
\put(5,3){\line(-1,0){4}}
\put(5,3){\line(0,1){3}}
\put(5,3){\line(0,-1){3}}
\put(5,3){\vector(-1,-1){2}}
\put(2.5,0.8){$Z^\sigma_{\gamma(\cO(-1)[2])}$}
\put(5,3){\vector(-3,4){2}}
\put(1.9,5.5){$Z^\sigma_{\gamma(\cO(1))}$}
\put(5,3){\vector(-1,3){0.3}}
\put(4.1,4.2){$Z^\sigma_{\gamma(\cO)}$}
\put(5,3){\vector(1,-3){0.3}}
\put(5.5,2){$Z^\sigma_{\gamma(\cO[1])}$}
\put(5,3){\vector(1,1){2}}
\put(7.1,4.8){$Z^\sigma_{\gamma(\cO(-1)[1])}$}
\put(5.25,3.25){\circle*{0.05}}
\put(5.5,3.25){\circle*{0.05}}
\put(5.75,3.25){\circle*{0.05}}
\put(6,3.25){\circle*{0.05}}
\put(6.25,3.25){\circle*{0.05}}
\put(6.5,3.25){\circle*{0.05}}
\put(6.75,3.25){\circle*{0.05}}
\put(7,3.25){\circle*{0.05}}
\put(7.25,3.25){\circle*{0.05}}
\put(7.5,3.25){\circle*{0.05}}
\put(7.75,3.25){\circle*{0.05}}
\put(8,3.25){\circle*{0.05}}
\put(8.25,3.25){\circle*{0.05}}
\put(8.5,3.25){\circle*{0.05}}
\put(5.5,3.5){\circle*{0.05}}
\put(5.75,3.5){\circle*{0.05}}
\put(6,3.5){\circle*{0.05}}
\put(6.25,3.5){\circle*{0.05}}
\put(6.5,3.5){\circle*{0.05}}
\put(6.75,3.5){\circle*{0.05}}
\put(7,3.5){\circle*{0.05}}
\put(7.25,3.5){\circle*{0.05}}
\put(7.5,3.5){\circle*{0.05}}
\put(7.75,3.5){\circle*{0.05}}
\put(8,3.5){\circle*{0.05}}
\put(8.25,3.5){\circle*{0.05}}
\put(8.5,3.5){\circle*{0.05}}
\put(5.75,3.75){\circle*{0.05}}
\put(6,3.75){\circle*{0.05}}
\put(6.25,3.75){\circle*{0.05}}
\put(6.5,3.75){\circle*{0.05}}
\put(6.75,3.75){\circle*{0.05}}
\put(7,3.75){\circle*{0.05}}
\put(7.25,3.75){\circle*{0.05}}
\put(7.5,3.75){\circle*{0.05}}
\put(7.75,3.75){\circle*{0.05}}
\put(8,3.75){\circle*{0.05}}
\put(8.25,3.75){\circle*{0.05}}
\put(8.5,3.75){\circle*{0.05}}
\put(6,4){\circle*{0.05}}
\put(6.25,4){\circle*{0.05}}
\put(6.5,4){\circle*{0.05}}
\put(6.75,4){\circle*{0.05}}
\put(7,4){\circle*{0.05}}
\put(7.25,4){\circle*{0.05}}
\put(7.5,4){\circle*{0.05}}
\put(7.75,4){\circle*{0.05}}
\put(8,4){\circle*{0.05}}
\put(8.25,4){\circle*{0.05}}
\put(8.5,4){\circle*{0.05}}
\put(6.25,4.25){\circle*{0.05}}
\put(6.5,4.25){\circle*{0.05}}
\put(6.75,4.25){\circle*{0.05}}
\put(7,4.25){\circle*{0.05}}
\put(7.25,4.25){\circle*{0.05}}
\put(7.5,4.25){\circle*{0.05}}
\put(7.75,4.25){\circle*{0.05}}
\put(8,4.25){\circle*{0.05}}
\put(8.25,4.25){\circle*{0.05}}
\put(8.5,4.25){\circle*{0.05}}
\put(6.5,4.5){\circle*{0.05}}
\put(6.75,4.5){\circle*{0.05}}
\put(7,4.5){\circle*{0.05}}
\put(7.25,4.5){\circle*{0.05}}
\put(7.5,4.5){\circle*{0.05}}
\put(7.75,4.5){\circle*{0.05}}
\put(8,4.5){\circle*{0.05}}
\put(8.25,4.5){\circle*{0.05}}
\put(8.5,4.5){\circle*{0.05}}
\put(1.5,3){\circle*{0.05}}
\put(1.5,3.25){\circle*{0.05}}
\put(1.5,3.5){\circle*{0.05}}
\put(1.5,3.75){\circle*{0.05}}
\put(1.5,4){\circle*{0.05}}
\put(1.5,4.25){\circle*{0.05}}
\put(1.5,4.5){\circle*{0.05}}
\put(1.5,4.75){\circle*{0.05}}
\put(1.5,5){\circle*{0.05}}
\put(1.5,5.25){\circle*{0.05}}
\put(1.5,5.5){\circle*{0.05}}
\put(1.5,5.75){\circle*{0.05}}
\put(1.5,6){\circle*{0.05}}
\put(1.75,3){\circle*{0.05}}
\put(1.75,3.25){\circle*{0.05}}
\put(1.75,3.5){\circle*{0.05}}
\put(1.75,3.75){\circle*{0.05}}
\put(1.75,4){\circle*{0.05}}
\put(1.75,4.25){\circle*{0.05}}
\put(1.75,4.5){\circle*{0.05}}
\put(1.75,4.75){\circle*{0.05}}
\put(1.75,5){\circle*{0.05}}
\put(1.75,5.25){\circle*{0.05}}
\put(1.75,5.5){\circle*{0.05}}
\put(1.75,5.75){\circle*{0.05}}
\put(1.75,6){\circle*{0.05}}
\put(2,3){\circle*{0.05}}
\put(2,3.25){\circle*{0.05}}
\put(2,3.5){\circle*{0.05}}
\put(2,3.75){\circle*{0.05}}
\put(2,4){\circle*{0.05}}
\put(2,4.25){\circle*{0.05}}
\put(2,4.5){\circle*{0.05}}
\put(2,4.75){\circle*{0.05}}
\put(2,5){\circle*{0.05}}
\put(2,5.25){\circle*{0.05}}
\put(2,6){\circle*{0.05}}
\put(2.25,3){\circle*{0.05}}
\put(2.25,3.25){\circle*{0.05}}
\put(2.25,3.5){\circle*{0.05}}
\put(2.25,3.75){\circle*{0.05}}
\put(2.25,4){\circle*{0.05}}
\put(2.25,4.25){\circle*{0.05}}
\put(2.25,4.5){\circle*{0.05}}
\put(2.25,4.75){\circle*{0.05}}
\put(2.25,5){\circle*{0.05}}
\put(2.25,5.25){\circle*{0.05}}
\put(2.25,6){\circle*{0.05}}
\put(2.5,3){\circle*{0.05}}
\put(2.5,3.25){\circle*{0.05}}
\put(2.5,3.5){\circle*{0.05}}
\put(2.5,3.75){\circle*{0.05}}
\put(2.5,4){\circle*{0.05}}
\put(2.5,4.25){\circle*{0.05}}
\put(2.5,4.5){\circle*{0.05}}
\put(2.5,4.75){\circle*{0.05}}
\put(2.5,5){\circle*{0.05}}
\put(2.5,5.25){\circle*{0.05}}
\put(2.5,6){\circle*{0.05}}
\put(2.75,3){\circle*{0.05}}
\put(2.75,3.25){\circle*{0.05}}
\put(2.75,3.5){\circle*{0.05}}
\put(2.75,3.75){\circle*{0.05}}
\put(2.75,4){\circle*{0.05}}
\put(2.75,4.25){\circle*{0.05}}
\put(2.75,4.5){\circle*{0.05}}
\put(2.75,4.75){\circle*{0.05}}
\put(2.75,5){\circle*{0.05}}
\put(2.75,5.25){\circle*{0.05}}
\put(2.75,6){\circle*{0.05}}
\put(3,3){\circle*{0.05}}
\put(3,3.25){\circle*{0.05}}
\put(3,3.5){\circle*{0.05}}
\put(3,3.75){\circle*{0.05}}
\put(3,4){\circle*{0.05}}
\put(3,4.25){\circle*{0.05}}
\put(3,4.5){\circle*{0.05}}
\put(3,4.75){\circle*{0.05}}
\put(3,5){\circle*{0.05}}
\put(3,5.25){\circle*{0.05}}
\put(3,6){\circle*{0.05}}
\put(3.25,3){\circle*{0.05}}
\put(3.25,3.25){\circle*{0.05}}
\put(3.25,3.5){\circle*{0.05}}
\put(3.25,3.75){\circle*{0.05}}
\put(3.25,4){\circle*{0.05}}
\put(3.25,4.25){\circle*{0.05}}
\put(3.25,4.5){\circle*{0.05}}
\put(3.25,4.75){\circle*{0.05}}
\put(3.25,5){\circle*{0.05}}
\put(3.25,5.25){\circle*{0.05}}
\put(3.25,5.5){\circle*{0.05}}
\put(3.25,5.75){\circle*{0.05}}
\put(3.25,6){\circle*{0.05}}
\put(3.5,3){\circle*{0.05}}
\put(3.5,3.25){\circle*{0.05}}
\put(3.5,3.5){\circle*{0.05}}
\put(3.5,3.75){\circle*{0.05}}
\put(3.5,4){\circle*{0.05}}
\put(3.5,4.25){\circle*{0.05}}
\put(3.5,4.5){\circle*{0.05}}
\put(3.5,4.75){\circle*{0.05}}
\put(3.5,5){\circle*{0.05}}
\put(3.5,5.25){\circle*{0.05}}
\put(3.5,5.5){\circle*{0.05}}
\put(3.5,5.75){\circle*{0.05}}
\put(3.5,6){\circle*{0.05}}
\put(3.75,3){\circle*{0.05}}
\put(3.75,3.25){\circle*{0.05}}
\put(3.75,3.5){\circle*{0.05}}
\put(3.75,3.75){\circle*{0.05}}
\put(3.75,4){\circle*{0.05}}
\put(3.75,4.25){\circle*{0.05}}
\put(3.75,4.5){\circle*{0.05}}
\put(3.75,4.75){\circle*{0.05}}
\put(3.75,5){\circle*{0.05}}
\put(3.75,5.25){\circle*{0.05}}
\put(3.75,5.5){\circle*{0.05}}
\put(3.75,5.75){\circle*{0.05}}
\put(3.75,6){\circle*{0.05}}
\put(4,3){\circle*{0.05}}
\put(4,3.25){\circle*{0.05}}
\put(4,3.5){\circle*{0.05}}
\put(4,3.75){\circle*{0.05}}
\put(4,4){\circle*{0.05}}
\put(4,4.25){\circle*{0.05}}
\put(4,4.5){\circle*{0.05}}
\put(4,4.75){\circle*{0.05}}
\put(4,5){\circle*{0.05}}
\put(4,5.25){\circle*{0.05}}
\put(4,5.5){\circle*{0.05}}
\put(4,5.75){\circle*{0.05}}
\put(4,6){\circle*{0.05}}
\put(4.25,3){\circle*{0.05}}
\put(4.25,3.25){\circle*{0.05}}
\put(4.25,3.5){\circle*{0.05}}
\put(4.25,3.75){\circle*{0.05}}
\put(4.25,4){\circle*{0.05}}
\put(4.25,4.75){\circle*{0.05}}
\put(4.25,5){\circle*{0.05}}
\put(4.25,5.25){\circle*{0.05}}
\put(4.5,3){\circle*{0.05}}
\put(4.5,3.25){\circle*{0.05}}
\put(4.5,3.5){\circle*{0.05}}
\put(4.5,3.75){\circle*{0.05}}
\put(4.5,4){\circle*{0.05}}
\put(4.75,3){\circle*{0.05}}
\put(4.75,3.25){\circle*{0.05}}
\put(4.75,3.5){\circle*{0.05}}
\put(4.75,3.75){\circle*{0.05}}
\end{picture}
\end{center}

\end{proof}

\begin{lem} \label{lem_coincide_bound_R}
We have $\fD^{\PP^2}_{u,v}
=S(\fD^\iin_{u,v})$ in restriction to the boundary of $\bar{U}^\iin_{0,R}$ and in restriction to the boundary of 
$\bar{U}^\iin_{0,L}$.
\end{lem}

\begin{proof}
Lemma \ref{lem_coincide_bound_R} is a limiting case of Lemma
\ref{lem_coincide_int_R} in the same way as 
Lemma \ref{lem_coincide_bound_T}
is a limiting case of Lemma 
\ref{lem_coincide_int_T}.

On the boundary $|\fd_0^+|$
(resp.\ $|\fd_1^-|$)
of 
$\bar{U}^\iin_{0,R}$,
 the only $\gamma \in \Gamma$ such that $Z_\gamma^\sigma \in i \R_{>0}$ are positive multiples of 
$\gamma(\cO[1])$ (resp. $\gamma(\cO(1))$).
The corresponding moduli space of stable objects is a moduli space 
of representations of $Q_0$ of dimension vector $(0,n,0)$ (resp. $(0,0,n)$), so is empty if $n>1$
(a representation of dimension $(0,n,0)$ is 
necessarily the direct sum of $n$ copies of the simple representation of dimension $(0,1,0)$, and so cannot be stable if $n>1$), and is a point if $n=1$.
So, using Definition 
\ref{def_ray_stability_scattering}
and Definition \ref{def_scattering_D_in_u_v},
the scattering diagrams $\fD^{\PP^2}_{u,v}$ and $S(\fD^\iin_{u,v})$ coincide in restriction to the boundary of 
$\bar{U}^\iin_{0,R}$.

Similarly, on the boundary $|\fd_0^-|$
(resp.\ $|\fd_{-1}^+|$)
of 
$\bar{U}^\iin_{0,L}$,
 the only $\gamma \in \Gamma$ such that $Z_\gamma^\sigma \in i \R_{>0}$ are positive multiples of 
$\gamma(\cO)$ (resp.\ $\gamma(\cO(-1)[1])$).
The corresponding moduli space of stable objects is a moduli space 
of representations of $Q_0$ of dimension vector $(0,n,0)$ (resp. $(n,0,0)$), so is empty if $n>1$, and is a point if $n=1$.
So, using Definition 
\ref{def_ray_stability_scattering}
and Definition \ref{def_scattering_D_in_u_v},
the scattering diagrams $\fD^{\PP^2}_{u,v}$ and $S(\fD^\iin_{u,v})$ coincide in restriction to the boundary of 
$\bar{U}^\iin_{0,L}$.
\end{proof}

We can now end the proof of Theorem 
\ref{thm_coincide_initial}: the scattering diagrams 
$\fD^{\PP^2}_{u,v}$ and 
$S(\fD^\iin_{u,v})$ coincide in restriction to
$\bar{U}_{0,T}^\iin$ according to Lemma
\ref{lem_coincide_int_T} and Lemma
\ref{lem_coincide_bound_T}, 
in restriction to
$\bar{U}_{0,L}^\iin$ and 
$\bar{U}_{0,R}^\iin$
according to Lemma 
\ref{lem_coincide_int_R} and Lemma
\ref{lem_coincide_bound_R},  and so, as
\[\bar{U}^\iin_0=\bar{U}^\iin_{0,T}
\cup \bar{U}^\iin_{0,L}
\cup \bar{U}^\iin_{0,R}\,,\]
in restriction to $\bar{U}^\iin_0$, and so,
using the action of 
$\psi(1)$, in restriction to $\bar{U}^\iin$.

\section{The equality $\fD^{\PP^2}_{u,v}=S(\fD_{u,v}^\iin)$}
\label{section_main_result}

In \cref{section_criterion} we prove that a consistent scattering diagram which
has the same initial data as $S(\fD_{u,v}^\iin)$ in fact coincides with $S(\fD_{u,v}^\iin)$.
In 
\cref{subsection_main_result} we use 
Theorem \ref{thm_consistency} and Theorem
\ref{thm_coincide_initial} to show that 
$\fD^{\PP^2}_{u,v}$ satisfies this criterion, and so we get our first main result, Theorem \ref{thm_main_precise}, 
stated as Theorem \ref{main_thm} in the introduction, that is
the equality $\fD^{\PP^2}_{u,v}=S(\fD_{u,v}^\iin)$.

\subsection{Criterion for $\fD=S(\fD^\iin_{u,v})$}
\label{section_criterion}

In \cref{section_scattering_final} we defined a scattering diagram 
$S(\fD^\iin_{u,v})$ on $U$ for 
$(M,\fg_{u,v})$.
In this section we prove 
Proposition \ref{prop_criterion}, which gives a sufficient criterion to prove that a scattering diagram $\fD$ on $U$
for $(M,\fg_{u,v})$ coincides with $S(\fD^\iin_{u,v})$.

\begin{defn} \label{def_parent}
Let $\fD$ be a scattering diagram on 
$U$ for $(M,\fg_{u,v})$, and let $\fd$ and $\fd'$ be two rays of 
$\fD$. The ray $\fd$ is a \emph{parent} of
the ray $\fd'$ if:
\begin{itemize}
\item[(i)] $\fd$ is bounded, that is, we have 
$|\fd|=\Init(\fd)-[0,T_\fd]m_\fd$ for some 
$T_{\fd} \in \R_{>0}$.
\item[(ii)] the endpoint of $\fd$ coincides with the initial point $\Init(\fd')$ of 
$\fd'$.
\item[(iii)] $\fd'$ defines an outgoing ray of the local scattering diagram $\fD_{\Init(\fd')}$.
\item[(iv)] $\fd$ defines an ingoing ray of the 
local scattering diagram 
$\fD_{\Init(\fd')}$.
\item[(v)] $\varphi_{\Init(\fd')}(m_\fd) \leqslant \varphi_{\Init(\fd')}(m_{\fd'})$.
\end{itemize}
\end{defn}

It follows from condition (ii) in the Definition 
\ref{def_scattering} of a scattering diagram that a given ray $\fd'$ has finitely many parents.

\begin{lem} \label{lem_desc_phi}
Let $\fD$ be a scattering diagram on 
$U$ for $(M,\fg_{u,v})$, and let $\fd$ and $\fd'$ be two rays of 
$\fD$. If $\fd$ is a parent of $\fd'$, then 
\[ \varphi_{\Init(\fd)}(m_{\fd}) < \varphi_{\Init(\fd')}(m_{\fd'}) \,. \]
\end{lem}

\begin{proof}
By Definition \ref{def_parent}, the endpoint of $\fd$ is the initial point of $\fd'$, and so we have $\varphi_{\Init(\fd)}(m_{\fd})<\varphi_{\Init(\fd')}(m_{\fd})$ by Lemma \ref{lem_increasing}. On the other hand, we have 
$\varphi_{\Init(\fd')}(m_{\fd}) \leqslant \varphi_{\Init(\fd')}(m_{\fd'})$ by Definition \ref{def_parent}(v).
\end{proof}

\begin{defn} \label{def_descendant}
Let $\fD$ be a scattering diagram on 
$U$ for $(M,\fg_{u,v})$, and let $\fd$ and $\fd'$ be two rays of 
$\fD$. The ray $\fd'$ is a \emph{descendant} of the ray 
$\fd$ of if there exists a finite sequence of rays $\fd_j$ of $\fD$, for
$0 \leqslant j \leqslant N$, such that:
\begin{itemize}
\item[(i)] $\fd_0=\fd$.
\item[(ii)] $\fd_N=\fd'$
\item[(iii)] For every $0 \leqslant j \leqslant N-1$, the ray $\fd_j$ is a parent of the ray $\fd_{j+1}$.
\end{itemize}
\end{defn}


\begin{defn}
Let $\fD$ be a scattering diagram on 
$U$ for $(M,\fg_{u,v})$, and let $\fd$ be a ray of 
$\fD$. A ray $\fd$ of 
$\fD$ is an \emph{absolute ancestor} if $\fd$ has no parent.
\end{defn}

\begin{lem} \label{lem_x_bound}
Let $\fD$ be a scattering diagram 
on $U$ for 
$(M, \fg_{u,v})$, and 
let $\fd$ and $\fd'$ be two rays of $\fD$
such that $\fd'$ is a descendant of $\fd$. 
Then, writing $\Init(\fd)=(x,y)$ and 
$\Init(\fd')=(x',y')$,
we have 
\[|x'-x|\leqslant \frac{1}{2}\varphi_{\Init(\fd')}(m_{\fd'})\,.\] 
\end{lem}

\begin{proof}
It is a consequence of Definition 
\ref{def_descendant} and Lemma \ref{lem_increasing}.
\end{proof}

\begin{lem} \label{lem_all_descendant}
Let $\fD$ be a consistent scattering diagram on 
$U$ for 
$(M,\fg_{u,v})$ such that 
$\fD$ coincides with $S(\fD^\iin_{u,v})$ in restriction to 
$\bar{U}^\iin$, and let 
$\fd'$ be a ray of $\fD$. Then there exists 
a ray $\fd$ of $\fD$, which is an absolute ancestor, such that 
$\fd'$ is a descendant of $\fd$. 
\end{lem}

\begin{proof}
Writing $\Init(\fd')=(x_{\fd'},y_{\fd'})$,
we denote 
\[ \bar{U}_{\fd'} \coloneq \{ (x,y) \in \bar{U}\,|
x^2+2y \leqslant x_{\fd'}^2+2y_{\fd'}\,,
|x-x_{\fd'}| \leqslant \frac{1}{2} 
\varphi_{\Init(\fd')}(m_{\fd'})\}\,.\]
According to Lemma \ref{lem_x2+2y_increasing} and 
Lemma \ref{lem_x_bound}, if $\fd$ is a ray of
$\fD$ such that $\fd'$ is a descendant of 
$\fd$, then we have 
$\Init(\fd) \in \bar{U}_{\fd'}$.

We denote 
\[ K_{\fd'} \coloneq \{ (x,y) \in \bar{U}\,|\,
\frac{1}{4} \leqslant x^2+2y \leqslant x_{\fd'}^2+2y_{\fd'}\,,
|x-x_{\fd'}| \leqslant \frac{1}{2} 
\varphi_{\Init(\fd')}(m_{\fd'})\}\,.\]
The set $K_{\fd'}$ is a compact subset of $U$ and we have $\bar{U}_{\fd'}-K_{\fd'} \subset \bar{U}^\iin$.

As $\fD$ is consistent, every ray of 
$\fD$ is either an absolute ancestor or 
is the descendant of an other ray.
So, if we assume by contradiction that 
$\fd'$ is not the descendant of an absolute ancestor, we can find an infinite sequence of rays 
$\fd_j$, $j \in \NN$, such that 
$\fd_0=\fd'$, and such that, for every $j \in \NN$, $\fd_{j+1}$ is a parent
of $\fd_j$. If there exists $j \in \NN$ such that $\Init(\fd_j) \in \bar{U}^\iin$, 
then we get a contradiction as by assumption $\fD$ coincides with $S(\fD^\iin)$ and every ray of 
$S(\fD^\iin)$ is a descendant of an absolute ancestor. If not, then we have 
$\Init(\fd_j) \in K_{\fd'}$ for every 
$j \in \NN$, and as $\varphi_{\Init(\fd_j)}(m_{\fd_j}) \leqslant \varphi_{\Init(\fd')}(m_{\fd'})$
for every $j \in \NN$ by Lemma \ref{lem_desc_phi}, we get a contradiction with condition (iii) of Definition \ref{def_scattering} of a scattering diagram.
\end{proof}

\begin{lem} \label{lem_absolute_ancestor}
Let $\fD$ be a consistent scattering diagram on $U$ for $(M,\fg_{u,v})$. Then, every absolute ancestor ray of $\fD$ intersects $\bar{U}^\iin$.
\end{lem}

\begin{proof}
Let $\fd$ be an absolute ancestor ray of 
$\fD^{\PP^2}_{u,v}$. We claim that $\Init(\fd) \in \bar{U}^\iin$.
Indeed, assume by contraction that $\Init(\fd) \notin \bar{U}^\iin$.
Then we have in particular $\Init(\fd) \in U$ and consistency 
of $\fD$ at the point 
$\Init(\fd)$ contradicts the assumption that $\fd$ is an absolute ancestor ray.
\end{proof}

\begin{prop} \label{prop_criterion}
Let $\fD$ be a consistent scattering diagram on $U$ for
$(M,\fg_{u,v})$ which coincides with $S(\fD_{u,v}^\iin)$
in restriction to $\bar{U}^\iin$.
Then, we have
$\fD=S(\fD_{u,v}^\iin)$.
\end{prop}

\begin{proof}
We prove by induction on $k \in \Z_{\geqslant 0}$ that $\fD$ coincides with 
$S(\fD_{u,v}^\iin)$ in restrictions to the set of rays $\fd$ such that $\varphi_{\Init(\fd)} (m_\fd) \leqslant k$. 

We first show the result for $k=0$. 
A ray $\fd$ with $\varphi_{\Init(\fd)} (m_\fd) =0$ is necessarily an absolute ancestor: if not, a parent $\fd'$ of $\fd$ would have to satisfy $\varphi_{\Init(\fd')}(m_{\fd'})<0$ by Lemma \ref{lem_desc_phi}, and this would be a contradiction with (ii) of the Definition \ref{def_scattering} of a scattering diagram on $U$.
By Lemma \ref{lem_absolute_ancestor}, every absolute ancestor ray of 
$\fD$ intersects $\bar{U}^\iin$, and so the assumption that $\fD$ and $S(\fD_{u,v}^\iin)$ coincide in restriction to
$\bar{U}^\iin$  implies that $\fD$ and $S(\fD_{u,v}^\iin)$ have the same rays $\fd$ with $\varphi_{\Init(\fd)} (m_\fd) =0$.

Before treating the induction step, we remark that a similar argument also shows that for any point $\sigma \in U$ around which we have a non-trivial local scattering, and for any ray $\fd$ of $\fD$ containing $\sigma$, we have $\varphi_{\sigma}(m_\fd) \geqslant 1$. Indeed, by Lemma \ref{lem_all_descendant}, every ray $\fd$ of $\fD$ is descendant from absolute ancestors, and by Lemma \ref{lem_desc_phi}, $\varphi(m_\fd)$ is increasing from ancestors to descendants, so it is enough to prove the result for points $\sigma$ where we have a non-trivial local scattering of absolute ancestors of $\fD$.
But by Lemma \ref{lem_absolute_ancestor} and the assumption that $\fD$ and $S(\fD_{u,v}^\iin)$ coincide in restriction to $\bar{U}^\iin$, this follows from the same property for $S(\fD_{u,v}^\iin)$, which holds by the explicit description of 
$S(\fD_{u,v}^\iin)$ (see Lemmas \ref{lem_first_scattering}-\ref{lem_varphi_first-scattering}).

We can now treat the induction step. Assume that the result holds for $k$. Let $\fd$ be a ray of $\fD$ with $k<\varphi_{\Init(\fd)}(m_\fd)\leqslant k+1$. Then $\Init(\fd) \in U$ (else $\fd$ would be an initial ray of $S(\fD_{u,v}^\iin)$ with
$\varphi_{\Init(\fd)}(m_\fd)=0$), and so by the previous remark, we have $\varphi_{\Init(\fd)}(m_{\fd'}) \geqslant 1$ for every ray $\fd'$ of $\fD$ containing $\Init(\fd)$. Arguing as in the proof of Proposition \ref{prop_consistent_completion}, it follows that the ray $\fd$ is uniquely determined by the local consistency of $\fD$ around $\Init(\fd)$ from the rays 
$\fd'$ of $\fD$ containing $\Init(\fd)$ and satisfying $\varphi_{\Init(\fd)}(m_{\fd'}) \leqslant k$. The same is true for $S(\fD_{u,v}^\iin)$. By Lemma \ref{lem_increasing}, these rays $\fd'$ satisfy $\varphi_{\Init(\fd')}(m_{\fd'})<
\varphi_{\Init(\fd)}(m_{\fd'})\leqslant k$, and so by the induction hypothesis are identical for $\fD$ and $S(\fD^\iin)$.
\end{proof}

\subsection{Proof of the main result}
\label{subsection_main_result}
In \cref{section_scattering_final} we have defined $S(\fD^\iin_{u,v})$ a scattering diagram 
on $U$ for $(M, \fg_{u,v})$, in a completely algorithmic way: as consistent completion of an explicitly given initial scattering diagram $\fD^\iin_{u,v}$.
In \cref{section_scattering_from_stability} we have defined 
$\fD_{u,v}^{\PP^2}$, a scattering diagram on $U$ for 
$(M, \fg_{u,v})$, in terms of intersection Hodge polynomials 
of moduli spaces of Bridgeland semistable objects in $\D^b(\PP^2)$.
Our main result, stated as Theorem
\ref{main_thm} in the Introduction, is that these two scattering diagrams coincide:

\begin{thm}\label{thm_main_precise}
We have the equality
$\fD^{\PP^2}_{u,v}=S(\fD^\iin_{u,v})$
of scattering diagrams on $U$ for 
$(M, \fg_{u,v})$.
\end{thm}

\begin{proof}
According to Proposition \ref{prop_criterion}, it is enough to show that the scattering diagram $\fD^{\PP^2}_{u,v}$ is consistent and coincides with 
$S(\fD^\iin_{u,v})$ in restriction to 
$\bar{U}^\iin$. 
The scattering diagram 
$\fD_{u,v}^{\PP^2}$ is consistent by Theorem 
\ref{thm_consistency}
and coincides with $S(\fD^\iin_{u,v})$ in 
restriction to 
$\bar{U}^\iin$ by Theorem \ref{thm_coincide_initial}. 
\end{proof}

\begin{thm}\label{thm_main_precise_limit}
We have the equality 
$\fD_{q^-}^{\PP^2}=S(\fD^\iin_{q^-})$
of scattering diagrams on $U$ for 
$(M,\fg_{q^-})$.
Similarly, we have the equality
$\fD_{\cl^-}^{\PP^2}=S(\fD^\iin_{\cl^-})$
of scattering diagrams on $U$ for 
$(M,\fg_{\cl^-})$.
\end{thm}

\begin{proof}
This follows from Theorem \ref{thm_main_precise} by the specialization $u=v=q^{\frac{1}{2}}$ and the semiclassical limit $q^{\frac{1}{2}}
\rightarrow 1$.
\end{proof}

\begin{thm}\label{thm_main_precise_sign}
We have the equality 
$\fD_{q^+}^{\PP^2}=S(\fD^\iin_{q^+})$
of scattering diagrams on $U$ for 
$(M,\fg_{q^+})$.
Similarly, we have the equality
$\fD_{\cl^+}^{\PP^2}=S(\fD^\iin_{\cl^+})$
of scattering diagrams on $U$ for 
$(M,\fg_{\cl^+})$.
\end{thm}

\begin{proof}
It will follow from Theorem \ref{thm_main_precise_limit}. We have to compare the Lie algebras $\fg_{q^-}$ and 
$\fg_{q^+}$. The Lie bracket of $\fg_{q,-}$ is the commutator of the associative algebra 
$A_{q^-}=\bigoplus_{m \in M} \Q(q^{\pm \frac{1}{2}})$, where 
$z^m \cdot z^{m'} \coloneq 
(-1)^{\langle m,m'\rangle}
q^{\frac{\langle m,m'\rangle}{2}}
z^{m+m'}$. Similarly, the 
Lie bracket of $\fg_{q,-}$ is the commutator of the associative algebra 
$A_{q^+}=\bigoplus_{m \in M} \Q(q^{\pm \frac{1}{2}})$, where 
$z^m \cdot z^{m'} \coloneq 
q^{\frac{\langle m,m'\rangle}{2}}
z^{m+m'}$.

We consider $\sigma \colon M=\Z^2 \mapsto \{\pm 1\}$, $(a,b) \mapsto (-1)^{ab+a+b}$.
One checks easily (or see
\cite[Lemma 8.3]{bousseau2018quantum_tropical}) that 
$\sigma$ is a quadratic refinement, in the sense that 
\[\sigma(m+m')=
(-1)^{\langle m,m'\rangle} \sigma(m)\sigma(m')\]
for every $m, m'\in M$.
It follows that $F_\sigma \colon A_{q^-} 
\mapsto A_{q^+}$, 
$z^m \mapsto \sigma(m)z^m$ is an isomorphism of associative algebras.
If $\fD$ is a scattering diagram for 
$\fg_{q^-}$, of rays $\fd=(|\fd|,H_\fd)$, we denote by $F_\sigma(\fD)$ the scattering diagram for $\fg_{q^+}$ of rays 
$\fd=(|\fd|,F_\sigma(H_\fd))$. As $F_\sigma$ is an isomorphism, $F_\sigma(\fD)$ is consistent if and only if $\fD$ is consistent.

From the explicit description of the scattering diagrams 
$\fD^\iin_{q^-}$ and $\fD^\iin_{q^+}$ in Definitions \ref{def_scattering_q_-} and 
\ref{def_scattering_q_+}, we have 
$F_\sigma(\fD^\iin_{q^-})=\fD^\iin_{q^+}$
(indeed $\sigma(\ell m_n^+)
=\sigma(\ell m_n^-)=(-1)^{\ell(n+n+1)}
=(-1)^{\ell}$ for every $n \in \Z$ and 
$\ell \geqslant 1$).
By uniqueness of the consistent completion given by Proposition \ref{prop_consistent_completion}, we deduce that $F_\sigma(S(\fD^\iin_{q^-}))=S(\fD^\iin_{q^+})$. 

On the other hand, we claim that 
$F_\sigma(\fD^{\PP^2}_{q^-})
=\fD^{\PP^2}_{q^+}$. From the explicit description given by Definition
\ref{def_scattering_stability_q},
it is enough to show that, for every 
$\gamma \in \Gamma$, we have 
$(-1)^{(\gamma,\gamma)}=\sigma(m_\gamma)$.
But for $\gamma=(r,d,\chi)$, we have by Definition \ref{def_bilinear_form},
$(\gamma,\gamma)=-3dr-r^2-d^2+2r\chi$, so
$(-1)^{(\gamma,\gamma)}
=(-1)^{dr+r+d}$, and as $m_\gamma=(r,-d)$
by Definition \ref{def_m_gamma}, we also have $\sigma(m_\gamma)=(-1)^{dr+r+d}$.

Thus, we get the equality $\fD_{q^+}^{\PP^2}=S(\fD^\iin_{q^+})$ by applying $F_\sigma$ to the equality 
$\fD_{q^-}^{\PP^2}=S(\fD^\iin_{q^-})$
given by Theorem \ref{thm_main_precise_limit}. The equality 
$\fD_{\cl^+}^{\PP^2}=S(\fD^\iin_{\cl^+})$ 
follows by taking the semiclassical limit
$q^{\frac{1}{2}} \rightarrow 1$.
\end{proof}

\section{Applications to moduli spaces of Gieseker semistable sheaves}
\label{section_application_gieseker}
In \cref{section_algorithm} we show how Theorem \ref{thm_main_precise} can be used to obtain an algorithm
computing intersection Hodge numbers of moduli spaces $M_\gamma$ of Gieseker semistable sheaves.
As an application, we prove Theorem 
\ref{thm_hodge}, stated as Theorem 
\ref{thm_hodge_intro}, that is, 
the fact that these intersection Hodge 
numbers are concentrated in bidegree $(p,p)$. In \cref{section_real_moduli} we prove results on the cohomology of the moduli spaces $M_\gamma$ viewed as real algebraic varieties.
In \cref{section_tree_decomposition} we explain that our scattering algorithm leads naturally to a decomposition indexed by trees of the Poincar\'e polynomial of $M_\gamma$ and we give explicit examples.

\subsection{Algorithm}
\label{section_algorithm}
 
For every $\gamma \in \Gamma$, we denote by
$M_\gamma$ the moduli space of $S$-equivalence classes of Gieseker semistable sheaves on 
$\PP^2$ of class $\gamma$. 
The moduli space $M_\gamma$ is a projective scheme, singular in general, which 
can be constructed by geometric invariant theory, see for example \cite{huybrechts2010geometry}.
The moduli space $M_\gamma^\st$ of Gieseker stable sheaves is a quasiprojective scheme, open in 
$M_\gamma$. It follows directly from Serre duality and from the definition
of Gieseker semistability that $\Ext^2(E,E)=0$ for every Gieseker semistable sheaf $E$
with $\gamma(E) \notin \Gamma^0$, and so the moduli space $M_\gamma^\st$ of stable object is smooth
if $\gamma \notin \Gamma^0$. In fact the moduli space $M_\gamma^\st$ is also smooth if 
$\gamma \in \Gamma^0$, as it is $\PP^2$ if $\gamma=(0,0,1)$, and the empty set else.

As in  \cref{section_intersection_invariants} for $M_\gamma^\sigma$,
we define intersection Hodge numbers $Ih^{p,q}(M_\gamma)$, 
intersection Betti numbers $Ib_k(M_\gamma)$, and intersection Euler characteristics
$Ie^+(M_\gamma)$. If $M_\gamma=M_\gamma^{\st}$, which happens for example if 
$\gamma$ is a primitive element of the lattice $\Gamma$, then $M_\gamma$
is a smooth projective varieties and the intersection Hodge numbers, Betti numbers,
Euler characteristics, are the usual Hodge numbers, Betti numbers, Euler characteristics
of $M_\gamma$.

The following Lemma
gives the relation between the 
moduli spaces $M_\gamma^\sigma$ of Bridgeland-semistable objects 
and the moduli spaces $M_\gamma$ of Gieseker semistable
sheaves. It is a precise version of the general idea that 
the notion of stability in the sense of Gieseker can be recovered as an asymptotic version
of Bridgeland stability condition.

\begin{lem} \label{lem_comparison_bridgeland_gieseker}
We fix $\gamma=(r,d,\chi)\in \Gamma$
such that $r \geqslant 0$.
Then:
\begin{itemize}
\item[(i)] If $r=0$ and $d=0$, we have $M_\gamma^{\sigma}=M_\gamma$
and $M_\gamma^{\sigma-\st}=M_\gamma^\st$ for every $\sigma \in U$.
\item[(ii)] If $r=0$ and $d>0$, we have $M_\gamma^{\sigma}
=M_\gamma$ and $M_\gamma^{\sigma-\st}=M_\gamma^\st$ for $\sigma=(x,y) \in U$ with $y$ large enough.
\item[(iii)] If $r>0$, we have  
$M_\gamma^{\sigma}=M_\gamma$ and $M_\gamma^{\sigma-\st}=M_\gamma^\st$ for 
$\sigma=(x,y) \in U$, 
$x < \frac{d}{r}$ and $y$ large enough.
\end{itemize}
\end{lem}

\begin{proof}
See the proof of \cite[Proposition 6.2]{MR3010070}.
\end{proof}

We now explain how Theorem \ref{thm_main_precise} gives an algorithm to compute the intersection Hodge numbers $Ih^{p,q}(M_\gamma)$.

We fix $\gamma \in \Gamma$. 
If $\gamma =(0,0,1)$, then $M_\gamma=M_\gamma^\st=\PP^2$
so $Ih^{p,q}(M_\gamma)=h^{p,q}(\PP^2)$. If $\gamma \in \Gamma^0$
and $\gamma \neq (0,0,1)$, then $M_\gamma^\st$ is empty so $Ih^{p,q}(M_\gamma)=0$, for every $p$ and $q$.

So we can assume that $\gamma \notin \Gamma^0$. In such cases, it follows from the explicit description of $L_\gamma$ given in \cref{section_scattering_from_stability}
and from Lemma \ref{lem_comparison_bridgeland_gieseker} that
$L_\gamma$ intersects the region in $U$ where $M_{\gamma'}^\sigma=M_{\gamma'}$,
$M_{\gamma'}^{\sigma-\st}=M_{\gamma'}^\st$, and so $Ih^{p,q}(M_{\gamma'}^\sigma)=Ih^{p,q}(M_{\gamma'})$
for every $\gamma' \in \Gamma_\gamma$.
Let $\Omega_\gamma(u^{\frac{1}{2}},v^{\frac{1}{2}})$ defined from the 
polynomials $Ih_{\gamma'}(u^{\frac{1}{2}},v^{\frac{1}{2}})$ as in Definition 
\ref{def_omega}.
If there exists $j \in J_\gamma$ such that $\bar{R}_{\gamma,j}$ is a half-line, then we have  $H_{\fd_\gamma,j}=\Omega_\gamma(u^{\frac{1}{2}},v^{\frac{1}{2}}) z^{m_\gamma}$ (and $\Omega_\gamma(u^{\frac{1}{2}},v^{\frac{1}{2}}) \neq 0$).
If not, this means that $\Omega_\gamma(u^{\frac{1}{2}},v^{\frac{1}{2}})= 0$.
In all cases, one can read off $\Omega_\gamma(u^{\frac{1}{2}},v^{\frac{1}{2}})$
from $\fD_{u,v}^{\PP^2}$.
So, after having determined by induction the $Ih_{\gamma',j}(u^{\frac{1}{2}},
v^{\frac{1}{2}})$ for $\gamma' \in \Gamma_\gamma$, $\gamma' \neq \gamma$, we can read off 
$Ih_\gamma(u^{\frac{1}{2}},v^{\frac{1}{2}})$
from $\fD^{\PP^2}_{u,v}$.
But according to Theorem
\ref{thm_main_precise}, we have 
$\fD_{u,v}^{\PP^2}=S(\fD^\iin_{u,v})$. As 
$S(\fD^\iin_{u,v})$ can be computed in an algorithmic way (see the remark ending
\cref{section_initial}), it is also the case 
for $Ih_\gamma(u^{\frac{1}{2}}, v^{\frac{1}{2}})$.

In order to get an effective algorithm, we have to know, for a given $\gamma \in \Gamma$, how to bound the number of steps necessary in the construction of $S(\fD_{u,v}^{\iin})$
to compute 
$Ih_\gamma(u^{\frac{1}{2}},v^{\frac{1}{2}})$. Li-Zhao \cite{MR3921322} have given an algorithm to compute the actual walls for 
$\gamma$, and from there, we can get a bound on the value $x^2+2y$ for the initial point $(x,y)$ of the half-line $\bar{R}_{\gamma,j}$.
Using Lemma \ref{lem_x_bound}, we get a bound on the $x$-coordinates of the relevant initial rays of $S(\fD^\iin_{u,v})$. Thus, we can algorithmically compute a compact set $K_{\gamma}$ in $\bar{U}$ such that $Ih_\gamma(u^{\frac{1}{2}},v^{\frac{1}{2}})$ is computed by the restriction of 
$S(\fD^\iin_{u,v})$ to $K_\gamma$.
On the other hand, by combination of 
Lemma \ref{lem_scattering_forward} and Lemma \ref{lem_increasing}, we know that 
$\varphi_\sigma(m_\fd)$ is increasing when moving from the initial rays by successive local scatterings. Thus, we can bound $\varphi_\sigma(m_\fd)$ for $\fd$ a ray contributing by successive scatterings to the formation of $\bar{R}_{\gamma,j}$, and for $\sigma \in |\fd|$, by 
$\varphi_{(x,y)}(m_\gamma)$.
By condition (iii) of Definition \ref{def_scattering}, and the description of $S(\fD^\iin_{u,v})$ in restriction to 
$\bar{U}^\iin$ at the beginning of
\cref{section_statement_initial}, the restriction of $S(\fD^\iin_{u,v})$ to
$K_\gamma$ contains finitely many such rays . By Lemma \ref{lem_scattering_forward}, we know when they are exhausted: when all the newly added rays have $\varphi_{\Init(\fd)}(m_\fd)>\varphi_{(x,y)}(m_\gamma)$.

We can now prove that, for every $\gamma \in \Z^3$, the intersection cohomology of $M_\gamma$ is concentrated in Hodge bidegrees $(p,p)$, that is Theorem 
\ref{thm_hodge_intro} of the introduction.

\begin{thm} \label{thm_hodge}
For every $\gamma \in \Gamma$, we have $Ih^{p,q}(M_\gamma)=0$ if $p \neq q$.
\end{thm}

\begin{proof}
As we explained above how to compute 
the intersection Hodge 
polynomials
\[Ih_\gamma(u^{\frac{1}{2}},v^{\frac{1}{2}}) =
(-(uv)^{\frac{1}{2}})^{- \dim M_\gamma}
\sum_{p,q=0}^{\dim M_\gamma}
(-1)^{p+q}
Ih^{p,q}(M_\gamma) u^p v^q \]
from the scattering diagram $\fD_{u,v}^{\PP^2}=S(\fD^\iin_{u,v})$, the result follows from Corollary 
\ref{cor_q_rationality}. In short, the result is true for the initial rays of $\fD^\iin_{u,v}$ (the corresponding moduli spaces are points), and then propagates by wall-crossing.
\end{proof}

As we recalled in the Introduction, Theorem \ref{thm_hodge} is well-known
 \cite{MR1228610, MR1351502, MR2304330} if $\gamma$ is primitive: in this case, semistability coincides with stability,
$M_\gamma$ is smooth, intersection cohomology coincides with ordinary cohomology, and one can show that the cohomology ring is generated by the Künneth components of the Chern classes of the universal sheaf. In general, $M_\gamma$ is singular and Theorem \ref{thm_hodge}
is much less clear. In fact, Theorem 
\ref{thm_hodge} does not seem to appear previously in the 
literature, 
even if, with the extra assumption $r>0$, a different proof can be obtained from the results proved in \cite{MR3874687}.

Theorem \ref{thm_hodge} implies immediately the following Corollary
\ref{cor_hodge}.

\begin{cor}\label{cor_hodge}
For every $\gamma \in \Gamma$, we have:
\begin{itemize}
\item[(i)] $Ib_{2k+1}(M_\gamma)=0$ for every $k \in \NN$.
\item[(ii)] $Ie^+(M_\gamma) \in \NN$, and, if 
$M_\gamma^\st$ is nonempty, $Ie^+(M_\gamma)
\in \Z_{\geqslant 1}$.
\end{itemize}
\end{cor}

\subsection{Real algebraic geometry}
\label{section_real_moduli}

We now discuss an application to the real algebraic geometry of the moduli spaces 
$M_\gamma$.
We equip $\PP^2$ with its natural real structure whose real locus is the real projective plane $\R \PP^2$. As the definition of Gieseker semistable sheaves
makes sense over essentially any base, it follows that for every 
$\gamma \in \Z^3$, the moduli space $M_\gamma$ has a natural real structure and we denote by $M_\gamma(\R)$ its real locus.
In particular, there is a natural action of $\Gal(\C/\R)=\Z/2$ on the intersection cohomology groups $IH^k(M_\gamma,\Q)$.

\begin{thm} \label{thm_real_moduli}
For every $\gamma \in \Gamma$, and for every $0 \leqslant p \leqslant \dim M_\gamma$, the complex conjugation in
$\Gal(\C/\R)$ acts as $(-1)^p$
on $IH^{2p}(M_\gamma,\Q)$.
\end{thm}

\begin{proof}
In all the arguments used in the proof of Theorem \ref{thm_main_precise}, we can replace the intersection Hodge 
polynomials valued in $\Z[u,v]$ by the class of the intersection cohomology in the Grothendieck group of the category of
mixed Hodge structures with action of 
$\Gal(\C/\R)$. In all the explicit formulas defining 
$S(\fD^\iin_{u,v})$, we only have to replace $uv$ by 
$H^2_c(\A^1,\Q)$. The analogue of Theorem \ref{thm_main_precise} is then an equality of scattering diagrams with values in the Grothendieck group of the category of
mixed Hodge structures with action of 
$\Gal(\C/\R)$.  The result then follows from the fact that the complex conjugation in $\Gal(\C/\R)$ acts as $-1$ on $H^2_c(\A^1,\Q)$. In short, the logic is identical to the one used to prove Theorem \ref{thm_hodge}: the result is true for the initial rays of $\fD^{\iin}_{u,v}$
(the corresponding moduli spaces are points), and then propagates by wall-crossing.
\end{proof}

If $\gamma$ is primitive, then
$M_\gamma$ is smooth, so $M_\gamma(\R)$ is a compact manifold. We denote by
$e(M_\gamma(\R))$ its topological Euler
characteristic.

\begin{cor} \label{cor_real_moduli}
For every $\gamma \in \Gamma$ primitive, we have 
\[ e(M_\gamma(\R))
=\sum_{p=0}^{\dim M_\gamma}
(-1)^p b_{2p}(M_\gamma)
=\sum_{p=0}^{\dim M_\gamma}
(-1)^p h^{p,p}(M_\gamma) \,.\]
\end{cor}

\begin{proof}
For $\gamma$ primitive, $M_\gamma$ is smooth, intersection cohomology coincides with usual cohomology, and so the result follows from Theorem \ref{thm_real_moduli} 
by the Lefschetz fixed point formula applied to the complex conjugation.
\end{proof}

It is possible to give a simpler proof of Corollary \ref{cor_real_moduli}.
According to \cite{MR1228610, MR1351502, MR2304330} if $\gamma$ is primitive, then the cohomology ring of $M_\gamma$ is generated by the Künneth components of the Chern classes of the universal sheaf on 
$M_\gamma \times \PP^2$. So the result follows from the fact that complex conjugation acts as $(-1)^p$ on the
$p$-th Chern class of a real sheaf, and as 
$(-1)^p$ on $H^{2p}(\PP^2,\Q)$.

\subsection{Tree decompositions}
\label{section_tree_decomposition}
According to Theorem \ref{thm_hodge}, the study of the intersection Hodge numbers 
$Ih^{p,q}(M_\gamma)$ can be reduced to the study of the intersection Betti numbers 
$Ib_{2k}(M_\gamma)$.
In particular, the intersection Poincar\'e polynomial 
\[ P(M_\gamma)\coloneq 
\sum_{p=0}^{\dim M_\gamma} Ib_{2p}(M_\gamma) q^p\]
has nonnegative coefficients.

The scattering algorithm of \cref{section_algorithm} induces a decomposition 
\[ P(M_\gamma)=\sum_{T \in \mathcal{T}_\gamma} P_T(M_\gamma)\,,\]
where $\mathcal{T}_\gamma$ is a set of oriented weighted trees immersed in $\bar{U}$, and where each 
$P_T(M_\gamma)$ is a polynomial in $q$ with integer coefficients.
Each $T \in \mathcal{T}_\gamma$ is contained in the support of $\fD^{\PP^2}_{u,v}=S(\fD^\iin_{u,v})$, has roots at some of the initial points $s_n$ of $\fD^\iin_{u,v}$, and has a unique unbounded leave coinciding with $\bar{R}_{\gamma,j}$, where $j$ is the unique $j \in J_\gamma$ such that $\bar{R}_{\gamma,j}$ is a half-lines. Each edge of $T$ is weighted by a positive integer and oriented in the direction of increasing value of 
$\varphi_\sigma(m_\fd)$. Each non-root vertex of $T$ has finitely many ingoing edges and one outgoing edge, and the tropical balancing condition is satisfied at each such vertex. The only obstruction to the embedding of $T$ in $\bar{U}$ is the fact that two roots of $T$ can map to the same initial point $s_n$. It is related to the fact that, in the support of $\fD^\iin_{u,v}$, there are two rays coming out from each $s_n$, and $T$ might contain these two rays.

Indeed, such trees naturally index the various terms obtained by successive local scatterings in the algorithmic construction of $S(\fD_{u,v}^\iin)$. The fact that the contribution of each tree is a polynomial in $q$ with integer coefficients follows from the formalism of admissible series of Kontsevich-Soibelman (see \cite[\S 6]{MR2851153}). We conjecture that the coefficients of the polynomials $P_T(M_\gamma)$ are nonnegative, and that  $q^{-(\dim M_\gamma -\deg P_T(M_\gamma))}P_T(M_\gamma)$ are palindromic polynomials with constant term $1$.

Even if our result holds only at the numerical level of Betti numbers, we can think of the various terms indexed by $T$ as coming from various locally closed strata in $M_\gamma$, of codimension $\dim M_\gamma -\deg P_T(M_\gamma)$, parametrizing Gieseker semistable sheaves which are obtained from the line bundles 
$\cO(n)$ by a precise pattern of exact
triangles in $\D^b(\PP^2)$.

This decomposition of $P(M_\gamma)$ seems to be new in general and probably deserves further study. 
We make only one observation.
The moduli spaces of Gieseker semistable sheaves supported in dimension $1$, that is with $\gamma=(0,d,\chi)$, have been quite explicitly studied for low values of $d$
in 
\cite{MR2795234, MR3082879, MR3160600, MR3217411, MR3319919, MR3275289, MR3324766}.
One key aspect of these studies is the construction of explicit locally closed decompositions of the moduli spaces $M_{(0,d,\chi)}$ (for low degree, 
$d \leqslant 6$), whose strata are characterized by the existence of a 
resolution by a direct sum of line bundles of some numerical type. For $d \leqslant 5$, one can check explicitly that 
our decomposition according to the trees $T \in \mathcal{T}_\gamma$
matches the decomposition of $P(M_\gamma)$
induced by these decompositions, and for 
$d=(6,1)$, it matches the slightly more refined decomposition according to the
first destabilizing wall, see \cite{MR3319919}.

A quite simple example: shape of the unique tree $T$ contributing 
to $\gamma=(0,1,1)$ (the precise embedding in $\bar{U}$ can be recovered from the explicit description of the initial rays of $\fD^\iin_{u,v}$ and from the tropical balancing condition):
\begin{center}
\setlength{\unitlength}{1cm}
\begin{picture}(10,3)
\thicklines
\put(2.8,0.2){$s_{-1}$}
\put(5,0.2){$s_0$}
\put(3,0.5){\circle*{0.1}}
\put(5,0.5){\circle*{0.1}}
\put(3,0.5){\vector(1,1){1}}
\put(5,0.5){\vector(-1,1){1}}
\put(4,1.5){\vector(0,1){1}}
\end{picture}
\end{center}
Above the vertex, extensions of $\cO(-1)[1]$ by $\cO$ become stable, corresponding to the fact that all the elements of 
$M_{(0,1,1)}$ are of the form $\cO_l$ for $l$ a line in $\PP^2$, and so admit a resolution of the form 
\[ 0 \rightarrow \cO(-1) \rightarrow \cO \rightarrow \cO_l \rightarrow 0\,.\]
We can check that $P_T(M_{(0,1,1)})
=1+q+q^2$, which is indeed the Poincar\'e polynomial of 
$M_{(0,1,1)} \simeq \PP^2$. More complicated examples are given in the proof of 
Theorem \ref{thm_chi_indep_test} in the following section \S \ref{section_chi_indep}.

The above trees can be viewed as examples of tropical disks in the sense of 
\cite[Definition 5.1]{cps} and are
essentially identical to the attractor trees discussed in a more general and partially conjectural context in
\cite[\S 3.2]{MR3330788}. They are of the same nature as the attractor flow trees of the supergravity literature 
\cite{MR1792870}, except that, in the supergravity context, the trees can have a root at an attractor point, which is a smooth point of the moduli space, analogue of our $U$, whereas all our trees end at a singular point of the moduli space
(the points $s_n$ where the central charge of $\cO(n)$ goes to $0$). The supergravity context should be relevant for compact Calabi-Yau 3-folds, whereas the present paper is about the noncompact Calabi-Yau 3-fold $K_{\PP^2}$.

\subsection{Test of the $\chi$-independence conjecture: proof of Theorem \ref{thm_chi_indep_test}} \label{section_chi_indep}

We prove Theorem \ref{thm_chi_indep_test} that is, for every $d \leq 4$, the intersection 
Poincar\'e polynomial $P(M_{(0,d,\chi)})$ is independent of $\chi$.
In order to save space, we will use Theorem \ref{thm_chi_indep_gcd}, proved 
in \cite{bousseau2019takahashi}, according to which 
$P(M_{(0,d,\chi)})$ only depends on $\chi$ through $\gcd(d,\chi)$. 

For $d=1$ and $d=2$, the moduli space $M_{(0,d,\chi)}$ is isomorphic to the linear system of degree $d$ curves and so does not depend on $\chi$: there is nothing to prove.

For $d=3$, there are two cases to consider, that we can choose to be $\chi=1$
and $\chi=3$, and there is something to prove.
For every $n \in \NN$, we denote $[n]_q:=\sum_{k=0}^{n-1}q^k$. 

\begin{prop}
We have 
\[ P(M_{(0,3,1)})=P(M_{(0,3,3)})=[9]_q [3]_q\,.\]
In particular, Theorem \ref{thm_chi_indep_test} holds for $d=3$.
\end{prop}

\begin{proof}
There is a unique tree contributing to 
$\gamma=(0,3,1)$:
\[T=\]
\begin{center}
\setlength{\unitlength}{1cm}
\begin{picture}(10,4)
\thicklines
\put(4,0.5){\circle*{0.1}}
\put(6,0.5){\circle*{0.1}}
\put(8,0.5){\circle*{0.1}}
\put(4,0.2){$s_{-2}$}
\put(6,0.2){$s_{-1}$}
\put(8,0.2){$s_0$}
\put(4,0.5){\vector(1,1){1}}
\put(6,0.5){\vector(-1,1){1}}
\put(8,0.5){\vector(-1,1){2}}
\put(5,1.5){\vector(1,1){1}}
\put(7,1.5){\vector(-1,1){1}}
\put(6,2.5){\vector(0,1){1}}
\put(4,1){$2$}
\end{picture}
\end{center}
The uniqueness of $T$ is related to the fact that every sheaf $E$ in $M_{(0,3,1)}$
admits a resolution of the form 
\[ 0 \rightarrow \cO(-2)^{\oplus 2} \rightarrow \cO(-1) \oplus \cO \rightarrow E \rightarrow 0\,, \]
see \cite{MR2795234}.
We can check running the scattering algorithm that 
\[ P(M_{(0,3,1)})=P_T(M_{(0,3,1)})=[9]_q [3]_q \,.\]

There are two trees contributing to $\gamma=(0,3,3)$:
\[T_0=\]
\begin{center}
\setlength{\unitlength}{1cm}
\begin{picture}(10,2)
\thicklines
\put(6,0.5){\circle*{0.1}}
\put(8,0.5){\circle*{0.1}}
\put(6,0.2){$s_{-1}$}
\put(8,0.2){$s_0$}
\put(6,0.5){\vector(1,1){1}}
\put(8,0.5){\vector(-1,1){1}}
\put(7,1.5){\vector(0,1){1}}
\put(6.2,1){$3$}
\put(7.6,1){$3$}
\end{picture}
\end{center}

\[T_1=\]
\begin{center}
\setlength{\unitlength}{1cm}
\begin{picture}(10,4)
\thicklines
\put(4,0.5){\circle*{0.1}}
\put(10,0.5){\circle*{0.1}}
\put(4,0.2){$s_{-2}$}
\put(10,0.2){$s_1$}
\put(4,0.5){\vector(1,1){3}}
\put(10,0.5){\vector(-1,1){3}}
\put(7,3.5){\vector(0,1){1}}
\end{picture}
\end{center}

The tree $T_0$ corresponds to the codimension $0$ stratum in 
$M_{(0,3,3)}$ of sheaves $E$ admitting a resolution of the form
\[ 0 \rightarrow \cO(-1)^{\oplus 3} \rightarrow \cO^{\oplus 3} \rightarrow E \rightarrow 0\,,\]
and the tree $T_1$ corresponds to the codimension $1$ stratum in $M_{(0,3,3)}$
of sheaves $E$ admitting a resolution of the form 
\[ 0 \rightarrow \cO(-2) \rightarrow \cO(1) \rightarrow E \rightarrow 0\,,\]
see \cite{MR2795234}. We can check running the scattering algorithm that 
\[P_{T_0}(M_{(0,3,3)})=[9]_q (1+q^2) \,,\]
\[P_{T_1}(M_{(0,3,3)})=[9]_q q\,,\]
and so 
\[P(M_{(0,3,3)})=P_{T_0}(M_{(0,3,3)})+P_{T_1}(M_{(0,3,3)})=[9]_q[3]_q \,.\]
\end{proof}

For $d=4$, there are three cases to consider, that we can chose to be 
$\chi=1$, $\chi=2$, and $\chi=4$.

\begin{prop}
We have 
\[ P(M_{(0,4,1)})=P(M_{(0,4,2)})=P(M_{(0,4,4)})=[12]_q (1+q+4q^2+4q^3+4q^4+q^5+q^6) \,.\] 
In particular, Theorem \ref{thm_chi_indep_test} holds for $d=4$.
\end{prop}

\begin{proof}
There are two trees contributing to 
$\gamma=(0,4,1)$:
\[T_0=\]
\begin{center}
\setlength{\unitlength}{1cm}
\begin{picture}(10,2.8)
\thicklines
\put(4,0.5){\circle*{0.1}}
\put(6,0.5){\circle*{0.1}}
\put(8,0.5){\circle*{0.1}}
\put(4,0.2){$s_{-2}$}
\put(6,0.2){$s_{-1}$}
\put(8,0.2){$s_0$}
\put(4,0.5){\vector(1,1){1}}
\put(6,0.5){\vector(-1,1){1}}
\put(8,0.5){\vector(-1,1){2}}
\put(5,1.5){\vector(1,1){1}}
\put(6,2.5){\vector(0,1){1}}
\put(4,1){$3$}
\put(5.8,1){$2$}
\end{picture}
\end{center}

\[T_1=\]
\begin{center}
\setlength{\unitlength}{1cm}
\begin{picture}(10,4.5)
\thicklines
\put(2,0.5){\circle*{0.1}}
\put(6,0.5){\circle*{0.1}}
\put(8,0.5){\circle*{0.1}}
\put(2,0.2){$s_{-3}$}
\put(6,0.2){$s_{-1}$}
\put(8,0.2){$s_0$}
\put(2,0.5){\vector(1,1){3}}
\put(6,0.5){\vector(1,1){1}}
\put(8,0.5){\vector(-1,1){1}}
\put(7,1.5){\vector(-1,1){2}}
\put(5,3.5){\vector(0,1){1}}
\put(7.5,1.5){$2$}
\end{picture}
\end{center}

Using the notation of  
\cite[\S 3]{MR2795234}, $T_0$ corresponds to the
codimension $0$ stratum $X_0$ in $M_{(0,4,1)}$ of sheaves $E$ admitting a resolution
of the form 
\[ 0 \rightarrow \cO(-3)^{\oplus 3} \rightarrow 
\cO(-1)^{\oplus 2} \oplus \cO \rightarrow E \rightarrow 0\,,\]
and $T_1$ corresponds to the codimension $2$ stratum $X_1$ in $M_{(0,4,1)}$ of sheaves $E$ admitting a resolution
of the form 
\[ 0 \rightarrow \cO(-3) \oplus \cO(-1) \rightarrow \cO^{\oplus 2} \rightarrow E \rightarrow 0\,.\]
We can check running the scattering algorithm that 
\[P_{T_0}(M_{(0,4,1)})=[12]_q (1+q+3q^{2}+3q^3+3q^4+q^5+q^6) \,,\]
\[P_{T_1}(M_{(0,4,1)})=[12]_q[3]_q  q^2\,,\]
and so 
\[P(M_{(0,4,1)})=P_{T_0}(M_{(0,4,1)})+P_{T_1}(M_{(0,4,2)})=[12]_q
 (1+q+4q^2+4q^3+4q^4+q^5+q^6)\,.\]
The moduli space $M_{(0,4,1)}$ is smooth and its Poincar\'e polynomial has been previously computed by torus localization in \cite{MR3217411}. 

There are three trees contributing to 
$\gamma=(0,4,2)$:
\[T_0=\]
\begin{center}
\setlength{\unitlength}{1cm}
\begin{picture}(10,4)
\thicklines
\put(4,0.5){\circle*{0.1}}
\put(4,0.2){$s_{-2}$}
\put(8,0.2){$s_0$}
\put(4,0.5){\vector(1,1){2}}
\put(8,0.5){\vector(-1,1){2}}
\put(6,2.5){\vector(0,1){1}}
\put(4.8,1.8){$2$}
\put(7.2,1.8){$2$}
\end{picture}
\end{center}

\[T_1=\]
\begin{center}
\setlength{\unitlength}{1cm}
\begin{picture}(10,3)
\thicklines
\put(4,0.5){\circle*{0.1}}
\put(6,0.5){\circle*{0.1}}
\put(8,0.5){\circle*{0.1}}
\put(4,0.2){$s_{-2}$}
\put(6,0.2){$s_{-1}$}
\put(8,0.2){$s_0$}
\put(4,0.5){\vector(1,1){1}}
\put(6,0.5){\vector(-1,1){1}}
\put(6,0.5){\vector(1,1){1}}
\put(8,0.5){\vector(-1,1){1}}
\put(5,1.5){\vector(1,1){1}}
\put(7,1.5){\vector(-1,1){1}}
\put(6,2.5){\vector(0,1){1}}
\put(8,1){$2$}
\put(4,1){$2$}
\end{picture}
\end{center}

\[T_2=\]
\begin{center}
\setlength{\unitlength}{1cm}
\begin{picture}(10,5)
\thicklines
\put(2,0.5){\circle*{0.1}}
\put(10,0.5){\circle*{0.1}}
\put(2,0.2){$s_{-3}$}
\put(10,0.2){$s_1$}
\put(2,0.5){\vector(1,1){4}}
\put(10,0.5){\vector(-1,1){4}}
\put(6,4.5){\vector(0,1){1}}
\end{picture}
\end{center}

Using the notation of  
\cite[\S 4]{MR2795234}, $T_0$ corresponds to the
codimension $0$ stratum $X_0$ in $M_{(0,4,2)}$ of sheaves $E$ admitting a resolution
of the form 
\[ 0 \rightarrow \cO(-2)^{\oplus 2}
\rightarrow \cO^{\oplus 2} 
\rightarrow E \rightarrow 0\,,\]
 $T_1$ corresponds to the codimension $1$ stratum $X_1$ in $M_{(0,4,2)}$ of sheaves $E$ admitting a resolution
of the form 
\[ 0 \rightarrow \cO(-2)^{\oplus 2}\oplus 
\cO(-1)
\rightarrow \cO(-1) \oplus \cO^{\oplus 2} 
\rightarrow E \rightarrow 0\,,\]
and $T_2$ corresponds to the codimension $3$ stratum $X_2$ in $M_{(0,4,2)}$ of sheaves $E$ admitting a resolution
of the form 
\[ 0 \rightarrow \cO(-3)
\rightarrow \cO(1) 
\rightarrow E \rightarrow 0\,,\]
We can check running the scattering algorithm that 
\[ P_{T_0}(M_{(0,4,2)})=[12]_q(1+2q^2+2q^4+q^6)\,,\]
\[ P_{T_1}(M_{(0,4,2)})=[12]_q q(1+2q+3q^2+2q^3+q^4)\,,\]
\[P_{T_2}(M_{(0,4,2)})=[12]_q q^3\,,\]
and so 
\[ P(M_{(0,4,2)})=[12]_q (1+q+4q^2+4q^3+4q^4+q^5+q^6)\,.\]

There are two trees contributing to $\gamma=(0,4,4)$:
\[T_0=\]
\begin{center}
\setlength{\unitlength}{1cm}
\begin{picture}(10,2)
\thicklines
\put(6,0.5){\circle*{0.1}}
\put(8,0.5){\circle*{0.1}}
\put(6,0.2){$s_{-1}$}
\put(8,0.2){$s_0$}
\put(6,0.5){\vector(1,1){1}}
\put(8,0.5){\vector(-1,1){1}}
\put(7,1.5){\vector(0,1){1}}
\put(6.2,1){$4$}
\put(7.6,1){$4$}
\end{picture}
\end{center}

\[T_1=\]
\begin{center}
\setlength{\unitlength}{1cm}
\begin{picture}(10,4)
\thicklines
\put(4,0.5){\circle*{0.1}}
\put(6,0.5){\circle*{0.1}}
\put(8,0.5){\circle*{0.1}}
\put(10,0.5){\circle*{0.1}}
\put(4,0.2){$s_{-2}$}
\put(6,0.2){$s_{-1}$}
\put(8,0.2){$s_0$}
\put(10,0.2){$s_1$}
\put(4,0.5){\vector(1,1){3}}
\put(6,0.5){\vector(1,1){1}}
\put(8,0.5){\vector(-1,1){1}}
\put(7,1.5){\vector(1,1){1}}
\put(7,3.5){\vector(0,1){1}}
\put(10,0.5){\vector(-1,1){2}}
\put(8,2.5){\vector(-1,1){1}}
\end{picture}
\end{center}
Using the notation of  
\cite[\S 5]{MR2795234}, $T_0$ corresponds to the
codimension $0$ stratum $X_0$ in $M_{(0,4,4)}$ of sheaves $E$ admitting a resolution
of the form 
\[ 0 \rightarrow \cO(-1)^{\oplus 4} \rightarrow \cO^{\oplus 4} \rightarrow E \rightarrow 0 \,,\]
and
$T_1$ corresponds to the
codimension $1$ stratum $X_1$ in $M_{(0,4,4)}$ of sheaves $E$ admitting a resolution
of the form 
\[ 0 \rightarrow  \cO(-2) \oplus \cO(-1) \rightarrow \cO \oplus \cO(1)
\rightarrow E \rightarrow 0 \,.\]

We can check running the scattering algorithm that 
\[ P_{T_0}(M_{(0,4,4)})=[12]_q(1+2q^2+q^3+2q^4+q^6)\,,\]
\[ P_{T_1}(M_{(0,4,4)})=[12]_q[3]_q^2 q\,,\]
and so 
\[ P(M_{(0,4,4)})=[12]_q (1+q+4q^2+4q^3+4q^4+q^5+q^6)\,.\]
\end{proof}

\vspace{+8 pt}
\noindent
Institute for Theoretical Studies \\
ETH Zurich \\
8092 Zurich, Switzerland \\
pboussea@ethz.ch

\end{document}